\documentclass{article}
\usepackage[utf8]{inputenc}
\usepackage{amsmath}
\usepackage{amsthm}
\usepackage{amsfonts}
\usepackage{amssymb}
\usepackage{tikz}
\usepackage{quiver}
\usepackage{bussproofs}
\usepackage[colorlinks]{hyperref}
\usepackage{dsfont}
\usepackage{bbm}

\theoremstyle{plain} 
\newtheorem{thm}{Theorem}[subsection]
\newtheorem{lem}[thm]{Lemma}

\newtheorem{cor}[thm]{Corollary}
\theoremstyle{definition}
\newtheorem{dfn}[thm]{Definition}
\newtheorem{exam}[thm]{Example}
\newtheorem{rem}[thm]{Remark}
\newtheorem{nota}[thm]{Notation}

\newtheorem{phil}[thm]{Philosophical Comment}

\def\WCN{\mathsf{\mathsf{WC}\hspace{-3pt}-\hspace{-2pt}\mathsf{N}}}
\def\UWCN{\mathsf{\mathsf{UWC}\hspace{-3pt}-\hspace{-2pt}\mathsf{N}}}

\title{An Introduction to Categorical Proof Theory}
\author{Amirhossein Akbar Tabatabai\\
Bernoulli Institute, University of Groningen}
\date{ }

\begin{document}

\maketitle

\tableofcontents

\newpage

\section{A Historical Introduction} \label{SecHistory}

\begin{quote}
``\emph{It is equally stupid and simple to consider mathematics to be just an axiom system as it is to see a tree as nothing but a quantity of planks} \cite{vanStigt}." \;\; \textsc{L.E.J. Brouwer}
\end{quote}

Proof theory was the child of a great panic. Attacked by a wild wave of paradoxes, the foundation of mathematics was in great danger, and to save it, Hilbert came up with a formalistic program to reduce mathematics (or at least its consistency) to its trusted finitary core. His idea was to first formalize mathematics as a purely syntactical theory and then show the \emph{consistency} of that theory using what he vaguely called the \emph{finitary methods}. This way, he hoped to bring peace back to the shaking foundation. From the very beginning, though, it was quite clear that for most interesting theories (e.g. Peano arithmetic), the usual strategy of proving consistency by model construction was not an option, simply because the models of such theories were usually infinite and hence beyond our trusted core. Therefore, the only remaining option was to work directly with the syntactical proofs, the \emph{finitary} and tame objects that mathematicians use every day to communicate. As expected, Hilbert's strategy was to investigate the structure of \emph{all} possible proofs of a proposition in a purely combinatorial way to hopefully rule out the contradiction as a provable proposition. Even more ambitious, through a pure formal investigation of the possible proofs of a theorem, he hoped to be able to prove the \emph{conservativity} of mathematics over its finitary core by eliminating all possible ideal infinitary entities as helpful auxiliary tools that we use for our convenience, yet we know they have no real effect.

As G\"{o}del showed in 1931, the program, as stated above, fails dramatically. In fact, there is no way to prove the consistency of mathematics (not even finitary mathematics) within finitary mathematics itself. This failure to prove consistency also shows the failure to prove conservativity, as the consistency of finitary mathematics is a simple finitary claim about the non-existence of a proof as a finitary object. This consistency statement, however, is provable in mathematics using the axiom of the existence of an infinite set, while it is unprovable in finitary mathematics based on what G\"{o}del proved. As disappointing as it seems, the story, of course, has not ended with this huge negative impact of G\"{o}del's theorems, and we are going to explain how. But before going any further, let us pause for a moment to consider the role of proofs in Hilbert's proof theory.

\subsection{Proofs as Auxiliary Objects}

As is clear even from our brief explanation of Hilbert's program, Hilbert employs proofs as a \emph{tool} to learn something about the theorems of a given theory. For him, what has the central role in proof theory is not the proofs but the \emph{provability}. Consistency is the unprovability of contradiction, and conservativity means that whatever is provable in a bigger theory is also provable in the smaller one. Even in many formalizations of constructive mathematics, where the constructions— and hence the proofs— must have the central role, people usually follow Hilbert-style proof theory and focus mostly on provability rather than on proofs. Such an interest in the shadows of an entity rather than the entity itself, such a move from the intensional to the extensional, from a notion to its extension, has similar incarnations in many mathematical fields. For instance, in the early days of the theory of computation, the central role was given to computability, or to a lesser extent to the functions that algorithms compute, rather than to the algorithms themselves. As is clear today, the crown of the central notion in the theory of computation must be awarded to computation itself, and not to mere computability. In a similar fashion, one can argue that proof theory must be elevated to a theory about proofs rather than merely provability.

To illustrate how auxiliary proofs are even within their own home of proof theory, let us consider the most basic questions a mathematician might ask herself when encountering a new mathematical concept: namely, what is the \emph{definition} of the concept, and what is its corresponding \emph{identity}? For the first question, the answer seems straightforward. A proof is a \emph{syntactical} object constructed within a syntactical calculus, starting from certain axioms and following certain rules. This definition aligns with the aims of Hilbert's program, as it provides a sound and complete definition for the program's main interest: provability. However, it is fair to say that this definition is more of a practical first draft than a polished final answer. Here are two reasons to support this claim.

The first concern is the role of syntax in the definition. A proof, regardless of its nature, must be a \emph{formal} entity distant from any sort of content. This is a philosophically acceptable assumption. However, the definition provided does not encompass form in its general sense. Instead, it relies on a specific kind of syntax as a substitute, and this reduction introduces its own problems.
In Hilbert's program, equating the formal with the syntactical serves a technical purpose, as it allows for the use of combinatorial techniques to investigate the structure of proofs. However, one can argue that form is a much richer notion than any type of syntax can capture, and even if it could, the result would not meet modern standards of mathematics. To illustrate this, let us consider an abstract mathematical entity, such as a ring. A ring is indeed constructed from elements, but what a ring theorist is truly interested in is the abstract structure of the ring beyond the nature of its elements. This structure is what can be referred to as the form of the ring. 

Is it possible to reduce this form to a kind of syntax? For rings—and let us emphasize that this is not necessarily true in general—the answer is affirmative. In fact, one can observe this possibility by examining the nineteenth-century approach to ``ring theory", which employed polynomials as its syntax. The theory begins with a given set of variables as the generic, meaningless generators of the ring and then imposes certain equations as the expected relationships among the resulting polynomials. Although it is possible to develop ring theory in this manner, it is clear that the form of rings extends beyond their polynomial representations. Mathematically, it is fair to assert that the significant advances in ring theory during the twentieth century owe their success to liberation from such polynomial presentations or any other form of presentation, for that matter. Taking this analogy seriously, one might even contend that the overly tight interdependence of proofs and syntax is a primary reason behind the relative narrowness of proof theory:
\begin{quote}
\emph{There ought to be a
vibrant specialty of ``proof theory". There is a subject with this title,
started by David Hilbert in his attempt to employ finitistic methods
to prove the correctness of classical mathematics... In
1957, at a famous conference in Ithaca, proof theory was recognized
as one of the four pillars of mathematical logic (along with model
theory, recursion theory and set theory). But the resulting proof
theory is far too narrow to be an adequate pillar...} \cite{MacLane}
\end{quote}
The form-to-syntax reduction is not merely an abstract philosophical problem; it also carries practical implications. For instance, consider the immaterial yet necessary choices that any use of syntax compels us to make. Is a proof a sequence of formulas in a given Hilbert-style system, or is it a tree of formulas in the sense of natural deduction? There are numerous proof systems that employ various syntactical choices, making it difficult to assert that any single system is the true representation of real abstract formal proofs. 
Most likely, all these proof systems offer different presentations of a single abstract notion, akin to the various presentations of a ring through different generators and relations. The ring exists beyond its presentations, and so does the proof

The second problem with the conventional definition of a proof is its \emph{reductive} nature. It characterizes a proof as a sequence or a tree of propositions, thereby reducing proofs to formulas. This reductive approach is analogous to the usual interpretation of an interval as a collection of points, a function as a specific type of relation, and movement as a sequence of snapshots. From a philosophical standpoint, it is far from trivial to believe whether the former can be faithfully reduced to the latter. Even mathematically speaking, one can argue that, even if this reduction is indeed possible, the outcome may not be desirable, as it could restrict the independence of the former and thus inhibit some of its alternative yet interesting interpretations.

For instance, consider the following three examples. First, we have Brouwer's refusal to reduce the continuum to its points, which unearthed alternative notions of real numbers that later reemerged as natural constructions in classical geometry via Grothendieck toposes \cite{MacLaneMoerjdijk}. These alternative notions compel all functions over the reals to be continuous or even smooth, as a geometer might expect. They also allow for the existence of a non-zero real number $\epsilon$ satisfying $\epsilon^2 = 0$, making the development of differential geometry easier and even synthetically possible  \cite{SDG}. 

Secondly, consider the point-free approach to real numbers, where subsets are not merely collections of points but specific equivalence relations over the open sets of the reals. This shift in perspective provides more refined subsets and addresses some anomalies, such as the existence of non-measurable subsets of real numbers \cite{LocaleSimpson}. 

Finally, consider the conceptual leap in modern mathematics that places homomorphisms on equal footing with their algebraic structures. This leap has enabled the elegant progress of modern mathematics. Taking the analogy in these three examples seriously, one can argue that a proof, whatever it may be, is more than a mere collection of propositions. It must be an independent notion that interacts with the propositions but is not necessarily reducible to them. In short, it is a tree, not a quantity of planks.

So far, we have examined several arguments against the conventional definition of a proof, illustrating the satisfaction of Hilbert's proof theory with a minimal working definition of proofs as mere auxiliary tools within the theory. Now, let us continue this line of reasoning by addressing the second significant yet overlooked problem: the identity problem, which has two facets: \emph{proof equivalence} and \emph{propositional identity}. As noted, Hilbert's proof theory has largely ignored these issues. The reason is straightforward: these identities were unrelated to the provability that the program was originally concerned with. Consequently, the program had no impetus to investigate these problems seriously. This lack of interest further underscores the auxiliary role that proofs play in Hilbert's proof theory. Where else in mathematics, the central notions of a theory do not even have their own notion of identity?

As a concluding remark in this subsection, let us briefly elucidate the two identity problems mentioned and their interrelationship. The first is the \emph{equivalence problem} for proofs. This problem is profound and intriguing, connecting programming languages on one side and \emph{coherence problems} in category theory and homotopy theory \cite{Identity} on the other. As we will explain later, proof equivalence is typically founded on \emph{redundancy elimination}. For instance, consider the following two derivations:
\begin{center}
	\begin{tabular}{c c c}
	    \AxiomC{$ $}
	    \noLine
		\UnaryInfC{$\mathsf{D}_1$}
		\noLine
		\UnaryInfC{$A$}
		 \AxiomC{$ $}
	    \noLine
		\UnaryInfC{$\mathsf{D}_2$}
		\noLine
		\UnaryInfC{$B$}
	    \RightLabel{\footnotesize$\wedge I$} 
		\BinaryInfC{$A \wedge B$}
		\RightLabel{\footnotesize$\wedge E_1$} 
		\UnaryInfC{$A$}
		\DisplayProof \hspace{20pt}
		&
		\AxiomC{$ $}
	    \noLine
		\UnaryInfC{$\mathsf{D}_1$}
		\noLine
		\UnaryInfC{$A$}
		\DisplayProof
	\end{tabular}
\end{center}
One might argue that these derivations are equivalent, as the left one is merely the derivation $\mathsf{D}_1$ with a redundant step that does not alter the essence of the proof. However, when redundancies are absent, the meaning of equivalence becomes unclear. For instance, consider the following two derivations:
\begin{center}
	\begin{tabular}{c c c}
	    \AxiomC{$A \wedge A$}
	    \RightLabel{\footnotesize$\wedge_1 E$} 
		\UnaryInfC{$A$}
		\DisplayProof \hspace{10pt}
		&
		\AxiomC{$A \wedge A$}
	    \RightLabel{\footnotesize$\wedge_2 E$} 
		\UnaryInfC{$A$}
		\DisplayProof
		
	\end{tabular}
\end{center}
It is not intuitively clear whether they are equivalent in any meaningful sense. Additionally, the problem becomes even more complex when we are tasked with comparing two proofs that exist within different systems.

For the identity problem concerning propositions, there is fortunately an intuitive definition: two propositions are considered identical if and only if they hold the same meaning, allowing one to be used in place of the other in any context. The challenge lies in formalizing this \emph{propositional identity} in a formal and abstract manner without delving into the specifics of their meanings.

The most straightforward formalization that comes to mind is the logical equivalence relation, which asserts that two propositions logically imply each other. In other words, there exists a proof $\mathsf{D}_1$ of $B$ from $A$ and a proof $\mathsf{D}_2$ of $A$ from $B$. However, this formalization proves to be inadequate. To illustrate this, let us represent the propositions $A$ and $B$ as sets of their respective proofs. In this context, logical equivalence merely indicates the existence of a function mapping from the set of proofs for $A$ to the set of proofs for $B$, and vice versa. 
If $A$ and $B$ are truly identical propositions, one would expect their associated sets of proofs to also be identical. This necessitates that there is a bijection between the two sets, which can be induced by composing the two proofs, $\mathsf{D}_1$ and $\mathsf{D}_2$. Therefore, we require these proofs to be inverses of each other, such that $\mathsf{D}_1 \circ \mathsf{D}_2 =id$ and $\mathsf{D}_2 \circ \mathsf{D}_1 = id$, where $\circ$ denotes the composition of the two proofs and $id$ represents the trivial proof of a proposition by itself. 

It is evident that the formalization of propositional identity necessitates a formalization of equivalence between the proofs themselves. Furthermore, it is important to note that even when our primary concern is the propositions rather than the proofs or their provability, certain purely propositional issues may still require an examination of proofs and their equivalence.

\subsection{Gentzen and his Consistency Proof}

In the opening of this section, we explained how Hilbert's proof theory emerged as a natural response to the foundational crisis of the early twentieth century, and we deviated from the narrative to emphasize how auxiliary the notion of a proof is in Hilbert's proof theory. Now, it is time to return to the rest of the story. We stated that G\"{o}del demonstrated the impossibility of proving the consistency of mathematics (even finitary mathematics) within finitary mathematics itself. This is actually not entirely accurate. The formal statement of G\"{o}del's incompleteness theorem shows that no recursively axiomatizable theory that is strong enough can prove its own consistency, and this theorem applies to Hilbert's program if we assume that his finitary mathematics is recursively axiomatizable. Although such an assumption is quite reasonable, it can also be argued that there is no way to formalize all of finitary mathematics but only its proper parts. Therefore, for any consistent theory $T$, it may still be possible to find a strong enough finitary theory $F_T$ that is sufficiently powerful to prove the consistency of $T$. The only thing that G\"{o}del's incompleteness theorem establishes is that this finitary theory $F_T$ cannot be a subtheory of the theory $T$. 
This is precisely how Gentzen began to investigate the situation. For him, the theory $T$ was Peano Arithmetic, $\mathsf{PA}$, for which he identified the corresponding finitary theory $F_T$ as the theory $\mathsf{PRA}$, a very basic quantifier-free theory, along with quantifier-free transfinite induction up to the ordinal $\epsilon_0$. One might object that such a theory cannot be considered finitary, as it operates with infinite ordinals. However, it is important to recognize that what we are actually working with is merely a finitary representation of the ordinals below $\epsilon_0$ as complex combinatorial objects, and the set-theoretical language of ordinals is employed solely to aid our understanding of this intricate combinatorial structure.

What we are truly interested in here is not Gentzen's proof of the consistency of arithmetic, but rather his strategy. We observed that to prove the consistency of a theory, one must investigate all possible proofs of a contradiction in a purely combinatorial manner. However, it should have been quite clear from the very beginning that proofs, as combinatorial objects, are extremely complex and challenging to analyze. To understand why, we do not need to delve into any specific proof theory; it is sufficient to consider how mathematicians grapple with their conjectures every day. If it were easy to determine what is provable and what is not, any conjecture could be resolved with a sufficient amount of combinatorial effort in proof theory.

However, this observation does not imply that there is absolutely no hope in investigating the proofs of simple statements in natural theories, such as the contradiction in $\mathsf{PA}$. To systematically investigate all possible proofs, Gentzen devised an elegant strategy. The first step involved developing a well-behaved and easily controllable proof system. In fact, he created two elegant proof systems known as \emph{natural deduction} and \emph{sequent calculus}. The primary characteristic of these systems, aside from their natural usability, is the capacity to transform their proofs into simpler forms in a certain formal sense. This process is referred to as \emph{normalization} in \emph{natural deduction} and \emph{cut elimination} in \emph{sequent calculus}. By achieving the simplest possible form of a proof, one may gain the insight needed to discern its structure; using that structure, one can demonstrate that the contradiction does not have a simple proof, and thus not any proof, within the system. It is even possible to employ this framework to prove certain conservativity results by showing that the simplest proofs in a stronger theory actually reside within the weaker one.
We hope the reader can recognize the irony here. To accomplish the primary objective of proof theory, namely proving consistency, even though we are primarily interested in provability, we are mathematically compelled to be meticulous about the formalization we choose for proofs, the structure of those proofs within that formalization, and the proof transformations that serve as witnesses for proof equivalence. In short, understanding the full structure of the proofs is essential for comprehending the seemingly simpler concept of provability!

Inspired by Gentzen's elegant work, interest in proof theory flourished. However, it primarily followed Gentzen's footsteps in investigating the form of proofs to gain insights into the associated theories, ranging from their consistency to the characterization of the functions that the theory can define. It also inherited the Hilbertian attitude, wherein proofs remained syntactical, mostly confined to one formalization and serving the theorems they prove. It is true that proof theorists became increasingly interested in the structure of proofs, but this was largely because they were compelled to do so.

\subsection{Categorical Proof Theory}

Gentzen's elegant work also gave birth to an alternative type of proof theory—a theory in which proofs play the main role, elevating their status from the planks of provability to a truly alive tree of proofs. The first such theory is Dag Prawitz's \emph{general proof theory}, a term he coined to distinguish his approach from what he termed \emph{reductive proof theory} \cite{IdeaPrawitz,GenPrawitz}.
The latter essentially represents the Hilbertian proof theory, using proofs as tools to reduce one theory to another to demonstrate consistency and conservativity. In contrast, general proof theory shifts the focus from provable propositions to the proofs themselves. Here, proofs are treated as elements of an \emph{algebraic structure} where the rules function as its operations. The study of this structure is the primary goal of general proof theory.

However, this algebraic structure is presented to us in an indirect manner. The situation resembles the style of nineteenth-century matrix theory. While matrices can be defined as certain algebraic entities with corresponding operations, these arrays of numbers actually represent the linear transformations that are of primary interest. The challenge lies in the fact that a single linear transformation can be represented by multiple matrices, depending on the basis used for the underlying vector spaces. Since we are interested in the transformations rather than their representations, our initial task is to determine when two matrices represent the same transformation, or in other words, when they are \emph{equivalent}. A similar scenario occurs in general proof theory. Proofs are represented within a proof system, specifically natural deduction, and the main challenge is to find an appropriate notion of equivalence between them. 

Prawitz's solution was intriguing; he utilized proof transformations and demonstrated that any proof can be simplified to a unique simplest possible form known as the \emph{normal form}. Thus, two proofs can be considered equivalent if they have the same normal form. This can be likened to algebraic simplifications. For example, two polynomials, $x(x+1)+x$ and $x^2+x+x$, are deemed equivalent when their simplest forms, namely $x^2+2x$, are equal. General proof theory, however, extends beyond this equivalence problem. It also seeks to understand the meanings of logical constants based on their roles within proofs. Nonetheless, it is fair to assert that the equivalence problem is central to the theory.

As intriguing as Prawitz's general proof theory is, it comes with its own limitations. It conflates form with syntax, reduces proofs to propositions, and confines itself to a single fixed system—namely, natural deduction.\footnote{It is important to clarify that general proof theory does not entirely restrict itself to syntax. Prawitz explicitly differentiates proofs from their syntactical representations. Moreover, for him, natural deduction is not merely one syntactical system among many; rather, it serves as a canonical framework for elucidating how proofs behave and what the logical constants signify. One can easily discern the early influences of the later categorical approach within these ostensibly syntactical aspects of general proof theory. That said, it is also notable that during the normalization process, proofs are identified through the syntactical derivations of propositions, making the process an algebraic simplification of syntax, where the aforementioned drawbacks emerge.} At this juncture, a reader might argue that while the notion of liberation from syntax and reduction to propositions is appealing, it remains unclear whether a non-reductionist definition of proofs is even feasible. A similar skepticism can be applied to linear transformations: if we aim to avoid representations, how can we discuss linear transformations directly? The answer lies in understanding the two orthogonal mindsets typically involved in defining a mathematical notion. 

The \emph{internal} mindset favors constructing new mathematical entities from existing ones, using the latter as building blocks. However, this approach can inadvertently incorporate unintended aspects, such as construction methods, into the new mathematical entity. To eliminate these extraneous elements, one must introduce an equivalence relation to collapse them, working with the entity only up to this relation to maintain statements about the entity rather than its representation. The process of constructing proofs from propositions and deriving linear transformations from real numbers exemplifies this internal mindset.
Dually, the \emph{external} mindset focuses on defining a mathematical entity based on its \emph{relative} behavior rather than on its constituents or construction. The modern approach to linear algebra is a prime example of this external mindset, as it does not define vectors and linear transformations in terms of arrays of real numbers; instead, it characterizes them as abstract objects that behave as expected within their respective contexts.

Modern mathematics owes its external mindset to Hilbert's strong emphasis on form, positing that mathematics is concerned with form rather than content. However, Hilbert's proof theory is more of an internal endeavor; it seeks to define a proof using a more fundamental notion, which, in this context, is the proposition. To modernize Hilbert's proof theory in alignment with his vision, we must axiomatize the algebraic structure of a proof system, where both propositions and proofs reside. To achieve this, we first need a natural mathematical candidate for such two-sorted structures, which is identified as a category. 

Categories are algebraic structures comprised of two largely independent entities: objects and morphisms that connect these objects. One can conceptualize a proof system as a category in which the objects represent propositions and the morphisms signify the deductions between them. This approach leads to an attempt to find an axiomatization of proof systems framed as \emph{structured categories}.
This perspective on proof theory is termed \emph{categorical proof theory}. Inspired by Lawvere's thesis \cite{Lawverethesis}, the theory began with Lambek's comprehensive exploration of proof systems as presented in \cite{LamI,LamII,LamIII}, and was further developed in \cite{LambekScott,Szabo}. For a succinct overview, see \cite{PrologueDosen,Identity}.

Interestingly, the abstract structure of proof systems uncovered by categorical proof theory is not confined strictly to proof theory; it is a well-recognized construct across various branches of mathematics known as the \emph{cartesian closed structure} and  \emph{bicartesian closed structure}, depending on connectives the language uses. Leveraging the richness of this structure allows one to hypothesize that many mathematical frameworks, ranging from computable functions to continuous maps, can be viewed as different forms of proof systems. By utilizing these mathematical structures as semantic interpretations of conventional syntactical proofs, we can gain intriguing insights into both the nature of proofs and the concept of provability.

As a concluding remark to this part, we will summarize the potential solutions that categorical proof theory offers for the problems previously discussed. Firstly, regarding the issue of form-to-syntax reduction, proof systems within categorical proof theory are viewed as algebraic structures, where their essence is captured in the abstract structure they embody. This perspective parallels how a ring can be regarded as a formal object independent of any syntactical representation.
Consequently, in categorical proof theory, the specific representations of proof systems through various syntaxes become inconsequential. Two proof systems are deemed equivalent if their underlying structures align. Turning to the issue of reducing proofs to propositions, it is noteworthy that in a category, morphisms cannot be reduced to objects; they maintain an independence analogous to the relationship between deductions and proofs.
In addressing the equivalence problem, it is important to recognize that a category inherently possesses its own notion of equality among morphisms. Thus, in categorical proof theory, there is no need to define the equivalence of proofs, as this concept emerges as a fundamental notion. However, this does not permit arbitrary definitions, as we expect our category to exhibit certain desirable properties. Finally, the identity of propositions is effectively captured by the established notion of isomorphic objects within categories.

\subsection{Type Theory}

Parallel to the categorical approach to proof theory, there is a type-theoretical approach that is more syntactic and thus closer to the usual practices of proof theory. Despite the primary focus of logical theories on formulas, the language of type theory incorporates two fundamental notions: types and terms. Over this syntax, a type theory consists of a series of type and term formers, along with rules that describe the expected behavior of these formers.

Type theories have many informal interpretations. Types and terms can be understood as propositions and proofs, sets and elements, or data types and programs. This versatility arises from an unexpected similarity between these concepts, known as the \emph{Curry-Howard correspondence} or \emph{propositions as types} dictum \cite{howard,LecCurryHow}. For example,  the simplest type theory is the typed lambda calculus, designed to capture the behavior of sets and functions in the presence of products and function spaces \cite{LambekScott}. This type theory can be extended with sums as disjoint unions, resulting in a syntactical system for sets as well as propositional deductions. Additionally, it is possible to extend typed lambda calculus to dependent type theories, where types can depend on the elements of other types. This allows us to formalize families of sets, on the one hand and first-order deductions, on the other.

Type theories are typically motivated by their role in formalizing data types and programs rather than proofs and propositions. However, some type theories have foundational or even proof-theoretical motivations. For instance, the well-known Martin L\"{o}f type theory \cite{martinlof} is a fundamental dependent type theory influenced by Prawitz's work on the equivalence of proofs via normalization, aiming to formalize what constructions—and, consequently, proofs—mean in constructive mathematics. Due to this connection with data types, type theory garners significant interest from the theoretical computer science community, and it is fair to say that most developments in both categorical and type-theoretical approaches to proofs are inspired by these interests. 

Finally, to connect categorical proof theory to type theory, one can interpret the former as the categorical semantics for the latter. For example, cartesian closed categories provide a semantics for the typed lambda calculus. In this chapter, we will not cover type theories; for more information on type theories, especially their connection with categories, see \cite{Crole,Bart,Taylor,HofmannSeman,LambekScott}.

\subsection{Landmarks in Categorical Proof Theory}

As mentioned before, categorical proof theory essentially began with Lambek \cite{LamI,LamII,LamIII}, who sought to formalize \emph{propositional} deductive systems as structured categories. He starts with substructural proof systems and demonstrates that the required structure is the \emph{monoidal structure}. Then, he moves to the intuitionistic deductive systems over the fragment $\{\top, \wedge, \to \}$ and shows that the structure is nothing other than the well-known \emph{cartesian closed structure}. Later, Lambek shows that the categorical approach to propositional intuitionistic proofs and the type-theoretical one using typed lambda calculus coincide \cite{LambekScott}.  For a thorough investigation of the intuitionistic case, see \cite{Szabo,LambekScott}. For the substructural setting, especially linear logic, see \cite{1,2,3,4,SeelyLinear,dePaiva}

Lambek's use of categorical machinery is beneficial for both category theorists and proof/type theorists. For the first community, it provides a combinatorially manageable syntax to investigate the properties of categories. Specifically, it aids in axiomatizing equalities between morphisms in a given family of categories in terms of a few easy-to-check equalities. Such problems are referred to as \emph{coherence problems}, with the primordial example being MacLane's work on monoidal and symmetric monoidal categories \cite{99}. In fact, Lambek was the one who introduced Gentzen’s proof-theoretical methods into category theory, which Mac Lane and Kelly later exploited to solve a major coherence problem for closed categories \cite{81}. For more on these coherence problems and their connection to proof theory, see the comprehensive book by Do\v{s}en and Petri{\'c} \cite{DosenBook}. It is also worth consulting \cite{101}.

For the proof/type theory community, category theory provides a rich semantics that is useful for providing counterexamples. Here, there are three main problems of interest. First, the \emph{existence problem} concerns the existence of terms in a type or proofs of a proposition. The second is the \emph{equivalence problem}, which addresses the equivalence of terms or proofs, and the third is the \emph{identity problem}, asking whether two types or propositions are identical. In each of these problems, one can use categories to demonstrate that some types (resp. propositions) are empty (resp. unprovable), some terms (resp. proofs) are unequal, or some types (resp. propositions) are non-isomorphic; for more see Section \ref{SecCatSemantics}. One can even go further and ask whether a given family of categories is \emph{complete} for the problem, meaning that to solve the problem, it is sufficient to check only the interpretations in those categories.

For the existence problem, since we are only concerned with the existence of proofs, it is sufficient to relax the categorical structure to posets. Finding concrete yet complete posets is an interesting line of inquiry. For instance, regarding the existence of intuitionistic propositional proofs, it is enough to consider the bicartesian closed poset $\mathcal{O}(\mathbb{R})$, which consists of the open sets of the topological space $\mathbb{R}$ \cite{Nick}. Although it is not strictly necessary, using concrete non-posetal categories for the existence problem is also intriguing. The first example of such a category was $\mathbf{Set}^{\mathbb{Z}}$, which was studied by Laüchli \cite{Lach}. His original paper does not employ categorical language; instead, it assigns a pair $(X, \sigma_X)$ of a set and a bijection on it to propositions, and interprets deductions as usual functions that respect the bijection. The modern categorical presentation, along with its generalizations, can be found in \cite{Makkai}. The categorical interpretation is also generalized to the first-order setting in \cite{M1,M2}.

For the equivalence problem, Friedman \cite{Fri} and, independently, Plotkin \cite{P1,P2} demonstrated that beta-eta conversion is complete for deriving all equalities between the (simply-typed) lambda-definable functionals in the category of sets. Building on this result, \v{C}ubri\'{c} showed that the category of sets is complete for the equivalence problem \cite{Cubric} for the fragment $\{\top, \wedge, \to\}$ of intuitionistic proofs. The result is further reduced to the category of finite sets by Solovev \cite{Soloviev}. This reduction can be employed to prove the decidability of the equivalence of proofs in the mentioned fragment. Later, Simpson \cite{Simpson} established that a cartesian closed category is complete for the equivalence problem for the fragment $\{\top, \wedge, \to\}$ if and only if it is not a preordered set. This theorem has a significant philosophical consequence, as it states that the equivalence between proofs is the maximum congruence relation one can impose on proofs without collapsing the entire proof structure to mere provability. For further details, see \cite{DosenMax}.
Recently, Scherer \cite{scherer} demonstrated that the category of finite sets is complete for the equivalence problem of the full propositional intuitionistic language leading again to the decidability of the equivalence of proofs problem. Using the argument provided by Simpson \cite{SimpOnline}, one can show that a bicartesian closed category is complete for the full language if and only if it is not a preordered set.

For the identity problem concerning the fragment $\{\wedge, \top, \to\}$, it is known that a finite set of simple identities (see Subsection \ref{SubsectionIdentity}) axiomatizes all identities between propositions. It has also been proven that the category of finite sets is complete for this problem, which leads to the decidability of identities between propositions in the fragment. For the fragment $\{\wedge, \vee, \top, \to\}$, it is known that the canonical finite set of axioms, referred to as \emph{Tarski's high school identities}, is not complete. Furthermore, \cite{Fiore} showed that the problem is not finitely axiomatizable.
For the full language, the situation is more complicated. It is still not finitely axiomatizable \cite{Fiore}. We also know that the category of sets cannot be complete, as it considers the formulas $\neg p \vee \neg \neg p$ and $\top$ to be isomorphic, despite the fact that they are not even intuitionistically equivalent. To the best of our knowledge, the decidability of the propositional identity problem for the the last two fragments remains open. For a survey of the results and applications of the propositional identity problem, see \cite{IsoSurvey,IsomorphismOfTypes}.

So far, we have discussed the propositional language. To some extent, the theory has been generalized to more expressive languages by Lawvere \cite{AdjLawvere,EqLawvere} using hyperdoctrines, and to first-order languages by B\'{e}nabou \cite{Ben}, Seely \cite{seelythesis,Seely2}, and Makkai \cite{M1,M2} through a special family of fibrations. In this chapter, we will not cover these investigations. However, to provide a glimpse into the use of categorical machinery in higher-order settings, we will restrict our focus to the special yet rich case of \emph{realizability}. The main idea is to use maps in a structured category or terms in a type theory as constructions. These constructions are usually called \emph{realizers}. Realizability originated with Kleene, who employed algorithms as realizers \cite{Kleene}. This approach proved useful for establishing meta-mathematical theorems about constructive arithmetic. For instance, one can prove the consistency of classically invalid Church-Turing thesis, which asserts that all functions on natural numbers are computable. It is also beneficial to extract computational information from constructive proofs. For example, if $\forall x \exists y A(x, y)$ is provable in constructive arithmetic, then there exists a computable function $f: \mathbb{N} \to \mathbb{N}$ such that $A(n, f(n))$ holds for any $n \in \mathbb{N}$. This machinery has been generalized in various ways \cite{RealTr}. The categorical version we are interested in emerged in \cite{EffHyland}, where Hyland used algorithms to realize higher-order theories. His framework was further generalized in \cite{Lars,Bauer} to allow for the use of structured categories or typed theories to formalize constructions. For a comprehensive explanation, see \cite{LarsThesis,Bauer,JaapBook,bauer2}. 

\subsection{A Word to the Reader}
The categorical language can pose a significant conceptual barrier for those wishing to understand the applications of categorical methods in proof theory. To mitigate this issue and reach a broader audience—particularly among proof theorists—we have decided to structure the material without assuming prior familiarity with category theory. We only require a very basic understanding of mathematical logic, including the fundamentals of natural deductions, the basics of order theory and topology, and a basic familiarity with the theory of computation. 

As a consequence, the chapter begins with a slow and somewhat extensive introduction to category theory. In our explanation, we employ numerous examples and motivate the introduced categorical concepts using logical and, in particular, proof-theoretical ideas. 
To make the introduction easier to digest, we refrain from using any unnecessary categorical notions. For instance, we avoid discussing limits and colimits, representable functors, or adjunctions. This categorical part is addressed in Sections \ref{SecCatAndFunct} and \ref{SecUnivCons}, where we also introduce the basic structure of a propositional proof system. Readers with a solid background in category theory may skim through these sections to see the connection with proof theory and start reading from Section \ref{SectionBHKInterpretation}. 
In Section \ref{SectionBHKInterpretation}, we will explain and formalize the BHK interpretation and demonstrate how it fits into the broader context of categorical proof theory. In Section \ref{SectionClassical}, we will discuss classical deductions. In Section \ref{SecCatSemantics}, we will utilize the categorical framework established in the previous sections as a form of semantics for syntactical proofs. To transition from propositional to higher-order cases, in Section \ref{SectionTheories}, we will introduce higher-order theories in general and some constructive theories of arithmetic in particular. In Section \ref{SecRealizability}, we will apply categorical proof theory to gain insights into general higher-order theories using a semantical machinery called realizability. Finally, in Section \ref{SecRealizabilityForArith}, we will use realizability to prove some consistency results in arithmetic and also extract information from the arithmetical proofs. In each section, we will provide numerous philosophical discussions and examples to convey the fundamental ideas, the constructions, and, at times, sketches of the proofs. However, we will generally leave the computational details of the examples to the reader. We will also prove the simpler theorems to demonstrate how the concepts work in practice while referring readers to external sources for the proofs of more complex results.\\

\noindent \textbf{Acknowledgement.}
We wish to thank the organizers of the TACL 2022 summer school and Rosalie Iemhoff for their invitation, the participants of the summer school and the fundamental computing team at the Bernoulli institute who organized and actively participated in the lectures on the present text and the editors of the present volume, especially Maria Manuel Clementino, for their support and patience without which this chapter would never come to the existence. We are also grateful to Steve Awodey, Andrej Bauer, David Pym, Pino Rosolini, Alex Simpson, Thomas Streicher and the anonymous referees for their very helpful comments, suggestions and corrections on the first draft. Support by the FWF project P 33548, the Czech Academy of Sciences (RVO 67985840) and the GA\v{C}R grant 23-04825S is also gratefully acknowledged.

\section{Preliminaries}

In this section, we will recall some basic notions from order theory, algebra, propositional proof theory and computability theory, to refresh our memory, emphasize the less known points and fix the notation. Many other notions such as spaces, compactness and local compactness in topology, first-order theories in logic and Turing machines in computability theory are assumed without any further explanation.

A \emph{preordered set} is a pair $(P, \leq)$ of a set $P$ and a preorder $\leq$, i.e., a reflexive and transitive binary relation $\leq \; \subseteq P \times P$, i.e., $x \leq x $, for any $x \in P$ and if $x \leq y$ and $y \leq z$ then $x \leq z$, for any $x, y, z \in P$. If this relation is also anti-symmetric, i.e., if $x \leq y$ and $y \leq x$ implies $x=y$, for any $x, y \in P$, then the preorder is called a \emph{partial order} and $(P, \leq)$ is called a \emph{poset}. 
A poset $(P, \leq)$ is called \emph{linear}, if for any $x, y \in P$, either $x \leq y$ or $y \leq x$. A poset $(P, \leq)$ is called a \emph{tree}, if for any element $x \in P$, the poset $\{y \in P \mid y \leq x\}$ with the induced order is finite and linear. A tree is called \emph{rooted}, if it has the least element, i.e., an element $m \in P$ such that $m \leq x$, for any $x \in P$.

For any two preordered sets $(P, \leq_P)$ and $(Q, \leq_Q)$, an \emph{order-preserving} map $f: (P, \leq_P) \to (Q, \leq_Q)$ is a function $f: P \to Q$ such that $f(x) \leq_Q f(y)$, if $x \leq_P y$. It is called an \emph{order-embedding} if it is order-preserving and $f(x) \leq_Q f(x)$ implies $x \leq_P y$, for any $x, y \in P$.
Two prototypical examples of posets that we will use in this chapter are the poset of the opens of a topological space $X$, denoted by $\mathcal{O}(X)$, and the poset of all the upward-closed subsets (upset, for short) of a preordered set $(P, \leq)$ with the inclusion, denoted by $U(P, \leq)$. Note that the upsets of $(P, \leq)$ also form a topology called the \emph{upset topology} and any space with a topology induced in that way is called an \emph{Alexandrov} space. For a poset $(P, \leq)$, define $\mathbf{2}^{(P, \leq)}$ as the poset of all order-preserving maps from $(P, \leq)$ into $\mathbf{2}$, where $\mathbf{2}$ is the poset $(\{0, 1\}, \leq)$ and $\leq$ is the usual order on numbers.
Note that $U(P, \leq)$ is isomorphic to $\mathbf{2}^{(P, \leq)}$. Hence, we use them interchangeably.  

A \emph{monoid} is a tuple $(M, \cdot, e)$, where $M$ is a set, $\cdot : M \times M \to M$ is a binary operation on $M$ and $e \in M$ is an element such that $e \cdot x=x \cdot e=x$, for any $x \in M$ and $x \cdot (y \cdot z)=(x \cdot y) \cdot z$, for any $x, y, z \in M$. \emph{Groups} are the monoids where any element has an inverse, i.e., for any $x \in M$, there exists $y \in M$ such that $x \cdot y=y \cdot x=e$. 

Let $\mathcal{L}_p=\{\top, \bot, \wedge, \vee, \to\}$ be the propositional language and recall that $\neg A$ is an abbreviation for $A \to \bot$.
We define the \emph{natural deduction} system $\mathbf{NJ}$ by its usual rules:
\begin{center}
 \begin{tabular}{c c c} 
 \AxiomC{$ $}
  \RightLabel{$(\top)$} 
   \UnaryInfC{$\top$}
 \DisplayProof \quad \quad
 &
 \AxiomC{$\mathsf{D}$}
 \noLine
 \UnaryInfC{$\bot$}
  \RightLabel{$(\bot)$} 
   \UnaryInfC{$A$}
 \DisplayProof
\end{tabular}
\end{center}
\begin{center}
 \begin{tabular}{c c c} 
 \AxiomC{$\mathsf{D}_1$}
 \noLine
 \UnaryInfC{$A$}
  \AxiomC{$\mathsf{D}_2$}
  \noLine
   \UnaryInfC{$B$}
 \RightLabel{$(I \wedge)$} 
 \BinaryInfC{$A \wedge B$}
 \DisplayProof
 &
  \AxiomC{$\mathsf{D}$}
 \noLine
 \UnaryInfC{$A \wedge B$}
  \RightLabel{$(E \wedge_1)$} 
   \UnaryInfC{$A$}
 \DisplayProof
 &
  \AxiomC{$\mathsf{D}$}
 \noLine
 \UnaryInfC{$A \wedge B$}
  \RightLabel{$(E \wedge_2)$} 
   \UnaryInfC{$B$}
 \DisplayProof
\end{tabular}
\end{center}

\begin{center}
 \begin{tabular}{c c c} 
  \AxiomC{$\mathsf{D}$}
 \noLine
 \UnaryInfC{$A$}
  \RightLabel{$(I \vee_1)$} 
   \UnaryInfC{$A \vee B$}
 \DisplayProof
 &
    \AxiomC{$\mathsf{D}$}
 \noLine
 \UnaryInfC{$B$}
  \RightLabel{$(I \vee_2)$} 
   \UnaryInfC{$A \vee B$}
 \DisplayProof
 &
  \AxiomC{$\mathsf{D}$}
 \noLine
 \UnaryInfC{$A \vee B$}
 \AxiomC{$[A]^i$}
 \noLine
  \UnaryInfC{$\mathsf{D}_1$}
  \noLine
   \UnaryInfC{$C$}
    \AxiomC{$[B]^j$}
    \noLine
  \UnaryInfC{$\mathsf{D}_2$}
  \noLine
   \UnaryInfC{$C$}
 \RightLabel{$(E \vee_{i,j})$} 
 \TrinaryInfC{$C$}
 \DisplayProof
\end{tabular}
\end{center}
\begin{center}
 \begin{tabular}{c c} 
 \AxiomC{$[A]^i$}
 \noLine
   \UnaryInfC{$\mathsf{D}$}
 \noLine
 \UnaryInfC{$B$}
  \RightLabel{$(I \to_i)$} 
   \UnaryInfC{$A \to B$}
 \DisplayProof
 &
 \AxiomC{$\mathsf{D}_1$}
 \noLine
 \UnaryInfC{$A$}
  \AxiomC{$\mathsf{D}_2$}
  \noLine
   \UnaryInfC{$A \to B$}
 \RightLabel{$(E \to)$} 
 \BinaryInfC{$B$}
 \DisplayProof
\end{tabular}
\end{center}
where $i$ and $j$ are natural numbers to refer to the assumptions. One can compose these rules to construct rooted trees that we will use to define derivations. Here are some conditions we impose on the use of superscripts in trees. First, all leaves are formulas with a superscript. Second, in any tree, any number $i$ must be used for a unique formula. However, it is possible to have one formula with different superscripts. Third, in $(E\vee_{i, j})$, the indices $i$ and $j$ cannot occur in $\mathsf{D}$, the index $i$ cannot occur in $\mathsf{D}_2$ and $j$ cannot occur in $\mathsf{D}_1$.

We define a \emph{derivation} of $A$ from the set $\Gamma=\{[\gamma_1]^{i_1}, [\gamma_2]^{i_2}, \ldots, [\gamma_n]^{i_n}\}$ of assumptions as a tree with the leaves inside $\Gamma$. Note that it is not necessarily that every formula in $\Gamma$ appears in the tree. When the set of assumptions is clear from the context, especially when $\Gamma$ is equal to the set of formulas in the leaves, we omit $\Gamma$ and simply talk about the derivations of $A$. If $\mathsf{D}$ is a derivation of $A$ from the assumptions $\Gamma$, a sub-derivation $\mathsf{E}$ of $\mathsf{D}$ is a full subtree of $\mathsf{D}$. It proves the formula in its root and its assumptions can be canonically computed from $\Gamma$ by adding the eliminated assumptions in the path from the root of $\mathsf{E}$ to the root of $\mathsf{D}$. If $\mathsf{D}$ is a derivation of $A$ from the assumptions in $\Gamma$ and for any $[\gamma_i]^{n_i}$, we have a derivation $\mathsf{D}_i$ of $\gamma_i$ from $\Sigma$, then by composition of $\mathsf{D}$ with the $\mathsf{D}_i$'s, we mean gluing $\mathsf{D}_i$ to the leaf $[\gamma_i]^{n_i}$ of $\mathsf{D}$ and eliminating the $[-]^{n_i}$'s. The resulting tree is a derivation of $A$ from $\Sigma$.  
We identify two derivations if one is reachable by the other by changing the names of the superscripts in a bijective manner. For instance, the following two derivations of $B \to A \wedge B$ from the assumptions $\{[A]^3, [B]^5\}$ and $\{[A]^4, [B]^7\}$, respectively, are considered as equal:
 \begin{center}
 \begin{tabular}{c c c}
 \AxiomC{$[A]^3$}
 \AxiomC{$[B]^5$}
\RightLabel{$(I\wedge)$} 
\BinaryInfC{$A \wedge B$}
\RightLabel{$(I\to_5)$} 
 \UnaryInfC{$B \to A \wedge B$}
 \DisplayProof
   &
 $\quad \equiv \quad$
 &
\AxiomC{$[A]^4$}
 \AxiomC{$[B]^7$}
\RightLabel{$(I\wedge)$} 
\BinaryInfC{$A \wedge B$}
\RightLabel{$(I\to_7)$} 
 \UnaryInfC{$B \to A \wedge B$}
 \DisplayProof
\end{tabular}
\end{center}
Consider the following family of \emph{basic $\beta$-equivalences}:
 \begin{center}
 \begin{tabular}{c c c}
 \AxiomC{$\mathsf{D}_1$}
 \noLine
\UnaryInfC{$A$}
 \AxiomC{$\mathsf{D}_2$}
 \noLine
\UnaryInfC{$B$}
\RightLabel{$(I\wedge)$} 
 \BinaryInfC{$A \wedge B$}
 \RightLabel{$(E\wedge)$} 
  \UnaryInfC{$A$}
 \DisplayProof
  &
 $\quad \equiv \quad$
 &
  \AxiomC{$\mathsf{D}_1$}
 \noLine
\UnaryInfC{$A$}
 \DisplayProof
\end{tabular}
\end{center}

\begin{center}
 \begin{tabular}{c c c c}
 \AxiomC{$[A]^i$}
 \noLine
 \UnaryInfC{$\mathsf{D}_2$}
 \noLine
\UnaryInfC{$B$}
\RightLabel{$(I\to_i)$}
 \normalsize\UnaryInfC{$A \to B$}
 \AxiomC{$\mathsf{D}_1$}
 \noLine
\UnaryInfC{$A$}
\RightLabel{$(E\to)$} 
  \BinaryInfC{$B$}
 \DisplayProof
 &
 $\quad \equiv \quad$
 &
 \AxiomC{$\mathsf{D}_1$}
 \noLine
\UnaryInfC{$A$}
 \noLine
 \UnaryInfC{$\mathsf{D}_2$}
 \noLine
\UnaryInfC{$B$}
 \DisplayProof
\end{tabular}
\end{center}

   \begin{center}
 \begin{tabular}{c c c}
 \AxiomC{$\mathsf{D}$}
 \noLine
\UnaryInfC{$A$}
\RightLabel{$(I\vee)$} 
\UnaryInfC{$A \vee B$}
 \AxiomC{$[A]^i$}
  \noLine
\UnaryInfC{$\mathsf{D}_1$}
 \noLine
\UnaryInfC{$C$}
 \AxiomC{$[B]^j$}
   \noLine
\UnaryInfC{$\mathsf{D}_2$}
 \noLine
\UnaryInfC{$C$}
\RightLabel{$(E\vee_{i, j})$} 
 \normalsize
 \TrinaryInfC{$C$}
 \DisplayProof
  &
 $\quad \equiv \quad$
 &
  \AxiomC{$\mathsf{D}$}
 \noLine
\UnaryInfC{$A$}
\noLine
\UnaryInfC{$\mathsf{D}_1$}
\noLine
\UnaryInfC{$C$}
 \DisplayProof
\end{tabular}
\end{center}
    
    \vspace{20 pt}
   \begin{center}
 \begin{tabular}{c c c}
 \AxiomC{$\mathsf{D}$}
 \noLine
\UnaryInfC{$B$}
\RightLabel{$(I\vee)$} 
\UnaryInfC{$A \vee B$}
 \AxiomC{$[A]^i$}
  \noLine
\UnaryInfC{$\mathsf{D}_1$}
 \noLine
\UnaryInfC{$C$}
 \AxiomC{$[B]^j$}
   \noLine
\UnaryInfC{$\mathsf{D}_2$}
 \noLine
\UnaryInfC{$C$}
\RightLabel{$(E\vee_{i, j})$} 
 \normalsize
 \TrinaryInfC{$C$}
 \DisplayProof
  &
 $\quad \equiv \quad$
 &
  \AxiomC{$\mathsf{D}$}
 \noLine
\UnaryInfC{$B$}
\noLine
\UnaryInfC{$\mathsf{D}_2$}
\noLine
\UnaryInfC{$C$}
 \DisplayProof
\end{tabular}
\end{center} 
Moreover, consider the following \emph{basic $\eta$-equivalences}:
    \begin{center}
        \begin{tabular}{c c c}
 \AxiomC{$\mathsf{D}_1$}
 \noLine
\UnaryInfC{$A \wedge B$}
 \UnaryInfC{$A$}
 \AxiomC{$\mathsf{D}_1$}
 \noLine
\UnaryInfC{$A \wedge B$}
 \UnaryInfC{$B$}
  \BinaryInfC{$A \wedge B$}
 \DisplayProof $\quad \equiv \quad \mathsf{D}_1$
 & \hspace{15pt}
  \AxiomC{$\mathsf{D}_2$}
 \noLine
\UnaryInfC{$A \to B$}
 \AxiomC{$[A]^1$}
  \BinaryInfC{$B$}
  \scriptsize\RightLabel{$\! \scriptsize{1}$}
 \normalsize\UnaryInfC{$A \to B$}
 \DisplayProof  $\quad \equiv \quad \mathsf{D}_2$
        \end{tabular}
    \end{center}

\begin{center}
        \begin{tabular}{c}
 \AxiomC{$A \vee B$}
 
 \AxiomC{$[A]^1$}
\UnaryInfC{$A \vee B$}
\noLine
\UnaryInfC{$\mathsf{D}_3$}
\noLine
 \UnaryInfC{$C$}
 \AxiomC{$[B]^2$}
\UnaryInfC{$A \vee B$}
\noLine
\UnaryInfC{$\mathsf{D}_3$}
\noLine
 \UnaryInfC{$C$}
\TrinaryInfC{$C$}
\DisplayProof  $\quad \equiv \quad \mathsf{D}_3$
        \end{tabular}
    \end{center} 
Note that in the right side, we always consider the used derivations from the left side as a sub-derivation. Hence, their assumptions are equal. For $\top$ and $\bot$, the basic $\eta$-equivalences are the identifications of any two derivations of a fix formula from $\bot$ and any two derivations of $\top$ from a fixed family of assumptions. These basic $\beta$- and $\eta$-equivalences\footnote{These equivalences also appear in lambda calculus and more complex type theories as the basic computational rules on terms.} give rise to an equivalence relation defined as the least equivalence relation on derivations extending the above-mentioned $\beta$- and $\eta$-equivalences and respecting the composition of derivations. We always consider derivations up to this equivalence relation.

The systems $\mathbf{NM}$ and $\mathbf{NN}$ are defined similarly for the fragments $\mathcal{L}_p-\{\bot\}$ and $\mathcal{L}_p-\{\bot, \vee\}$, respectively, by eliminating all the rules for the missing connectives.
The \emph{intuitionistic propositional logic}, $\mathsf{IPC}$, is defined as the set of all formulas $A \in \mathcal{L}_p$ such that there exists a derivation of $A$ in $\mathbf{NJ}$ from the empty set of assumptions. Similarly, one can define the \emph{minimal propositional logic}, $\mathsf{MPC}$, using the system $\mathbf{NM}$.

The class of all \emph{primitive recursive functions} are defined as the least class of functions over natural numbers, containing the basic functions $Z=0$, $s(x)=x+1$ and $I^i_k(x_1, x_2, \ldots, x_k)=x_i$ for $1 \leq i \leq k$ and closed under composition and \emph{primitive recursion}, i.e.,
for any primitive recursive functions $g: \mathbb{N}^k \to \mathbb{N}$ and $h: \mathbb{N} \times \mathbb{N}^k \times \mathbb{N} \to \mathbb{N}$, the function $f: \mathbb{N} \times \mathbb{N}^k \to \mathbb{N}$ defined by:
\[
\begin{cases}
f(0, \bar{x})=g(\bar{x}) & \\
f(y+1, \bar{x})=h(y, \bar{x}, f(y, \bar{x})) &\\
\end{cases}
\]
is also primitive recursive.
Turing machines are defined in their usual way. It is possible to encode Turing machines by natural numbers. We call these numbers the \emph{codes} of the algorithms or machines. Recall that there is a \emph{universal Turing machine} $U$ that reads the pair $(e, n)$ of the code $e$ of an algorithm and the input $n$ and provides the output of the algorithm $e$ on the input $n$. We denote $U(e, n)$ by $e \cdot n$. Note that $U$ is a partial function. When $U(e, n)$ is (resp. is not) defined, i.e., when the algorithm $e$ halts (resp. does not halt) on the input $n$, we write $e \cdot n \downarrow$ (resp. $e \cdot n \uparrow$). One interesting theorem that we will use later is Kleene's \emph{$S^{m}_n$ theorem}:
\begin{thm}
There exists a primitive recursive function $S$ such that for any code $e$ of a function $f(x, y)$ with two variables, we have $S(e, x) \cdot y=f(x, y)$, for any $x, y \in \mathbb{N}$.
\end{thm}

\section{Categories, Functors and Natural Transformations} \label{SecCatAndFunct}

In this section, we will introduce the basic notions of categories, functors, and natural transformations. We will also present many instances of these concepts to provide concrete toy examples to familiarize readers with such abstract notions. These examples will also be required later in the chapter when we aim to provide a categorical formalization for a proof system.

\subsection{Categories}

Let us start this subsection with the definition of a category.

\begin{dfn}
A \emph{category} $\mathcal{C}$ consists of the following data:
\begin{itemize}
\item[$\bullet$]
a collection of \emph{objects}, denoted by $\mathrm{Ob}(\mathcal{C})$,
\item[$\bullet$]
a collection of \emph{morphisms} or \emph{maps}, denoted by $\mathrm{Mor}(\mathcal{C})$,
\item[$\bullet$]
for any morphism $f \in \mathrm{Mor}(\mathcal{C})$, an object $s(f)$ called the \emph{source} of $f$,
\item[$\bullet$]
for any morphism $f \in \mathrm{Mor}(\mathcal{C})$, an object $t(f)$ called the \emph{target} of $f$,
\item[$\bullet$]
for any object $A \in \mathrm{Ob}(\mathcal{C})$, a morphism $id_A$, called the \emph{identity} on $A$, 
\item[$\bullet$]
for any two morphisms $f, g \in \mathrm{Mor}(\mathcal{C})$ such that $s(f)=t(g)$, a morphism $f \circ g$, called the \emph{composition} of $f$ and $g$,
\end{itemize}
satisfying the following properties:
\begin{itemize}
\item[$\bullet$]
$s(id_A)=t(id_A)=A$,
\item[$\bullet$]
$s(f \circ g)=s(g)$ and $t(f \circ g)=t(f)$,
\item[$\bullet$]
$f \circ id_A=f=id_B \circ f$, if $s(f)=A$ and $t(f)=B$,
\item[$\bullet$]
$f \circ (g \circ h)=(f \circ g) \circ h$.
\end{itemize}
Since composition is a partial function, by the above equalities, we mean that if one side is defined, the other must also be defined, and they must be equal.
\begin{table}[h!]
    \centering
\[\begin{tikzcd}
	&&&& {D} \\
	{A} &&& {B} \\
	\\
	\\
	&&& {C}
	\arrow["{h}"', from=2-1, to=2-4, curve={height=30pt}]
	\arrow["{f}", from=2-1, to=2-4, curve={height=-30pt}]
	\arrow["{g}" description, from=2-1, to=2-4]
	\arrow["{i}"', from=2-4, to=5-4]
	\arrow["{j}", from=2-4, to=1-5]
	\arrow["{k}"', from=5-4, to=1-5, dotted]
	\arrow["{if=ig=ih}"', from=2-1, to=5-4, curve={height=10pt}]
	\arrow["{jf=jg=jh}", from=2-1, to=1-5, curve={height=-40pt}]
\end{tikzcd}\]  
    \caption{A schematic diagram of a category}
    \label{tab:my_label}
\end{table}
For any $f \in \mathrm{Mor}(\mathcal{C})$, we present the data $s(f)=A$ and $t(f)=B$ by $f: A \to B$. For any two objects $A, B \in \mathrm{Ob}(\mathcal{C})$, by $\mathcal{C}(A, B)$ or $\mathrm{Hom}_{\mathcal{C}}(A, B)$, we mean the collection of all morphisms $f:A \to B$. A category is called \emph{small} if $\mathrm{Mor}(\mathcal{C})$ is a set. It is called \emph{locally small} if $\mathrm{Hom}_{\mathcal{C}}(A, B)$ is a set, for any two objects $A$ and $B$. Sometimes, for simplicity, we drop $\circ$ in $f \circ g$ and denote the composition by $fg$.
A map $f: A \to B$ is called an \emph{isomorphism} iff there exists $g: B \to A$ such that $gf=id_A$ and $fg=id_B$. If there is an isomorphism between $A$ and $B$, we write $A \cong B$.
\end{dfn}

\begin{phil}
To have an informal interpretation in mind, consider the objects of a category as the entities of a given discourse, the maps as the transformations between these entities, the composition as the composition of these transformations, and the identity as the do-nothing transformation. More specifically, objects in a category can be interpreted as \emph{propositions}, and morphisms $f: A \to B$ as the \emph{deductions} of $B$ from the assumption $A$. Thus, the identity map $id_A: A \to A$ can be seen as the canonical deduction of $A$ from itself, and the composition as the process of applying two deductions sequentially. In this sense, any category can be viewed as a kind of \emph{proof system}.
\end{phil}

In the following, we will present many examples of categories. These examples are intended to provide a proper sense of what categories are. In each example, we will also attempt to interpret the introduced category as a proof system.

\begin{exam}
The collection of all (resp. finite) sets as the objects and functions as the morphisms, with the usual composition and identity, constitutes a category denoted by $\mathbf{Set}$ (resp. $\mathbf{FinSet}$). We can think of the set $A$ as the set of all \emph{proofs} of the proposition associated with $A$, denoted by $\phi_{A}$. Then, the function $f: A \to B$ can be read as a \emph{deduction} of $\phi_{B}$ from $\phi_{A}$, transforming any proof of $\phi_{A}$ into a proof of $\phi_{B}$.
\end{exam}

\begin{exam}
The collection of all topological spaces as the objects and continuous functions as the morphisms, with the usual composition and identity, constitutes a category, denoted by $\mathbf{Top}$. We can think of the space $X$ as the space of all proofs of the proposition associated to $X$, denoted by $\phi_X$. These proofs can be infinitary. However, the topology of the space ensures that only a finite amount of information encoded in such infinitary proofs is used. More precisely, we can think of an open subset of $X$ as a subset of all possible proofs of $\phi_X$ that extend a given finite amount of information. For instance, a proof of $\forall x \in \mathbb{N}, A(x)$ can be an infinite sequence $\{\pi_n\}_{n \in \mathbb{N}}$, where each $\pi_n$ is a proof of $A(n)$, and the open sets are the subsets consisting of all sequences $\{\pi_n\}_{n \in \mathbb{N}}$ that extend a finite initial segment. In this interpretation, a deduction of $Y$ from $X$ is a \emph{finitary} function transforming any proof of $X$ into a proof of $Y$, where finitary means that it only uses a finite amount of information from its input to compute the output. The continuity condition simply encodes this finiteness condition. For example, in our case, a finitary deduction of $B$ from $\forall x \in \mathbb{N}, A(x)$ must only use a finite initial segment of the proof $\{\pi_n\}_{n \in \mathbb{N}}$ of $\forall x \in \mathbb{N}, A(x)$.
\end{exam}

\begin{exam} \textit{(Discrete categories)}
A category $\mathcal{C}$ is called \emph{discrete} if it has only identity maps. Therefore, any set can be considered as a small discrete category. For example, the set $\{A,B,C\}$ can be viewed as the following category:
\[\begin{tikzcd}[ampersand replacement=\&]
	A \&\& B \&\& C
	\arrow["{id_A}", from=1-1, to=1-1, loop, in=50, out=130, distance=15mm]
	\arrow["{id_B}", from=1-3, to=1-3, loop, in=50, out=130, distance=15mm]
	\arrow["{id_C}", from=1-5, to=1-5, loop, in=50, out=130, distance=15mm]
\end{tikzcd}\]
\end{exam}

\begin{exam} \textit{(Finite categories)}
These are some finite categories:
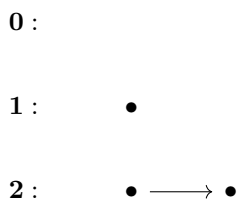
\[\begin{tikzcd}[ampersand replacement=\&]
	{\mathbf{0}:} \\
	{\mathbf{1}:} \& \bullet \\
	{\mathbf{2}:} \& \bullet \& \bullet
	\arrow[from=3-2, to=3-3]
\end{tikzcd}\]
The category $\mathbf{0}$ has no objects and no morphisms. The category $\mathbf{1}$ has one object and one morphism, namely, the identity of the object. We did not draw the identity map in the above diagram and will continue this practice of omitting identity maps in diagrams. The category $\mathbf{2}$ has two objects and one non-identity morphism. Throughout this chapter, we always use the notation $\mathbf{2}$ for this category, which should not be confused with the discrete category containing two elements.  
\end{exam}

\begin{exam}\textit{(Preorders)}
Any preordered set $(P, \leq)$ can be interpreted as a small category, where the objects are the elements of $P$, and there is a unique morphism from $p$ to $q$ iff $p \leq q$. Note that the identity and composition correspond to the reflexivity and transitivity of the preorder, respectively.
\end{exam}

\begin{phil}
It is useful to think of preordered sets as the shadows of usual categories, reducing all transformations between two objects to a single \emph{transformability} between them. In the logical interpretation, this means that we collapse all deductions between two statements into one \emph{deductibility} map.
\end{phil}

\begin{exam}\textit{(Monoids)}
Any monoid $(M, \cdot, e)$ can be interpreted as a category with a single object $*$, where each element of $M$ is a morphism from $*$ to $*$, $e$ is the identity morphism $id_*$, and $\cdot$ represents the composition of morphisms.
\end{exam}

\begin{phil}(\emph{Groups and identity})
Being monoids, groups can also be read as categories with one object $*$. One can read any element of the group or equivalently any morphism over $*$ as a proof of the equality $*=*$. With this interpretation, it is easy to see that the identity $e$ is the \emph{trivial} proof of $*=*$ and its existence is the proof-sensitive version of the reflexivity of the identity relation, while the composition and the inverse operation of the group are the proof-sensitive versions of its transitivity and symmetry, respectively. 
\end{phil}

\begin{rem}
A category combines the two extreme cases previously mentioned: a preordered set and a monoid. The first handles the existence of different objects in a category, while the second addresses the different morphisms between any two objects. From a proof-theoretic perspective, one can say that the first describes the propositions and the deductibility relation between them, whereas the second incorporates the actual deductions into the picture.
\end{rem}

\begin{exam}(\textit{Baby variable sets})
Consider the collection of functions
\[\begin{tikzcd}
	{A_1} \\
	\\
	{A_0}
	\arrow["{f}", from=3-1, to=1-1]
\end{tikzcd}\]
as the objects and define a morphism $\alpha: f \to g$ to be the pair of functions $(\alpha_0, \alpha_1)$, where $\alpha_0 : A_0 \to B_0$ and $\alpha_1 : A_1 \to B_1 $ such that $\alpha_1 f=g \alpha_0$: 
\[\begin{tikzcd}[ampersand replacement=\&]
	{A_1} \&\& {B_1} \\
	\\
	{A_0} \&\& {B_0}
	\arrow["{\alpha_1}", from=1-1, to=1-3]
	\arrow["f", from=3-1, to=1-1]
	\arrow["{\alpha_0}"', from=3-1, to=3-3]
	\arrow["g"', from=3-3, to=1-3]
\end{tikzcd}\]
These collections, with the evident composition and identity, constitute a category denoted by $\mathbf{Set}^{\mathbf{2}}$. Any object $f: A_0 \to A_1$ in $\mathbf{Set}^{\mathbf{2}}$ can be interpreted as a \emph{variable set}, varying over the discrete structure of time $\{0 \leq 1\}$. The set $A_0$ represents the elements of the variable set available at the moment $t=0$, while the set $A_1$ represents its elements at the moment $t=1$. Moving from $t=0$ to $t=1$, there are three main possibilities: either new elements are created, some elements remain unchanged (up to renaming), or some distinct elements in $A_0$ become equal in $A_1$. These scenarios are all formalized by the function $f$. The elements outside the range of $f$ represent new elements at $t=1$, while the elements within the range come from $t=0$, reflecting the latter two possibilities. Any map between these variable sets naturally corresponds to a pair of maps, each representing the state of the sets at different moments in time, and must respect the changes that occur over time. From a proof-theoretical perspective, these variable sets can be interpreted as propositions whose proofs evolve over time. New proofs can be discovered, or two existing proofs may be found to be identical upon closer examination.
\end{exam}

\begin{rem}
In the previous example, there is nothing special about the structure $\{0 \leq 1\}$ and it can be replaced by any other preordered set or even by any small category. We will see how later. 
\end{rem}

\begin{exam}(\textit{Dynamical systems})
A pair $(A, \sigma_A)$ of a set $A$ and a function $\sigma_A: A \to A$
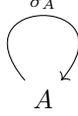
\[\begin{tikzcd}[ampersand replacement=\&]
	A
	\arrow["{\sigma_A}", from=1-1, to=1-1, loop, in=50, out=130, distance=15mm]
\end{tikzcd}\]
is called a \emph{dynamic system}, where $A$ is interpreted as the set of the states of the system and $\sigma_A: A \to A$ as its dynamism. Let $(A, \sigma_A)$ and $(B, \sigma_B)$ be two dynamic systems. A map $f: A \to B$ is called \emph{equivariant} if $f \circ \sigma_A=\sigma_B \circ f$:
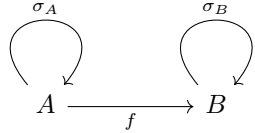
\[\begin{tikzcd}[ampersand replacement=\&]
	A \&\& B
	\arrow["{\sigma_A}", from=1-1, to=1-1, loop, in=50, out=130, distance=15mm]
	\arrow["f"', from=1-1, to=1-3]
	\arrow["{\sigma_B}", from=1-3, to=1-3, loop, in=50, out=130, distance=15mm]
\end{tikzcd}\]
The equivariant maps are simply the functions that respect the dynamism, as we expect from any map between dynamic systems.
The collection of all dynamic systems as the objects and equivariant maps as the morphisms with the evident composition and identity constitute a category, denoted by $\mathbf{Set}^{(\mathbb{N}, +)}$ or $\mathbf{Set}^{\mathbb{N}}$, for short. (We will see the motivation of this rather strange denotation later). If we restrict the morphisms to equivariant maps that are also bijections, we reach the category of \emph{reversible dynamic systems}, denoted by $\mathbf{Set}^{(\mathbb{Z}, +)}$ or $\mathbf{Set}^{\mathbb{Z}}$, for short. It is possible to apply the same idea to define dynamic topological systems, dynamic monoids, and so on, using a topological space with one continuous map over it, a monoid with a homomorphism over it, and so on. Moreover, it is worth mentioning that a reversible dynamic system $(A, \sigma_A)$ can be read as a proposition. The set $A$ stores all possible proofs of that proposition, the bijection $\sigma_A$ is a way to make the proof $a \in A$ identical to $\sigma_A(a)$ and we want to consider the proofs up to that identity. In this sense, an equivariant map $f: (A, \sigma_A) \to (B, \sigma_B)$ is just a function to transform the proofs of $A$ to the proofs of $B$, respecting the newly put symmetry on the proofs.
\end{exam}

\begin{exam}
Let $\mathcal{L}_p=\{\top, \bot, \wedge, \vee, \to\}$ be the propositional language and $A$ and $B$ be two formulas in $\mathcal{L}_p$. By a morphism $f: A \to B$, we mean a derivation of $B$ from the assumption $A$ in $\mathbf{NJ}$, up to $\beta\eta$-equivalence. The collection of all the formulas in $\mathcal{L}_p$ and morphisms with the identity on $A$ as the trivial derivation of $A$ from itself and composition as simply putting derivations one after another is a category. We denote this category by $\mathbf{NJ}$. In a similar fashion, one can restrict oneself to the fragments $\mathcal{L}_p-\{\bot\}$ and $\mathcal{L}_p-\{\bot, \vee\}$ and define the categories $\mathbf{NM}$ and $\mathbf{NN}$, respectively. 
\end{exam}

We conclude this subsection by presenting some simple methods for constructing new categories from existing ones.

\begin{exam}(\textit{Opposite category})
Let $\mathcal{C}$ be a category. By its dual (opposite), denoted by $\mathcal{C}^{op}$, we mean a category with the same collections of objects and morphisms as of $\mathcal{C}$ with the source and the target assignments swapped and $f \circ' g=g \circ f$:
\[\begin{tikzcd}[ampersand replacement=\&]
	\& A \&\&\&\& A \\
	\\
	B \&\& C \&\& B \&\& C \\
	\& {\mathcal{C}} \&\&\&\& {\mathcal{C}^{op}}
	\arrow["f"', from=1-2, to=3-1]
	\arrow["h", from=1-2, to=3-3]
	\arrow["g"', from=3-1, to=3-3]
	\arrow["f", from=3-5, to=1-6]
	\arrow["h"', from=3-7, to=1-6]
	\arrow["g", from=3-7, to=3-5]
\end{tikzcd}\]
\end{exam}

\begin{exam}(\textit{Product of categories})
Let $\mathcal{C}$ and $\mathcal{D}$ be two categories. By $\mathcal{C} \times \mathcal{D} $, we mean the category with the pairs $(C, D)$ as the objects, where $C$ and $D$ are the objects of $\mathcal{C}$ and $\mathcal{D}$, respectively, and the pairs $(f, g): (C, D) \to (E, F)$ as the morphisms, where $f: C \to E$ and $g: D \to F$ are morphisms in  $\mathcal{C}$ and $\mathcal{D}$, respectively.
Note that this construction generalizes the product of monoids and posets, on the one hand, and the product of sets as discrete small categories, on the other.
\end{exam}

\begin{exam}(\textit{Preorder and Poset reflection})
Let $\mathcal{C}$ be a small category. By its \emph{preorder reflection}, we mean the preordered set $(\mathrm{Ob}(\mathcal{C}), \leq)$ where $A \leq B$ iff $\mathrm{Hom}(A, B)$ is non-empty. The \emph{poset reflection} of $\mathcal{C}$, denoted by $\mathrm{Po}(\mathcal{C})$, is the poset of the preorder reflection, i.e., the objects are the equivalence classes of the objects of the preorder reflection under the equivalence $A \sim B$ if $A \leq B$ and $B \leq A$ while the order is just the order induced by the preorder reflection. 
\end{exam}

\subsection{Functors}
Reading a category as a proof system, one naturally needs to formalize a transformation between proof systems as a map that respects the compositional structure of deductions.
\begin{dfn}(\textit{Functors})
Let $\mathcal{C}$ and $\mathcal{D}$ be two categories. By a \emph{functor} $F: \mathcal{C} \to \mathcal{D}$ we mean a pair of two assignments $F_0$ and $F_1$, such that $F_0$ maps any object $A$ of $\mathcal{C}$ to an object of $\mathcal{D}$, denoted by $F_0(A)$, and $F_1$ maps any morphism $f$ of $\mathcal{C}$ to a morphism in $\mathcal{D}$, denoted by $F_1(f)$, respecting the source, target, identity and composition operations as depicted in the following schematic diagram:
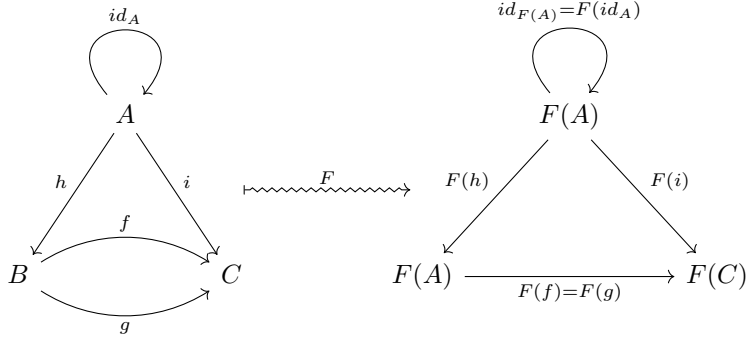
\[\begin{tikzcd}[ampersand replacement=\&]
	\& A \&\&\&\& {F(A)} \\
	\&\& {} \&\& {} \\
	B \&\& C \&\& {F(A)} \&\& {F(C)}
	\arrow["{id_A}", from=1-2, to=1-2, loop, in=50, out=130, distance=15mm]
	\arrow["h"', from=1-2, to=3-1]
	\arrow["i", from=1-2, to=3-3]
	\arrow["{id_{F(A)}=F(id_A)}", from=1-6, to=1-6, loop, in=50, out=130, distance=15mm]
	\arrow["{F(h)}"', from=1-6, to=3-5]
	\arrow["{F(i)}", from=1-6, to=3-7]
	\arrow["F", squiggly, maps to, from=2-3, to=2-5]
	\arrow["f", curve={height=-18pt}, from=3-1, to=3-3]
	\arrow["g"', curve={height=18pt}, from=3-1, to=3-3]
	\arrow["{F(f)=F(g)}"', from=3-5, to=3-7]
\end{tikzcd}\]
Usually, for simplicity, one drops the subscripts in $F_0$ and $F_1$ and denotes both of the assignments by $F$.
\end{dfn}
\begin{phil}
It is possible to read a functor $F: \mathcal{C} \to \mathcal{D}$ as a way to interpret the discourse $\mathcal{C}$ into the discourse $\mathcal{D}$, as a way to realize $\mathcal{C}$ in $\mathcal{D}$, as a $\mathcal{C}$-indexed family in $\mathcal{D}$ or a $\mathcal{C}$-variable object in $\mathcal{D}$. When we interpret $\mathcal{C}$ and $\mathcal{D}$ as two proof systems, $F$ may be read as an interpretation of the first proof system into the second. 
\end{phil}
\begin{exam}
Homomorphisms between monoids and order-preserving maps between preordered sets are examples of functors. In fact, as categories are a common generalization of preordered sets and monoids, functors are also a common generalization of monoid homomorphisms and order-preserving maps.
\end{exam}

\begin{exam}
The assignment mapping any set $A$ to its powerset $P(A)$ and any function $f: A \to B$ to the function $P(f): P(A) \to P(B)$, defined by $P(f)(S)=f[S]=\{f(a) \mid a \in S \}$ is a functor from $\mathbf{Set}$ to itself. Similarly, the functor $P^{\circ}: \mathbf{Set}^{op} \to \mathbf{Set}$, mapping any set $A$ to its powerset $P(A)$ and any  map $f: B \to A$ in $\mathbf{Set}^{op}$ (i.e., a function $f: A \to B$) to the function $P^{\circ}(f): P(B) \to P(A)$, defined by $P^{\circ}(f)(S)=f^{-1}(S)=\{a \in A \mid f(a) \in S\}$ is a functor.
\end{exam}

\begin{exam}
The assignment mapping any object $(A, B)$ in $\mathbf{Set} \times \mathbf{Set}$ to the set $A \times B$ and any morphism $(f, g): (A, B) \to (C, D)$ of $\mathbf{Set} \times \mathbf{Set}$ to the function $f \times g: A \times B \to C \times D$ defined by $[f \times g](a, b)=(f(a), g(b))$ is a functor, denoted by $(-)\times (-): \mathbf{Set} \times \mathbf{Set} \to \mathbf{Set}$. Reading proof-theoretically, we can interpret this functor as the \emph{conjunction operation} encoding the idea that a proof of a conjunction is a pair of the proofs of each component.
\end{exam}

\begin{exam}
The assignment mapping any object $(A, B)$ in $\mathbf{Set} \times \mathbf{Set}$ to the set $A + B=\{(0, a) \mid a \in A \} \cup \{(1, b) \mid b \in B\}$ and any morphism $(f, g): (A, B) \to (C, D)$ of $\mathbf{Set} \times \mathbf{Set}$ to the function $f + g: A + B \to C + D$ defined by $[f + g](0, a)=(0, f(a))$ and $[f + g](1, b)=(1, g(b))$ is a functor, denoted by $(-)+(-): \mathbf{Set} \times \mathbf{Set} \to \mathbf{Set}$. Reading proof-theoretically, we can interpret this functor as the \emph{disjunction operation} encoding the idea that a proof of a disjunction is a pair, where the first component states which of the disjuncts we are proving and the second component provides the proof.
\end{exam}

\begin{exam}
Let $A$ be a fixed set. Define the assignment $(-)^A : \mathbf{Set} \to \mathbf{Set}$, mapping the set $B$ to $B^A$ i.e., the set of all functions from $A$ to $B$ and mapping the function $f: B \to C$ to the function $f^A : B^A \to C^A$, defined by $f^A(g)=fg$. Then, $(-)^A: \mathbf{Set} \to \mathbf{Set}$ is a functor, generalizing the finite power functor $A \mapsto A^n$ generated by the iteration of the product functor. Similarly, it is possible to define the assignment $A^{(-)}: \mathbf{Set}^{op} \to \mathbf{Set}$, mapping the set $B$ to the set $A^B=\{f : B \to A\}$ and the map $f: C \to B$ in $\mathbf{Set}^{op}$ (i.e., the function $f: B \to C$) to the function $A^f : A^C \to A^B$, defined by $A^f(g)=gf$. Then, $A^{(-)}: \mathbf{Set}^{op} \to \mathbf{Set}$ is a functor, generalizing the functor $P^{\circ}=2^{(-)}$. It is possible to unify these two functors in one functor $(-)^{(-)}: \mathbf{Set}^{op} \times \mathbf{Set} \to \mathbf{Set}$ mapping $(B, C)$ to $C^B$ and $(f, g)$ to the function $g^f: C_1^{B_1} \to C_2^{B_2}$ defined by $g^f(h)=ghf$, where $f: B_2 \to B_1$ and $g: C_1 \to C_2$ are two functions. Reading proof-theoretically, we can interpret this functor as the \emph{implication operation} encoding the idea that a proof of the implication $B \to C$ is a function mapping any proof of $B$ to a proof of $C$.
\end{exam}

\begin{exam}(\textit{Forgetful functors})
It is customary to have a category of some structures and then, for some reason, forget some of the structures the objects possess and the maps preserve. This way, one can think of the same data as it lives in a richer context. Functors can formalize such situations, and in these settings, they are called \emph{forgetful functors} as they forget some parts of their input data. Let us provide three examples of such situations. First, the assignment $U$ mapping any topological space $X$ and any continuous map $f: X \to Y$ in $\mathbf{Top}$ to themselves in $\mathbf{Set}$ is a functor. The functor $U$ forgets that there is a topology on the spaces that the continuous maps respect. For the second example, for any $i \in \{0, 1\}$, define the assignment $\pi_i : \mathbf{Set}^{\mathbf{2}} \to \mathbf{Set}$, by mapping any $f: A_0 \to A_1$ to $A_i$ and any $\alpha: f \to g$ to $\alpha_i: A_i \to B_i$, where $f: A_0 \to A_1$ and $g: B_0 \to B_1$. Both $\pi_0$ and $\pi_1$ are functors. By making two snapshots of a variable set in the two possible moments, $\pi_0$ and $\pi_1$ forget that the variable set actually varies. For the third, consider the assignment $U: \mathbf{Set}^{\mathbb{N}} \to \mathbf{Set}$ mapping any dynamic system to its set and any equivariant map to itself. The assignment $U$ is a functor that forgets the dynamism of the dynamic space.
\end{exam}

Conversely to the idea behind the forgetful functors, sometimes we want to add more structure to the existing object, and of course, there might be many ways to add these structures. In the following, we will see some examples of such a situation.

\begin{exam}
Define the assignment $Disc: \mathbf{Set} \to \mathbf{Top}$ mapping the set $X$ to the discrete space $X$ (i.e., where all subsets are open) and the function $f: X \to Y$ to itself. This assignment is a functor. Similarly, if we define $InDisc: \mathbf{Set} \to \mathbf{Top}$, mapping the set $X$ to the indiscrete space $X$ (i.e., where the opens are $\varnothing$ and $X$) and the function $f: X \to Y$ to itself, we also obtain a functor. 
\end{exam}

\begin{exam}
Define the assignment $St: \mathbf{Set} \to \mathbf{Set}^{\mathbb{Z}}$ mapping the set $X$ to the \emph{static} dynamic system $(X, id_X)$ and the function $f: X \to Y$ to itself. This assignment is a functor.
\end{exam}

There are several functors that are generally available over arbitrary categories. Here are a few examples.

\begin{exam}
Let $\mathcal{C}$, $\mathcal{D}$ and $\mathcal{E}$ be some categories. Then, note that the identity functor $id_{\mathcal{C}} : \mathcal{C} \to \mathcal{C}$ mapping any object and morphism of $\mathcal{C}$ to itself is a functor. Moreover, if $F: \mathcal{D} \to \mathcal{E}$ and $G: \mathcal{C} \to \mathcal{D}$ are two functors, then the composition $FG: \mathcal{C} \to \mathcal{E}$ with its evident definition is also a functor. For any object $D$ of $\mathcal{D}$, the assignment $\Delta_D: \mathcal{C} \to \mathcal{D}$ mapping all objects of $\mathcal{C}$ to $D$ and all maps of $\mathcal{C}$ to $id_D$ is a functor.
\end{exam}

\begin{exam}
Let $\mathcal{C}$ be a locally small category. The assignment $\mathrm{Hom}_{\mathcal{C}}: \mathcal{C}^{op} \times \mathcal{C} \to \mathbf{Set}$, defined by $\mathrm{Hom}_{\mathcal{C}}(A, B)=\{f: A \to B \mid f \in \mathrm{Mor}(\mathcal{C})\}$ and $\mathrm{Hom}_{\mathcal{C}}(g, h) : \mathrm{Hom}_{\mathcal{C}}(A, B) \to \mathrm{Hom}_{\mathcal{C}}(C, D)$ as $\mathrm{Hom}_{\mathcal{C}}(g, h)(f)=hfg$, for any maps $f: A \to B$, $g: C \to A$ and $h: B \to D$ in $\mathcal{C}$, is a functor. This functor captures the whole structure of the category $\mathcal{C}$ and is called the Hom functor.
\end{exam}

\begin{exam}
Let $\mathcal{C}$ be a locally small category. For any object $A$ in $\mathcal{C}$, there is a canonical functor $y^A=\mathrm{Hom}_{\mathcal{C}}(A, - ): \mathcal{C} \to \mathbf{Set}$, capturing the behavior of the maps above $A$. It is defined by $y^A(B)=\mathrm{Hom}_{\mathcal{C}}(A, B)$ and $y^A(f): \mathrm{Hom}_{\mathcal{C}}(A, B) \to \mathrm{Hom}_{\mathcal{C}}(A, C)$ as $y^A(g)=fg$, for any $f: B \to C$. Similarly, there is a canonical functor $y_A=\mathrm{Hom}_{\mathcal{C}}(-, A): \mathcal{C}^{op} \to \mathbf{Set}$, capturing the behavior of the maps below $A$. It is defined by $y_A(B)=\mathrm{Hom}_{\mathcal{C}}(B, A)$ and $y_A(f) : \mathrm{Hom}_{\mathcal{C}}(C, A) \to \mathrm{Hom}_{\mathcal{C}}(B, A)$ as $y_A(f)(g)=gf$, for any $f: B \to C$.
These functors are just the Hom functor where we fix one of the components.
\end{exam}

\begin{exam}
Let $\mathcal{C}$ be a small category. Define the assignment $\pi: \mathcal{C} \to \mathrm{Po}(\mathcal{C})$, mapping an object $A$ to its class $[A]$ and a map $f: A \to B$ to the unique map from $[A]$ to $[B]$. This assignment is a functor. From a proof-theoretical perspective, we can interpret this functor as the \emph{collapsing operation}, which forgets all deductions between two propositions and retains only the record of their mere existence, encoded in the deductibility relation.
\end{exam}

To conclude this subsection, we will introduce three families of functors that will be important later in this chapter.

\begin{dfn}
A functor $F: \mathcal{C} \to \mathcal{D}$ is called \emph{faithful} if for any two morphisms $f, g: A \to B$ in $\mathcal{C}$, if $F(f)=F(g)$, then $f=g$.
\end{dfn}

\begin{exam}
Any functor with a preordered set as its domain is faithful, as in any preordered set, there is at most one map between any two objects. A homomorphism between two monoids is faithful iff it is injective.
The forgetful functors from $ \mathbf{Top}$ and $\mathbf{Set}^{\mathbb{N}}$ to $\mathbf{Set}$ are faithful, simply because they map the continuous functions and the equivariant maps to themselves. The functors $Disc, InDisc: \mathbf{Set} \to \mathbf{Top}$ and $St: \mathbf{Set} \to \mathbf{Set}^{\mathbb{Z}}$ are all faithful, as they also do not change the functions. The functor $\pi_0: \mathbf{Set}^{\mathbf{2}} \to \mathbf{Set}$ is not faithful, as it forgets the value of the variable set at the point $t=1$. More precisely, consider the inclusion function $i: \{0\} \to \{0, 1\}$ as an object of $\mathbf{Set}^{\mathbf{2}}$ and define the map $\alpha: i \to i$ by setting $\alpha_0: \{0\} \to \{0\}$ as the identity function and $\alpha_1: \{0, 1\} \to \{0, 1\}$ as the switch function, i.e., $\alpha_1(0)=1$ and $\alpha_1(1)=0$:
\[\begin{tikzcd}[ampersand replacement=\&]
	{\{0, 1\}} \&\& {\{0, 1\}} \\
	\\
	{\{0\}} \&\& {\{0\}}
	\arrow["i", from=3-1, to=1-1]
	\arrow["i"', from=3-3, to=1-3]
	\arrow["{\alpha_0}"', from=3-1, to=3-3]
	\arrow["{\alpha_1}", from=1-1, to=1-3]
\end{tikzcd}\]
It is clear that $\pi_0(\alpha)=\alpha_0=\pi_0(id_i)$, but $\alpha \neq id_i$ as $\alpha_1 \neq id_{\{0, 1\}}$.
\end{exam}

\begin{dfn}
A functor $F: \mathcal{C} \to \mathcal{D}$ is called \emph{full} if for any morphism $g: F(A) \to F(B)$ in $\mathcal{D}$, there exists a map $f: A \to B$ such that $F(f)=g$. It is called \emph{weakly-full} if the non-emptiness of $\mathrm{Hom}_{\mathcal{D}}(F(A), F(B))$ implies the non-emptiness of $\mathrm{Hom}_{\mathcal{C}}(A, B)$.
\end{dfn}

\begin{exam}\label{ExamOfFull}
An order-preserving map between two posets is weakly-full iff it is full iff it is an order-embedding. A homomorphism between two monoids is always weakly-full but it is full iff it is surjective. 
The forgetful functors from $ \mathbf{Top}$ and $\mathbf{Set}^{\mathbb{N}}$ to $\mathbf{Set}$ are not full as there are functions that are not continuous or equivariant. However, these forgetful functors are both weakly-full. The forgetful functor $\pi: \mathcal{C} \to \mathrm{Po}(\mathcal{C})$ is full and hence weakly-full. The functors $Disc, Indisc : \mathbf{Set} \to \mathbf{Top}$ are trivially full. 
The functor $\pi_0: \mathbf{Set}^{\mathbf{2}} \to \mathbf{Set}$ is not full, because not all functions give rise to a map between variable sets. 
For instance, consider the inclusion function $!: \{0, 1\} \to \{0\}$ and the identity function $id_{\{0, 1\}}: \{0, 1\} \to \{0, 1\}$ as the objects of $\mathbf{Set}^{\mathbf{2}}$ and let $\alpha_0=id_{\{0, 1\}}:  \{0, 1\} \to \{0, 1\}$ be the identity function: 
\[\begin{tikzcd}[ampersand replacement=\&]
	{\{0\}} \&\& {\{0, 1\}} \\
	\\
	{\{0, 1\}} \&\& {\{0, 1\}}
	\arrow["{!}", from=3-1, to=1-1]
	\arrow["{id_{\{0, 1\}}}"', from=3-3, to=1-3]
	\arrow["{\alpha_0=id_{\{0, 1\}}}"', from=3-1, to=3-3]
	\arrow["{\alpha_1}", from=1-1, to=1-3]
\end{tikzcd}\]
It is clear that there is no function $\alpha_1: \{0\} \to \{0, 1\}$ to make the diagram commutative and hence $\alpha_0$ is not $\pi_0(\alpha)$, for any map $\alpha: \, ! \to id_{\{0, 1\}}$.
\end{exam}

\begin{phil}
We noted that from a proof-theoretic perspective, functors should be understood as interpretations between proof systems. In this context, faithful functors can be viewed as faithful interpretations that preserve the structure of deductions. Full functors, on the other hand, are interpretations that do not introduce any new deductions between propositions. Weakly full functors may add some new deductions, but they do not render non-deductible propositions deductible. In other words, they do not alter the logic of the proof system.
\end{phil}

\subsection{Natural Transformations}
Reading a category as a proof system and a functor as an interpretation, one can go one step further to think of the possible formalization of the transformations between the interpretations. 
\begin{dfn}(\textit{Natural transformations})
Let $\mathcal{C}$ and $\mathcal{D}$ be two categories and $F, G: \mathcal{C} \to \mathcal{D}$ be two functors. By a \emph{natural transformation} $\alpha: F \Rightarrow G$, depicted as:

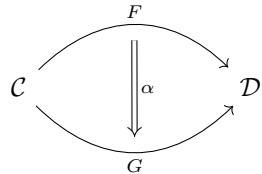
\[\begin{tikzcd}
	{\mathcal{C}} &&& {\mathcal{D}}
	\arrow["{G}"{name=0, swap}, from=1-1, to=1-4, curve={height=30pt}]
	\arrow["{F}"{name=1}, from=1-1, to=1-4, curve={height=-30pt}]
	\arrow[Rightarrow, "{\alpha}", from=1, to=0, shorten <=6pt, shorten >=6pt]
\end{tikzcd}\]

\noindent we mean an assignment mapping any object of $\mathcal{C}$ to a morphism $\alpha_C: F(C) \to G(C)$ in $\mathcal{D}$ such that for any morphism $f: A \to B$ in $\mathcal{C}$, the following diagram commutes:

\[\begin{tikzcd}
	{F(A)} && {G(A)} \\
	\\
	{F(B)} && {G(B)}
	\arrow["{\alpha_A}", from=1-1, to=1-3]
	\arrow["{F(f)}"', from=1-1, to=3-1]
	\arrow["{G(f)}", from=1-3, to=3-3]
	\arrow["{\alpha_B}"', from=3-1, to=3-3]
\end{tikzcd}\]
\end{dfn}

\begin{phil}
Reading the functors $F, G: \mathcal{C} \to \mathcal{D}$ as $\mathcal{C}$-variable objects in $\mathcal{D}$, a natural transformation $\alpha: F \Rightarrow G$ serves as a natural candidate for a transformation between these variable objects. Naturally, any transformation between variable objects must specify the way we change the object $F(C)$ to the object $G(C)$ in $\mathcal{D}$, for each parameter $C \in \mathrm{Ob}(\mathcal{C})$. These changes cannot be arbitrary; they must respect the variation of the parameter $C$ in $\mathcal{C}$, as illustrated in the following schematic diagram:
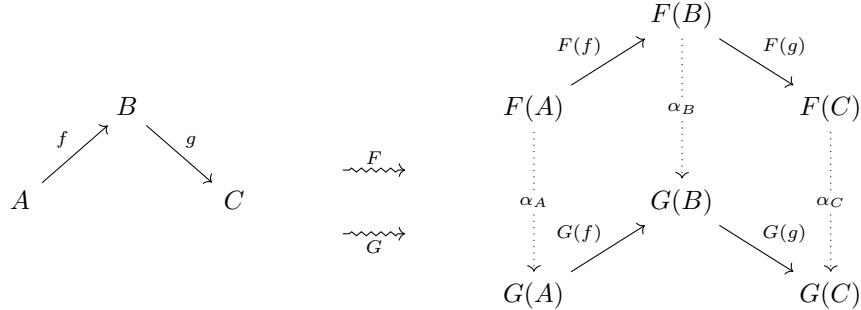
\[\begin{tikzcd}
	&&&&&& {F(B)} \\
	& {B} &&&& {F(A)} && {F(C)} \\
	{A} && {C} & {} & {} && {G(B)} \\
	&&&&& {G(A)} && {G(C)}
	\arrow["{f}", from=3-1, to=2-2]
	\arrow["{g}", from=2-2, to=3-3]
	\arrow["{F(f)}", from=2-6, to=1-7]
	\arrow["{F(g)}", from=1-7, to=2-8]
	\arrow["{G(f)}", from=4-6, to=3-7]
	\arrow["{G(g)}", from=3-7, to=4-8]
	\arrow["{\alpha_A}" description, from=2-6, to=4-6, dotted]
	\arrow["{\alpha_B}" description, from=1-7, to=3-7, dotted]
	\arrow["{\alpha_C}" description, from=2-8, to=4-8, dotted]
	\arrow["{F}", from=3-4, to=3-5, shift left=5, squiggly]
	\arrow["{G}"', from=3-4, to=3-5, shift right=5, squiggly]
\end{tikzcd}\]
It is intuitively helpful to think of a natural transformation as a way of changing $F(C)$ to $G(C)$, \emph{smoothly} in the parameter $C$.

It is also possible to interpret the functors $F, G: \mathcal{C} \to \mathcal{D}$ as two \emph{construction methods} that take an object in $\mathcal{C}$ and transform it into an object in $\mathcal{D}$. When can we consider $F$ and $G$ to be ``equal" as methods of construction? We certainly do not want to restrict ourselves to strict equality, which requires the functors to be the same on both objects and morphisms, as this is overly restrictive. For instance, consider $F, G: \mathbf{Set} \to \mathbf{Set}$ as $F(A)=A \times \{0\}$ and $G(A)=A \times \{1\}$. Although $F$ and $G$ are not strictly equal, they must be considered as the same methods of construction, as they are only different up to an isomorphism. Using this criterion, one natural candidate for the intended equality between $F$ and $G$ is the existence of an isomorphism between $F(A)$ and $G(A)$, for any object $A$ in $\mathcal{C}$. However, it is clear that any random assignment of isomorphisms between $F(A)$ and $G(A)$ does not suffice; the isomorphisms must be assigned in a \textit{uniform} way, as we want $F$ and $G$ to be equal as two methods of construction not two mere structureless assignments. This uniformity requires that the isomorphisms be somewhat independent of the specific choice of the object $A$. Of course, one may object that the isomorphisms clearly depend on the object $A$ (the source and the target of the isomorphism, for instance), but at the same time it is intuitively meaningful to talk about the constructions that apply the \textit{same} method to \textit{different} objects. An example may be more illuminating. Consider the canonical isomorphism $s_{A, B}: A \times B \to B \times A$ defined by $s_{A, B}(a, b)=(b, a)$ that shows the order in the product of two sets is not important. This map clearly depends on the choice of $A$ and $B$, but at the same time it is defined in a uniform way of ``swapping the elements in a pair" which does not use the sets in an essential way. Natural transformations is historically developed for the sole purpose of capturing this very intuition of uniformity.   
\end{phil}

\begin{exam}
The assignment $s: id_{\mathbf{Set}} \Rightarrow P$ defined by $s_A: A \to P(A)$ as $s_A(a)=\{a\}$ is a natural transformation. It is natural, because for any map $f: A \to B$, if $f$ maps $a \in A$ to $f(a) \in B$, then $P(f)$ maps $\{a\}$ to $f[\{a\}]=\{f(a)\}$:

\[\begin{tikzcd}
	{A} && {P(A)} && {a} && {\{a\}} \\
	\\
	{B} && {P(B)} && {f(a)} && {\{f(a)\}}
	\arrow["{f}"', from=1-1, to=3-1]
	\arrow["{f[-]}", from=1-3, to=3-3]
	\arrow["{\{-\}}", from=1-1, to=1-3]
	\arrow["{\{-\}}"', from=3-1, to=3-3]
	\arrow["{f}"', from=1-5, to=3-5, maps to]
	\arrow["{\{-\}}", from=1-5, to=1-7, maps to]
	\arrow["{\{-\}}"', from=3-5, to=3-7, maps to]
	\arrow["{f[-]}", from=1-7, to=3-7, maps to]
\end{tikzcd}\]

\end{exam}

\begin{exam}
Let $Ex: \mathbf{Set} \times \mathbf{Set} \to \mathbf{Set} \times \mathbf{Set}$ be the exchange functor, i.e, $Ex(A, B)=(B, A)$ and $Ex(f, g)=(g, f)$ and $(-) \times (-): \mathbf{Set} \times \mathbf{Set} \to \mathbf{Set}$ be the product functor. Then, the assignment $s: (-) \times (-) \Rightarrow [(-) \times (-) ] \circ Ex $ defined by $s_{(A, B)}: A \times B \to B \times A$ as $s_{A \times B}(a, b)=(b, a)$ is a natural transformation as for any map $(f, g): (A, B) \to (C, D)$ in $\mathbf{Set} \times \mathbf{Set}$, we have:
\[\begin{tikzcd}
	{A \times B} && {B \times A} && {(a, b)} && {(b, a)} \\
	\\
	{ C \times D} && {D \times C} && {(f(a), g(b))} && {(g(b), f(a))}
	\arrow["{ f \times g}"', from=1-1, to=3-1]
	\arrow["{g \times f}", from=1-3, to=3-3]
	\arrow["{s_{(A, B)}}", from=1-1, to=1-3]
	\arrow["{s_{(C, D)}}"', from=3-1, to=3-3]
	\arrow["{f \times g}"', from=1-5, to=3-5, maps to]
	\arrow["{s_{(A, B)}}", from=1-5, to=1-7, maps to]
	\arrow["{s_{(C, D)}}"', from=3-5, to=3-7, maps to]
	\arrow["{g \times f}", from=1-7, to=3-7, maps to]
\end{tikzcd}\]
\end{exam}

\begin{exam}\label{NatforPoset}
Let $(P, \leq_P)$ and $(Q, \leq_Q)$ be posets and $F, G: (P, \leq_P) \to (Q, \leq_Q)$ be two order-preserving maps. Then, there is at most one natural transformation $\alpha : F \Rightarrow G$, as there is at most one map from $F(p)$ to $G(p)$, for any $p \in P$. This unique natural transformation exists iff $F(p) \leq_{Q} G(p)$, for any $p \in P$:
\[\begin{tikzcd}[ampersand replacement=\&]
	{F(p)} \&\& {G(p)} \\
	\\
	{F(q)} \&\& {G(q)}
	\arrow["{\alpha_p}", shift left=2, draw=none, from=1-1, to=1-3]
	\arrow["{\alpha_q}"', shift right=3, draw=none, from=3-1, to=3-3]
	\arrow["{\huge{\leq}}"{marking}, draw=none, from=1-1, to=3-1]
	\arrow["{\huge{\leq}}"{marking}, draw=none, from=1-3, to=3-3]
	\arrow["{\huge{\leq}}"{marking}, draw=none, from=1-1, to=1-3]
	\arrow["{{\huge \leq}}"{marking}, draw=none, from=3-1, to=3-3]
\end{tikzcd}\]
A similar phenomenon occurs when we replace $(P, \leq_P)$ with any other category. Specifically, for any category $\mathcal{C}$ and any functors $F, G: \mathcal{C} \to (Q, \leq_Q)$, there is at most one natural transformation $\alpha: F \Rightarrow G$, and it exists if $F(A) \leq_Q G(A)$ for every object $A$ in $\mathcal{C}$.
\end{exam}

\begin{exam}
Let $M$ and $N$ be two monoids and $F, G: M \to N$ be two homomorphisms. Then, a natural transformation $\alpha : F \Rightarrow G$, by definition, is an element $\alpha_*= n \in N$ such that $nF(m)=G(m)n$, for any $m \in M$:
\[\begin{tikzcd}[ampersand replacement=\&]
	{F(*)} \&\& {G(*)} \\
	\\
	{F(*)} \&\& {G(*)}
	\arrow["{\alpha_*=n}", from=1-1, to=1-3]
	\arrow["{G(m)}", from=1-3, to=3-3]
	\arrow["{F(m)}"', from=1-1, to=3-1]
	\arrow["{\alpha_*=n}"', from=3-1, to=3-3]
\end{tikzcd}\]
\end{exam}

\begin{exam}\label{NaturalYoneda}
Let $\mathcal{C}$ be a category and $f: A \to B$ be a map. Then, the assignment $y_f: \mathrm{Hom}(-, A) \to \mathrm{Hom}(-, B)$ defined by $(y_f)_C=f \circ (-)$ is a natural transformation as for any map $g: C \to D$:
\[\begin{tikzcd}
	{\mathrm{Hom}(D, A)} && {\mathrm{Hom}(D, B)} \\
	\\
	{\mathrm{Hom}(C, A)} && {\mathrm{Hom}(C, B)}
	\arrow["{f \circ (-)}", from=1-1, to=1-3]
	\arrow["{f \circ (-)}"', from=3-1, to=3-3]
	\arrow["{(-) \circ g}"', from=1-1, to=3-1]
	\arrow["{(-) \circ g}", from=1-3, to=3-3]
\end{tikzcd}\]
\end{exam}

\begin{exam}(\textit{Non-natural transformations})
Let $\alpha$ be an assignment choosing a function $\alpha_A: A \to A$ for any set $A$. Then $\alpha: id_{\mathbf{Set}} \Rightarrow id_{\mathbf{Set}}$ is a natural transformation iff $\alpha_A=id_A$, for any set $A$. It is clear that setting $\alpha_A=id_A$ for any set $A$ makes $\alpha$ a natural transformation. For the converse, assume $\alpha$ is a natural transformation and consider the following commutative diagram:
\[\begin{tikzcd}
	{\{0\}} && {\{0\}} \\
	\\
	A && A
	\arrow["{\alpha_{\{0\}}}", from=1-1, to=1-3]
	\arrow["\hat{a}"', from=1-1, to=3-1]
	\arrow["\hat{a}", from=1-3, to=3-3]
	\arrow["{\alpha_A}"', from=3-1, to=3-3]
\end{tikzcd}\]
where $a \in A$ is an arbitrary element and $\hat{a}(0)=a$. It is clear that $\alpha_{\{0\}}=id_{\{0\}}$. Hence, $\alpha_A \circ \hat{a}=\hat{a}$. Applying both sides on $0$, we have $\alpha_A(a)=a$. As $a \in A$ is arbitrary, we reach $\alpha_A=id_A$. Using the above characterization, it is enough to pick assignments that are not identity over at least one $A$ to have non-natural transformations. 
\end{exam}
With functors and natural transformations at our disposal, we can generalize the notion of a variable set by generalizing the underlying temporal structure from the poset $\{0 \leq 1\}$ to any small category:
\begin{exam}(\emph{Variable sets over $\mathcal{C}$})
Let $\mathcal{C}$ be a small category. Then, a \emph{variable set} over $\mathcal{C}$ is a functor $F: \mathcal{C} \to \mathbf{Set}$. Variable sets over $\mathcal{C}$ with natural transformations and canonical identity and composition form a category denoted by $\mathbf{Set}^{\mathcal{C}}$. The exponentiation-like notation suggests that  $\mathbf{Set}^{\mathcal{C}}$ collects all functors from $\mathcal{C}$ to $\mathbf{Set}$ similar to collecting all functions from $A$ to $B$ in $B^A$.
It is evident that if $\mathcal{C}$ is the category $\mathbf{2}$, then a variable set over $\mathcal{C}$ corresponds to the notion of a variable set we had previously established. Additionally, when $\mathcal{C}$ is taken to be the monoid $(\mathbb{N}, +)$ or $(\mathbb{Z}, +)$, variable sets over $\mathcal{C}$ represent dynamic systems and reversible dynamic systems, respectively. This observation not only unifies these diverse categories under one notion of variable sets but also justifies the notation that we employ for them.

One can use the Hom functor to map any small category $\mathcal{C}$ to $\mathbf{Set}^{\mathcal{C}^{op}}$. More precisely, define $y: \mathcal{C} \to \mathbf{Set}^{\mathcal{C}^{op}}$ by $y_A=\mathrm{Hom}(-, A)$ on objects and $y_f: \mathrm{Hom}(-, A) \to \mathrm{Hom}(-, B)$ as defined in Example \ref{NaturalYoneda} on morphisms. It is easy to check that $y$ is actually a functor. This functor plays a crucial role in category theory and is known as the \emph{Yoneda embedding}. It is known that $y$ is a full and faithful functor. Therefore, one can always pretend that any small category is actually a category of variable sets. 
\end{exam}

\begin{phil}
One of the prominent foundational paradigms in mathematics is Brouwerian intuitionism. Among many other things, the paradigm believes that mathematics is a mental narrative constructed by a creative subject for herself. Like any narrative, this mathematical story evolves over time as new constructions are added and new properties are proved. In this context, truth in mathematics is viewed as temporal and dynamic, which can be effectively represented by our variable sets in $\mathbf{Set}^{\mathcal{C}}$, where $\mathcal{C}$ is a suitable category that captures the progression of time. There are many valid formalizations of this notion of time. For instance, the simplest formalization that comes to mind is the set of all natural numbers and its usual order encoding the instances and the arrow of time. However, in this example, we focus on Brouwer's own understanding of time: 

\begin{quote}
\textit{This perception of a move of time may be described as the falling apart of a life moment into two distinct things, one of which gives way to the other, but is retained by memory. If the twoity thus born is divested of all quality, it passes into the empty form of the common substratum of all twoities. And it is this common substratum, this empty form, which is the basic intuition of mathematics} \cite{qBrouwerCamLec}.
\end{quote}

To formalize this notion of time, we use  $[n]=\{0, 1, \ldots, n-1\}$, for $n \geq 0$, as the objects to represent the $n$th moment of time and, for any $n \leq m$, we define the morphisms from $[n]$ to $[m]$ as a function $f: [m] \to [n]$, where $f(i)=i$, for any $i < n$. The equation  $f(i)=j$ represents the creation process of moments by encoding the fact that the moment $i$ has been created from the moment $j$. Therefore, the condition $f(i)=i$ just says that when we are at the $n$th moment, the moment $i < n$ is fixed throughout the creation process, and only the moments greater than or equal to $n$ are newly created. As it is expected, the category $\mathbf{Set}^{\mathcal{C}}$ leads to an interesting intuitionistic dynamic version of sets. What is surprising, though, is the fact that some of these variable growing sets are in some sense \textit{completed}, and the category of these completed sets satisfies all classical axioms of set theory except the axiom of choice. Hence, it can serve as a model to prove the unprovability of the axiom of choice from $\mathsf{ZF}$. The details are, of course, beyond the scope of this chapter. For more, see \cite{MacLaneMoerjdijk}.
    
\end{phil}

\section{Some Universal Constructions} \label{SecUnivCons}

Category theory is renowned for its \emph{relative approach}, in which the properties of an entity are defined by its relationships with other entities rather than by its internal components. For example, when we work with sets in a categorical context, we define the properties of a given set based on the morphisms (functions) that connect it to other sets, rather than inspecting the elements within the set itself.

As the relative approach completely ignores the internal structure of the entities, it is hardly clear whether it is actually capable of defining any non-trivial property. For instance, consider the task of defining the singleton set $\{0\}$ in $\mathbf{Set}$. An analysis reveals that such a task is impossible as any two isomorphic objects in a category exhibit indistinguishable relative behavior. Therefore, a definition for $\{0\}$ would also apply to $\{1\}$ which is impossible. The first lesson from this observation is that any purely relative definition overlooks the differences between isomorphic objects. If we choose to pursue a relative perspective, we must accept this limitation. With that observation, let us change the task from defining the set $\{0\}$ to defining ``being a singleton". Fortunately, we can see that this time, the definition is possible. A set $A$ is a singleton iff for any set $B$, there exists exactly one map from $B$ to $A$. To transform this modest observation to a general method of producing relative definitions, we must notice that our above definition of singletons is based on the informal intuition that singletons are the ``\emph{greatest}" objects with a certain property, i.e., the property of having a map from any object into them. Using this vague form of being greatest is not accidental. It is clear that if we must define everything relatively, then the only things we can define are the best objects satisfying a relative behavior with respect to some of the others. This being the best can take one of two forms: either it can be the least object or the greatest object with respect to some property. Such definitions are referred to as defining by a \emph{universal property}, and this section is dedicated to exploring these definitions.

\subsection{Cartesian Structure}
As we have already seen, singleton sets are defined by the property of having exactly one map from any arbitrary source into them. This is a crucial notion that is useful in many categories, so let us give it a name.
\begin{dfn}(\emph{Terminal objects})
An object $A$ is called \emph{terminal} if for any object $B$, there exists a \emph{unique} map from $B$ to $A$. This unique map is denoted by $!_B: B \to A$ and if $B$ is clear from the context, just by $!$.
\end{dfn}

\begin{exam}\label{ExampleOfTerminal}
In $\mathbf{FinSet}$ and $\mathbf{Set}$, an object is terminal iff it is a singleton. As we only care about the cardinality of the set and not its element, we usually denote these terminal objects by a common name $\{*\}$.  
In a poset $(P, \leq)$, a terminal object is by definition an element $a \in P$ such that for any $b \in P$, we have $b \leq a$. Hence, a terminal object is the greatest element of the poset if it exists. 
In $\mathbf{Set}^{\mathbb{N}}$ and $\mathbf{Set}^{\mathbb{Z}}$, a terminal object is a pair in the form $(\{*\}, id_{\{*\}})$. The same also holds for the category $\mathbf{Set}^{\mathcal{C}}$, where a terminal object is a constant functor $\Delta_{\{*\}}: \mathcal{C} \to \mathbf{Set}$. 
In $\mathbf{NJ}$, the proposition $\top$ is a terminal object. The reason is that first, there is always a derivation from $A$ to $\top$, for any proposition $A$. Second, for a fixed formula $A$, as we identified any two such derivations by the $\eta$-equivalence relation, the map from $A$ to $\top$ is also unique. The same claim also holds for $\mathbf{NN}$ and $\mathbf{NM}$.
\end{exam}

\begin{exam}(\textit{Non-existence of terminal objects})
Terminal objects do not necessarily exist. For instance, a poset without a greatest element such as $(\mathbb{N}, \leq)$ or a non-trivial monoid does not have a terminal object. As another example, the category of sets with at least two elements does not have a terminal object.
\end{exam}

\begin{rem}
Define the order $\leq$ on the objects of $\mathcal{C}$ by setting $B \leq A$ if there is a morphism $f: B \to A$. In this sense, a terminal object $A$ can be intuitively thought of as the ``greatest" object in the category, as there is a map from any object $B$ into $A$. However, we must use ``greatest" with caution, as being the greatest in this order is not sufficient to be a terminal object. There is an additional condition requiring the uniqueness of morphisms from $B$ to $A$, which is not captured by our \emph{informal} order-theoretic description.
\end{rem}

\begin{rem}
A category can have many terminal objects. For instance, all singletons in $\mathbf{Set}$ are terminal. However, up to an isomorphism, there is at most one terminal object. To prove this, let $A$ and $B$ be two terminal objects. Then, by the universal property of both $A$ and $B$, there are maps $f:A \to B$ and $g: B \to A$. Now, consider the maps $fg: B \to B$ and $gf: A \to A$. By the uniqueness condition for both $A$ and $B$, we have $fg=id_B$ and $gf=id_A$. Hence, $f:A \to B$ is an isomorphism. Now, as there is at most one terminal object up to an isomorphism and our relative approach sees everything up to isomorphisms, it is safe to denote all terminal objects by one reserved name, i.e., $1$, and call it \emph{the} terminal object. 
\end{rem}

\begin{phil}(\emph{Terminal objects as the trivial truth})
Reading any category as a proof system, a terminal object is a \emph{trivially true proposition}, i.e., a proposition with a unique deduction from any given assumption.
One may argue that although the intuition behind the uniqueness condition is clear for the singletons in $\mathbf{Set}$, the same claim for $\top$ is a bit artificial and this is even evident in $\mathbf{NJ}$, where we must use the apparently artificial $\eta$-equivalence to ensure this condition. To defend our interpretation, let us emphasize that the logical constant $\top$ is present in the language to formalize the \emph{trivial} truth. Whatever it means, the trivial truth must be provable as it is true and have exactly one proof as it is trivial. Note that if we forget about the triviality part, then there is no need for $\top$ in the language as there are many other provable statements like $p \to p$ playing the same role as $\top$ does. In fact,
in proof-irrelevant approaches, all provable statements are considered equivalent, while this should not hold in a proof-sensitive approach. For instance, the formula $p \to p$ and $\top$ must be considered as different formulas, because there are many deductions of $p$ from $p$ and hence many proofs of $p \to p$ while $\top$ has only one proof.
\end{phil}

The second universal construction we are interested in is the binary product. Before going any further, it is useful to pause and consider cartesian products in $\mathbf{Set}$, and to try to define them in a relative fashion without referring to their internal structure.

\begin{dfn}\label{DefProduct}(\emph{Binary products})
Let $A$ and $B$ be two objects in a category $\mathcal{C}$. An object $C$ together with two morphisms $p_0: C \to A$ and $p_1: C \to B$ is called a \emph{binary product} of $A$ and $B$ if for any object $D$ in $\mathcal{C}$ and any morphisms $f: D \to A$ and $g: D \to B$, there exists a \emph{unique} map $h: D \to C$ such that:
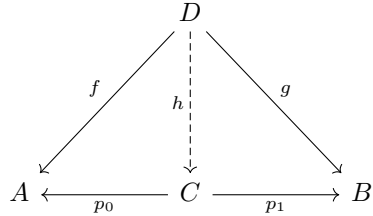
\[\begin{tikzcd}[ampersand replacement=\&]
	\&\& D \\
	\\
	\\
	A \&\& C \&\& B
	\arrow["{p_0}", from=4-3, to=4-1]
	\arrow["{p_1}"', from=4-3, to=4-5]
	\arrow["f"', from=1-3, to=4-1]
	\arrow["g", from=1-3, to=4-5]
	\arrow["h"', dashed, from=1-3, to=4-3]
\end{tikzcd}\]
The maps $p_0$ and $p_1$ are called the \emph{projections} of the binary product. The unique map $h$ is denoted by $\langle f, g \rangle$. 
Note that $p_0 \circ \langle f, g \rangle=f$ and  $p_1 \circ \langle f, g \rangle=g$. As we will mostly work with binary products in this chapter, whenever we say a product, unless specified otherwise, we always mean a binary product.
\end{dfn}

\begin{rem}\label{equivalentUniquenessProduct}
The uniqueness condition in Definition \ref{DefProduct} can be equivalently phrased as the identity $\langle p_0 h, p_1 h \rangle=h$, for any $h: D \to C$. To show the equivalence, it is clear that having the identity, we can reach the uniqueness, because if $p_0h=p_0h'=f$ and $p_1h=p_1h'=g$, then $h=\langle p_0 h, p_1 h \rangle=\langle p_0 h', p_1 h' \rangle=h'$. Conversely, if we have the uniqueness condition, set $h'=\langle p_0 h, p_1 h \rangle$. As $p_0h'=p_0h=f$ and $p_1h'=p_1h=g$, then by uniqueness, we have $h=h'$. Writing universal conditions as equalities is useful in practice. However, it is also conceptually important as it helps to construct the free categories with that universal structure, by first producing all possible required objects and maps freely and then making the quotient over the equalities to ensure that they have the properties they should.
\end{rem}

\begin{rem}\label{HighlevelProduct}
There is a high-level definition for binary products that is conceptually more elegant but computationally harder to work with. Here is the definition. A binary product of $A$ and $B$ is an object $C$ together with an isomorphism between the functors $\mathrm{Hom}(-, C): \mathcal{C}^{op} \to \mathbf{Set}$ and $\mathrm{Hom}(-, A) \times \mathrm{Hom}(-, B): \mathcal{C}^{op} \to \mathbf{Set}$ in the category $\mathbf{Set}^{\mathcal{C}^{op}}$. We do not prove the equivalence between the two definitions here. However, it is a good exercise in order to become fluent in the categorical language.
\end{rem}

\begin{rem}
Consider all diagrams of the shape: 
\[\begin{tikzcd}[ampersand replacement=\&]
	A \&\& D \&\& B
	\arrow["f", from=1-3, to=1-1]
	\arrow["g"', from=1-3, to=1-5]
\end{tikzcd}\]
where $A$ and $B$ are fixed and we change $D$ and the maps $f:D \to A$ and $g: D \to B$. Using maps between the middle objects, there is a natural order on these diagrams. In the following picture, we say that the lower diagram is greater than the upper one:
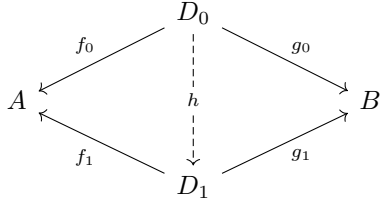
\[\begin{tikzcd}[ampersand replacement=\&]
	\&\& {D_0} \\
	A \&\&\&\& B \\
	\&\& {D_1}
	\arrow["{f_0}"', from=1-3, to=2-1]
	\arrow["{g_0}", from=1-3, to=2-5]
	\arrow["h"{description}, dashed, from=1-3, to=3-3]
	\arrow["{f_1}", from=3-3, to=2-1]
	\arrow["{g_1}"', from=3-3, to=2-5]
\end{tikzcd}\]
if there is $h:D_0 \to D_1$ making the whole diagram commutative. In this sense, the universal definition of binary product roughly states that the diagram
\[\begin{tikzcd}[ampersand replacement=\&]
	A \&\& C \&\& B
	\arrow["{p_0}", from=1-3, to=1-1]
	\arrow["{p_1}"', from=1-3, to=1-5]
\end{tikzcd}\]
is the ``greatest" among all the diagrams with the mentioned shape. Again, we put the ``greatest" in quotation, as being the greatest in the above order is not enough to be a product due to the missing uniqueness condition.
\end{rem}

\begin{exam}
In $\mathbf{FinSet}$ and $\mathbf{Set}$, the usual cartesian product $A \times B$ together with its projections $p_0: A \times B \to A$ and $p_1: A \times B \to B$, is a binary product of the sets $A$ and $B$.
In a poset $(P, \leq)$, a product of $a, b \in P$ is by definition the greatest lower bound of the subset $\{a, b\}$ i.e., an element $c$ such that $c \leq a$ and $c \leq b$ and for any $d \in P$, if $d \leq a$ and $d \leq b$, then $d \leq c$. The product in a poset is usually called \emph{meet}. For the typical poset $(P, \subseteq)$, where $P$ is a set of subsets of a set $X$, if $P$ is closed under intersection, then the product (meet) is the intersection itself.
In $\mathbf{Set}^{\mathbb{N}}$, a product of $(A, \sigma_A)$ and $(B, \sigma_B)$ is the dynamic system $(A \times B, \sigma_A \times \sigma_B)$ together with the evident projections, where $[\sigma_A \times \sigma_B](a, b)=(\sigma_A(a), \sigma_B(b))$. Note that the projections respect the dynamisms and hence they live in $\mathbf{Set}^{\mathbb{N}}$. A similar situation happens in $\mathbf{Set}^{\mathbb{Z}}$. More generally, in $\mathbf{Set}^{\mathcal{C}}$, a product of two functors $F, G: \mathcal{C} \to \mathbf{Set}$ is the functor $H: \mathcal{C} \to \mathbf{Set}$, together with the maps $p_0: H \to F$ and $p_1: H \to F$ such that $H$ maps the object $A$ to $F(A) \times G(A)$ and the morphism $f: A \to B$ to the function $H(f): F(A) \times G(A) \to F(B) \times G(B)$, sending $(x, y)$ to $(F(f)(x), G(f)(y))$. The projection $p_0: H \to F$ is defined by setting $(p_0)_{A}: F(A) \times G(A) \to F(A)$ as the usual projection. The map $p_1: H \to G$ is defined similarly. It is not hard to see that $H$ is actually a functor, $p_0$ and $p_1$ are actually maps in $\mathbf{Set}^{\mathcal{C}}$ and the data is actually a product.
\end{exam}

\begin{exam}
In $\mathbf{NJ}$, the product of $A$ and $B$ is the conjunction $A \wedge B$ together with the following two derivations as $p_0: A \wedge B \to A$ and $p_1: A \wedge B \to B$:
\begin{center}
	\begin{tabular}{c c}
	    \AxiomC{$A \wedge B$}
	    \RightLabel{\footnotesize$\wedge E_1$} 
		\UnaryInfC{$A$}
		\DisplayProof \hspace{10pt}
		&
		\AxiomC{$A \wedge B$}
	    \RightLabel{\footnotesize$\wedge E_2$} 
		\UnaryInfC{$B$}
		\DisplayProof
	\end{tabular}
\end{center}
For any two derivations $\mathsf{D}$ and $\mathsf{D}'$, take $\langle \mathsf{D}, \mathsf{D}' \rangle$ as the derivation:
\begin{center}
	\begin{tabular}{c}
		\AxiomC{$\mathsf{D}$}
		\noLine
		\UnaryInfC{$A$}
		\AxiomC{$\mathsf{D}'$}
		\noLine
		\UnaryInfC{$B$}
	    \RightLabel{\footnotesize$\wedge I$} 
		\BinaryInfC{$A \wedge B$}
		\DisplayProof
	\end{tabular}
\end{center}
To prove $p_0(\langle \mathsf{D}, \mathsf{D}' \rangle)=\mathsf{D}$, note that the left hand side of the equality is the derivation:
\begin{center}
	\begin{tabular}{c c}
		\AxiomC{$\mathsf{D}$}
		\noLine
		\UnaryInfC{$A$}
		\AxiomC{$\mathsf{D}'$}
		\noLine
		\UnaryInfC{$B$}
	    \RightLabel{\footnotesize$\wedge I$} 
		\BinaryInfC{$A \wedge B$}
		\RightLabel{\footnotesize$\wedge E_1$} 
		\UnaryInfC{$A$}
		\DisplayProof
	\end{tabular}
\end{center}
which is $\beta$-equivalent to $\mathsf{D}$. As we consider the derivations up to $\beta$-equivalence, we have $p_0(\langle \mathsf{D}, \mathsf{D}' \rangle)=\mathsf{D}$. The proof for the other identity, i.e., $p_1(\langle \mathsf{D}, \mathsf{D}' \rangle)=\mathsf{D}'$ is similar.
For the uniqueness, by Remark \ref{equivalentUniquenessProduct}, it is enough to prove $\langle p_0 \mathsf{D}, p_1 \mathsf{D} \rangle=\mathsf{D}$, for any derivation $\mathsf{D}$ of $A \wedge B$. Notice that the left hand side of the equality is:
\begin{center}
\begin{tabular}{c}
		\AxiomC{$ $}
		\noLine
		\UnaryInfC{$\mathsf{D}$}
		\noLine
		\UnaryInfC{$A \wedge B$}
		\RightLabel{\footnotesize$\wedge E_1$} 
		\UnaryInfC{$A$}
		\AxiomC{$ $}
		\noLine
		\UnaryInfC{$\mathsf{D}$}
		\noLine
		\UnaryInfC{$A \wedge B$}
		\RightLabel{\footnotesize$\wedge E_2$} 
		\UnaryInfC{$B$}
	    \RightLabel{\footnotesize$\wedge I$} 
		\BinaryInfC{$A \wedge B$}
		\DisplayProof
	\end{tabular}
\end{center}
which is $\eta$-equivalent to $\mathsf{D}$. Therefore, $\langle p_0 \mathsf{D}, p_1 \mathsf{D} \rangle=\mathsf{D}$. A similar construction is in place for the categories $\mathbf{NN}$ and $\mathbf{NM}$. Note that conjunction, its rules, and their $\beta \eta$-equivalences form the minimal machinery required to make conjunction into the product in $\mathbf{NJ}$. The only intuitively non-trivial aspect is the $\beta \eta$-equivalences, which might be seen as artificial conditions imposed on derivations to ensure the categorical interpretation works. However, in Philosophical Comment \ref{ProdAsConhj}, we argue that these equivalences are also natural conditions to assume.
\end{exam}

\begin{exam}(\textit{Non-existence of binary products})
Consider the poset $(P, \subseteq)$ of all infinite subsets of $\mathbb{N}$. Then, the product (meet) of the set of even numbers and the set of odd numbers does not exist, as there is no infinite set below both of them, let alone the greatest of such lower bounds. For another example, consider a non-trivial group seen as a category. Then, the product of the category's only object $*$ with itself does not exist. Because, if it does, it must be $*$. Let us assume that $p_0, p_1 \in G$ are the projections of the product. Then, for any $g, h \in G$, there is $i \in G$ such that $p_0i=g$ and $p_1i=h$. Therefore, $p_1p_0^{-1}g=h$, for any $g, h \in G$, which is impossible. Just put $h=p_1p_0^{-1}$ and $g \neq e$.
\end{exam}

\begin{phil}(\emph{Product as the conjunction}) \label{ProdAsConhj}
Reading any category as a proof system, we can interpret a product of two propositions as their conjunction. However, looking deeply into the definition, one may also object that although the diagram commutativity and the uniqueness conditions for the cartesian product of sets is natural, the proof-theoretical counterpart, i.e., the $\beta$- and $\eta$-equivalences are not and hence the identification of the conjunction with the product may seem to be too demanding. To justify the choice, note that by Remark \ref{HighlevelProduct}, a product of $A$ and $B$ is just an object $C$ in $\mathcal{C}$ together with an isomorphism between the functors $\mathrm{Hom}(-, C)$ and $\mathrm{Hom}(-, A) \times \mathrm{Hom}(-, B)$ in the category $\mathbf{Set}^{\mathcal{C}^{op}}$. This just says that the deductions of $A \wedge B$ from $D$ is in one-to-one and \emph{uniform} correspondence with the pairs of deductions of $A$ and $B$ from $D$. This, of course, can be considered as a natural formalization for the conjunction operator. The projections, the diagram commutation and the uniqueness conditions, however, are inessential details to present this universal definition in a low level language of maps.
\end{phil}

\begin{lem}\label{UniqenessOfProduct}
The product is unique up to a canonical isomorphism, i.e., for any two products $(C, p_0, p_1)$ and $(D, q_0, q_1)$ of $A$ and $B$, the canonical map $i: C \to D$ induced by the universality of the product is an isomorphism. 
\end{lem}
\begin{proof}
Let $(C, p_0, p_1)$ and $(D, q_0, q_1)$ be both products of the two objects $A$ and $B$. Therefore, there are maps $i: C \to D$ and $j: D \to C$ such that the diagram:
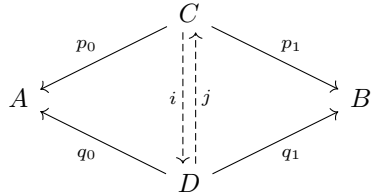
\[\begin{tikzcd}
	&& C \\
	A &&&& B \\
	&& D
	\arrow["{p_0}"', from=1-3, to=2-1]
	\arrow["{p_1}", from=1-3, to=2-5]
	\arrow["{q_0}", from=3-3, to=2-1]
	\arrow["{q_1}"', from=3-3, to=2-5]
	\arrow["j"', shift right=1, dashed, from=3-3, to=1-3]
	\arrow["i"', shift right=1, dashed, from=1-3, to=3-3]
\end{tikzcd}\]
commutes. Therefore, the map $ji: C \to C$ makes the following diagram commute:
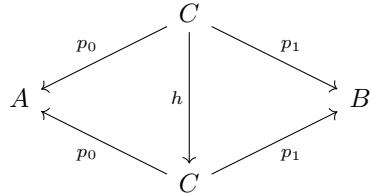
\[\begin{tikzcd}
	&& C \\
	A &&&& B \\
	&& C
	\arrow["{p_0}"', from=1-3, to=2-1]
	\arrow["{p_1}", from=1-3, to=2-5]
	\arrow["{p_0}", from=3-3, to=2-1]
	\arrow["{p_1}"', from=3-3, to=2-5]
	\arrow["h"', from=1-3, to=3-3]
\end{tikzcd}\]
By the universality of the product $(C, p_0, p_1)$, this map must be unique. As $id_C: C \to C$ does exactly the same thing, $ji=id_C$. Similarly, one can show that $ij=id_D$. Therefore, $i: C \to D$ is an isomorphism.
\end{proof}

\begin{rem}
Here are two remarks. First, note that the isomorphism $i: C \to D$ in Lemma \ref{UniqenessOfProduct} is not just an isomorphism between $C$ and $D$. It also transforms the projections to each other, i.e., $q_0i=p_0$ and $q_1i=p_1$. This is what we should expect because a product is not just an object but the tuple of an object and two projections, and the isomorphism between the products must also interact with these projections. Secondly, using Lemma \ref{UniqenessOfProduct}, we are allowed to denote a product of $A$ and $B$ and their projections with a reserved name $A \times B$ and $p^{A, B}_0: A \times B \to A$ and $p^{A, B}_1: A \times B \to B$. When there is no risk of confusion, we simply denote  $p^{A, B}_0$ and $p^{A, B}_1$ by $p_0$ and $p_1$.
\end{rem}

\begin{dfn}
    A category is called  \emph{cartesian} if it has a terminal object and all binary products.
\end{dfn}

\begin{exam}
All categories $\mathbf{FinSet}$, $\mathbf{Set}$, $\mathbf{Top}$, $\mathbf{Set}^{\mathcal{C}}$, $\mathbf{NN}$, $\mathbf{NM}$ and $\mathbf{NJ}$ are cartesian. A cartesian poset is usually called a \emph{bounded meet-semilattice}.
\end{exam}

Usually, universal constructions are functorial. Let us explain the phenomenon by an example. Assume that the product of $A$ and $B$ and the product of $C$ and $D$ exist. Now, assume that we are given two maps $f: A \to C$ and $g: B \to D$. We intend to come up with a \emph{canonical map} $f \times g: A \times B \to C \times D$. For that purpose, first note that $p_0^{A, B}: A \times B \to A$ and $p_1^{A, B}: A \times B \to B$. Thus, $fp_0^{A, B}: A\times B \to C$ and $gp_1^{A, B}: A \times B \to D$. Therefore, $\langle fp_0^{A, B}, gp_1^{A, B} \rangle: A \times B \to C \times D$:
\[\begin{tikzcd}[ampersand replacement=\&]
	A \&\& {A \times B} \&\& B \\
	\\
	C \&\& {C \times D} \&\& D
	\arrow["{p^{A, B}_0}"', from=1-3, to=1-1]
	\arrow["{p_1^{A, B}}", from=1-3, to=1-5]
	\arrow["{p_0^{C, D}}", from=3-3, to=3-1]
	\arrow["{p_1^{C, D}}"', from=3-3, to=3-5]
	\arrow["f"', from=1-1, to=3-1]
	\arrow["g", from=1-5, to=3-5]
	\arrow["{gp_1^{A, B}}", from=1-3, to=3-5]
	\arrow["{fp_0^{A, B}}"', from=1-3, to=3-1]
	\arrow["{f \times g}", dashed, from=1-3, to=3-3]
\end{tikzcd}\]
One can then define $f \times g: A \times B \to C \times D$ as $\langle fp_0^{A, B}, gp_1^{A, B} \rangle$. Using this definition, one can show that the assignment $(-) \times (-): \mathcal{C} \times \mathcal{C} \to \mathcal{C}$ mapping the pair $(A, B)$ to $A \times B$ and $(f, g):(A, B) \to (C, D)$ to $f \times g: A \times B \to C \times D$ is a functor. We call this functor a \emph{binary product} functor. Note that for each pair of objects, there are different choices for their binary product and hence there may be different product functors on a given category. However, as all those different choices are isomorphic, one can see that there is at most one binary product functor, up to an isomorphism. Using this observation and by a slight abuse of language, we denote all binary product functors with one reserved notation $(-) \times (-): \mathcal{C} \times \mathcal{C} \to \mathcal{C}$ and call it \emph{the} binary product functor.\\

Using the universal property of the terminal object and the products, it is easy to prove the following list of isomorphisms:

\begin{thm}\label{ComAsso}
Let $\mathcal{C}$ be a cartesian category. Then:
\begin{itemize}
\item
$A \times 1 \cong A$, via the map $p_0: A \times 1 \to A$.
\item
$A \times B \cong B \times A$, via the map $\langle p_1, p_0 \rangle: A \times B \to B \times A$. 
\item
$(A \times B) \times C \cong A \times (B \times C)$, via the canonical map $\langle r_0q_0, \langle r_1q_0, q_1 \rangle \rangle:(A \times B) \times C \to A \times (B \times C)$, where $q_i=p_i^{(A \times B),C}$ and $r_i=p_i^{A, B}$.
\end{itemize}
\end{thm}
\begin{proof}
We only prove the first claim and leave the rest to the reader. Consider the map $\langle id_A, ! \rangle: A \to A \times 1$. We claim that this map is the converse of $p_0$, i.e., $p_0\langle id_A, ! \rangle=id_A$ and $\langle id_A, ! \rangle p_0=id_{A \times 1}$. The first equality is clear. For the second, it is enough to show that both $\langle id_A, ! \rangle p_0$ and $id_{A \times 1}$ make the following diagram commute:
\[\begin{tikzcd}
	A && {A \times 1} && 1 \\
	\\
	\\
	&& {A \times 1}
	\arrow["{p_1}", from=1-3, to=1-5]
	\arrow["{p_0}"', from=1-3, to=1-1]
	\arrow["{\langle id_A, ! \rangle p_0}"{pos=0.7}, dashed, from=4-3, to=1-3]
	\arrow["{p_0}", from=4-3, to=1-1]
	\arrow["{p_1=!}"', from=4-3, to=1-5]
\end{tikzcd}\]
which is trivially true. Therefore, by the uniqueness condition in the definition of the product, we have $\langle id_A, ! \rangle p_0=id_{A \times 1}$.
\end{proof}

Using binary products, it is possible to define the product of any finite family $A_0, A_1, \ldots, A_n$ of objects and denote it by $A_0 \times A_1 \times \ldots \times A_n$ or $\prod_{i=0}^n A_i$. Of course, we have to specify the way we split the product into binary products and by definition, we always start from the left, i.e., $\prod_{i=0}^n A_i=(\ldots (((A_0 \times A_1) \times A_2) \times \ldots ) \times A_n)$. For the projections of this product, for any $0 \leq j \leq n$, there is a canonical combination of binary projections that goes from $\prod_{i=0}^n A_i$ to $A_j$. We denote this map by $p_j$. Moreover, for any families of maps $\{f_i: C \to A_i\}_{i=0}^n$, there is a canonical map from $C$ to $\prod_{i=0}^n A_i$ that we denote by $\langle f_i \rangle_{i=0}^n$. Notice that using Theorem \ref{ComAsso}, it is clear that the way we split the product into binary products and the order of them is not important in the end.

\subsection{Cocartesian Structure}

The cartesian structure explained in the previous subsection has a dual that is also significant in categorical proof theory.

\begin{dfn}(\emph{Initial objects})
An object $A$ is called \emph{initial} if for any object $B$, there exists a unique map from $A$ to $B$. This unique map is denoted by $!_B: A \to B$ and if $B$ is clear from the context, just by $!$.  Similar to terminal objects, there is at
most one initial object up to an isomorphism. Therefore, we can denote all initial objects with a reserved name $0$ and call it \emph{the} initial object. 
\end{dfn}

\begin{exam}
In $\mathbf{FinSet}$ and $\mathbf{Set}$, the initial object is the empty set $\varnothing$. 
In a poset $(P, \leq)$, the initial object is by definition an element $a \in P$ such that for any $b \in P$, we have $a \leq b$. Hence, the initial object is the least element of the poset. 
In $\mathbf{Set}^{\mathbb{N}}$ and $\mathbf{Set}^{\mathbb{Z}}$, the initial object is the pair $(\varnothing, id_{\varnothing})$. More generally, in $\mathbf{Set}^{\mathcal{C}}$, the initial object is the constant functor $\Delta_{\varnothing}: \mathcal{C} \to \mathbf{Set}$, mapping every object to the empty set and any morphism of $\mathcal{C}$ to the identity function over $\varnothing$.
In $\mathbf{NJ}$, the proposition $\bot$ is an initial object. The reasoning is similar to that of terminality of $\top$ presented in Example \ref{ExampleOfTerminal}. Notice the role of the $\eta$-equivalence relation on derivations to ensure the uniqueness condition.
\end{exam}

\begin{phil}(\emph{Initial objects as the trivializing inconsistency})
Reading any category as a proof system, we can interpret the initial object as the trivializing inconsistency, i.e., the proposition $\bot$ whose assumption makes all propositions deductible and all such deductions equivalent. Note that this type of definition is different from what we had for $\top$ and conjunction, as here, we identify a proposition by what and how it can prove the other propositions rather than by the form of its deductions. 
Trying to define $\bot$ in the latter way, one may argue that we must define $\bot$ as a proposition with no proof and then, as a derived property, show that such a proposition proves any other proposition in a trivial way.
To see why, note that a deduction is a way to transform the proofs of the assumption to the proofs of the conclusion, and as $\bot$ has no proofs, there is only one unique way to map this empty set to the set of proofs of any other proposition.
We will come back to this discussion when we reach the BHK interpretation in Section \ref{SectionBHKInterpretation}.
\end{phil}

\begin{exam}(\textit{Non-existence of initial objects})
For an easier example, consider the poset $(\mathbb{Z}, \leq)$. This poset has no least element and hence no initial object. For another, note that any non-trivial monoid seen as a category does not have an initial object. Moreover, the category of non-empty sets with functions as its morphisms does not have an initial object. 
\end{exam}

\begin{dfn}(\emph{Binary coproducts})
Let $A$ and $B$ be two objects. An object $C$ together with two morphisms $i_0: A \to C$ and $i_1: B \to C$ is called a \emph{coproduct} of $A$ and $B$ if for any object $D$ and any morphisms $f: A \to D$ and $g: B \to D$, there exists a unique map $h: C \to D$ such that:
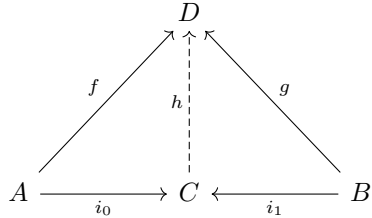
\[\begin{tikzcd}[ampersand replacement=\&]
	\&\& D \\
	\\
	\\
	A \&\& C \&\& B
	\arrow["{i_0}"', from=4-1, to=4-3]
	\arrow["{i_1}", from=4-5, to=4-3]
	\arrow["f", from=4-1, to=1-3]
	\arrow["g"', from=4-5, to=1-3]
	\arrow["h", dashed, from=4-3, to=1-3]
\end{tikzcd}\]
The maps $i_0$ and $i_1$ are called the \emph{injections} of the coproduct. The unique $h$ is denoted by $(f, g)$. Similar to what we had for the other universal constructions, the coproduct of $A$ and $B$ is unique up to (a canonical) isomorphism and hence we can denote it and their injections with reserved names $A + B$ and $i^{A, B}_0: A \to A+B$ and $i^{A, B}_1: B \to A + B$, respectively. When there is no risk of confusion, we simply denote  $i^{A, B}_0$ and $i^{A, B}_1$ by $i_0$ and $i_1$. Finally, as we will mostly work with binary coproducts in this chapter, whenever we say a coproduct, unless specified otherwise, we always mean a binary coproduct.
\end{dfn}

\begin{rem}\label{equivalentUniquenessCoproduct}
The uniqueness condition above can be equivalently phrased as the identity $(hi_0, hi_1)=h$, for any $h: C \to D$. The proof is similar to that of the product case.
\end{rem}

\begin{rem}\label{HighlevelCoproduct}
Similar to products, there is also a high-level definition for coproducts. A coproduct of $A$ and $B$ is an object $C$ together with an isomorphism between the functors $\mathrm{Hom}(C, -): \mathcal{C} \to \mathbf{Set}$ and $\mathrm{Hom}(A, -) \times \mathrm{Hom}(B, -): \mathcal{C} \to \mathbf{Set}$ in the category $\mathbf{Set}^{\mathcal{C}}$. We will not prove the equivalence between the two definitions. 
\end{rem}

\begin{rem}
Notice that the universal definition of the binary coproduct roughly states that the diagram
\[\begin{tikzcd}[ampersand replacement=\&]
	A \&\& C \&\& B
	\arrow["{i_0}"', from=1-1, to=1-3]
	\arrow["{i_1}", from=1-5, to=1-3]
\end{tikzcd}\]
is the ``least" between all such diagrams with fixed end nodes $A$ and $B$.
\end{rem}

\begin{exam}
In $\mathbf{FinSet}$ and $\mathbf{Set}$, the coproduct of $A$ and $B$ is the disjoint union $A+B=\{(0, a) \mid a \in A\} \cup \{(1, b) \mid b \in B\}$ together with the injections $i_0: A \to A+B$ and $i_1: B \to A+B$, defined by $i_0(a)=(0, a)$ and $i_1(b)=(1, b)$.
In a poset $(P, \leq)$, the coproduct of $a, b \in P$ is by definition the least upper bound of the subset $\{a, b\}$ i.e., an element $c$ such that $a \leq c$ and $b \leq c$ and for any $d \in P$, if $a \leq d$ and $b \leq d$, then $c \leq d$. Coproducts in posets are usually called \emph{joins}. For the typical poset $(P, \subseteq)$, where $P$ is a set of subsets of a set $X$, if $P$ is closed under binary union, the coproduct (join) is the union, itself.
In $\mathbf{Set}^{\mathbb{N}}$ and $\mathbf{Set}^{\mathbb{Z}}$, the coproduct of $(A, \sigma_A)$ and $(B, \sigma_B)$ is $(A + B, \sigma_{A+B})$, together with the usual injections, where $[\sigma_{A+B}](0, a)=(0, \sigma_A(a))$ and $[\sigma_{A+B}](1, b)=(1, \sigma_B(b))$. Note that the injections respect the dynamism. More generally, a coproduct of two objects $F: \mathcal{C} \to \mathbf{Set}$ and $G: \mathcal{C} \to \mathbf{Set}$ in $\mathbf{Set}^{\mathcal{C}}$ is the functor $H:\mathcal{C} \to \mathbf{Set}$ mapping the object $A$ to $H(A)=F(A)+G(A)$ and the map $f: A \to B$ to the function $H(f): F(A)+G(A) \to F(B)+G(B)$ defined by $H(f)(0, x)=(0, F(f)(x))$ and $H(f)(1, y)=(1, G(f)(y))$. The injections $i_0: F \to H$ and $i_1: G \to H$ are defined by $(i_0)_A: F(A) \to F(A)+G(A)$ and $(i_1)_A: G(A) \to F(A)+G(A)$ as the usual injections. It is easy to prove that $H: \mathcal{C} \to \mathbf{Set}$ is a functor, both $i_0$ and $i_1$ are natural transformations and the data is actually a coproduct. 
\end{exam}

\begin{exam}
In $\mathbf{NJ}$, the coproduct of $A$ and $B$ is the disjunction $A \vee B$ together with the two derivations:
\begin{center}
	\begin{tabular}{c c c}
	    \AxiomC{$A$}
	    \RightLabel{\footnotesize$\vee I_1$} 
		\UnaryInfC{$A \vee B$}
		\DisplayProof \hspace{10pt}
		&
		\AxiomC{$B$}
	    \RightLabel{\footnotesize$\vee I_2$} 
		\UnaryInfC{$A \vee B$} \hspace{10pt}
		\DisplayProof
	\end{tabular}
\end{center}
as the injections. For any two derivations $\mathsf{D}$ and $\mathsf{D}'$, define the map $(\mathsf{D}, \mathsf{D}')$ as the derivation: 
\begin{center}
	\begin{tabular}{c c c}
	
		\AxiomC{$A \vee B$}
		\AxiomC{$[A]$}
		\noLine
		\UnaryInfC{$\mathsf{D}$}
		\noLine 
		\UnaryInfC{$C$} 
		\AxiomC{$[B]$}
		\noLine
		\UnaryInfC{$\mathsf{D}'$}
		\noLine 
		\UnaryInfC{$C$} 
	    \RightLabel{\footnotesize$\vee E$} 
		\TrinaryInfC{$C$}
		\DisplayProof
	\end{tabular}
\end{center}
To prove $(\mathsf{D}, \mathsf{D}') i_0=\mathsf{D}$, note that the left hand side is the derivation:
\begin{center}
	\begin{tabular}{c c}
        \AxiomC{$A$}
\RightLabel{\footnotesize$\vee I_1$} 
        \UnaryInfC{$A \vee B$}
		\AxiomC{$[A]$}
		\noLine
		\UnaryInfC{$\mathsf{D}$}
		\noLine 
		\UnaryInfC{$C$} 
		\AxiomC{$[B]$}
		\noLine
		\UnaryInfC{$\mathsf{D}'$}
		\noLine 
		\UnaryInfC{$C$} 
	    \RightLabel{\footnotesize$\vee E$} 
		\TrinaryInfC{$C$}
		\DisplayProof
	\end{tabular}
\end{center}
which is $\beta$-equivalent to $\mathsf{D}$. As we consider the derivations up to $\beta$-equivalence, we have $(\mathsf{D}, \mathsf{D}') i_0=\mathsf{D}$. The proof for the other identity, i.e.,  $(\mathsf{D}, \mathsf{D}') i_1=\mathsf{D}'$ is similar.
For the uniqueness, by Remark \ref{equivalentUniquenessCoproduct}, it is enough to prove $(\mathsf{D}i_0, \mathsf{D}i_1)=\mathsf{D}$, for any derivation $\mathsf{D}$ with the assumption $A \vee B$. Notice that the left hand side is:
\begin{center}
        \begin{tabular}{c}
 \AxiomC{$A \vee B$}
 
 \AxiomC{$[A]^1$}
 \RightLabel{\footnotesize$\vee I_1$} 
\UnaryInfC{$A \vee B$}
\noLine
\UnaryInfC{$\mathsf{D}$}
\noLine
 \UnaryInfC{$C$}
 \AxiomC{$[B]^2$}
 \RightLabel{\footnotesize$\vee I_2$} 
\UnaryInfC{$A \vee B$}
\noLine
\UnaryInfC{$\mathsf{D}$}
\noLine
 \UnaryInfC{$C$}
 \RightLabel{\footnotesize$\vee E$} 
\TrinaryInfC{$C$}
\DisplayProof
        \end{tabular}
    \end{center}
which is $\eta$-equivalent to $\mathsf{D}$.
\end{exam}

\begin{phil}(\emph{Coproduct as the disjunction})
Reading any category as a proof system, we can interpret the coproduct of two propositions as their disjunction. However, looking deeply into the definition,
one may also object that although the diagram commutativity and the uniqueness conditions for the disjoint union of sets is natural, the proof-theoretical counterpart, i.e., the $\beta$- and $\eta$-equivalences are not and hence the identification of the disjunction with the coproduct may seem to be too demanding. To justify the choice, note that by Remark \ref{HighlevelCoproduct}, a coproduct of $A$ and $B$ is just an object $C$ in $\mathcal{C}$ together with an isomorphism between the functors $\mathrm{Hom}(C, -)$ and $\mathrm{Hom}(A, -) \times \mathrm{Hom}(B, -)$ in the category $\mathbf{Set}^{\mathcal{C}}$. This just says that the deductions of $D$ from the assumption $A + B$ is in one-to-one and \emph{uniform} correspondence with the pairs of deductions of $D$ from $A$ and $B$. This, of course, can be considered as a natural formalization for the disjunction operator. The projections, the diagram commutation and the uniqueness conditions, however, are inessential details to present this universal definition in a low level language of maps.
Note that this type of definition is again based on what and how a proposition can prove the other propositions rather than by the form of its deductions.
Trying to define the disjunction in the latter way, one may argue that we must define $A \vee B$ as a proposition whose proofs are either the proofs of $A$ or the proofs of $B$ and we know which one is the case. Then again, as a derived property, we can show that the deductions of $D$ from $A \vee B$ is in one-to-one correspondence with the pairs of the deductions of $A$ and $B$ from $D$. The reason simply is that to transform the proofs of $A \vee B$ to the proofs of $D$, as a proof of $A \vee B$ is either a proof of $A$ or a proof of $B$ and we know which case is happening, it is necessary and sufficient to transform the proofs of $A$ and the proofs of $B$ into the proofs of $D$. We will come back to this discussion when we reach the BHK interpretation in Section \ref{SectionBHKInterpretation}.
\end{phil}

\begin{dfn}
    A category is called  \emph{cocartesian} if it has the initial object and all the binary coproducts.
\end{dfn}

\begin{exam}
All categories $\mathbf{FinSet}$, $\mathbf{Set}$, $\mathbf{Top}$, $\mathbf{Set}^{\mathcal{C}}$, and $\mathbf{NJ}$ are cocartesian. A cocartesian poset is usually called a \emph{bounded join-semilattice}.
\end{exam}

\begin{exam}(\textit{Non-existence of binary coproducts})
Consider the poset $(P, \subseteq)$ of all subsets of $\mathbb{N}$ whose complement is infinite. Then, the coproduct (join) of the set $E$ of even numbers and $O$ of odd numbers does not exists, as the only subset above both of them is $\mathbb{N}$ whose complement is empty and as such finite.
For another example, consider a non-trivial group considered as a category. Similar to the product case, it is easy to prove that the coproduct of $*$ with itself does not exist in this category.
\end{exam}

Similar to what we had for the product, the coproduct can also rise to a functor. Assume that both the coproduct of $A$ and $B$ and the coproduct of $C$ and $D$ exist. Now, assume that we are given two morphisms $f: A \to C$ and $g: B \to D$. We intend to come up with a \emph{canonical map} $f + g: A + B \to C + D$. For that purpose, first note that $i^{C, D}_0: C \to C+D$ and $i_1^{C, D}: D \to C+D$. Thus, $i_0^{C,D}f: A \to C+D$ and $i_1^{C, D}g: B \to C+ D$. Therefore, $(i_0^{C, D}f, i_1^{C, D}g): A + B \to C + D$:
\[\begin{tikzcd}[ampersand replacement=\&]
	A \&\& {A + B} \&\& B \\
	\\
	C \&\& {C + D} \&\& D
	\arrow["{i_0^{A, B}}", from=1-1, to=1-3]
	\arrow["{i_1^{A, B}}"', from=1-5, to=1-3]
	\arrow["{i_0^{C, D}}"', from=3-1, to=3-3]
	\arrow["{i_1^{C, D}}", from=3-5, to=3-3]
	\arrow["f"', from=1-1, to=3-1]
	\arrow["g", from=1-5, to=3-5]
	\arrow["{f + g}", dashed, from=1-3, to=3-3]
	\arrow["{i_0^{C, D}f}"', from=1-1, to=3-3]
	\arrow["{i_1^{C, D}g}", from=1-5, to=3-3]
\end{tikzcd}\]
One can then define $f + g: A + B \to C + D$ as $(i_0^{C, D}f, i_1^{C, D}g)$. It is easy to see that this definition makes the assignment $(-) + (-): \mathcal{C} \times \mathcal{C} \to \mathcal{C}$ mapping the pair $(A, B)$ to $A + B$ and the morphism $(f, g):(A, B) \to (C, D)$ to $f + g: A + B \to C + D$ a functor. This functor is called a \emph{coproduct functor}. Again, as all coproduct functors are isomorphic, we use one reserved notation $(-)+(-):\mathcal{C} \times \mathcal{C} \to \mathcal{C}$ for all of them and call it \emph{the} coproduct functor.

Moreover, similar to products, we can define the coproduct of any finite family $A_0, \ldots, A_n$ of objects and denote it by $A_0+ \ldots + A_n$ or $\sum_{i=0}^n A_i$. Using the injections, one can come up with a canonical map $i_j: A_j \to \sum_{i=0}^n A_i$, for any $0 \leq j \leq n$. Moreover, for any family $\{f_i: A_i \to C\}_{i=0}^n$ of maps, there is a canonical map from $\sum_{i=0}^n A_i$ to $C$ that we denote by $(f_i)_{i=0}^n$.

\subsection{Exponential Objects}
The third universal construction we are interested in is the exponential. Exponentials provide the machinery to internalize the structure of the category within itself.
\begin{dfn}\label{DefExpObj}(\emph{Exponential object})
Let $\mathcal{C}$ be a cartesian category and $A$ and $B$ be two objects. An object $C$ together with a morphism $ev: C \times A \to B$ is called an \emph{exponential object} if for any $f: D \times A \to B$, there exists a unique $g: D \to C$ such that:
\[\begin{tikzcd}[ampersand replacement=\&]
	{D \times A} \\
	\\
	{C \times A} \&\& B
	\arrow["ev"', from=3-1, to=3-3]
	\arrow["f", from=1-1, to=3-3]
	\arrow["{g \times id_A}"', from=1-1, to=3-1]
\end{tikzcd}\]
Again, it is easy to see that exponentials are unique up to a canonical isomorphism. Hence, we can safely refer to them as \emph{the} exponential object of $B$ by $A$ and denote it and its evaluation map by the reserved names $[A, B]$ (or sometimes $B^A$) and $ev^{A, B}: [A, B] \times A \to B$, respectively. Moreover, for any $f: C \times A \to B$, we denote its unique $g$ by $\lambda^{B, C}_A f: C \to [A, B]$. If there is no risk of confusion, we drop the superscripts in $ev^{A, B}$ and $\lambda^{B, C}_A f$ and write $ev$ and $\lambda_A f$, respectively. Note that the above diagram's commutativity is simply the equality $ev(\lambda_A f \times id_A)=f$.
\end{dfn}

\begin{rem}\label{equivalentUniquenessExp}
The uniqueness condition in Definition \ref{DefExpObj} can be equivalently phrased as the identity $\lambda_A ev(g \times id_A)=g$, for any $g: C \to [A, B]$.
\end{rem}

\begin{rem}
Consider all diagrams of the shape: 
\[\begin{tikzcd}[ampersand replacement=\&]
	{D \times A} \&\& B
	\arrow["f"', from=1-1, to=1-3]
\end{tikzcd}\]
where $A$ and $B$ are fixed and we change $D$ and the map $f:D\times A \to B$. Using maps between $D$'s, these diagrams have a natural order. In the following picture, we say that the lower diagram is greater than the upper one:
\[\begin{tikzcd}[ampersand replacement=\&]
	{D_0 \times A} \\
	\&\& B \\
	{D_1 \times A}
	\arrow["{f_0}", from=1-1, to=2-3]
	\arrow["{g \times id_A}"', dashed, from=1-1, to=3-1]
	\arrow["{f_1}"', from=3-1, to=2-3]
\end{tikzcd}\]
if there is $g:D_0 \to D_1$ making the whole diagram commutative. In this sense, the universal definition of the exponential object roughly states that the diagram
\[\begin{tikzcd}[ampersand replacement=\&]
	{C \times A} \&\& B
	\arrow["ev"', from=1-1, to=1-3]
\end{tikzcd}\]
is the ``greatest" among all the diagrams with similar shapes. Note that the uniqueness condition is again missing in this informal discussion.
\end{rem}

\begin{rem}\label{HighlevelExp}
Similar to other universal constructions, there is a high-level formalization for the exponential object. An exponential object is an object $C$ together with an isomorphism between the functors $\mathrm{Hom}(-, C)$ and $\mathrm{Hom}(- \times A, B)$ in the category $\mathbf{Set}^{\mathcal{C}^{op}}$. We do not prove the equivalence between the two definitions. However, one special instance of the isomorphism in the second definition, i.e., the one between $\mathrm{Hom}(1 \times A, B)$ or equivalently $\mathrm{Hom}(A, B)$ and $\mathrm{Hom}(1, [A, B])$ is worth explaining.
Here is the description of the isomorphism. For any $f: A \to B$, using $p_1: 1 \times A \to A$, we have the map $fp_1:1 \times A \to B$. By the universal property of the exponential $[A, B]$, we have the map $\lambda_A (fp_1): 1 \to [A, B]$ such that:
\[\begin{tikzcd}
	{1 \times A} && A \\
	\\
	{[A, B] \times A} && B
	\arrow["{ \lambda_A (fp_1) \times id_A}"', from=1-1, to=3-1]
	\arrow["ev"', from=3-1, to=3-3]
	\arrow["{p_1}", from=1-1, to=1-3]
	\arrow["f", from=1-3, to=3-3]
	\arrow["{fp_1}", from=1-1, to=3-3]
\end{tikzcd}\]
Conversely, having a map $g: 1 \to [A, B]$, we have
\[\begin{tikzcd}
	{1 \times A} && A \\
	\\
	{[A, B] \times A} && B
	\arrow["{g \times id_A}"', from=1-1, to=3-1]
	\arrow["ev"', from=3-1, to=3-3]
	\arrow["{\langle !, id_A \rangle}"', from=1-3, to=1-1]
\end{tikzcd}\]
as a map from $A$ to $B$. It is easy to see that these two operations are the converse to each other. For any $f: A \to B$, we denote $\lambda_A (fp_1)$ by $\lambda f$ and for any $g: 1 \to [A, B]$ and $a: X \to A$, we denote the map $ev(g \times id_A)\langle !, id_A \rangle a$, by $g \cdot a: X \to B$. It is easy to see that $(\lambda f) \cdot a=f \circ a$ and $\lambda (g \cdot id_A)=g$. Therefore, we can read $\lambda$ as the usual lambda operator and $\cdot$ as the application operation.
\end{rem}

\begin{exam}\label{ExponentialInSet}
In the categories $\mathbf{FinSet}$ and $\mathbf{Set}$, the exponential $[A, B]$ is the set $B^A$ consisting of all the functions from $A$ to $B$ with the morphism $ev: B^A \times A \to B$ defined by $ev(f, a)=f(a)$. Note that for any map $f: C \times A \to B$, the map $\lambda_A f: C \to B^A$ is defined by $(\lambda_A f)(c)(a)=f(c, a)$.
In the poset $(P, \leq)$, the exponential is by definition the greatest element $c$ such that $c \wedge a \leq b$, i.e., an element $c$ such that $c \wedge a \leq b$ and for any $d \in P$ if $d \wedge a \leq b$ then $d \leq c$. Exponential objects in posets are called \emph{Heyting implications} and denoted by $\to$. In the poset $(\mathbf{2}^{(P, \leq)}, \subseteq)$, the exponential $U \to V$ is $\{x \in P \mid \forall y \geq x \, (y \in U \to y \in V)\}$. For the poset $\mathcal{O}(X)$, where $X$ is a topological space, the exponential $U \to V$ is $int(U^c \cup V)$, where $int$ is the interior operation. 
\end{exam}

\begin{exam}\label{ExponentialSetZ}
In $\mathbf{Set}^{\mathbb{Z}}$, the exponential object $[(A, \sigma_A), (B, \sigma_B)]$ is the dynamical system $(C, \sigma_{C})$ where $C$ is the set of \textit{all} functions from $A$ to $B$ and $\sigma_{C}(f)=\sigma_B f \sigma_A^{-1}$. The evaluation map $ev: (C, \sigma_C) \times (A, \sigma_A) \to (B, \sigma_B)$ is just the usual set-theoretical evaluation map. Notice that $ev$ respects the dynamism as
$ev(\sigma_{C}(f), \sigma_A(x))=\sigma_{C}(f)(\sigma_A(x))=\sigma_B f \sigma_A^{-1}(\sigma_A(x))=\sigma_B(f(x))=\sigma_B(ev(f, x))$. For any equivariant map $f: (D, \sigma_D) \times (A, \sigma_A) \to (B, \sigma_B)$, the map $\lambda_A f: (D, \sigma_D) \to [(A, \sigma_A), (B, \sigma_B)]$ is defined by $(\lambda_A f)(d)(a)=f(d, a)$ and it is equivariant. The reason is that for any $d \in D$ and any $a \in A$, we have $(\lambda_A f)(\sigma_D(d))(a)=f(\sigma_D(d), a)$, by definition. Moreover, as $f$ is equivariant, we have  
\[
f(\sigma_D(d), a)=\sigma_B(f(d, \sigma_A^{-1}(a)))=\sigma_B[(\lambda_A f)(d) (\sigma_A^{-1}(a))].
\]
Therefore, $(\lambda_A f) (\sigma_D(d))=\sigma_B((\lambda_A f)(d)) \sigma_A^{-1}$.
\end{exam}

\begin{exam}
Let $\mathcal{C}$ be a small category. Then, in $\mathbf{Set}^{\mathcal{C}}$, the exponential object $[E, F]$ is the functor mapping the object $A$ to the set $\mathrm{Hom}_{\mathbf{set}^{\mathcal{C}}}(y^A \times E, F)$. The evaluation map is the natural transformation $ev: [E, F] \times E \to F$ defined by $ev_A: [E, F](A) \times E(A) \to F(A)$, mapping the pair $(\alpha, x)$ to $\alpha_A(id_A, x)$, for any $\alpha: y^A \times E \to F$ and $x \in E(A)$. It is an interesting exercise to set $\mathcal{C}$ as $\mathbf{1}$ or $(\mathbb{Z}, +)$, compute the exponential object described here and show that it is exactly what was described in Example \ref{ExponentialInSet} and Example \ref{ExponentialSetZ}.
\end{exam}

\begin{exam}\label{TopAndExp}
In $\mathbf{Top}$, not all exponential objects exist. For instance, there is no topological space and no evaluation map that can act as $[\mathbb{Q}, I]$, where $\mathbb{Q}$ is the subspace of rational numbers in $\mathbb{R}$ and $I=[0, 1]$ is the unit interval \cite{HandbookII}. In fact, for any topological space $X$, a necessary and sufficient condition for $X$ is known that guarantees the existence of the exponential $[X, Y]$ for \emph{all} spaces $Y$. This condition is called \emph{core-compactness}. To define this notion, let $U, V \subseteq X$ be two open subsets of $X$. The open $V$ is called \emph{way below} $U$, denoted by $V <\!\!< U$, if any open cover of $U$ has a finite subcover for $V$. A space $X$ is called core-compact if for every open neighborhood $U$ of a point $x \in X$, there exists an open neighborhood $V$ of $x$ with $V <\!\!< U$. If the space $X$ is Hausdorff, it is core-compact iff it is locally compact. If $X$ is core-compact, then the space $[X, Y]$ is simply the set of all continuous functions from $X$ to $Y$ with the topology with the subbasis
\[
O_{U, V}=\{f \in [X, Y] \mid U <\!\!< f^{-1}(V)\},
\]
for any opens $U \subseteq X$ and $V \subseteq Y$. The evaluation map is the usual set-theoretical evaluation map. If both $X$ and $Y$ are Hausdorff, then this topology is just the usual compact-open topology that has the subbasis
\[
O'_{K, V}=\{f \in [X, Y] \mid f[K] \subseteq V\},
\]
for any compact $K \subseteq X$ and any open $V \subseteq Y$. For instance, as the discrete space $\mathbb{N}$ is locally compact and Hausdorff, the exponential $\mathbb{N}^{\mathbb{N}}$ exists and it is just the set of \emph{all} functions over $\mathbb{N}$ together with the compact-open topology which is nothing but the usual product topology. Another example that we will use later is the exponential $[\{0,1\}^{\mathbb{N}}, \mathbb{N}]$ that also exists, as $\{0,1\}^{\mathbb{N}}$ is a compact and Hausdorff space and its topology is the compact-open topology.
\end{exam}

\begin{exam}
In $\mathbf{NJ}$, the exponential $[A, B]$ is the implication $A \to B$ together with the derivation:
\begin{center}
	\begin{tabular}{c c}
		\AxiomC{$(A \to B) \wedge A$}
  \RightLabel{\footnotesize$\wedge E_2$} 
        \UnaryInfC{$A$}
        \AxiomC{$(A \to B) \wedge A$}
        \RightLabel{\footnotesize$\wedge E_1$} 
        \UnaryInfC{$A \to B$}
	    \RightLabel{\footnotesize$\to E$} 
		\BinaryInfC{$B$}
		\DisplayProof
	\end{tabular}
\end{center}
as the evaluation map. For any derivation $\mathsf{D}$ for $B$ from $C \wedge A$, define $\lambda_A \mathsf{D}$ as:
\begin{center}
	\begin{tabular}{c c}
        \AxiomC{$C$}
		\AxiomC{$[A]$}
  \RightLabel{\footnotesize$\wedge I$} 
        \BinaryInfC{$C \wedge A$}
		\noLine
		\UnaryInfC{$\mathsf{D}$}
		\noLine
		\UnaryInfC{$B$}
		\RightLabel{\footnotesize$\to I$} 
		\UnaryInfC{$A \to B$}
  
		\DisplayProof 
	\end{tabular}
\end{center}
To show that $ev(\lambda_A \mathsf{D} \times id_A)=\mathsf{D}$, note that the right-hand side is:
\begin{center}
	\begin{tabular}{c c}
 
        \AxiomC{$C \wedge A$}
  \UnaryInfC{$C$}
		\noLine
		\UnaryInfC{$\lambda_A \mathsf{D}$}
		\noLine
		\UnaryInfC{$A \to B$}

  \AxiomC{$C \wedge A$}
  \UnaryInfC{$A$}
  
  \BinaryInfC{$(A \to B) \wedge A$}
  \noLine
		\UnaryInfC{$ev$}
  \noLine
  \UnaryInfC{$B$}
		\DisplayProof 
	\end{tabular}
\end{center}
Substituting $\lambda_A \mathsf{D}$ and $ev$ in this derivation, we have:
\begin{center}
	\begin{tabular}{c c}

        \AxiomC{$C \wedge A$}
  \UnaryInfC{$C$}
		 \AxiomC{$[A]$}
   
        \BinaryInfC{$C \wedge A$}
		\noLine
		\UnaryInfC{$\mathsf{D}$}
		\noLine
		\UnaryInfC{$B$}
		\RightLabel{\footnotesize$\to I$} 
		\UnaryInfC{$A \to B$}
  \AxiomC{$C \wedge A$}
        \UnaryInfC{$A$}
        
  \BinaryInfC{$(A \to B) \wedge A$}
  \UnaryInfC{$A$}

  \AxiomC{$C \wedge A$}
  \UnaryInfC{$C$}
		 \AxiomC{$[A]$}
        \BinaryInfC{$C \wedge A$}
		\noLine
		\UnaryInfC{$\mathsf{D}$}
		\noLine
		\UnaryInfC{$B$}
		\RightLabel{\footnotesize$\to I$} 
		\UnaryInfC{$A \to B$}

  \AxiomC{$C \wedge A$}
        \UnaryInfC{$A$}

  \BinaryInfC{$(A \to B) \wedge A$}
  \UnaryInfC{$A \to B$}
  
   \BinaryInfC{$B$}
		\DisplayProof 
	\end{tabular}
\end{center}
which is $\beta \eta$-equivalent to $\mathsf{D}$. The proof of the uniqueness part is similar.
\end{exam}

\begin{phil}(\emph{Exponentials as the implication})
Reading any category as a proof system, we can interpret the exponential as the implication. 
To justify, note that the exponential $[A, B]$ is just an object together with a natural isomorphism between $\mathrm{Hom}(- \times A, B)$ and $\mathrm{Hom}(-, [A, B])$. This just says that the deduction of $B$ from $D \wedge A$ are in one-to-one and uniform correspondence with the deductions of $A \to B$ from $D$, i.e., proving $A \to B$ is just proving $B$ with the addition of $A$ to the assumptions.
\end{phil}

\begin{dfn}
A cartesian category where all exponentials exist is called a \emph{cartesian closed category} ($\mathrm{CCC}$, for short). A $\mathrm{CCC}$ which has all the binary coproducts is called an \emph{almost bicartesian closed category} ($\mathrm{ABC}$, for short), and if it also has the initial object, a \emph{bicartesian closed category} ($\mathrm{BCC}$, for short). A $\mathrm{BCC}$ (resp. an $\mathrm{ABC}$) that is also a poset is called a \emph{Heyting algebra} (resp. an \emph{almost Heyting algebra}). By the $\mathrm{BC}$-structure, we mean the terminal and initial objects and the product, coproduct and exponential of a $\mathrm{BCC}$. This is an informal notion, putting a name on the structure that makes a category a $\mathrm{BCC}$. The $\mathrm{CC}$- and $\mathrm{AB}$-structures are defined accordingly. 
\end{dfn}

\begin{exam}
All categories $\mathbf{FinSet}$, $\mathbf{Set}$, $\mathbf{Set}^{\mathcal{C}}$, and $\mathbf{NJ}$ are $\mathrm{BCC}$'s. The categories $\mathbf{NN}$ and $\mathbf{NM}$ are $\mathrm{CCC}$ and $\mathrm{ABC}$, respectively.
\end{exam}

\begin{phil}
A $\mathrm{BCC}$ can be read as a proof system over the language $\mathcal{L}_p=\{\wedge, \vee, \to, \top, \bot\}$. Similarly, an $\mathrm{ABC}$ (resp. a $\mathrm{CCC}$) is a proof system over the fragment $\mathcal{L}_p-\{\bot\}=\{\wedge, \vee, \to, \top\}$ (resp. $\mathcal{L}_p-\{\vee, \bot\}=\{\wedge, \to, \top\}$). With this interpretation, Heyting algebras are trivialized proof systems over $\mathcal{L}_p$ where all deductions are collapsed to one deductibility order. The same also holds for the poset $\mathrm{ABC}$ and the poset $\mathrm{CCC}$. Moreover, it is worth mentioning that reading connectives as universal constructions naturally leads to intuitionistic proof systems. For classical proof systems, see Section \ref{SectionClassical}.
\end{phil}

\begin{exam}(\textit{Non-existence of exponential objects}) 
Consider the set $P$ consisting of all finite subsets of $\mathbb{N}$ together with the whole set $\mathbb{N}$. Note that the poset $(P, \subseteq)$ is cartesian with the terminal object $\mathbb{N}$ and the intersection as the product. We claim that the exponential $W=\{0\} \to \{1\}$ in the poset $(P, \subseteq)$ does not exist. If it does, it must have the property $S \cap \{0\} \subseteq \{1\}$ iff $S \subseteq W$, for any $S \in P$. As $S \cap \{0\} \subseteq \{1\}$ holds for any $S \subseteq \mathbb{N}-\{0\}$, we must have $\mathbb{N}-\{0\} \subseteq W$. Therefore, $W$ is not finite and hence $W=\mathbb{N}$. But then as $\mathbb{N} \subseteq W=\mathbb{N}$, we must have $\mathbb{N} \cap \{0\} \subseteq \{1\}$ which is impossible.  
\end{exam}

\begin{exam}
Let $\mathcal{C}$ be a non-preorder category with the initial and terminal objects and $0 \cong 1$. Then, $\mathcal{C}$ cannot have all the exponentials because if it can, then we must have
\[
\mathrm{Hom}(A, B) \cong \mathrm{Hom}(1 \times A, B) \cong \mathrm{Hom}(1, [A, B]) \cong \mathrm{Hom}(0, [A, B]).
\]
for any objects $A$ and $B$. As $0$ is the initial object, the rightmost set has exactly one element. Hence, $\mathrm{Hom}(A, B)$ must also be a singleton. This contradicts the assumption that $\mathcal{C}$ is not a preordered set. As a consequence, many algebraic categories, like the category of all groups and homomorphisms, do not have all exponential objects.
\end{exam}

\begin{exam}
Let $X$ be a fixed infinite set. Define $\mathcal{C}$ as the category of sets in the form $X^k$ for some $k \geq 0$ with the usual functions as the morphisms. This category is cartesian with the terminal object $X^0$ and the usual set-theoretical product as the product. We claim that the exponential $[X, X]$ does not exist in $\mathcal{C}$. Because if it does, it must be in the form $X^k$, for some $k \geq 0$. Therefore, $\mathrm{Hom}(X^0, [X, X])=\mathrm{Hom}(X^0, X^k)$ must be in one to one correspondence with $\mathrm{Hom}(X^0 \times X, X) \cong \mathrm{Hom}(X, X)$. However, the cardinality of the former equals the cardinality of $X$, while the cardinality of the latter is strictly greater than the cardinality of $X$. This phenomenon that there are more functions over an infinite set than the cardinality of the set itself is blocking the desirable construction of a cartesian closed category of \emph{sets} with a non-terminal object $X$ such that $[X, X] \cong X$. Such a situation can be useful to interpret lambda calculus where each element of the domain $X$ can also be read as a function on $X$ at the same time. Having such a dual nature is a usual phenomenon in computability theory where numbers are both the inputs and the codes of the algorithms. Such a $\mathrm{CCC}$ (using ordered/topological structures and not pure sets) was later discovered by Dana Scott \cite{LambekScott}. 
\end{exam}

Similar to what we had for the product and coproduct, the exponential can also rise to a functor. Assume that both the exponentials $[A, B]$ and $[C, D]$ exist. Now, assume that we are given $f: C \to A$ and $g: B \to D$. We intend to come up with a \emph{canonical map} $[f, g]: [A,B] \to [C, D]$. For that purpose, first note that $id_{[A, B]} \times f : [A, B] \times C \to [A, B] \times A$. Composing with $ev^{A, B}: [A, B] \times A \to B$ and $g: B \to D$, we will have the map $h: [A, B] \times C \to D$:
\[\begin{tikzcd}[ampersand replacement=\&]
	{[A, B] \times C} \&\& {[A, B] \times A} \\
	\\
	D \&\& B
	\arrow["{id_{[A, B]} \times f}", from=1-1, to=1-3]
	\arrow["{ev^{A, B}}", from=1-3, to=3-3]
	\arrow["h"', from=1-1, to=3-1]
	\arrow["g", from=3-3, to=3-1]
\end{tikzcd}\]
One can then define $[f, g]: [A, B] \to [C, D]$ as $\lambda_{C} h$. Now, define the assignment $[-, -]: \mathcal{C}^{op} \times \mathcal{C} \to \mathcal{C}$ by mapping the pair $(A, B)$ to $[A, B]$ and $(f, g):(A, B) \to (C, D)$ to $[f, g]: [A, B] \to [C, D]$. This assignment is a functor called the \emph{exponentiation functor}.

\subsection{Structure-preserving Functors}

So far, we have introduced terminal and initial objects, binary products and coproducts, and exponentials as proof-sensitive versions of logical constants. We also claimed that any $\mathrm{BCC}$ can be interpreted as a proof system over the language $\mathcal{L}_p$. The next natural step is to compare these proof systems using the natural machinery of functors. We expect these functors to preserve the logical constants of the systems. In this subsection, we will introduce these structure-preserving functors and then see some examples.

\begin{dfn}
Let $\mathcal{C}$ and $\mathcal{D}$ be two $\mathrm{BCC}$'s. A functor $F: \mathcal{C} \to \mathcal{D}$ is said to preserve:
\begin{itemize}
    \item 
the terminal (initial) object, if $F(A)$ is a terminal (initial) object in $\mathcal{D}$, for any terminal (initial) object in $\mathcal{C}$.
    \item
the product, if the diagram
\[\begin{tikzcd}
	{F(A)} && {F(C)} && {F(B)}
	\arrow["{F(p_0)}"', from=1-3, to=1-1]
	\arrow["{F(p_1)}", from=1-3, to=1-5]
\end{tikzcd}\]
is a product of $F(A)$ and $F(B)$ in $\mathcal{D}$, for any product 
\[\begin{tikzcd}
	A && C && B
	\arrow["{p_0}"', from=1-3, to=1-1]
	\arrow["{p_1}", from=1-3, to=1-5]
\end{tikzcd}\]
in the category $\mathcal{C}$.
    \item
the coproduct, if the diagram
\[\begin{tikzcd}
	{F(A)} && {F(C)} && {F(B)}
	\arrow["{F(i_0)}", from=1-1, to=1-3]
	\arrow["{F(i_1)}"', from=1-5, to=1-3]
\end{tikzcd}\]
is a coproduct of $F(A)$ and $F(B)$ in $\mathcal{D}$, for any coproduct
\[\begin{tikzcd}
	A && C && B
	\arrow["{i_0}", from=1-1, to=1-3]
	\arrow["{i_1}"', from=1-5, to=1-3]
\end{tikzcd}\]
in the category $\mathcal{C}$.
\end{itemize}
A product-preserving functor $F: \mathcal{C} \to \mathcal{D}$ preserves the exponentials, if the object $F(C)$ together with the evaluation map
\[\begin{tikzcd}[ampersand replacement=\&]
	{F(C \times A)} \&\& {F(B)}
	\arrow["{F(ev)}", from=1-1, to=1-3]
\end{tikzcd}\]
is an exponential in $\mathcal{D}$, for any exponential
\[\begin{tikzcd}[ampersand replacement=\&]
	{C \times A} \&\& B
	\arrow["ev", from=1-1, to=1-3]
\end{tikzcd}\]
in the category $\mathcal{C}$. Note that as $F$ preserves the product, $F(C \times A)$ is actually a product of $F(C)$ and $F(A)$.
\end{dfn}

\begin{rem}
Fixing a $\mathrm{BC}$-structure for $\mathcal{C}$ and $\mathcal{D}$, one can define the preservation equivalently in the following way. We say that $F$ preserves:
\begin{itemize}
    \item[$\bullet$]
the terminal object iff the canonical map $!: F(1_{\mathcal{C}}) \to 1_{\mathcal{D}}$ is an isomorphism.
    \item[$\bullet$]
the initial object iff the canonical map $!: 0_{\mathcal{C}} \to F(0_{\mathcal{D}})$ is an isomorphism.
\item
the product iff the canonical map
$\langle F(p_0), F(p_1) \rangle: F(A \times B) \to F(A) \times F(B)$ is an isomorphism.
    \item[$\bullet$]
the coproduct iff the canonical map $(F(i_0), F(i_1)) : F(A)+F(B) \to F(A+B)$ is an isomorphism.
 \item[$\bullet$]
the exponential iff the canonical map $\lambda_{F(A)} (F(ev) \circ g) : F([A, B]) \to [F(A), F(B)]$ is an isomorphism, where $g$ is the inverse of the canonical map $\langle F(p_0), F(p_1) \rangle: F([A, B] \times A) \to F([A, B]) \times F(A)$
\end{itemize}
Note that the second definition of preservation demands the canonical morphisms (i.e., the ones induced by the universality of the structure) to be an isomorphism. As we will see later, this is not equivalent to simply having an isomorphism (except for the initial and terminal objects). Even if we have a natural isomorphism, it is still different from having the canonical map as an isomorphism. By these weaker versions, what we really do is preserving the object of the structure (product, coproduct, exponential), while we must also preserve its involved maps (projections, injections, evaluation) as they are also a part of the structure.   
\end{rem}

\begin{dfn}
Let $\mathcal{C}$ and $\mathcal{D}$ be two $\mathrm{BCC}$'s. A functor $F: \mathcal{C} \to \mathcal{D}$ is called a \emph{$\mathrm{BC}$-functor} if it preserves the $\mathrm{BC}$-structure. The $\mathrm{AB}$- and $\mathrm{CC}$-functors are defined similarly, demanding the preservation of the structures available in the $\mathrm{ABC}$'s and $\mathrm{CCC}$'s, respectively.
\end{dfn}

\begin{exam}
The inclusion functor $i: \mathbf{FinSet} \to \mathbf{Set}$, mapping the finite sets and the functions between them to themselves is a $\mathrm{BC}$-functor. The reason is that the construction of the $\mathrm{BC}$-structure on finite sets is exactly the same as the one on all sets. 
\end{exam}

\begin{exam}
The forgetful functor $U: \mathbf{Set}^{\mathbb{Z}} \to \mathbf{Set}$ mapping $(A, \sigma_A)$ to $A$ and equivariant maps to themselves is a $\mathrm{BC}$-functor. To see why, it is enough to go back and check the construction of the $\mathrm{BC}$-structure in the two categories. The only non-trivial thing to mention is that for the set of the exponential object $[(A, \sigma_A), (B, \sigma_B)]$, we used the set $B^A$ of \emph{all} functions from $A$ to $B$ and not just the equivariant functions. This is consistent with our claim that $U$ preserves all exponentials. 
\end{exam}

\begin{exam}\label{OpenMaps}
Let $X$ and $Y$ be two topological spaces and $f: X \to Y$ be an \emph{open map}, i.e., a continuous map, where $f[U]$ is open in $Y$, for any open $U \subseteq X$. Then, the functor $f^{-1}: \mathcal{O}(Y) \to \mathcal{O}(X)$ is a $\mathrm{BC}$-functor. It is clear that it preserves all finite meets and finite joins, simply because the inverse image function always preserves finite intersections and finite unions. For the exponentiation, 
first notice that $f[W \cap f^{-1}(U)]=f[W] \cap U$, for any subsets $W \subseteq X$ and $U \subseteq Y$. Then, consider the following series of equivalences:
\[
W \subseteq f^{-1}(U) \to_X f^{-1}(V) \quad \text{iff} \quad W \cap f^{-1}(U) \subseteq f^{-1}(V) \quad \text{iff} \quad 
\]
\[
f[W \cap f^{-1}(U)] \subseteq V \quad \text{iff} \quad f[W] \cap U \subseteq V.
\]
As $f$ is an open map, $f[W]$ is open, for any open $W \in \mathcal{O}(X)$. Hence, $f[W] \cap U \subseteq V$ is equivalent to $f[W] \subseteq U \to_Y V$ which is also equivalent to $W \subseteq f^{-1}(U \to_Y V)$.
Therefore, as $W$ is an arbitrary open subset of $X$, we reach $f^{-1}(U \to_Y V)=f^{-1}(U) \to_X f^{-1}(V)$.
\end{exam}

\begin{exam}\label{p-morphism}
Let $(P, \leq_P)$ and $(Q, \leq_Q)$ be two posets and $f: (P, \leq_P) \to (Q, \leq_Q)$ be a \emph{p-morphism}, i.e., an order-preserving map such that for any $x \in P$ and $y \in Q$, if $f(x) \leq_Q y$, then there exists $z \in X$ such that $z \geq_P x$ and $f(z)=y$. Then, the functor $f^{-1}: U(Q, \leq_Q) \to U(P, \leq_P)$ is a $\mathrm{BC}$-functor. To prove, it is enough to show that putting the upset topologies on $P$ and $Q$ makes the map $f: P \to Q$ an open map. Let $U \subseteq P$ be an upset. Then, to show that $f[U]$ is also an upset, assume that $y \in f[U]$ and $y \leq_Q w$. Therefore, there is $x \in U$ such that $y=f(x)$. Hence, by the property of the function $f$, there exists $z \in P$ such that $z \geq_P x$ and $w=f(z)$. As $U$ is an upset and $x \in U$, we have $z \in U$. Hence, $w \in f[U]$. Therefore, $f[U]$ is an upset, which means that $f$ is an open map.
\end{exam}

\begin{exam} \label{PosetReflectionIsBCFunctor}
Let $\mathcal{C}$ be a $\mathrm{BCC}$. Then, the forgetful functor $\pi: \mathcal{C} \to \mathrm{Po}(\mathcal{C})$ is a $\mathrm{BC}$-functor. A similar claim holds for the $\mathrm{ABC}$'s and $\mathrm{CCC}$'s.
\end{exam}

\begin{exam}
Let $\mathcal{C}$ be a small $\mathrm{CCC}$. Then, the functor $y: \mathcal{C} \to \mathbf{Set}^{\mathcal{C}^{op}}$ is a $\mathrm{CC}$-functor. We only mention the main points in this example and leave the rest to the reader. For the terminal object, it is enough to note that $y(1)=\mathrm{Hom}(-, 1)$ is actually the constant functor $\Delta_{\{*\}}$ which is the terminal object in $\mathbf{Set}^{\mathcal{C}^{op}}$. For the product, note that the canonical map $y(A \times B) \to y(A) \times y(B)$ is the map $\mathrm{Hom}(-, A \times B) \to \mathrm{Hom}(-, A) \times \mathrm{Hom}(-, B)$ induced by a composition with projections. It is not hard to see that this map is actually an isomorphism. For the exponential, the  canonical map $y([A, B]) \to [y(A), y(B)]$ is the map $\mathrm{Hom}(-, [A, B]) \to \mathrm{Hom}(y_{-} \times y_A, y_B)$ induced by combining the evaluation and the Yoneda functor. Using the Yoneda embedding, it is enough to show that the canonical map $\mathrm{Hom}(-, [A, B]) \to \mathrm{Hom}(- \times A, B)$ induced by the evaluation map is an isomorphism. The latter is mentioned in Remark \ref{HighlevelExp}.
\end{exam}

\begin{exam}(\emph{Non-preservation of terminal objects and products})
Let $\mathcal{C}$ be a $\mathrm{CCC}$ and $X$ be a set with at least two elements. Then, the constant functor $\Delta_{X}: \mathcal{C} \to \mathbf{Set}$ is not a $\mathrm{CC}$-functor. It does not preserve the terminal object as otherwise, $X$ must have been a singleton. It does not preserve the product as otherwise, the canonical map $\Delta_X(A \times B) \to \Delta_X(A) \times \Delta_X(B)$ must be an isomorphism. Computing this map, we can see that it is actually the function $X \to X \times X$ mapping $x$ to $(x, x)$. However, as $X$ has more than one element, the map $x \mapsto (x, x)$ cannot be surjective. Notice that if $X$ is infinite, the objects $\Delta_X(A \times B)$ and $\Delta_X(A) \times \Delta_X(B)$, i.e., $X$ and $X \times X$ are isomorphic in $\mathbf{Set}$. This isomorphism also makes the functors $\Delta_X(- \times -): \mathcal{C}\times \mathcal{C} \to \mathcal{C}$ and $\Delta_X(-) \times \Delta_X(-): \mathcal{C}\times \mathcal{C} \to \mathcal{C}$ isomorphic. However, none of these facts imply that $\Delta_X$ is preserving the product. As we emphasized before, to preserve the product, it is important to have the \emph{canonical morphism} as an isomorphism.
\end{exam}

\begin{exam}(\emph{Non-preservation of exponentials})
Let $X$ and $Y$ be two topological spaces. Then, the inverse image of a continuous map $f: X \to Y$ is not necessarily a $\mathrm{CC}$-functor as it may not preserve the exponentials. For instance, take the continuous map $f: \mathbb{R} \to \mathbb{R}$ defined by $f(x)=x^2$. Let $U=(-\infty, 0)$ and note that $U \to \varnothing=int [0, +\infty)=(0, + \infty)$. Then, $f^{-1}(U \to \varnothing)=f^{-1}((0, + \infty))=\mathbb{R}-\{0\}$. However, $f^{-1}(U) \to f^{-1}(\varnothing)=\varnothing \to \varnothing = \mathbb{R}$ and hence $f^{-1}(U \to \varnothing) \neq f^{-1}(U) \to f^{-1}(\varnothing)$. Notice that $f: \mathbb{R} \to \mathbb{R}$ is not an open map, as $f[\mathbb{R}]=[0, +\infty)$ is not open. 
\end{exam}

\begin{exam}(\emph{Non-preservation of initial objects and  coproducts})
Let $\mathcal{C}$ be a small cocartesian category. Then, the functor $y: \mathcal{C} \to \mathbf{Set}^{\mathcal{C}^{op}}$ does not preserve the initial object or \emph{any} of the coproducts in $\mathcal{C}$. For the initial object, note that $y_0(C)=\mathrm{Hom}(C, 0)$ and this set is not empty for $C=0$. Therefore, $y_0$ cannot be the initial object of $\mathbf{Set}^{\mathcal{C}^{op}}$. Similarly, for the coproducts, notice that $y_{A+B}(C)=\mathrm{Hom}(C, A+B)$ which is not isomorphic to $\mathrm{Hom}(C, A)+\mathrm{Hom}(C, B)$ for $C=0$. Because the first set has exactly one element while the second has exactly two elements. Therefore, $y_{A+B}$ and $y_A+y_B$ are not isomorphic as objects of $\mathbf{Set}^{\mathcal{C}^{op}}$. 
\end{exam}

\begin{rem}
Note that in the previous example, we seriously used the existence of the initial object to prove that $y: \mathcal{C} \to \mathbf{Set}^{\mathcal{C}^{op}}$ does not preserve any existing coproducts. It is actually possible to prove it even without assuming that $\mathcal{C}$ is cocartesian. 
\end{rem}

\subsection{Free Structured Categories}
In the previous subsections, we identified the discourse of $\mathrm{BCC}$'s and $\mathrm{BC}$-functors as a natural categorical framework for formalizing proof systems over $\mathcal{L}_p$ and their interpretations. With this framework established, one might take a conceptual leap to seek a universal property that redefines the syntactical $\mathrm{BCC}$ $\mathbf{NJ}$ in a relative manner, thereby providing a definition for syntax in a syntax-free way.

\begin{thm}\label{Freness}
The category $\mathbf{NJ}$ is the free $\mathrm{BCC}$ generated by the infinite set $\{p_0, p_1, \ldots \}$ of objects, i.e., for any $\mathrm{BCC}$ $\mathcal{C}$ and any assignment $f$ of the objects of $\mathcal{C}$ to the atoms $\{p_0, p_1, \ldots \}$, there is a unique (up to isomorphism) $\mathrm{BC}$-functor $F: \mathbf{NJ} \to \mathcal{C}$ extending $f$, i.e., $F(p_i)=f(p_i)$, for any $i \in \mathbb{N}$. The same also holds for $\mathbf{NM}$ and $\mathbf{NN}$, as the free $\mathrm{ABC}$ and $\mathrm{CCC}$ generated by the set of objects $\{p_0, p_1, \ldots \}$, respectively.
\end{thm}
\begin{proof}
We only explain the case $\mathbf{NJ}$. The rest are similar. First, let us start with an informal explanation of why we expect $\mathbf{NJ}$ to be the free $\mathrm{BCC}$ generated by the set $\{p_0, p_1, \ldots\}$. The reason is simple. Reading category-theoretically, one can see that the propositions in the language $\mathcal{L}_p$ and the rules of $\mathbf{NJ}$ are designed to define a $\mathrm{BC}$-structure on $\mathbf{NJ}$ in a minimal manner.
For instance, the connective $\wedge$ defines the product object $A \wedge B$ of the objects $A$ and $B$. The rules $(E\wedge_1)$ and $(E\wedge_2)$ act as the projection maps for the product $A \wedge B$, and $(I\wedge)$ corresponds to the operation $\langle -, - \rangle$ on the maps, which is required for the products. Similarly, there is a correspondence between the disjunction and its rules with coproducts, the implication and its rules with exponentiation, $\top$ and its axiom with the terminal object, and $\bot$ and its axiom with the initial object.
Of course, there are some issues with the commutativity and the uniqueness conditions in the universal definition of the $\mathrm{BC}$-structure. These are resolved with the minimal equalities that $\beta\eta$-equivalences put on the derivations. Note that these equivalences are nothing but what one needs to ensure the commutativity and the uniqueness conditions that the $\mathrm{BC}$-structure of $\mathbf{NJ}$ demands. Therefore, we naturally expect $\mathbf{NJ}$ to be the ``least" possible $\mathrm{BCC}$ generated by the set $\{p_0, p_1, \ldots\}$ of objects.

To formally prove the freeness, let $f$ be an assignment of the objects of $\mathcal{C}$ to the atoms in $\{p_0, p_1, \ldots \}$. Our task is defining the $\mathrm{BC}$-functor $F: \mathbf{NJ} \to \mathcal{C}$ and proving its uniqueness up to an isomorphism. We will explain the main ingredients of the proof and leave the rest to the reader. First, fix a $\mathrm{BC}$-structure on $\mathcal{C}$, i.e., fix an initial and a terminal object and for any two objects, fix a product, a coproduct and an exponential for them. We will define the assignment $F$ on propositions and derivations in an inductive way. For propositions, set $F(p_i)=f(p_i)$, for any $i \in \mathbb{N}$, $F(\top)=1$, $F(\bot)=0$, $F(A \wedge B)=F(A) \times F(B)$, $F(A \vee B)=F(A)+F(B)$ and $F(A \to B)=[F(A), F(B)]$. 

For derivations, for any finite set of assumptions in the form $[A]^i$, define $F(\Gamma)$ as $\Pi_{\gamma \in \Gamma}F(\gamma)$, using the order on $\Gamma$ induced by the superscripts in $[A]^i$'s. If $\Gamma=\varnothing$, define $F(\Gamma)=1$.
Then, for any derivation $\mathsf{D}$ of $A$ with the set of assumptions $\Gamma$, we want to assign a map $F(\mathsf{D}): F(\Gamma) \to F(A)$. As the derivation $\mathsf{D}$ is a tree constructed from the rules of $\mathbf{NJ}$, it is enough to define $F(\mathsf{D})$ recursively on the sub-derivations of $\mathsf{D}$. If $\mathsf{D}$ is a one-node tree proving $[A]^i$, define $F(\mathsf{D})$ as the projection on the $i$th component of $F(\Gamma)$. For the other rules, use Table \ref{tableD}.

\begin{table}[h!]
 \centering
 \vline
\begin{tabular}{ c | c  }
\hline
 \textbf{Rules} & \textbf{Interpretations} \\
 \hline
\\
\footnotesize \AxiomC{$ $}
  \RightLabel{$(\top)$} 
   \UnaryInfC{$\top$}
 \DisplayProof
&
\footnotesize  $!_{F(\Gamma)} $
 \\
 \\
  \hline
  \\
\footnotesize \AxiomC{$\mathsf{D}$}
 \noLine
 \UnaryInfC{$\bot$}
  \RightLabel{$(\bot)$} 
   \UnaryInfC{$A$} 
 \DisplayProof 
& 
\footnotesize $!_A \circ F(\mathsf{D})$
 \\
 \\
  \hline
  \\
\footnotesize \AxiomC{$\mathsf{D}_1$}
 \noLine
 \UnaryInfC{$A$}
  \AxiomC{$\mathsf{D}_2$}
  \noLine
   \UnaryInfC{$B$}
 \RightLabel{$(I \wedge)$} 
 \BinaryInfC{$A \wedge B$}
 \DisplayProof 
& 
\footnotesize $\langle F(\mathsf{D}_1), F(\mathsf{D}_2) \rangle$
 \\
 \\
  \hline
\\
\footnotesize  \AxiomC{$\mathsf{D}$}
 \noLine
 \UnaryInfC{$A \wedge B$}
  \RightLabel{$(E \wedge_1)$} 
   \UnaryInfC{$A$}
 \DisplayProof
 \quad , \quad
  \AxiomC{$\mathsf{D}$}
 \noLine
 \UnaryInfC{$A \wedge B$}
  \RightLabel{$(E \wedge_2)$} 
   \UnaryInfC{$B$}
 \DisplayProof
&
\footnotesize $p_0 \circ F(\mathsf{D})$ \quad , \quad $p_1 \circ F(\mathsf{D})$
 \\
 \\
  \hline
\\
\footnotesize  \AxiomC{$\mathsf{D}$}
 \noLine
 \UnaryInfC{$A$}
  \RightLabel{$(I \vee_1)$} 
   \UnaryInfC{$A \vee B$}
 \DisplayProof 
\quad , \quad
    \AxiomC{$\mathsf{D}$}
 \noLine
 \UnaryInfC{$B$}
  \RightLabel{$(I \vee_2)$} 
   \UnaryInfC{$A \vee B$}
 \DisplayProof
& 
\footnotesize $i_0 \circ F(\mathsf{D})$ \quad , \quad $i_1 \circ F(\mathsf{D})$
 \\
 \\
 \hline
  \\
\footnotesize \AxiomC{$\mathsf{D}$}
 \noLine
 \UnaryInfC{$A \vee B$}
 \AxiomC{$[A]^i$}
 \noLine
  \UnaryInfC{$\mathsf{D}_1$}
  \noLine
   \UnaryInfC{$C$}
    \AxiomC{$[B]^j$}
    \noLine
  \UnaryInfC{$\mathsf{D}_2$}
  \noLine
   \UnaryInfC{$C$}
 \RightLabel{$(E \vee_{i,j})$} 
 \TrinaryInfC{$C$}
 \DisplayProof
& 
\footnotesize $ev \langle I \langle \lambda_i F(\mathsf{D}_1), \lambda_j F(\mathsf{D}_2) \rangle F(\mathsf{D}) \rangle$
 \\
 \\
 \hline
\\
\footnotesize  \AxiomC{$[A]^i$}
 \noLine
   \UnaryInfC{$\mathsf{D}$}
 \noLine
 \UnaryInfC{$B$}
  \RightLabel{$(I \to_i)$} 
   \UnaryInfC{$A \to B$}
 \DisplayProof
& 
\footnotesize $\lambda_{i} F(\mathsf{D})$
 \\
 \\
 \hline
\\
\footnotesize \AxiomC{$\mathsf{D}_1$}
 \noLine
 \UnaryInfC{$A$}
  \AxiomC{$\mathsf{D}_2$}
  \noLine
   \UnaryInfC{$A \to B$}
 \RightLabel{$(E \to)$} 
 \BinaryInfC{$B$}
 \DisplayProof
& 
\footnotesize $ev \circ \langle F(\mathsf{D}_2), F(\mathsf{D}_1) \rangle$
 \\
 \\
 \hline
\end{tabular}\vline
\caption{Inductive definition of $F(\mathsf{D})$}
\label{tableD}
\end{table}

To read the table, we need the following notational conventions. First, for the map $F(\mathsf{D}): F(\Gamma \cup [A]^i) \to F(B)$, we can use the commutativity of the product to transform $F(\Gamma \cup [A]^i)$ to $F(\Gamma) \times F(A)$, and then using $\lambda_A$, construct a map from $F(\Gamma)$ to $[F(A), F(B)]$. We call this map $\lambda_i F(\mathsf{D})$. Second, notice that in any $\mathrm{BCC}$, there is a canonical map $I: [X, Z] \times [Y, Z] \to [X+Y, Z]$. We use this map to interpret the rule $(E\vee)$.

The definition of $F(\mathsf{D})$ does not respect the renaming of superscripts as any such renaming may change the order of $F(\gamma)$'s in the product $F(\Gamma)=\Pi_{\gamma \in \Gamma} F(\gamma)$. However, as we will finally restrict ourselves to the derivations with one assumption, the map $F$ becomes well-defined at the end. The definition of $F(\mathsf{D})$ respects the $\beta \eta$-equivalences on $\mathsf{D}$. This holds due to the universal properties of the $\mathrm{BC}$-structure on $\mathcal{C}$. We leave its detailed proof to the reader. Then, to get a functor, one must restrict $F$ to the maps in $\mathbf{NJ}$, i.e., to the derivations with one assumption. It is easy to see that this restriction is a functor as $F$ is defined on a one-node tree with one assumption as the identity map and we defined $F$ by composing the values of $F$ on sub-derivations. The functor $F$ respects the $\mathrm{BC}$-structure by design. For instance, we defined $F(A \wedge B) = F(A) \times F(B)$, and the projections in $\mathbf{NJ}$, i.e., the rules $(E\wedge_1)$ and $(E\wedge_2)$ are mapped to the projections of $\mathcal{C}$. As $F$ clearly extends $f$, we constructed the $\mathrm{BC}$-functor we sought.

To prove the uniqueness, note that every proposition is constructed from the atoms and $\{\top, \bot\}$ using $\{\wedge, \vee, \to\}$ in a unique way, and any derivation is constructed from the rules. If $F$ is to be a $\mathrm{BC}$-functor, it must preserve the $\mathrm{BC}$-structure, and hence the definition of $F$ is essentially forced upon us. For instance, we must define $F(A \wedge B)$ as the product $F(A) \times F(B)$, and the application of $F$ on the derivation:
\begin{center}
    \begin{tabular}{c}
        \AxiomC{$A \wedge B$}
  \RightLabel{$(E \wedge_1)$} 
   \UnaryInfC{$A$}
 \DisplayProof
    \end{tabular}
\end{center}
must be the projection $p_0: F(A) \times F(B) \to F(A)$. Therefore,
the only choice we have is the choice of the $\mathrm{BC}$-structure on $\mathcal{C}$ which is unique up to isomorphism. Hence, $F$ must be unique up to an isomorphism. 
\end{proof}

Theorem \ref{Freness} characterizes the $\mathrm{BCC}$ $\mathbf{NJ}$ as the ``least" $\mathrm{BCC}$ containing the given set $\{p_0, p_1, \ldots\}$ of objects. This characterization, like those for products, coproducts, and other universal constructions, is a relative one based on universality and does not rely on the internal structure of the category $\mathbf{NJ}$. The method of constructing $\mathbf{NJ}$ is not crucial; there may be alternative syntactical methods to construct the same category, which can be seen as different \emph{presentations} of the same mathematical entity. In this sense, the definition of $\mathbf{NJ}$ provided by Theorem \ref{Freness} serves as a conceptual characterization and offers insight into how to define syntax in a \emph{presentation-free} manner.

\section{BHK Interpretation and Variable Sets} \label{SectionBHKInterpretation} 

In the previous section, we introduced some universal constructions to interpret the logical constants in the propositional language. For this purpose, we employed categorical language, where propositions are treated as objects and \emph{deductions} as morphisms. This framework effectively formalized the expected behavior of logical constants. However, some anomalies emerged, such as defining the constants $\bot$ and $\vee$ by their behavior as assumptions, based on how they prove other propositions rather than how they are proved themselves.
One could argue that these anomalies result from our choice of deduction as the primitive notion, while the primitive notion should actually be the \emph{proof}, i.e., a deduction with no assumption. Once proofs are specified, a deduction of $B$ from $A$ can always be defined as a \emph{construction} that maps proofs of $A$ to proofs of $B$.
But if proofs are taken as primitives, how do we define the logical connectives? The well-known Brouwer-Heyting-Kolmogorov (BHK) interpretation attempts to answer this question. It defines the proofs of a proposition recursively as follows:
\begin{itemize}
    \item 
    there is a canonical proof for $\top$,
    \item 
    there is no proof for $\bot$,
    \item
    a proof of $A \wedge B$ is a pair of a proof of $A$ and a proof of $B$,
    \item
    a proof of $A \vee B$ is either a proof of $A$ or a proof of $B$,
    \item
    a proof of $A \to B$ is a \emph{construction} that transforms any proof of $A$ to a proof of $B$.
\end{itemize}
The BHK interpretation is clearly based on a notion of \emph{construction}, as the implication clause clearly demonstrates. Even for the other clauses, we must specify a world of constructions from which the clauses select the ones considered as proofs. To formalize this notion of construction, a natural approach is to employ a category with sufficient structure as the world of constructions. Here, we interpret the objects as types and the morphism $f: A \to B$ as a construction of type $B$ based on a given input of type $A$.

There are many such categories, ranging from categories of recursive functions or continuous maps to the usual category of sets. In this section, we restrict ourselves to the classical category of sets to simplify the exposition. However, in Section \ref{SecRealizability} and Section \ref{SecRealizabilityForArith}, we will explore some alternative constructions to illustrate their usefulness and interesting properties. These alternative constructions are typically referred to as the \emph{realizability} interpretation.

Fixing constructions, we also need to incorporate the notion of \emph{time}, an important factor that is often ignored in the usual presentation of the BHK interpretation. To understand the role of time, consider the way a constructive mathematician rejects the axiom of the excluded middle. For her, $A \vee \neg A$ is not acceptable as a \emph{general law} because it asserts that for any proposition $A$, there is either a proof of $A$ or a proof that $A$ leads to a contradiction. This is not true in her view, as there are propositions that are potentially provable but have not been proved \emph{so far}, making them neither provable nor disprovable at this point in time.

However, in the static and timeless picture that the BHK interpretation presents, without any epistemological change, either $A$ is provable and we have its proof, or we do not have any proof because it is impossible to have one. In the latter case, since there is no proof for $A$, one can claim that any proof of $A$ leads to a proof of the contradiction, which provides a proof for $A \to \bot = \neg A$ by the BHK definition of the proofs of implication. Therefore, the BHK interpretation, with classical sets and functions as its world of constructions, validates the axiom of the excluded middle if we ignore the temporal aspect.
Recalling that time is the seed of \emph{Kripke semantics} for constructive theories, it is insightful to consider what Goodman interestingly highlighted:
\begin{quote}
\emph{Recursive realizability emphasizes the active aspect of constructive
mathematics. However, Kleene's notion has the weakness that it disregards
that aspect of constructive mathematics which concern epistemological change. Precisely that aspect of constructive mathematics which Kleene's notion neglects is emphasized by Kripke's semantics for intuitionistic logic. However, Kripke's notion makes it appear that the constructive mathematician is a passive observer of a structure which gradually reveals itself. What is lacking is the emphasis on the mathematician as active which Kleene's notion provides} \cite{Goodman}.
\end{quote}
Using the temporal structure, the informal BHK interpretation can be updated as follows:
\begin{itemize}
    \item 
    there is \emph{always} a canonical proof for $\top$,
    \item 
    there is \emph{never} a proof for $\bot$,
    \item
    a proof of $A \wedge B$ \emph{at some point} is a pair of a proof of $A$ and a proof of $B$ \emph{at the same point},
    \item
    a proof of $A \vee B$ \emph{at some point} is either a proof of $A$ or a proof of $B$ \emph{at the same point},
    \item
    a proof of $A \to B$ \emph{at some point} is a construction that transforms any proof of $A$ \emph{at any point in the future} to a proof of $B$ \emph{at that point}.
\end{itemize}
There is another necessary update to the BHK interpretation concerning the clauses for $\bot$ and $\vee$. We will explain the latter, as the former is similar. In the traditional form of the BHK interpretation, a proof of $A \vee B$ at some point is either a proof of $A$ or a proof of $B$ at that same point. However, this definition can be too stringent for some practices. For instance, if we are generating a sequence of zeros and ones one bit at a time, at stage $n$, we know that the next bit $a_{n+1}$ is either zero or one. Yet, we do not have a proof at that moment to determine which it will be.

This observation suggests weakening the clause for disjunction in the BHK interpretation to state that a proof of $A \vee B$ at time $n$ is a construction that will eventually become a proof of $A$ or $B$ as time progresses, regardless of how it progresses. To formalize this concept, we need a notion called \emph{coverage} or \emph{Grothendieck topology}, which defines how future possibilities can cover the present moment. This chapter will not delve into this concept, as it extends far beyond the scope of our already extensive survey. However, you can explore a basic version of it in the renowned \emph{Beth semantics} for intuitionistic logic \cite{vanDalen,vanDalenII}.

Using the two ingredients of construction and time, and setting aside the role of coverages, we will provide a formalization for a weaker form of the BHK interpretation in the following two subsections.

\subsection{Provability via the BHK interpretation}

Let us start with the simpler task of formalizing the BHK interpretation considered as the inductive definition of the provability relation rather than the actual proofs. As we saw before, this means restricting our structures from categories to posets, where all the morphisms are collapsed into an order relation. 

First, let $\mathcal{T}=(T, \leq)$ be a poset encoding the time structure. Note that time is not linear, as each point in time may branch into different possible futures. Ignoring explicit proofs and focusing solely on provability, a proposition is represented by an order-preserving map from $\mathcal{T}$ to $\mathbf{2}$, indicating whether the proposition is provable at each point in time. It is order-preserving because if a proposition is provable at some point $w \in T$, it remains provable as time progresses. Therefore, the set of all propositions is $\mathbf{2}^{\mathcal{T}}$. For the deductibility order between propositions, we define $A \leq B$ if and only if $A(w) \leq B(w)$ for every $w \in T$. This means that $A$ proves $B$ if the provability of $A$ at any point in time implies the provability of $B$ at that point.

For any proposition $A$ and any point $w \in T$, we write $w \Vdash A$ to denote $A(w) = 1$, meaning that $A$ is provable at $w$. Now, following the BHK interpretation and substituting proofs with provability, we define the provability of complex propositions as follows:
\begin{itemize}
\item
$w \nVdash \bot$, meaning that $\bot$ is never provable.
\item
$w\Vdash \top$, meaning that $\top$ is always provable.
\item
$w \Vdash A \wedge B$ iff $w \Vdash A$ and $w \Vdash B$, meaning that $A \wedge B$ is provable at some point iff both $A$ and $B$ are provable at that point. This is the shadow of the definition of the proofs of a conjunction as the pairs of the proofs of the conjuncts.
\item
$w \Vdash A \vee B$ iff either $w \Vdash A$ or $w \Vdash B$, meaning that $A \vee B$ is provable at some point iff either $A$ or $B$ is provable at that point. This is the shadow of the definition of the proofs of a disjunction as the proofs of either of the disjuncts.
\item
$w \Vdash A \to B$ iff for any $u \geq w$ if $u \Vdash A$ then $u \Vdash B$, meaning that $A \to B$ is provable at some point iff at any point in the future, when $A$ becomes provable, $B$ also becomes provable at that point. This is the shadow of the definition of the proofs of an implication as a construction transforming a proof of $A$ to a proof of $B$. When we have such a construction at $w$, we can apply it at all points in the future to any future proof of $A$ to get a proof of $B$.
\end{itemize}
It is straightforward to verify that the operations $\{\bot, \top, \wedge, \vee, \to\}$ on the poset $\mathbf{2}^{\mathcal{T}}$ align with its $\mathrm{BC}$-structure. For instance, for disjunction, we have $A \vee B \leq C$ if and only if $A \leq C$ and $B \leq C$, for any propositions $A$, $B$, and $C$.
Given this $\mathrm{BC}$-structure, we can conclude that the BHK interpretation yields a poset proof system, specifically a Heyting algebra, as discussed in the previous section.

Two points are worth noting here. First, the BHK interpretation defines provability rather than deductibility. Deductibility is simply the natural order on $\mathbf{2}^{\mathcal{T}}$ induced by the temporal structure of $\mathcal{T}$. This approach allows for a clearer definition of $\bot$ and $\vee$, as it directly addresses their provability rather than their role as assumptions for proving other propositions. Second, the BHK interpretation concretely formalizes poset proof systems, with propositions represented as order-preserving functions rather than abstract elements in a Heyting algebra.

Given these observations, one might wonder if the poset proof system we introduced earlier is more general than what the BHK interpretation provides. Fortunately, this is not the case:

\begin{thm}(Kripke representation for Heyting algebras \cite{vanDalen}) \label{KripkeRepForHeyting}
Let $\mathcal{H}$ be a Heyting algebra. Then, the canonical map $e: \mathcal{H} \to \mathbf{2}^{[\mathcal{H}, \mathbf{2}]}$ defined by $e(a)(f)=f(a)$ is a full and faithful $\mathrm{BC}$-functor, where $[\mathcal{H}, \mathbf{2}]$ is the poset of all $\mathrm{BC}$-functors from $\mathcal{H}$ to $\mathbf{2}$ with the pointwise order.    
\end{thm}

Therefore, any poset proof system can be viewed as a subsystem of the proof system derived from the BHK interpretation. In essence, the BHK interpretation offers a \emph{concrete representation} for all poset proof systems.

\subsection{Proofs via the BHK interpretation}

Having investigated the easier case of provability, we are now ready to focus on the original BHK interpretation. The first point is generalizing the structure of time from a poset to a small category $\mathcal{T}$. The idea is that not only the results of the growth of time are important, but also we care about \emph{how} this expansion happens. Therefore, in the categorical formalization of time, the objects of $\mathcal{T}$ are the points on the timeline, while its morphisms, such as $f: w \to u$, explain how $w$ grows to $u$.

Now, having the category $\mathcal{T}$ to formalize the temporal structure, a proposition is simply an assignment of sets to the objects of $\mathcal{T}$, specifying the proofs of the proposition at each point $w$, together with an assignment of functions to the maps of $\mathcal{T}$, specifying how proofs change through time. The natural way to formalize such data is by variable sets over $\mathcal{T}$, i.e., the functors from $\mathcal{T}$ to $\mathbf{Set}$. Therefore, we assume that $\mathbf{Set}^{\mathcal{T}}$ is the category of propositions. Note that a deduction of the proposition $B$ from the proposition $A$ in this category is defined as a natural transformation that is simply a way to map the proofs of $A$ to the proofs of $B$ at each point $w$ of $\mathcal{T}$ in a uniform and coherent way. This is compatible with our previous discussion of defining deductions as constructions that transform proofs and not as a primitive notion.

For any proposition $A$, write $w, x \Vdash A$ to denote $x \in A(w)$, meaning that $x$ is a proof of $A$ at the point $w$. Then, the BHK interpretation simply states:
\begin{itemize}
\item
there is no $x$ such that $w, x \Vdash \bot$,
\item
$w, * \Vdash \top$, where $*$ is the unique $x$ for which $w, x \Vdash \top$,
\item
$w, \langle x, y \rangle \Vdash A \wedge B$ iff $w, x \Vdash A$ and $w, y \Vdash B$,
\item
$w, \langle i, x \rangle \Vdash A \vee B$ iff either $i=0$ and $w, x \Vdash A$ or $i=1$ and $w, x \Vdash B$, 
\item
$w, \{\alpha_{u}\}_{u \in \mathcal{T}} \Vdash A \to B$ iff for any expansion $f: w \to u$ and any $x$ such that $u, x \Vdash A$, we have $u, \alpha_u(f, x) \Vdash B$.
\end{itemize}
In the last part, the sequence 
$\{\alpha_{u}\}_{u \in \mathcal{T}}$ must be coherent. To explain, note that for any map $g: u \to v$, there are two natural candidates to map the pair of $f: w \to u$ and $x$ as a proof of $A$ at $u$ to a proof of $B$ at $v$. One is using $\alpha_u$ to provide a proof of $B$ at $u$ and then apply $B(g)$ to reach a proof of $B$ at $v$. The other is first changing $(f, x)$ to $(gf, A(g)(x))$ and then apply $\alpha_v$. The coherency simply says that these two ways are equal:
\[\begin{tikzcd}[ampersand replacement=\&]
	{\mathrm{Hom}(w, u) \times A(u)} \&\& {B(u)} \\
	\\
	{\mathrm{Hom}(w, v) \times A(v)} \&\& {B(v)}
	\arrow["{\alpha_u}", from=1-1, to=1-3]
	\arrow["{\alpha_v}"', from=3-1, to=3-3]
	\arrow["{(g \circ -) \times A(g)}"', from=1-1, to=3-1]
	\arrow["{B(g)}", from=1-3, to=3-3]
\end{tikzcd}\]
In other words, we are simply saying that the assignment $\{\alpha_{u}: \mathrm{Hom}(w, u) \times A(u) \to B(u)\}_{u \in \mathcal{T}}$ is a natural transformation.

The astute reader will soon realize that these clauses cannot fully define complex propositions such as $A \wedge B$, $A \vee B$, and $A \to B$, since any proposition is represented as a functor from $\mathcal{T}$ into $\mathbf{Set}$, and what we've defined so far only covers the object part of the functor. However, there is a canonical method to extend this assignment to maps of $\mathcal{T}$, thereby obtaining a functor for each complex proposition. By doing so, it becomes clear that the operations $\{\bot, \top, \wedge, \vee, \to\}$ correspond precisely to the $\mathrm{BC}$-structure of the category $\mathbf{Set}^{\mathcal{T}}$. Therefore, it follows that the BHK interpretation yields a proof system in the sense described in the previous sections, specifically a $\mathrm{BCC}$.

Again, there are a few key points to consider. First, the BHK interpretation focuses on defining proofs rather than deductions. This approach resolves the peculiarities associated with $\bot$ and $\vee$, as they are now characterized by their proofs rather than by their deductive capacities. Second, it's important to emphasize that in this framework, propositions and deductions are represented concretely as variable sets and natural transformations, respectively, rather than as abstract objects and morphisms in a general $\mathrm{BCC}$.

One might wonder whether there is any advantage to using this abstract setting, or if an abstract $\mathrm{BCC}$ can be represented by a category of variable sets. To address this, note that if we shift from $\mathrm{BCC}$s to $\mathrm{CCC}$s, such a representation becomes feasible. Specifically, the Yoneda embedding provides a $\mathrm{CC}$-functor that maps the category $\mathcal{C}$ into $\mathbf{Set}^{\mathcal{C}^{op}}$ in a faithful and full manner. Leveraging a significant representation theorem from topos theory \cite{Moerdijk}, we can even opt for a temporal structure that is a poset, providing a robust connection between abstract and concrete settings.
\begin{thm}(Kripke representation for small $\mathrm{CCC}$'s)
Let $\mathcal{C}$ be a small $\mathrm{CCC}$. Then, there exist a poset $(P, \leq)$ and a full and faithful $\mathrm{CC}$-functor $F: \mathcal{C} \to \mathbf{Set}^{(P, \leq)}$.
\end{thm}
Now, let us return to the general setting of $\mathrm{BCC}$s. Contrary to what one might expect, it is not always possible to embed a $\mathrm{BCC}$ into a $\mathrm{BCC}$ of the form $\mathbf{Set}^{\mathcal{T}}$ via a full and faithful $\mathrm{BC}$-functor. In fact, some $\mathrm{BCC}$s cannot be mapped by any $\mathrm{BC}$-functor into categories of variable sets.

The reason for this limitation is straightforward. In the category $\mathbf{Set}^{\mathcal{T}}$, the object $1 + 1$ is not isomorphic to $1$ because $1 + 1$ has one element at each $w \in \mathcal{T}$, while $1$ always has a single element. However, in certain $\mathrm{BCC}$s, such as Heyting algebras, we have $1 \cong 1 + 1$, and this property is preserved by any $\mathrm{BC}$-functor, whether or not it is full or faithful. Philosophically, this distinction highlights that the BHK interpretation imposes a specific structure that ensures the proofs of $A$ and $B$ within $A \vee B$ are disjoint. In contrast, arbitrary $\mathrm{BCC}$s allow for a more flexible integration of proofs.
\begin{rem}
If the reader is familiar with Grothendieck toposes, which incorporate the coverage structures we omitted, they might suspect that using sheaves instead of presheaves could address this issue in the BHK interpretation. Unfortunately, this is not the case. The issue persists because, even in non-trivial Grothendieck toposes, we still have $1 + 1 \ncong 1$.
\end{rem}
As it may be clear now, the issue with the representation claim lies in the cocartesian structure of a $\mathrm{BCC}$.
To provide stronger evidence, we demonstrate that some cocartesian categories cannot be mapped into a category of variable sets via a weakly-full functor that preserves the cocartesian structure.
To prove this, we define the property $(*)$ in a cocartesian category as follows: For any diagram of the form:
\[\begin{tikzcd}
	C && A \\
	\\
	B && {A+B}
	\arrow["{i_1}"', from=3-1, to=3-3]
	\arrow["{i_0}", from=1-3, to=3-3]
	\arrow[from=1-1, to=1-3]
	\arrow[from=1-1, to=3-1]
\end{tikzcd}\]
where $i_0$ and $i_1$ are the injections, there is a map in $\mathrm{Hom}(C, 0)$. Note that the property $(*)$ is somehow stating that the coproduct is disjoint.

\begin{lem}
Let $\mathcal{C}$ and $\mathcal{D}$ be two cocartesian categories such that there is a weakly-full functor $F: \mathcal{C} \to \mathcal{D}$ preserving the initial object and the binary coproducts. Then, if $\mathcal{D}$ has the property $(*)$, then so does $\mathcal{C}$.
\end{lem}
\begin{proof}
Let the following diagram commute in $\mathcal{C}$:
\[\begin{tikzcd}
	C && A \\
	\\
	B && {A+B}
	\arrow["{i_1}"', from=3-1, to=3-3]
	\arrow["{i_0}", from=1-3, to=3-3]
	\arrow[from=1-1, to=1-3]
	\arrow[from=1-1, to=3-1]
\end{tikzcd}\]
Then, applying the functor $F$ and using the fact that $F$ preserves the coproducts, we can use $(*)$ to have a map in $\mathrm{Hom}(F(C), 0)$. Since $F(0)$ is an initial object, there is a map in $\mathrm{Hom}(F(C), F(0))$. As $F$ is weakly-full, there is a map in $\mathrm{Hom}(C, 0)$.
\end{proof}
Now, observe that for any category of variable sets (and all Grothendieck toposes), the property $(*)$ holds, as $A + B$ is the pointwise \emph{disjoint} union of sets. However, $(*)$ does not hold for non-trivial cocartesian posets. For example, consider a poset with $a \neq 0$. In such a poset, we have:
\[\begin{tikzcd}[ampersand replacement=\&]
	a \&\& a \\
	\\
	a \&\& {a \vee a=a}
	\arrow["\leq"{marking}, draw=none, from=3-1, to=3-3]
	\arrow["\leq"{marking}, draw=none, from=1-3, to=3-3]
	\arrow["\leq"{marking}, draw=none, from=1-1, to=1-3]
	\arrow["\leq"{marking}, draw=none, from=1-1, to=3-1]
\end{tikzcd}\]
However, since $a \nleq 0$, if such posets could be mapped via a weakly-full functor that preserves the cocartesian structure into a category of the form $\mathbf{Set}^{\mathcal{T}}$, then $\mathbf{Set}^{\mathcal{T}}$, having the property $(*)$, would force the poset to also possess this property. This leads to a contradiction.
It is crucial to note that the weakly-fullness condition cannot be omitted. Without this condition, any cocartesian category $\mathcal{C}$ could be trivially mapped into the initial object $0$ of a category of variable sets using the constant functor $\Delta_0$. While this functor is cocartesian, it is not necessarily weakly-full.

\section{Classical Deductions} \label{SectionClassical}

We have examined how $\mathrm{BCC}$'s offer a categorical formalization of intuitionistic proof systems and explored their interconnections, including the syntactical systems as free categories and the BHK interpretation as a source of a special family of proof systems. Given this comprehensive framework for intuitionistic proofs, it is natural to seek a comparable formalization for classical proof systems. Unfortunately, the situation is less satisfactory.
To proceed, we begin with a useful lemma:
\begin{lem}\label{NegativeProofsAreUnique}
In a $\mathrm{BCC}$, all maps in $\mathrm{Hom}(A, 0)$ are isomorphisms, for any object $A$. As a consequence, there is at most one map in $\mathrm{Hom}(B, \neg C)$, for any objects $B$ and $C$.
\end{lem}
\begin{proof}
First, observe that $0 \times A$ is an initial object. To prove, note that $\mathrm{Hom}(0 \times A, B) \cong \mathrm{Hom}(0, B^A)$ and as $0$ is initial, there is exactly one map in $\mathrm{Hom}(0, B^A)$ which also means there is exactly one map from $0 \times A$ to $B$, for any arbitrary object $B$. Hence, $0 \times A$ is initial.
Now, let $f: A \to 0$ be a map.
We claim that the map $!: 0 \to A$ is the inverse of $f: A \to 0$. As $0$ is initial, it is clear that $f!=id_0$. Therefore, we have to show that $!f=id_A$. For that purpose, consider the following diagram:
\[\begin{tikzcd}[ampersand replacement=\&]
	A \&\& {0 \times A} \&\& A \\
	\\
	A \&\& 0 \&\& A
	\arrow["{\langle f, id_A \rangle}", from=1-1, to=1-3]
	\arrow["{id_A}"', from=1-1, to=3-1]
	\arrow["{p_1}", from=1-3, to=1-5]
	\arrow["{p_0}"', from=1-3, to=3-3]
	\arrow["{id_A}", from=1-5, to=3-5]
	\arrow["f"', from=3-1, to=3-3]
	\arrow["{!}"', from=3-3, to=3-5]
\end{tikzcd}\]
The left square trivially commutes. The right square commutes as $0 \times A$ is an initial object. Therefore, the whole rectangular commutes. As the upper row is $id_A$, the lower row must also be $id_A$. Therefore, $!f=id_A$ and hence $f: A \to 0$ is an isomorphism.  
For the second part, first notice that there is at most one map in $\mathrm{Hom}(A, 0)$, for any object $A$. The reason is that if $\mathrm{Hom}(A, 0)$ is non-empty, by the first part, $A \cong 0$ which implies the initiality of $A$. Hence, $\mathrm{Hom}(A, 0)$ must be a singleton. Using this fact, as $\mathrm{Hom}(B, \neg C) \cong \mathrm{Hom}(B \times C, 0)$, it is trivial that $\mathrm{Hom}(B, \neg C)$ has at most one element.
\end{proof}

\begin{phil}
Reading proof theoretically, Lemma \ref{NegativeProofsAreUnique} simply states that the negation of a formula has at most one proof because any such proof is just a deduction of $\bot$ from the formula itself. This result is quite counter-intuitive for some logicians, as there are apparently some negative formulas with different proofs in mathematics. For these logicians, what Lemma \ref{NegativeProofsAreUnique} actually implies is that identifying $\bot$ with the initial object is too demanding and must be relaxed.
There are many ways to solve this issue.
For instance, one can identify $\bot$ as a \emph{weak initial object}, i.e., an object that has a (not necessarily unique) map to any other object. As another way, one can simply move to the setting of minimal logic, omitting $\bot$ completely from the language.
For other logicians, the uniqueness of the proofs of the negative propositions is not only intuitive and justifiable but also has some explanatory power. For these logicians, all proofs of a negative formula are equivalent because the proofs of the negative statements carry no information other than their mere truth or provability. This is a known intuition in constructive mathematics that negative formulas behave more like classical propositions where only truth plays a role and not the actual proofs. One can even say that Lemma \ref{NegativeProofsAreUnique} actually provides a formal proof for this intuition.
\end{phil}

Now, it is time to return to our task of providing a categorical formalization for classical proof systems. There are different interpretations of how classicality must be captured in the categorical setting. In the most obvious approach, one can argue that classical proofs, whatever they are, generalize intuitionistic proofs. Therefore, the category of classical deductions must at least be a $\mathrm{BCC}$. To incorporate the essence of classical proofs, we have the following four proposals:
\begin{itemize}
\item
\textbf{Logical equivalence:} There is a map from $\neg \neg A$ to $A$, for any object $A$. 
\item
\textbf{Isomorphism:} $A \cong \neg \neg A$, for any object $A$. 
\item
\textbf{Natural Isomorphism:} $A \cong \neg \neg A$, natural in $A$, i.e., there is a natural transformation $\alpha : id_{\mathcal{C}} \to \neg \neg (-)$ such that $\alpha_A: A \to \neg \neg A$ is an isomorphisms, for any object $A$.
\item
\textbf{Canonical Isomorphism:} The canonical map $\lambda_{\neg A} ev \langle p_1, p_0 \rangle$ from $A$ to $\neg \neg A$ is an isomorphism. 
\end{itemize}
Using Lemma \ref{NegativeProofsAreUnique}, there is at most one map in $\mathrm{Hom}(A, \neg\neg A)$. Thus, the last three options are all equivalent. However, the weakest form is different from the others. It only ensures the existence of a deduction of $A$ from $\neg \neg A$ for any proposition $A$. This cannot fully capture the essence of classical proofs, as classical mathematicians typically believe that the propositions $A$ and $\neg \neg A$ are \emph{identical}, not merely logically equivalent. Therefore, we are led to define a category of classical proofs as a $\mathrm{BCC}$ where $A \cong \neg \neg A$, for any object $A$.

\begin{thm}(Joyal) \label{Joyal}
A $\mathrm{BCC}$ in which $A \cong \neg \neg A$, for any object $A$ is a preordered set.
\end{thm}
\begin{proof}
By lemma \ref{NegativeProofsAreUnique}, there is at most one map in $\mathrm{Hom}(A, \neg \neg B)$. As $B \cong \neg \neg B$, there is at most one map in $\mathrm{Hom}(A, B)$, for any two objects $A$ and $B$. Hence, the category is a preordered set.
\end{proof}

Similar to how we interpreted negations, Theorem \ref{Joyal} can be read in two distinct ways. The first interpretation suggests that there are no non-trivial classical deductions, reflecting the idea that classical mathematics is grounded in truth or, equivalently, in deductibility where deductions have no substantive role beyond their mere existence. In this view, classical proofs are seen as fundamentally trivial or non-informative.

Conversely, the second interpretation offers a different perspective: while classical proofs may not be as informative as constructive proofs, they still contain some amount of information and are not trivially unique. Thus, Theorem \ref{Joyal} demonstrates that the combination of the $\mathrm{BC}$-structure and the isomorphism between $A$ and $\neg \neg A$ might be too restrictive. Consequently, alternative formalizations are needed to capture classical proof systems. Finding a satisfactory formalization has been a longstanding open problem with many proposed solutions.

There are primarily two approaches to developing such alternatives. The first approach maintains the $\mathrm{CC}$-structure but relaxes the cocartesian structure. The second approach acknowledges the importance of preserving the symmetry between conjunctive and disjunctive operations by weakening both cartesian and cocartesian structures while ensuring they remain dual to each other. Both of these approaches have many instances. To convey their taste, we will present one representative of each of these approaches: \emph{control categories} and \emph{classical categories}.

\subsection{Control Categories}

The first approach involves using the categorical reading of constructive proofs combined with a way to reduce classical deductions to the constructive ones. 
There are different ways to implement such a combination \cite{Hofmann,hofmannI,parigot1992lambdamu,ong,pym2001semantics,Selinger}. Here, we focus on a route leading to one specific example of \emph{control categories}. In this case, the reduction of classical deductions to the constructive ones is via \emph{negative translation}. Therefore, we must first generalize this translation from logics, i.e., the preorders, to categories. Then, the translation helps to identify classical deductions within constructive deductions and aims to axiomatize the structure of these deductions universally.

To present a generalization of the negative translation, recall that a \emph{boolean algebra} is a Heyting algebra $\mathcal{H}$ where $a = \neg \neg a$ for any $a \in \mathcal{H}$. Then:

\begin{thm}
Let $\mathcal{H}$ be an almost Heyting algebra and $r \in \mathcal{H}$ be an arbitrary element. Then, the subset $\mathcal{H}_r=\{x \to r \mid x \in \mathcal{H}\}$ of $\mathcal{H}$ together with the order of $\mathcal{H}$ is a boolean algebra.
\end{thm}
\begin{proof}
First, note that the following data actually describes the $\mathrm{BC}$-structure of $\mathcal{H}_r$:
\begin{itemize}
   \item 
$0_r=(1 \to r)=r$ and $1_r=(0 \to r)=1$,
    \item 
$(x \to r) \wedge_r (y \to r)=(x \vee y \to r)$,
    \item 
$(x \to r) \to_r (y \to r)=((x \to r) \wedge y) \to r$,
    \item
$(x \to r) \vee_r (y \to r)=(x \wedge y \to r)$.
\end{itemize}
To prove that it is boolean, note that $\neg_r (x \to r)=(x \to r) \to r$ and hence $\neg_r \neg_r (x \to r)=[((x \to r) \to r)] \to r$. It is easy to prove that the latter is actually $x \to r$. Therefore, $\neg_r\neg_r (x \to r)= (x \to r)$.
\end{proof}

One can lift the above technique from the posetal $\mathrm{ABC}$'s, i.e., the almost Heyting algebras to all $\mathrm{ABC}$'s \cite{Hofmann,Selinger}. Let $\mathcal{C}$ be an $\mathrm{ABC}$ and $R$ be an arbitrary object such that the canonical map $\lambda_{R^A} ev \langle p_1, p_0 \rangle: A \to R^{R^A}$ is monic, i.e., for any two maps $f, g: X \to A$, if $\lambda_{R^A} ev \langle p_1, p_0 \rangle f=\lambda_{R^A} ev \langle p_1, p_0 \rangle g$, then $f=g$. Consider the category $\mathcal{C}_R$ consisting of all the objects of the form $R^A$ in $\mathcal{C}$ as its objects and \emph{all} maps between these objects as its morphisms. The first observation is that the category $\mathcal{C}_R$ is actually a $\mathrm{CCC}$. Its terminal object is simply $1_{\mathcal{C}}^R$, which is also a terminal object in $\mathcal{C}$. For the product of $R^A$ and $R^B$, first consider the maps: 
\[\begin{tikzcd}[ampersand replacement=\&]
	{R^{A+B} \times A} \&\& {R^{A+B} \times (A+B)} \&\& R \\
	{R^{A+B} \times B} \&\& {R^{A+B} \times (A+B)} \&\& R
	\arrow["{id_{R^{A+B}} \times i_0}"', from=1-1, to=1-3]
	\arrow["ev"', from=1-3, to=1-5]
	\arrow["{id_{R^{A+B}} \times i_1}"', from=2-1, to=2-3]
	\arrow["ev"', from=2-3, to=2-5]
\end{tikzcd}\]
Then, it is enough to set $R^{A+B}$ as the product together with $\lambda_A ev(id_{R^{A+B}} \times i_0): R^{A+B} \to R^A$ and $\lambda_B ev(id_{R^{A+B}} \times i_1): R^{A+B} \to R^B$ as the projections. For the exponentiation of $R^A$ and $R^B$, first consider the map:
\[\begin{tikzcd}[ampersand replacement=\&]
	{R^{R^A \times B} \times R^A \times B} \&\& R
	\arrow["ev"', from=1-1, to=1-3]
\end{tikzcd}\]
Then, it is enough to set  $R^{R^A \times B}$ as the exponential object together with the evaluation map $\lambda_{B} ev: R^{R^A \times B} \times R^A \to R^B$. This completes the presentation of the $\mathrm{CC}$-structure of the category $\mathcal{C}_R$.

For the cocartesian structure, unlike the posetal case, the object $R$ is not necessarily the initial object of $\mathcal{C}_R$. For instance, set $\mathcal{C}=\mathbf{Set}$ and $R=\{0, 1\}$. Then, $\mathcal{C}_R$ is the category of all powersets together with all the functions between them. It is easy to see that $\{0, 1\}$ is not an initial object in $\mathcal{C}_R$. In fact, this category has no initial object. For the coproducts, the situation is worse. Our candidate for the coproduct, i.e., the operation mapping $R^A$ and $R^B$ to $R^{A \times B}$ is defined on the objects. However, there is no natural way to extend this assignment to the morphisms and gain a functor on $\mathcal{C}_R \times \mathcal{C}_R$.
In fact, it is possible to prove that this operation on the objects gives rise to a functor iff the category $\mathcal{C}_R$ is a preordered set \cite{Selinger}. 
Having said that, it is still possible to sacrifice the cocartesian structure and interpret $\bot$ and the disjunction as $R$ and the operation $(R^A, R^B) \mapsto R^{A \times B}$, respectively. Of course, the lack of the functorial behavior of the latter blocks our disjunction to operate on the deductions properly.  
Finally, for the classicality, one can see that $\mathcal{C}_R$ is classical in the sense of the logical equivalence we had above, i.e., there is a map $f: \neg \neg A \to A$, for any object $A$. In fact, a stronger property holds, i.e., the map $f: \neg \neg A \to A$ is the left inverse of the canonical map $\lambda_{\neg A} ev \langle p_1, p_0 \rangle$. One can axiomatize $\mathcal{C}_R$'s with their explained structures as the classical proof systems \cite{Selinger}. Those categories are called \emph{control categories}. 

\subsection{A General Categorical Interpretation}

We observed that employing the $\mathrm{BC}$-structure, along with the assumption $A \cong \neg \neg A$, leads to a complete collapse of classical proofs into mere provability. Since $A \cong \neg \neg A$ lies at the core of classical reasoning, one way to prevent this collapse is to weaken the $\mathrm{BC}$-structure. This requires revisiting our prior discussions to identify a weaker alternative for interpreting the connectives $\{\top, \bot, \wedge, \vee, \to\}$.

As mentioned earlier, in the second approach to handling the collapse, we aim to preserve the full duality between the conjunctive structure $\{\top, \wedge\}$ and the disjunctive structure $\{\bot, \vee\}$. To simplify our presentation, we therefore restrict ourselves to the complete and economical fragment $\{\top, \wedge, \to, \bot\}$, defining the remaining connectives by duality.\footnote{Alternative fragments, such as $\{\top, \wedge, \neg\}$ or $\{\top, \bot, \wedge, \vee, \neg\}$, are also viable and suited to specific purposes. For instance, the latter simplifies computation and facilitates the study of conjunction and disjunction, even in non-classical settings \cite{Seely,fuhrmann2006order,fuhrmann2007categorical}.}

To proceed, we begin with a category \(\mathcal{C}\). We interpret \(\{\top, \bot\}\) as objects of \(\mathcal{C}\), \(\wedge\) as a functor from \(\mathcal{C} \times \mathcal{C}\) to \(\mathcal{C}\), and \(\to\) as a functor from \(\mathcal{C}^{op} \times \mathcal{C}\) to \(\mathcal{C}\), subject to properties that reflect the expected behavior of these connectives. Naturally, we cannot approach this blindly; an intuitive understanding of these connectives is essential to guide the formulation of a coherent family of axioms. In the next subsubsection, we will outline this intuitive framework and the fundamental properties it suggests.  
Such an interpretation and its resulting formalism are not unique \cite{Lutz, hyland2002proof, hyland2004abstract, HylandClassical, dosen2013cut, DosenBook}. Here, we follow a path that leads to the specific example of \emph{classical categories}, introduced by Fuhrmann and Pym in \cite{fuhrmann2004geometry, fuhrmann2006order, fuhrmann2007categorical}. Classical categories elegantly capture the inherent categorical structure of classical deductions and offer an insightful solution to the longstanding open problem of providing a natural categorical semantics for classical logic that does not reduce merely to provability. Moreover, the interpretation is \emph{sound and complete} with respect to a natural form of classical proofs, known as $\mathbf{LK}$-proofs, making classical categories a solid and well-founded framework for categorical interpretations of classical deductions.

\subsubsection{$*$-Autonomous Categories}

Let us begin with an intuitive understanding of the conjunctive structure. Previously, we characterized the conjunctive structure as the product structure by interpreting the deductions of $A \wedge B$ from $C$ in terms of pairs of deductions of $A$ and $B$ from $C$. To develop a more general interpretation of the conjunctive structure, we must shift our focus from the \emph{deductions of} a conjunction to the \emph{deductions from} it. In other words, we must think of a conjunction as an \emph{assumption} rather than a \emph{conclusion}.

Following this intuition, $\top$ should be understood as the ``empty set of assumptions,'' since adding $\top$ to the assumptions has no effect on deductions. Similarly, $\wedge$ should be seen as a ``combining the assumptions'' operation, as assuming $A \wedge B$ has the same effect as assuming both $A$ and $B$. Guided by this intuition, several natural properties of $\top$ and $\wedge$ emerge.

For example, since ``combining the assumptions" is inherently symmetric and associative, we expect $A \wedge B \cong B \wedge A$ and $(A \wedge B) \wedge C \cong A \wedge (B \wedge C)$. Additionally, since adding nothing to $A$ should leave it unchanged, we anticipate $\top \wedge A \cong A \wedge \top \cong A$. This generalization of the monoid structure to categories, where we must also account for morphisms, is called a \emph{symmetric monoidal structure}. Hence, the conjunctive structure must at least be symmetric monoidal.

However, to define a symmetric monoidal structure formally, the above isomorphisms alone are insufficient. The reason is that, although these isomorphisms are \emph{not} equalities, we expect them to behave \emph{as if} they were. This means that any two ways of composing these isomorphisms to move from one object to another must yield the same result.\footnote{This issue is known as the \emph{coherency issue}, and its precise formal definition is rather subtle, varying depending on the context. Here, we use it only in an informal sense.} For instance, let $\sigma_{A, B}$ denote the isomorphism from $A \wedge B$ to $B \wedge A$. There are two distinct ways to derive $A \wedge B$ from itself: one is by using the identity map, and the other is by composing $\sigma_{B, A}$ with $\sigma_{A, B}$. This situation can be visualized in the following diagram:
\[\begin{tikzcd}
	{A \wedge B} && {B \wedge A} \\
	\\
	& {A \wedge B}
	\arrow["{\sigma_{A, B}}", from=1-1, to=1-3]
	\arrow["{id_{A \wedge B}}"', from=1-1, to=3-2]
	\arrow["{\sigma_{B, A}}", from=1-3, to=3-2]
\end{tikzcd}\]
Naturally, we expect this diagram to commute, as $\sigma_{A, B}$ and $\sigma_{B, A}$ are expected to act as identities.  

Given this discussion, completing the axiomatization of a symmetric monoidal structure requires enforcing all equalities between morphisms in the category that are induced by the above isomorphisms. Since there are numerous ways to compose these isomorphisms, one might find it challenging—or even seemingly impossible—to provide a clean and comprehensive axiomatization.
Fortunately, a few fundamental equalities, listed in the next definition, exist that, once established, suffice to derive all other required identities within the structure \cite{CWM,99}.

\begin{dfn}
By a \emph{symmetric monoidal structure} on $\mathcal{C}$, we mean a tuple $(\otimes, I, \alpha, \lambda, \rho, \sigma)$, where:
\begin{itemize}
    \item[$\bullet$] 
$\otimes: \mathcal{C} \times \mathcal{C} \to \mathcal{C}$ is a functor called the \emph{tensor product} and $I \in \mathcal{C}$ is an object called the \emph{unit object},
    \item[$\bullet$] 
$\alpha_{A, B, C}: (A \otimes B) \otimes C \to A \otimes (B \otimes C)$ is a natural isomorphism called the \emph{associator},
    \item[$\bullet$] 
$\lambda_A: I \otimes A \to A$ is a natural isomorphism called the \emph{left unitor},
\item[$\bullet$]
$\rho_A: A \otimes I \to A$ is a natural isomorphism called the \emph{right unitor}, and
\item[$\bullet$]
$\sigma_{A, B}: A \otimes B \to B \otimes A$ is a natural isomorphism called the \emph{braiding} 
\end{itemize}
such that the following diagrams commute:
\begin{itemize}
    \item[$\bullet$] \emph{triangle identity:}
\[\begin{tikzcd}
	{(A \otimes I) \otimes B} && {A \otimes (I \otimes B)} \\
	\\
	& {A \otimes B}
	\arrow["{\alpha_{A,I,B}}", from=1-1, to=1-3]
	\arrow["{\rho_A \otimes id_B}"', from=1-1, to=3-2]
	\arrow["{id_A \otimes \lambda_B}", from=1-3, to=3-2]
\end{tikzcd}\]
\item[$\bullet$] \emph{pentagon identity:}
\[\begin{tikzcd}
	& {(A \otimes B) \otimes (C \otimes D)} \\
	\\
	{((A \otimes B) \otimes C) \otimes D} && {A \otimes (B \otimes (C \otimes D))} \\
	\\
	{(A \otimes (B \otimes C)) \otimes D} && {A \otimes ((B \otimes C) \otimes D)}
	\arrow["{\alpha_{A, B,C \otimes D}}", from=1-2, to=3-3]
	\arrow["{\alpha_{A \otimes B,C,D}}", from=3-1, to=1-2]
	\arrow["{\alpha_{A,B,C} \otimes id_D}"', from=3-1, to=5-1]
	\arrow["{\alpha_{A,B\otimes C,D} }"', from=5-1, to=5-3]
	\arrow["{id_A \otimes \alpha_{B,C,D}}"', from=5-3, to=3-3]
\end{tikzcd}\]
\item[$\bullet$] \emph{hexagon identity:}
\[\begin{tikzcd}
	{(A \otimes B) \otimes C} && {A \otimes (B \otimes C)} && {(B \otimes C) \otimes A} \\
	\\
	{(B \otimes A) \otimes C} && {B \otimes (A \otimes C)} && {B \otimes (C \otimes A)}
	\arrow["{\alpha_{A, B, C}}", from=1-1, to=1-3]
	\arrow["{\sigma_{A, B} \otimes id_C}"', from=1-1, to=3-1]
	\arrow["{\sigma_{A, B\otimes C}}", from=1-3, to=1-5]
	\arrow["{\alpha_{B,C,A}}", from=1-5, to=3-5]
	\arrow["{\alpha_{B,A, C} }"', from=3-1, to=3-3]
	\arrow["{id_B \otimes \sigma_{A, C}}"', from=3-3, to=3-5]
\end{tikzcd}\]
\end{itemize}
and $\sigma_{A, B}\sigma_{B,A}=id_{A \otimes B}$. By a slight abuse of notation, we occasionally denote the symmetric monoidal structure $(\otimes, I, \alpha, \lambda,  \rho, \sigma)$ simply by $(\otimes, I)$. Moreover, we call the tuple $(\mathcal{C}, \otimes ,I)$ a \emph{symmetric monoidal category}.
\end{dfn}

\begin{phil}
The triangle identity states that the two canonical ways of absorbing the empty set between two assumptions are equal. The pentagon identity asserts that all canonical ways of rearranging parentheses—from left-to-right to right-to-left are equal. Lastly, the hexagon identity ensures that repeated swapping is equivalent to swapping with the combined set of assumptions.
\end{phil}

This completes our first attempt at axiomatizing the conjunctive structure. However, some properties of the conjunctive structure remain unaddressed.
Before proceeding further, let us incorporate the necessary structure for $\to$ and $\bot$. For implication, we follow the earlier idea: deductions of $A \to B$ from $C$ correspond bijectively and uniformly to deductions of $B$ from $C \wedge A$. Formally, this means:

\begin{dfn}
Let $(\otimes, I)$ be a symmetric monoidal structure on $\mathcal{C}$. The tuple $(\otimes, I, [-, -], \gamma)$ is called a \emph{symmetric monoidal closed} structure if $[-,-]: \mathcal{C}^{op} \times \mathcal{C} \to \mathcal{C}$ is a functor and 
\[
\gamma_{A, B, C}: \mathrm{Hom}_{\mathcal{C}}(C \otimes A, B) \to \mathrm{Hom}_{\mathcal{C}}(C, [A, B])
\]
is a natural isomorphism. By abuse of notation, we often omit $\gamma$ and refer to the symmetric monoidal closed structure simply as $(\otimes, I, [-, -])$.  Moreover, we call the tuple $(\mathcal{C}, \otimes ,I, [-, -])$ a \emph{symmetric monoidal closed category}.
\end{dfn}

\begin{exam}
Let $\mathcal{C}$ be a cartesian category. Then, by interpreting $\otimes$ as the binary product, $I$ as the terminal object, and taking $\alpha$, $\lambda$, $\rho$, and $\sigma$ in their canonical forms, we obtain a symmetric monoidal structure on $\mathcal{C}$. If $\mathcal{C}$ is also a $\mathrm{CCC}$, defining $[-, -]: \mathcal{C}^{op} \times \mathcal{C} \to \mathcal{C}$ as exponentiation with the canonical choice of $\gamma$ leads to a symmetric monoidal closed structure on $\mathcal{C}$.
\end{exam}

\begin{exam}\label{Monoid}
Let $(M, \cdot, e)$ be a commutative monoid, and let $\mathcal{C}$ be the discrete category over the set $M$. Define $(\otimes, I)$ as $(\cdot, e)$ and set all the required natural isomorphisms as identity maps. It is straightforward to verify that $(\otimes, I)$ forms a symmetric monoidal structure on $\mathcal{C}$. If we also assume that $(M, \cdot, e)$ is a group, by defining $[a, b] = a^{-1} b$, we can make the symmetric monoidal structure closed.
\end{exam}

\begin{phil}
By Example \ref{Monoid}, a symmetric monoidal structure is a generalization of the monoid structure, extending it from small discrete categories to arbitrary categories.
\end{phil}

\begin{exam}\label{RelAsSMC}
Let $\mathcal{C} = \mathbf{Rel}$ be the category with sets as objects and relations $R \subseteq A \times B$ as morphisms $R: A \to B$. The identity morphism on $A$ is given by the identity relation $\{(a, a) \mid a \in A\}$, and composition is the usual composition of relations.
Define $\otimes: \mathbf{Rel} \times \mathbf{Rel} \to \mathbf{Rel}$ as the cartesian product functor, where $A \otimes B = A \times B$ for sets $A$ and $B$, and for relations $R \subseteq A \times B$ and $S \subseteq C \times D$, we define $R \otimes S \subseteq (A \times C) \times (B \times D)$ by
\[
((a, c), (b, d)) \in R \otimes S \iff (a, b) \in R \text{ and } (c, d) \in S.
\]  
Let $I = \{*\}$. With the canonical choices for $\alpha$, $\lambda$, $\rho$, and $\sigma$, this defines a symmetric monoidal structure on $\mathbf{Rel}$. It is worth noting that this cartesian product functor is not the binary product of $\mathbf{Rel}$. The product and the coproduct functors are both the disjoint union functor on $\mathbf{Rel}$.

Next, define the functor $(-)^{\bot}: \mathbf{Rel}^{op} \to \mathbf{Rel}$ by:
\[
A^{\bot} = A, \quad R^{\bot} = R^{op},
\]  
where $R^{op} \subseteq B \times A$ is the converse relation, i.e., $(b, a) \in R^{op}$ if and only if $(a, b) \in R$, for any $a \in A$ and $b \in B$.
Now, define $[A, B] = A^{\bot} \otimes B$ with the canonical choice of the natural isomorphism $\gamma$. It is straightforward to verify that this setup satisfies the axioms of a symmetric monoidal closed structure. Throughout this section, whenever we refer to $\mathbf{Rel}$, we always mean this category equipped with this symmetric monoidal closed structure.
\end{exam}

\begin{exam}
Let $K$ be a fixed field, and let $\mathcal{C} = \mathbf{Vect}_K$ be the category of all $K$-vector spaces with $K$-linear maps as morphisms. Define $\otimes$ as the usual tensor product functor and set $I = K$. With the canonical choices for $\alpha$, $\lambda$, $\rho$, and $\sigma$, this gives a symmetric monoidal structure on $\mathbf{Vect}_K$.

Next, define the functor $[-, -]: \mathbf{Vect}_K^{op} \times \mathbf{Vect}_K \to \mathbf{Vect}_K$ by $[V, W] = \mathrm{Hom}_K(V, W)$ as the $K$-vector space of all $K$-linear maps from $V$ to $W$.
With the canonical choice for the natural isomorphism $\gamma$, we obtain a symmetric monoidal closed structure on $\mathbf{Vect}_K$. Restricting this structure to finite-dimensional $K$-vector spaces also yields a symmetric monoidal closed category, denoted by $\mathbf{FdVect}_K$.
Throughout this section, whenever we refer to $\mathbf{Vect}_K$ or $\mathbf{FdVect}_K$, we always mean these categories equipped with the symmetric monoidal closed structure described above.
\end{exam}

Finally, to formalize $\bot$, we need an object such that the functor  
$
\neg : \mathcal{C}^{op} \to \mathcal{C}
$  
defined by $\neg A = [A, \bot]$ becomes an involution. To explain, note that for any object $B$, we have the functor  
$
\neg_B: \mathcal{C}^{op} \to \mathcal{C}
$  
defined by $\neg_B A = [A, B]$. Using the natural isomorphism  
\[
\gamma_{X, Y, Z}: \mathrm{Hom}(X \otimes Y, Z) \to \mathrm{Hom}(X, [Y, Z])
\]  
and the identity map $\mathrm{id}_{[Y, Z]}: [Y, Z] \to [Y, Z]$, there is a canonical map  
$
[Y, Z] \otimes Y \to Z.
$  
Therefore, we obtain a map  
$
\neg_B A \otimes A \to B.
$  
Using the symmetry of $\otimes$, we get a map  
$
A \otimes \neg_B A \to B.
$ 
Again, applying $\gamma$, we derive a canonical map  
$
A \to \neg_B \neg_B A$. An object $B$ is called \emph{dualizing} if the canonical map $A \to \neg_B\neg_B A$ is an isomorphism, for any object $A$. To reflect the classical behavior of negation, we need to assume that $\bot$ is dualizing:

\begin{dfn}
Let $(\otimes, I, [-, -])$ be a symmetric monoidal closed structure on $\mathcal{C}$. The tuple $(\otimes, I, [-, -], \bot)$ is called a \emph{$*$-autonomous} structure on $\mathcal{C}$ if $\bot$ is a dualizing object of $\mathcal{C}$, i.e., the canonical map $A \to \neg_{\bot}\neg_{\bot} A$ is an isomorphism for any object $A$ of $\mathcal{C}$. We denote $\neg_{\bot}$ by $\neg$. A $*$-autonomous structure is called \emph{compact closed} if $[A, B] \cong \neg A \otimes B$, naturally in $A$ and $B$.
We call the tuple $(\mathcal{C}, \otimes, I, [-, -], \bot)$ a \emph{$*$-autonomous category} (resp. \emph{compact closed category}) if $(\otimes, I, [-, -], \bot)$ is a \emph{$*$-autonomous structure} (resp. \emph{compact closed structure}) on $\mathcal{C}$.
\end{dfn}

\begin{exam}
Using Theorem \ref{Joyal}, a $\mathrm{BCC}$ with its $\mathrm{CC}$-structure as the symmetric monoidal closed structure and $\bot$ as its initial object is $*$-autonomous if and only if it is a boolean algebra. Note that in a non-trivial boolean algebra, there is no reason to have $a \to b = \neg a \wedge b$. Because, this equality would imply that $1 = (0 \to 0) = (\neg 0 \wedge 0) = 0$, which is a contradiction. Therefore, non-trivial boolean algebras provide an example of a $*$-autonomous category that is not compact closed.

The category $\mathbf{Rel}$ is compact closed if we define $\bot = \{*\}$. It is $*$-autonomous because the canonical map $A \to \neg \neg A$ is the identity. It is compact closed because $[A, B]$ is defined as $A^{\bot} \otimes B$, and $\neg A \cong A^{\bot}$ by definition.
A similar result holds for $\mathbf{FdVect}_K$, if we set $\bot = K$. Since $V$ is finite-dimensional and $\neg V = \mathrm{Hom}_K(V, K) = V^*$, the canonical map 
\[
V \to \mathrm{Hom}_K(\mathrm{Hom}_K(V, K), K) = V^{**}
\]
is an isomorphism. Hence, $\mathbf{FdVect}_K$ is $*$-autonomous. It is also compact closed because $\mathrm{Hom}_K(V, W) \cong V^* \otimes W$, naturally in $V$ and $W$.
\end{exam}

\begin{rem}
Here are some remarks. First, $*$-autonomous categories are used to provide a natural categorical semantics for a fragment of classical linear logic \cite{barII,bar}. Adding more structure, we will use them to interpret deductions in classical logic as will be explained in the rest of this subsection. However, it is worth mentioning that as a special case of $*$-autonomous categories, compact closed categories (with additional structures) have also been proposed to formalize classical deductions \cite{hyland2004abstract,hyland2002proof}.
Second, compact closed categories, when equipped with additional structures, can be used to formalize some fundamental and interesting aspects of quantum mechanics \cite{abramsky2004categorical,abramsky2009categorical,heunenvicary}. To have an intuition why, recall that the usual formalization of quantum mechanics is based on $\mathbb{C}$-vector spaces. Then, one can interpret compact closed categories as a categorical abstraction of the category $\mathbf{FdVect}_{\mathbb{C}}$.
\end{rem}

\begin{rem}\label{IisBot}
In any $*$-autonomous category, we have $\neg I \cong \bot$ and $\neg \bot \cong I$.
For the first isomorphism, observe that since $X \otimes I \cong X$ naturally in $X$, we obtain the following sequence of natural isomorphisms:
\[
\mathrm{Hom}(X, Y) \cong \mathrm{Hom}(X \otimes I, Y) \cong \mathrm{Hom}(X, [I, Y]).
\]
It follows that $[I, Y] \cong Y$. In particular, for $Y = \bot$, we get $\neg I = [I, \bot] \cong \bot$.
For the second isomorphism, it suffices to use the fact that $\neg \neg I \cong I$.
Additionally, it is worth noting that in a compact closed category, we have $\bot \cong I$. To prove this, first notice that by the above observation $[I, I] \cong I$. Then, recall that in a compact closed category, $[I, I] \cong \neg I \otimes I \cong \neg I$. Therefore, we obtain $\neg I \cong I$. Given that $\neg I \cong \bot$, we conclude that $I \cong \bot$.
\end{rem}

\subsubsection{Cloning and Deleting}

In the previous subsubsection, we argued that the conjunctive part of the language possesses at least the structure of a symmetric monoidal category. However, this structure alone does not fully capture all the expected properties of conjunction. The reason is that two key properties of conjunction have not yet been addressed: \emph{cloning} and \emph{deleting}.\footnote{The terms ``cloning" and ``deleting" have an information-theoretic flavor, referring to the ability to duplicate or discard bits of information. In proof theory, these operations are known as contraction and weakening, respectively.}
In proofs, an assumption can be used multiple times or discarded entirely. Formally, for any proposition $A$, we expect to have deductions $A \to A \wedge A$ and $A \to \top$ to \emph{clone} the assumption $A$ and \emph{delete} it, respectively.

In a general symmetric monoidal category, there is no reason for these cloning and deleting morphisms to exist.

\begin{exam}
Let $\mathcal{M} = (\mathbb{N}, +, 0)$ be the usual monoid of natural numbers, and let $a \neq 0$. Following Example \ref{Monoid}, we can regard $\mathcal{M}$ as a symmetric monoidal category. Since morphisms in this setting are only identities, it is clear that no map $a \to 0$ or $a \to 2a$ exists, as either would imply $a = 0$, leading to a contradiction.
\end{exam}

To formalize the cloning and deleting processes, we assume the existence of two morphisms, $A \to A \wedge A$ and $A \to \top$, for every object $A$. However, as with the structural isomorphisms in a symmetric monoidal category, these maps must satisfy certain coherence conditions. To begin, let us examine the relationship between these two maps for a fixed object $A$.

\begin{dfn}\label{Comonoid}
Let $(\mathcal{C}, \otimes, I)$ be a symmetric monoidal category. A \emph{commutative comonoid structure} on an object $A$ in $(\mathcal{C}, \otimes, I)$ is a pair $(\langle \rangle_A, \Delta_A)$ of morphisms, where $\langle \rangle_A: A \to I$ represents deletion and $\Delta_A: A \to A \otimes A$ represents cloning, satisfying the following commutative diagrams:
\begin{itemize}
    \item \emph{co-associativity:}
\[\begin{tikzcd}
	{(A \otimes A) \otimes A} && {A \otimes (A \otimes A)} \\
	\\
	{A \otimes A} && {A \otimes A} \\
	& A
	\arrow["{\alpha_{A,A,A}}", from=1-1, to=1-3]
	\arrow["{\Delta_A \otimes id_A}", from=3-1, to=1-1]
	\arrow["{id_A \otimes \Delta_A}"', from=3-3, to=1-3]
	\arrow["{\Delta_A}", from=4-2, to=3-1]
	\arrow["{\Delta_A}"', from=4-2, to=3-3]
\end{tikzcd}\]
\item 
\emph{co-unitality:}
\[\begin{tikzcd}
	{A \otimes I} && {A\otimes A} && {I \otimes A} \\
	\\
	&& A
	\arrow["{id_A \otimes \langle \rangle_A }"', from=1-3, to=1-1]
	\arrow["{\langle \rangle_A  \otimes id_A}", from=1-3, to=1-5]
	\arrow["{\rho^{-1}_A}", from=3-3, to=1-1]
	\arrow["{\Delta_A}"{description}, from=3-3, to=1-3]
	\arrow["{\lambda^{-1}_A}"', from=3-3, to=1-5]
\end{tikzcd}\]
\item 
\emph{co-symmetry:}
\[\begin{tikzcd}
	&& A \\
	\\
	{A \otimes A} &&&& {A \otimes A}
	\arrow["{\Delta_A}"', from=1-3, to=3-1]
	\arrow["{\Delta_A}", from=1-3, to=3-5]
	\arrow["{\sigma_{A, A}}"', from=3-1, to=3-5]
\end{tikzcd}\]
\end{itemize}
A \emph{commutative comonoid} in $(\mathcal{C}, \otimes, I)$ is a tuple $(A, \langle \rangle_A, \Delta_A)$, where $A$ is an object in $\mathcal{C}$ and $(\langle \rangle_A, \Delta_A)$ is a commutative comonoid structure on $A$ in $(\mathcal{C}, \otimes, I)$.
\end{dfn}

\begin{phil}
The co-associativity diagram formalizes the intuition that ``the two canonical ways to obtain three copies of $A$ are equal." The co-unitality diagram expresses that ``cloning $A$ and then deleting one copy has no effect." Finally, the co-symmetry diagram ensures that ``the order of cloned copies is immaterial."
\end{phil}

\begin{exam}\label{ExamplesOfComonoid}
In a cartesian category with its symmetric monoidal structure as its cartesian structure, there is a canonical commutative comonoid structure on any object $A$. Specifically, one can define $\langle \rangle_A$ as the unique morphism $!_A: A \to 1$ and set $\Delta_A$ as the diagonal morphism $\langle id_A, id_A \rangle: A \to A \times A$. Co-associativity, co-unitality, and co-symmetry can be easily verified using the properties of the binary product and terminal object.

Similarly, in the symmetric monoidal category $\mathbf{Rel}$, any set $A$ can be equipped with a commutative comonoid structure using the relations $\langle \rangle_A: A \to \{*\}$ and $\Delta_A: A \to A \times A$, defined by $\langle \rangle_A = A \times \{*\}$ and $\Delta_A = \{(a, (a, a)) \mid a \in A\}$. The commutativity of the diagrams is straightforward to verify.
\end{exam}

\begin{rem}
The dual of a commutative comonoid is a \emph{commutative monoid}. It consists of an object $A$ together with maps $[]_A: I \to A$ and $\nabla_A: A \otimes A \to A$, which satisfy the duals of the equalities in Definition \ref{Comonoid}, namely \emph{associativity}, \emph{unitality}, and \emph{symmetry}. Commutative monoids in a symmetric monoidal category generalize the usual commutative monoids in algebra. One can interpret $\nabla_A: A \otimes A \to A$ as the product of the monoid and $[]_A: I \to A$ as its neutral element, while associativity, unitality, and symmetry formalize the expected properties of a commutative monoid. In fact, a commutative monoid in $\mathbf{Set}$, with its cartesian structure, is simply a commutative monoid in the traditional sense. Commutative monoids are abundant in mathematics. For another example, a commutative monoid in $\mathbf{Vect}_K$ corresponds to a commutative $K$-algebra. 
As in classical reasoning, the dual of the conjunctive structure is the disjunctive structure. A commutative monoid structure can also serve to witness the deductions $\bot \to A$ and $A \vee A \to A$.
\end{rem}

To handle the relationship between cloning and deleting different assumptions, we can introduce additional coherency equalities. These equalities should reflect how different assumptions can be cloned or deleted in relation to each other.

\begin{dfn}\label{ComoinOnCat}
Let $(\mathcal{C}, \otimes, I)$ be a symmetric monoidal category. A \emph{commutative comonoid structure on $(\mathcal{C}, \otimes, I)$} is an assignment of commutative comonoids $(A, \langle \rangle_A, \Delta_A)$ to any object $A$ such that
$\langle \rangle_I=id_I: I \to I$, $\Delta_I=\lambda_I^{-1}: I \to I \otimes I$ and
\[\begin{tikzcd}
	{(A \otimes B) \otimes (A \otimes B)} && {(A \otimes A) \otimes (B \otimes B)} \\
	\\
	& {A \otimes B}
	\arrow["\cong", from=1-1, to=1-3]
	\arrow["{\Delta_{A \otimes B}}", from=3-2, to=1-1]
	\arrow["{\Delta_A \otimes \Delta_B}"', from=3-2, to=1-3]
\end{tikzcd}\]
\[\begin{tikzcd}
	{I \otimes I} && I \\
	\\
	& {A \otimes B}
	\arrow["{\lambda_I}", from=1-1, to=1-3]
	\arrow["{\langle \rangle_A \otimes \langle \rangle_B}", from=3-2, to=1-1]
	\arrow["{\langle \rangle_{A \otimes B}}"', from=3-2, to=1-3]
\end{tikzcd}\]
where the map $\cong$ is the canonical map constructed by the associators and braiding. A comonoid structure is called \emph{uniform} if both $\langle \rangle_A: A \to I$ and $\Delta_A: A \to A \otimes A$ are natural in $A$.
\end{dfn}

\begin{phil}
The equalities in Definition \ref{ComoinOnCat} formalize the intuition that ``cloning $\{A, B\}$ is equivalent to cloning $A$ and cloning $B$",   ``deleting $\{A, B\}$ is the same as deleting $A$ and $B$" and  
``deleting and cloning the empty set has no effect". The uniformity condition simply states that the processes of cloning and deleting are applied uniformly, rather than being arbitrarily chosen for different assumptions.
\end{phil}

\begin{exam}
Both assignments of commutative comonoids presented in Example \ref{ExamplesOfComonoid} are commutative comonoid structures on their coressponding categories.
\end{exam}

This seemingly completes our proposal for the conjunctive structure. We start with a symmetric monoidal category $\mathcal{C}$ and then introduce a uniform commutative comonoid structure on $\mathcal{C}$ to simulate the missing processes of cloning and deleting.
Using this proposal, to formalize classical deduction, one might expect to employ a $*$-autonomous category $\mathcal{C}$ equipped with a uniform commutative comonoid structure. However, we are once again confronted with a collapse theorem, Corollary \ref{MonoidalCollapse}, which states that such a $*$-autonomous category must necessarily be a boolean algebra. To prove this theorem, we first show that the existence of a uniform comonoid structure forces the symmetric monoidal structure to collapse to the usual cartesian structure. This result is well-known among category theorists, although its proof is, to our knowledge, rarely found in print \cite{heunenvicary}. Consequently, we provide the proof here.

\begin{thm}\label{CollapseOfMonToCart}
Let $(\mathcal{C}, \otimes, I)$ be a symmetric monoidal category. If there is a uniform commutative comonoid structure on $(\mathcal{C}, \otimes, I)$, then the monoidal structure $(\otimes, I)$ must necessarily be the cartesian structure $(\times, 1)$.
\end{thm}
\begin{proof}
Let $(\langle \rangle_A, \Delta_A)$ be the comonoid structure on $A$. Since $\langle \rangle_A: A \to I$ and $\Delta_A: A \to A \otimes A$ are natural in $A$, for any map $f: A \to B$, the following diagram commutes:
\[\begin{tikzcd}
	{A \otimes A } && A && I \\
	\\
	{B \otimes B} && B && I
	\arrow["{f \otimes f}"', from=1-1, to=3-1]
	\arrow["{\Delta_A}"', from=1-3, to=1-1]
	\arrow["{\langle \rangle_A}", from=1-3, to=1-5]
	\arrow["f"', from=1-3, to=3-3]
	\arrow["{id_I}", from=1-5, to=3-5]
	\arrow["{\Delta_B}", from=3-3, to=3-1]
	\arrow["{\langle \rangle_B}"', from=3-3, to=3-5]
\end{tikzcd}\]
In words, we have $\langle \rangle_B \circ f = \langle \rangle_A$ and $\Delta_B \circ f = (f \otimes f) \circ \Delta_A$ for any map $f: A \to B$.

First, we prove that $I$ is a terminal object. Since there is always the map $\langle \rangle_A: A \to I$, it suffices to show that this map is unique. Let $g: A \to I$. Since $\langle \rangle_I = \text{id}_I$, we have $g = \text{id}_I \circ g = \langle \rangle_I \circ g = \langle \rangle_A$. Therefore, $g = \langle \rangle_A$, which implies that $I$ is a terminal object.
Second, we prove that $A \otimes B$ is the binary product of $A$ and $B$. Consider the maps $p_0: A \otimes B \to A$ and $p_1: A \otimes B \to B$ defined as follows:
\[\begin{tikzcd}
	{A \otimes B} && {A \otimes I} && A \\
	{A \otimes B} && {I \otimes B} && B
	\arrow["{id_A \otimes \langle \rangle_B}", from=1-1, to=1-3]
	\arrow["{\rho_A}", from=1-3, to=1-5]
	\arrow["{\langle \rangle_A \otimes id_B}"', from=2-1, to=2-3]
	\arrow["{\lambda_B}"', from=2-3, to=2-5]
\end{tikzcd}\]
These maps are intended to serve as projections. Moreover, for any two morphisms $f: C \to A$ and $g: C \to B$, define $\langle f, g \rangle: C \to A \otimes B$ by:
\[\begin{tikzcd}
	C && {C \otimes C} && {A \otimes B}
	\arrow["{\Delta_C}", from=1-1, to=1-3]
	\arrow["{f \otimes g}", from=1-3, to=1-5]
\end{tikzcd}\]
We prove that $p_0 \circ \langle f, g \rangle = f$ and $p_1 \circ \langle f, g \rangle = g$ for any $f: C \to A$ and $g: C \to B$. We will only prove the first equality; the second is similar. Consider the following diagram:
\[\begin{tikzcd}
	C && {C \otimes C} && {C \otimes I} && C \\
	\\
	&& {A \otimes B} && {A \otimes I} && A
	\arrow["{\Delta_C}", from=1-1, to=1-3]
	\arrow["{\langle f, g \rangle}"', from=1-1, to=3-3]
	\arrow["{id_C \otimes \langle \rangle_C}", from=1-3, to=1-5]
	\arrow["{f \otimes g}"', from=1-3, to=3-3]
	\arrow["{\rho_C}", from=1-5, to=1-7]
	\arrow["{f \otimes id_I}", from=1-5, to=3-5]
	\arrow["f", from=1-7, to=3-7]
	\arrow["{id_A \otimes \langle \rangle_B}"', from=3-3, to=3-5]
	\arrow["{\rho_A}"', from=3-5, to=3-7]
\end{tikzcd}\]
The right square commutes by the naturality of $\rho$. The left square commutes by the equality $\langle \rangle_B \circ g = \langle \rangle_C$. The commutativity of the triangle follows directly from the definition of $\langle f, g \rangle$. Now, by co-unitality, the upper row equals $id_C$ and the lower row equals $p_0$, by definition. Hence, we obtain $p_0 \circ \langle f, g \rangle = f$.

Now, to complete the proof of the second part, it is enough to prove $\langle p_0h, p_1h \rangle=h$, for any $h: C \to A \otimes B$. For that purpose, consider the following diagram:
\[\begin{tikzcd}
	C && {C \otimes C} && {A \otimes B} \\
	\\
	{A \otimes B} && {(A \otimes B) \otimes (A \otimes B)} && {A \otimes B} \\
	\\
	{A \otimes B} && {(A \otimes A) \otimes (B \otimes B)} && {A \otimes B}
	\arrow["{\Delta_C}", from=1-1, to=1-3]
	\arrow["h"', from=1-1, to=3-1]
	\arrow["{p_0h \otimes p_1h}", from=1-3, to=1-5]
	\arrow["{h \otimes h}", from=1-3, to=3-3]
	\arrow["{id_{A \otimes B}}", from=1-5, to=3-5]
	\arrow["{\Delta_{A \otimes B}}", from=3-1, to=3-3]
	\arrow["{id_{A \otimes B}}"', from=3-1, to=5-1]
	\arrow["{p_0 \otimes p_1}", from=3-3, to=3-5]
	\arrow["\cong", from=3-3, to=5-3]
	\arrow["{id_{A \otimes B}}", from=3-5, to=5-5]
	\arrow["{\Delta_A \otimes \Delta_B}"', from=5-1, to=5-3]
	\arrow["{p_0 \otimes p_1}"', from=5-3, to=5-5]
\end{tikzcd}\]
The upper left square commutes by the naturality of $\Delta$. The commutativity of the upper right square is clear. The lower left square commutes by the definition of a commutative comonoid structure on $(\mathcal{C}, \otimes, I)$. The lower right square commutes by straightforward applications of the identities for the associator and braiding.
Now, notice that the upper row is $\langle p_0 h, p_1 h \rangle$, by definition. The lower row is the identity, by co-unitality. Therefore, we have $\langle p_0 h, p_1 h \rangle = h$. This completes the proof that $A \otimes B$ is the binary product of $A$ and $B$.

To show that the tensor product is the binary product as a functor, it suffices to prove that $f \otimes g = f \times g$ for any $f: A \to C$ and $g: B \to D$. By definition, $(f \times g): A \otimes B \to C \otimes D$ is the unique map satisfying $p_0 \circ (f \times g) = f \circ p_0$ and $p_1 \circ (f \times g) = f \circ p_1$. Therefore, it is enough to show that $f \otimes g: A \otimes B \to C \otimes D$ also satisfies these equalities. We will prove the case for $p_0$; the other case is similar. Consider the following diagram:
\[\begin{tikzcd}
	{A \otimes B} && {A \otimes I} && A \\
	\\
	{C \otimes D} && {C \otimes I} && C
	\arrow["{id_A \otimes \langle \rangle_B}", from=1-1, to=1-3]
	\arrow["{f \otimes g}"', from=1-1, to=3-1]
	\arrow["{\rho_A}", from=1-3, to=1-5]
	\arrow["{f \otimes id_I}", from=1-3, to=3-3]
	\arrow["f", from=1-5, to=3-5]
	\arrow["{id_C \otimes \langle \rangle_D}"', from=3-1, to=3-3]
	\arrow["{\rho_C}"', from=3-3, to=3-5]
\end{tikzcd}\]
The right square commutes by the naturality of $\rho$. The left square commutes by the equality $\langle \rangle_B = \langle \rangle_D \circ g$. Since the rows are both $p_0$, we obtain $p_0 \circ (f \otimes g) = f \circ p_0$, which completes the proof.
\end{proof}

\begin{cor}\label{MonoidalCollapse}
Let $(\mathcal{C}, \otimes, I, [-, -], \bot)$ be a $*$-autonomous category. If there is a uniform commutative comonoid structure on $(\mathcal{C}, \otimes, I)$, then $\mathcal{C}$ must necessarily be a boolean algebra.
\end{cor}
\begin{proof}
By Theorem \ref{CollapseOfMonToCart}, the monoidal structure must be the cartesian structure. Consequently, $I$ is the terminal object, and $\otimes$ corresponds to the binary product.
First, note that the collapse of $\otimes$ to the product forces the functor $[-, -]: \mathcal{C}^{op} \times \mathcal{C} \to \mathcal{C}$ to be the exponentiation. Hence, $\mathcal{C}$ is a $\mathrm{CCC}$.
Second, using the isomorphism $A \cong \neg \neg A$, we can easily construct a natural isomorphism
\[
\mathrm{Hom}(\neg A, \neg B) \cong \mathrm{Hom}(B, A).
\]
Applying this isomorphism, along with the fact that $\bot \cong \neg I$ from Remark \ref{IisBot}, we have
\[
\mathrm{Hom}(\bot, A) \cong \mathrm{Hom}(\neg I, \neg \neg A) \cong \mathrm{Hom}(\neg A, I).
\]
Since $I$ is terminal, the last set is a singleton. Consequently, $\mathrm{Hom}(\bot, A)$ is also a singleton, which implies that $\bot$ is an initial object. Similarly, one can show that the functor $\oplus: \mathcal{C} \times \mathcal{C} \to \mathcal{C}$, defined by $A \oplus B=\neg (\neg A \otimes \neg B)$ is the binary coproduct. This establishes that $\mathcal{C}$ is a $\mathrm{BCC}$ satisfying $A \cong \neg \neg A$ for all $A$.
Finally, by Theorem \ref{Joyal}, it follows that $\mathcal{C}$ must be a boolean algebra.
\end{proof}

\begin{phil}\emph{(Lafont Example)} 
For readers familiar with sequent calculi, there is a stronger argument showing that not only classical reasoning but also multi-conclusion reasoning leads to a collapse. This argument is usually called \emph{Lafont example} \cite{fuhrmann2004geometry,fuhrmann2006order,fuhrmann2007categorical} as it is a variant of an example credited to Lafont \cite{ProofandTypes}.
To explain, consider a multi-conclusion sequent calculus with the weakening rule (denoted by $W$) and the contraction rule (denoted by $C$). Then, the following two proofs must be considered equivalent:
\begin{center}
	\begin{tabular}{c c c}
	    \AxiomC{$\mathcal{D}_1$}
        \noLine
        \UnaryInfC{$\Gamma_1 \Rightarrow \Delta_1, C$}
         \AxiomC{$\mathcal{D}_2$}
         \noLine
         \UnaryInfC{$\Gamma_2 \Rightarrow \Delta_2$}
        \RightLabel{$W$}
         \UnaryInfC{$\Gamma_2, C \Rightarrow \Delta_2$}
         \RightLabel{$cut$}
		\BinaryInfC{$\Gamma_1, \Gamma_2 \Rightarrow \Delta_1, \Delta_2$}
		\DisplayProof \quad \quad \quad
		&
     \AxiomC{$\mathcal{D}_2$}
        \noLine
        \UnaryInfC{$\Gamma_2 \Rightarrow \Delta_2$}
        \RightLabel{$W$}
		\UnaryInfC{$\Gamma_1, \Gamma_2 \Rightarrow \Delta_1, \Delta_2$}
		\DisplayProof  
	\end{tabular}
\end{center}
The reason is that the formula $C$ is added to the assumptions of the right-hand-side deduction in a dummy way. Therefore, using $\mathcal{D}_1$ to provide a proof for it must have no effect on the result, apart from adding other dummy formulas, i.e., $\Gamma_1$ and $\Delta_1$ instead of $C$. Note that this is a proof-theoretic version of the naturality of the deletion map. Dually, we also have the following equality:
\begin{center}
\begin{tabular}{cc}
	    \AxiomC{$\mathcal{D}_1$}
        \noLine
        \UnaryInfC{$\Gamma_1 \Rightarrow \Delta_1$}
        \RightLabel{$W$}
        \UnaryInfC{$\Gamma_1 \Rightarrow \Delta_1, C$}
         \AxiomC{$\mathcal{D}_2$}
         \noLine
         \UnaryInfC{$\Gamma_2, C \Rightarrow \Delta_2$}
         \RightLabel{$cut$}
		\BinaryInfC{$\Gamma_1, \Gamma_2 \Rightarrow \Delta_1, \Delta_2$}
		\DisplayProof \quad \quad \quad
       &
      \AxiomC{$\mathcal{D}_1$}
        \noLine
        \UnaryInfC{$\Gamma_1 \Rightarrow \Delta_1$}
        \RightLabel{$W$}
		\UnaryInfC{$\Gamma_1, \Gamma_2 \Rightarrow \Delta_1, \Delta_2$}
		\DisplayProof 
	\end{tabular}
    \end{center}
Now, using the above equalities, let us simplify the following proof tree:
\begin{center}
	\begin{tabular}{c c c}
	    \AxiomC{$\mathcal{D}_1$}
        \noLine
        \UnaryInfC{$\Gamma \Rightarrow \Delta$}
         \RightLabel{$W$}
        \UnaryInfC{$\Gamma \Rightarrow \Delta, C$}
         \AxiomC{$\mathcal{D}_2$}
         \noLine
         \UnaryInfC{$\Gamma \Rightarrow \Delta$}
          \RightLabel{$W$}
         \UnaryInfC{$\Gamma, C \Rightarrow \Delta$}
		\BinaryInfC{$\Gamma, \Gamma \Rightarrow \Delta, \Delta$}
         \RightLabel{$C$}
        \UnaryInfC{$\Gamma \Rightarrow \Delta$}
		\DisplayProof 
	\end{tabular}
\end{center}
Using the above equalities, this proof tree is equal to both of the following proof trees:
\begin{center}
	\begin{tabular}{c c c}
	     \AxiomC{$\mathcal{D}_1$}
        \noLine
        \UnaryInfC{$\Gamma \Rightarrow \Delta$}
        \RightLabel{$W$}
		\UnaryInfC{$\Gamma, \Gamma \Rightarrow \Delta, \Delta$}
         \RightLabel{$C$}
        \UnaryInfC{$\Gamma \Rightarrow \Delta$}
		\DisplayProof \quad \quad
        &
      \AxiomC{$\mathcal{D}_2$}
        \noLine
        \UnaryInfC{$\Gamma \Rightarrow \Delta$}
        \RightLabel{$W$}
		\UnaryInfC{$\Gamma, \Gamma \Rightarrow \Delta, \Delta$}
         \RightLabel{$C$}
        \UnaryInfC{$\Gamma \Rightarrow \Delta$}
		\DisplayProof 
	\end{tabular}
\end{center}
Hence, the latter two proof-trees are equal. Now, observe that applying contraction to weakening must lead to identity. This is the proof-theoretic presentation of the co-unitality condition of a commutative comonoid. Therefore, using the equality established above, we reach the equality 
$\mathcal{D}_1 = \mathcal{D}_2$, which proves the collapse of deductions of any sequent $\Gamma \Rightarrow \Delta$.
Note that the argument does not use any negation or classicality. However, it critically depends on multi-conclusion type reasoning.
\end{phil}

As mentioned earlier, compact closed categories (equipped with additional structures) can be used to formalize certain aspects of quantum mechanics. Since compact closed categories are a special case of $*$-autonomous categories, one can prove a stronger collapse theorem than Corollary \ref{MonoidalCollapse}, commonly known as the \emph{no-deleting theorem,} which states that there is no uniform way to forget information in the quantum world.

\begin{thm}(No-deleting)
Let $(\mathcal{C}, \otimes, I, [-,-], \bot)$ be a compact closed category. If there is a natural transformation $\langle \rangle_A: A \to I$ such that $\langle \rangle_I=id_I$, then $\mathrm{Hom}_{\mathcal{C}}(A, B)$ is a singleton, for any objects $A$ and $B$. 
\end{thm}
\begin{proof}
First, we can use a similar argument to the one in the first part of the proof of Theorem \ref{CollapseOfMonToCart} to show that $I$ is terminal. Then, as $\bot \cong I$ by Remark \ref{IisBot}, and using the isomorphism $B \cong \neg \neg B$, we obtain the following series of isomorphisms:
\[
\mathrm{Hom}_{\mathcal{C}}(A, B) \cong \mathrm{Hom}_{\mathcal{C}}(A, \neg \neg B) \cong \mathrm{Hom}_{\mathcal{C}}(A \otimes \neg B, \bot) \cong \mathrm{Hom}_{\mathcal{C}}(A \otimes \neg B, I).
\]
Since $I$ is terminal, the last set is a singleton. Therefore, $\mathrm{Hom}_{\mathcal{C}}(A, B)$ is also a singleton.
\end{proof}

\subsubsection{Classical Categories}
By Corollary \ref{MonoidalCollapse}, it appears that imposing a uniform commutative comonoid structure on a $*$-autonomous category is too strong a requirement. Since all the other structures are naturally desirable to retain, we must weaken the uniformity condition i.e., the following commutativity:
\[\begin{tikzcd}
	{A \otimes A } && A && I \\
	\\
	{B \otimes B} && B && I
	\arrow["{f \otimes f}"', from=1-1, to=3-1]
	\arrow["{\Delta_A}"', from=1-3, to=1-1]
	\arrow["{\langle \rangle_A}", from=1-3, to=1-5]
	\arrow["f"', from=1-3, to=3-3]
	\arrow["{id_I}", from=1-5, to=3-5]
	\arrow["{\Delta_B}", from=3-3, to=3-1]
	\arrow["{\langle \rangle_B}"', from=3-3, to=3-5]
\end{tikzcd}\]
For this purpose, we need to introduce a new structure into our categorical framework: \emph{simplifications}. Intuitively, for any fixed $A$ and $B$, there is a partial order between deductions from $A$ to $B$ that encodes the possibility of simplifying one deduction into another. Formally, we write $f_1 \leq f_2$ if $f_1$ can be simplified to $f_2$, where $f_1, f_2: A \to B$.
For example, let $f: A \to B$ be a deduction and let $\langle \rangle_B: B \to I$ be the deletion map for $B$. Then, the deduction $\langle \rangle_B \circ f: A \to I$ can be simplified to the deduction $\langle \rangle_A: A \to I$. The reasoning behind this simplification is that instead of first using $A$ to prove $B$ and then forgetting $B$, we can simply forget $A$ from the outset.  

The crucial observation of \cite{fuhrmann2004geometry,fuhrmann2006order,fuhrmann2007categorical} is that this process of simplification is not necessarily symmetric and must therefore be formalized using a \emph{partial order} on deductions rather than \emph{equality}. Accordingly, we should have $\langle \rangle_B \circ f \leq \langle \rangle_A$, but there is no need to assume the equality $\langle \rangle_B \circ f = \langle \rangle_A$, which corresponds to the naturality condition that leads to a collapse.

Putting this in proof-theoretic terms, using this ordering on Lafont example, the proof tree:

\begin{center}
	\begin{tabular}{c c c}
	    \AxiomC{$\mathcal{D}_1$}
        \noLine
        \UnaryInfC{$\Gamma \Rightarrow \Delta$}
         \RightLabel{$W$}
        \UnaryInfC{$\Gamma \Rightarrow \Delta, C$}
         \AxiomC{$\mathcal{D}_2$}
         \noLine
         \UnaryInfC{$\Gamma \Rightarrow \Delta$}
          \RightLabel{$W$}
         \UnaryInfC{$\Gamma, C \Rightarrow \Delta$}
		\BinaryInfC{$\Gamma, \Gamma \Rightarrow \Delta, \Delta$}
         \RightLabel{$C$}
        \UnaryInfC{$\Gamma \Rightarrow \Delta$}
		\DisplayProof 
	\end{tabular}
\end{center}
is less than or equal to both $\mathcal{D}_1$ and $\mathcal{D}_2$, as it can be simplified to both of these proof trees. However, this does not mean that $\mathcal{D}_1$ and $\mathcal{D}_2$ are equal.

To formalize this intuitive notion of simplification, one should start with a $*$-autonomous category and introduce a partial order on each $\mathrm{Hom}_{\mathcal{C}}(A, B)$ in a coherent manner.

\begin{dfn}
A $*$-autonomous category $(\mathcal{C}, \otimes, I, [-,-], \gamma, \bot)$ equipped with a partial order on each $\mathrm{Hom}_{\mathcal{C}}(A, B)$ is called \emph{order-enriched} if composition and the functors $\otimes$ and $[-,-]$ are monotone with respect to the partial order, and if the function $\gamma_{A, B, C}: \mathrm{Hom}_{\mathcal{C}}(C \otimes A, B) \cong \mathrm{Hom}_{\mathcal{C}}(C, [A, B])$ is an order-isomorphism for all $A$, $B$, and $C$.
\end{dfn}

\begin{exam}
Any $*$-autonomous category, when equipped with equality as the partial order, is order-enriched. As another example, consider the $*$-autonomous category $\mathbf{Rel}$. For any two relations $R, S \subseteq A \times B$ in $\mathrm{Hom}_{\mathbf{Rel}}(A, B)$, define $R \leq S$ if and only if $R \subseteq S$. Since both the cartesian product and the operation $(-)^{\bot}: \mathbf{Rel}^{op} \to \mathbf{Rel}$ are monotone, it is straightforward to verify that composition and the functors $\otimes$ and $[-, -]$ are monotone, and that the bijection $\gamma_{A, B, C}$ is an order-isomorphism.
\end{exam}

Now, we can weaken the naturality of cloning and deleting from equality to inequality, thereby formalizing the idea that one deduction can be simplified into another.

\begin{dfn}\label{ClassicalCat}
An order-enriched $*$-autonomous category, equipped with a commutative comonoid structure, is called \emph{classical} if it satisfies the following relaxed version of naturality:
\[\begin{tikzcd}
	A && I && A && {A \otimes A} \\
	& \leq &&&& \leq \\
	B && I && B && {B \otimes B}
	\arrow["{\langle \rangle_A}", from=1-1, to=1-3]
	\arrow["f"', from=1-1, to=3-1]
	\arrow["{id_I}", from=1-3, to=3-3]
	\arrow["{\Delta_A}", from=1-5, to=1-7]
	\arrow["f"', from=1-5, to=3-5]
	\arrow["{f \otimes f}", from=1-7, to=3-7]
	\arrow["{\langle \rangle_B}"', from=3-1, to=3-3]
	\arrow["{\Delta_B}"', from=3-5, to=3-7]
\end{tikzcd}\]
In words, for any $f: A \to B$, we have $\langle \rangle_B \circ f \leq \langle \rangle_A$ and $\Delta_B \circ f \leq (f \otimes f) \circ \Delta_A$.
\end{dfn}

\begin{exam}\label{RelAsClassical}
The trivial example of a classical category is a boolean algebra, where the partial order is simply equality. For a non-trivial example, consider the order-enriched $*$-autonomous category $\mathbf{Rel}$ with the comonoid structure presented in Example \ref{ExamplesOfComonoid}, where the partial order is given by inclusion. To prove that $\mathbf{Rel}$ is a classical category, observe the following: for any $R \subseteq A \times B$, we have
\[
\langle \rangle_B \circ R = \{(a, *) \mid \exists b \in B, (a, b) \in R\} \subseteq A \times \{*\} = \langle \rangle_A.
\]
This represents an inequality, which becomes an equality if and only if $R$ is a total relation.  
For the other condition, we note that
\[
\Delta_B \circ R = \{(a, (b, b)) \mid (a, b) \in R\},
\]
while
\[
(R \otimes R) \circ \Delta_A = \{(a, (b, c)) \mid (a, b) \in R \; \text{and} \; (a, c) \in R\}.
\]
Therefore, we have $\Delta_B \circ R \subseteq (R \otimes R) \circ \Delta_A$. This inequality becomes an equality if and only if $R$ is functional.  
\end{exam}

\begin{rem}
The category $\mathbf{Rel}$ in Example \ref{RelAsClassical} illustrates that the notion of a classical category does not collapse to a boolean algebra. To expand on this observation, it is easy to see that the canonical structure induced from classical categories $\mathcal{C}$ and $\mathcal{D}$ to their product $\mathcal{C} \times \mathcal{D}$ forms a classical category itself. By using this closure under the product, one can take a non-trivial boolean algebra $\mathcal{B}$ and obtain the classical category $\mathbf{Rel} \times \mathcal{B}$. This classical category is neither compact closed nor a preordered set. Hence, classical categories that are not compact closed are not restricted to preordered sets. 
\end{rem}

\begin{rem}(\emph{Dummett categories})
The original definition of classical categories \cite{fuhrmann2004geometry, fuhrmann2006order, fuhrmann2007categorical} differs, yet is equivalent, to the one we present here \cite{Lutz}. It begins with a \emph{symmetric linearly distributive category}, which is a category equipped with two symmetric monoidal structures: the conjunctive $(\otimes, I)$ and the disjunctive $(\oplus, \bot)$. It also includes a natural transformation $\delta: A \otimes (B \oplus C) \to (A \otimes B) \oplus C$, subject to certain coherency identities, ensuring \emph{distributivity}. Subsequently, an order-enrichment is added coherently, along with a commutative comonoid structure for $(\otimes, I)$ and, dually, a commutative monoid structure for $(\oplus, \bot)$, subject to additional coherency equalities and inequalities.These categories are called \emph{Dummett categories}. They capture \emph{multi-conclusion} deductions in a propositional language without negation and implication, thereby addressing the issue raised by Lafont example and providing a non-trivial categorical interpretation of such deductions. Classical categories are then defined by adding an involutive negation functor that appropriately maps the conjunctive and disjunctive structures to each other.
In this chapter, we have chosen to work with $*$-autonomous categories, as they are simpler to introduce, requiring fewer structures and thus involving fewer coherence issues. Moreover, they are more accessible, aligning better with both the spirit and structure of this chapter. However, it is worth noting that the route via Dummett categories is more general and provides a richer perspective than the one we have taken here.
\end{rem}

Although the idea of simplification and its non-symmetric behavior is natural, the inequalities we used to weaken naturality might seem a bit arbitrary to the reader. However, this is not the case.
To explain, first note that, similar to natural deduction proofs for intuitionistic logic, there is a syntactical calculus for classical proofs called \emph{sequent calculus $\mathbf{LK}$}. Moreover, similar to the $\beta \eta$-reductions in natural deduction, there is a natural process of simplification for $\mathbf{LK}$-proofs called \emph{cut reduction} \cite{basicproof}. The inequalities in the definition of a classical category are two instances of cut reduction, and they are actually enough to axiomatize all such cut reductions. In a more formal form:
\begin{thm}[Soundness-completeness]
Let \( T \) be a set of inequalities between \(\mathbf{LK}\)-proofs of the same sequent, and let \(\mathcal{D}_1\) and \(\mathcal{D}_2\) be two such \(\mathbf{LK}\)-proofs. Then, \(\mathcal{D}_1 \preceq \mathcal{D}_2\) is derivable from \( T \) by cut-reductions if and only if \( F(\mathcal{D}_1) \leq F(\mathcal{D}_2) \) for any classical category \(\mathcal{C}\) and any interpretation \( F \) of \(\mathbf{LK}\)-proofs into \(\mathcal{C}\) that satisfies the inequalities in \( T \).
\end{thm}
As we have not formally defined $\mathbf{LK}$-proofs and the cut-reduction process, notions such as ``interpretation" and ``derivable by cut reductions" have been used informally in the previous theorem. However, the theorem is stated and proved in its precise form in \cite{fuhrmann2004geometry, fuhrmann2006order, fuhrmann2007categorical}.
This result demonstrates that classical categories, along with their order-enrichment, faithfully capture the categorical structure underlying sequent-style classical deductions and their cut reduction, providing further evidence for the naturality of these categories.

The connection between \(\mathbf{LK}\)-proofs and classical categories can also be expressed in a more categorical form. Recall that for \(\mathrm{BCC}\)s, we saw that natural deduction proofs, up to \(\beta \eta\)-equivalence, form a \(\mathrm{BCC}\) that is also the free \(\mathrm{BCC}\) generated by the set \(\{p_0, p_1, \ldots\}\).  
In a similar vein, \(\mathbf{LK}\)-proofs\footnote{technically, proof nets for classical logic introduced in \cite{robinson} as the less bureaucratic version of sequent-style proofs} and their simplifications via cut reduction form a classical category. This category is the \emph{free classical category} generated by the set of atoms \(\{p_0, p_1, \ldots\}\) as proved in \cite{fuhrmann2006order}.  
Thus, classical categories capture the \emph{exact structure and coherency identities} underlying classical deductions. Similar to the intuitionistic case, this freeness provides a syntax-free presentation of classical deductions.

The categorical structures of both Dummett and classical categories are studied in detail in \cite{fuhrmann2004geometry, fuhrmann2006order, fuhrmann2007categorical}. In addition to mathematical and syntactical examples of these categories, there is also a \emph{combinatorial} construction based on the \emph{classical propositional matrix method} presented in \cite{pym2014proof}. It is also worth noting that by strengthening classical categories, one can develop a categorical treatment of first-order classical deductions, as explored in \cite{mckinley2006thesis}.

\section{Categories as Semantics}\label{SecCatSemantics}

In the previous sections, we introduced $\mathrm{BCC}$'s (resp. $\mathrm{ABC}$'s and $\mathrm{CCC}$'s) as the proof systems for the language $\mathcal{L}_p$ (resp. $\mathcal{L}_p-\{\bot\}$ and $\mathcal{L}_p-\{\bot, \vee\}$).\footnote{In the following discussion, we will only focus on the full language and the $\mathrm{BCC}$'s. However, everything we say also holds for the two fragments.} One can visualize the class of all $\mathrm{BCC}$'s and the $\mathrm{BC}$-functors between them as the \emph{models} for the abstract proofs. The free $\mathrm{BCC}$, i.e., the category $\mathbf{NJ}$, is living high above and captures the pure notion of proof, where no unnecessary structure is present.
Then, below $\mathbf{NJ}$, we have other $\mathrm{BCC}$'s, the shadows of the free $\mathrm{BCC}$. As much as we go down in the picture, the $\mathrm{BCC}$'s become more concrete and hence easier to work with. Finally, on the ground level, we have the Heyting algebras, where all the deductions are collapsed to one map encoding the mere \emph{deductibility}. One can use this picture as some sort of semantics for the deductions, where the free $\mathrm{BCC}$ plays the role of the \emph{syntax}, all the others are the models of that syntax and the $\mathrm{BC}$-functors from $\mathbf{NJ}$ into a $\mathrm{BCC}$ $\mathcal{C}$ are the \emph{interpretations} of the syntax inside the model $\mathcal{C}$.

To see how such a semantical setting works, first, let us recall the three problems we explained in Section \ref{SecHistory}:
\begin{itemize}
\item
\emph{The existence problem}: given two propositions $A$ and $B$, if there is a map from $A$ to $B$ in $\mathbf{NJ}$.
\item
\emph{The equivalence problem}: given two maps $f, g: A \to B$ in $\mathbf{NJ}$, if they are equal in $\mathbf{NJ}$. Note that the problem is not just the pure equality of two syntactical entities as the equality of the derivations in $\mathbf{NJ}$ is considered up to the $\beta\eta$-equivalence. 
\item 
\emph{The identity problem}: given two propositions $A$ and $B$, if there is an isomorphism between them in $\mathbf{NJ}$.
\end{itemize}
For any of these problems, if the answer is positive—i.e., if there exists a deduction or the deductions are equal or the propositions are identical—we must provide the evidence in the free $\mathrm{BCC}$. However, for the negative claims, when there is no deduction, the deductions are unequal, or the propositions are non-identical, as any positive claim is inherited from the free category to all its shadows, it is enough to prove the claim for a shadow category, where hopefully its concrete nature makes the computation easier. Here, we use some examples to show how $\mathrm{BCC}$'s can be used to prove such negative claims.

\begin{exam}(\emph{The existence problem})
We want to show that the proposition $p \vee \neg p$ is not provable, i.e., there is no map in $\mathrm{Hom}_{\mathbf{NJ}}(\top, p \vee \neg p)$. For the sake of contradiction, assume $\mathrm{Hom}_{\mathbf{NJ}}(\top, p \vee \neg p) \neq \varnothing$. Then, for any $\mathrm{BCC}$ $\mathcal{C}$ and any $\mathrm{BC}$-functor $F: \mathbf{NJ} \to \mathcal{C}$, there must be a map in $\mathrm{Hom}_{\mathcal{C}}(1, F(p)+0^{F(p)})$. To reach a contradiction, it is enough to find suitable $\mathrm{BCC}$ and $\mathrm{BC}$-functor $F: \mathbf{NJ} \to \mathcal{C}$ such that $\mathrm{Hom}_{\mathcal{C}}(1, F(p)+0^{F(p)})$ is empty. We provide two instances of such a suitable choice. First,
as we only care about the existence of maps, one may aim for the lowest possible level of the $\mathrm{BCC}$'s, i.e., the Heyting algebras. For instance, take $\mathcal{C}$ as the Heyting algebra $\mathcal{O}(\mathbb{R})$ and set $f(p)=(0, \infty)$. Then, by the freeness of $\mathbf{NJ}$, there is a $\mathrm{BC}$-functor $F: \mathbf{NJ} \to \mathcal{O}(\mathbb{R})$ such that
$F(p)=f(p)$. Now, note that $0^{F(p)}=(-\infty, 0)$ and hence $F(p) +0^{F(p)}=\mathbb{R}-\{0\}$. Therefore, there is no map in $\mathrm{Hom}_{\mathcal{C}}(1, F(p)+0^{F(p)})$ as $\mathbb{R} \nsubseteq \mathbb{R}-\{0\}$. 

As another instance and to have a more BHK-style argument, let $\mathcal{C}=\mathbf{Set}^{\mathbb{Z}}$ and set $f(p)=(\{0, 1\}, \sigma)$, where $\sigma(0)=1$ and $\sigma(1)=0$. Again, $f$ gives rise to a $\mathrm{BC}$-functor $F: \mathbf{NJ} \to \mathbf{Set}^{\mathbb{Z}}$.
Note that $0^{F(p)}$ is empty, as there is no function from $\{0, 1\}$ to $\varnothing$. Now, if there is a map in $\mathrm{Hom}(1, F(p) +0^{F(p)})$, as $1=(\{*\}, id_{\{*\}})$, there is an invariant element in $F(p)+0^{F(p)}$. As $0^{F(p)}$ is empty, there must be an invariant element in $\{0, 1\}$ which is impossible.
\end{exam}

\begin{exam}(\emph{The equivalence problem}) \label{ExamEqualityProblem}
We prove that the following two derivations:
\begin{center}
	\begin{tabular}{c c c}
$\mathsf{D}_1:$   \AxiomC{$p \wedge p$}
	    \RightLabel{\footnotesize$\wedge E_1$} 
		\UnaryInfC{$p$}
		\DisplayProof 
  \quad \quad \quad
		&
$\mathsf{D}_2:$		\AxiomC{$p \wedge p$}
	    \RightLabel{\footnotesize$\wedge E_2$} 
		\UnaryInfC{$p$}
		\DisplayProof
	\end{tabular}
\end{center}
are not equal in $\mathbf{NJ}$. Assume $\mathsf{D}_1=\mathsf{D}_2$. Then, for any $\mathrm{BCC}$ $\mathcal{C}$ and any $\mathrm{BC}$-functor $F: \mathbf{NJ} \to \mathcal{C}$, the maps $F(\mathsf{D}_1)$ and $F(\mathsf{D}_2)$ must be equal. Therefore, to prove the inequality, it is again enough to find suitable $\mathrm{BCC}$ and $\mathrm{BC}$-functor $F: \mathbf{NJ} \to \mathcal{C}$ such that $F(\mathsf{D}_1) \neq F(\mathsf{D}_2)$ in $\mathcal{C}$. Notice that for this problem, we cannot use the preordered sets, as any two morphisms there are equal. Set $\mathcal{C}=\mathbf{Set}$ and $f(p)=\{0, 1\}$. Then, the resulting $\mathrm{BC}$-functor $F: \mathbf{NJ} \to \mathbf{Set}$ maps $\mathsf{D}_1$ and $\mathsf{D}_2$ to the projections $p_0: \{0, 1\} \times \{0, 1\} \to \{0, 1\}$ and $p_1: \{0, 1\} \times \{0, 1\} \to \{0, 1\}$, respectively. These two functions are not equal in $\mathbf{Set}$. Therefore, $\mathsf{D}_1$ and $\mathsf{D}_2$ cannot be equal in $\mathbf{NJ}$.
\end{exam}

\begin{exam}(\emph{The identity problem})
We show that the propositions $p$ and $p \wedge p$ are not identical in $\mathbf{NJ}$. 
Assume $p \wedge p \cong p$. Then, for any $\mathrm{BCC}$ $\mathcal{C}$ and any $\mathrm{BC}$-functor $F: \mathbf{NJ} \to \mathcal{C}$, the objects $F(p)$ and $F(p) \times F(p)$ must be isomorphic in $\mathcal{C}$. Therefore, to prove the claim, it is again enough to find suitable $\mathrm{BCC}$ and $\mathrm{BC}$-functor $F: \mathbf{NJ} \to \mathcal{C}$ such that $F(p) \ncong F(p) \times F(p)$.
Again, we cannot use preordered sets here, as for them, the logical equivalence and the propositional identity are the same and $p$ and $p \wedge p$ are logically equivalent. Set $\mathcal{C}=\mathbf{Set}$ and $f(p)=\{0, 1\}$. Then, the resulting $\mathrm{BC}$-functor $F: \mathbf{NJ} \to \mathbf{Set}$ maps $p \wedge p$ and $p$ to $\{0, 1\} \times \{0, 1\}$ and $\{0, 1\}$, respectively and $\{0, 1\} \times \{0, 1\}$ and $\{0, 1\}$ are clearly not isomorphic in $\mathbf{Set}$.
\end{exam}

These examples demonstrate how $\mathrm{BCC}$'s are useful in providing counter-models for negative claims. However, even for positive claims, sometimes working with the shadow of the syntax can be helpful, as it might be easier to handle compared to the pure syntax.
To be more precise, for any of the three mentioned problems, it is clear that if the positive claim (the existence of deductions, the equivalence of derivations, or the existence of isomorphisms) holds for all $\mathrm{BCC}$'s, it also holds for $\mathbf{NJ}$, simply because $\mathbf{NJ}$ is itself a $\mathrm{BCC}$.
This completeness-style result is not as informative as one might expect. To make it more informative, one may want to restrict the family of all $\mathrm{BCC}$'s to a complete subfamily such that working within that family is considerably easier than dealing with the raw syntax. This is the task of the following subsections, which will address the three problems discussed above.

\subsection{The Existence Problem}

For the existence problem, let us first define the suitable notion of completeness for a family of $\mathrm{BCC}$'s.

\begin{dfn}
A family $\mathfrak{F}$ of $\mathrm{BCC}$'s is called \emph{existence-complete} for $\mathbf{NJ}$ if for any two propositions $A$ and $B$ satisfying $\mathrm{Hom}_{\mathbf{NJ}}(A, B)=\varnothing$, there exists $\mathcal{C} \in \mathfrak{F}$ and a $\mathrm{BC}$-functor $F: \mathbf{NJ} \to \mathcal{C}$ such that $\mathrm{Hom}_{\mathcal{C}}(F(A), F(B))=\varnothing$. A $\mathrm{BCC}$ is called \emph{existence-complete} for $\mathbf{NJ}$, if it is existence-complete as a singleton family. For $\mathbf{NM}$ and $\mathbf{NN}$, a similar definition is in place, replacing the $\mathrm{BC}$-structure by $\mathrm{AB}$- or $\mathrm{CC}$-structures, respectively.
\end{dfn}

\begin{phil}
Reading any $\mathrm{BCC}$ $\mathcal{C}$ a proof system and any $\mathrm{BC}$-functor $F: \mathbf{NJ} \to \mathcal{C}$ as an interpretation of the pure deductions by the deductions living in $\mathcal{C}$, the existence-completeness simply states that if $B$ is not intuitionistically deductible from $A$, then there is an interpretation $F$ into a $\mathrm{BCC}$ in $\mathfrak{F}$, where $F(A)$ does not prove $F(B)$.
\end{phil}

\begin{rem}
There is a stronger version of the existence-completeness of $\mathcal{C}$ for $\mathbf{NJ}$ that simply demands the existence of a weakly-full $\mathrm{BC}$-functor $F: \mathbf{NJ} \to \mathcal{C}$. This is stronger as it is the uniform version of the existence-completeness, where one $\mathrm{BC}$-functor $F$ works for all the formulas $A$ and $B$.
\end{rem}

As the existence problem is just about the existence of morphisms, it is reasonable to think that $\mathcal{C}$ and $\mathrm{Po}(\mathcal{C})$ have the same behavior in this respect. The next corollary provides a proof for this intuition. 

\begin{lem}\label{PosetReflectionExistenceProblem}
\begin{itemize}
    \item[$(i)$]
The existence of a weakly-full $\mathrm{BC}$-functor $F: \mathbf{NJ} \to \mathcal{C}$ is equivalent to the existence of a weakly-full $\mathrm{BC}$-functor $F: \mathbf{NJ} \to \mathrm{Po}(\mathcal{C})$.   
   \item[$(ii)$]
The existence-completeness of a family $\mathfrak{F}$ of $\mathrm{BCC}$'s for $\mathbf{NJ}$ is equivalent to that of the family $\{\mathrm{Po}(\mathcal{C}) \mid \mathcal{C} \in \mathfrak{F} \}$.
\end{itemize}
A similar claim also holds for $\mathbf{NM}$ and $\mathbf{NN}$.
\end{lem}
\begin{proof}
For $(i)$, recall from Example \ref{ExamOfFull} and Example \ref{PosetReflectionIsBCFunctor} that the poset reflection of a $\mathrm{BCC}$ is a Heyting algebra and the canonical forgetful map $\pi: \mathcal{C} \to \mathrm{Po}(\mathcal{C})$ is a full and hence weakly-full $\mathrm{BC}$-functor. As the composition of weakly-full functors is weakly-full, one direction of $(i)$ is clear. For the other direction, let $G: \mathbf{NJ} \to \mathrm{Po}(\mathcal{C})$ be a weakly-full $\mathrm{BC}$-functor. Then, as $\pi$ is surjective on the objects, it is possible to find the objects $h(p_i)$ in $\mathcal{C}$ such that $\pi h(p_i)=G(p_i)$. Then, by the freeness of $\mathbf{NJ}$, Theorem \ref{Freness}, there is a $\mathrm{BC}$-functor $H: \mathbf{NJ} \to \mathcal{C}$ such that $H(p_i)=h(p_i)$ and hence $\pi H(p_i)=G(p_i)$. We claim that $\pi H \cong G$:
\[\begin{tikzcd}[ampersand replacement=\&]
	\&\& {\mathcal{C}} \\
	\\
	{\mathbf{NJ}} \&\& {\mathrm{Po}(\mathcal{C})}
	\arrow["\pi", from=1-3, to=3-3]
	\arrow["G"', from=3-1, to=3-3]
	\arrow["H", dashed, from=3-1, to=1-3]
\end{tikzcd}\]
The reason simply is that $\pi H(p_i)=G(p_i)$, for any $i \geq 0$ and hence by the universality of the freeness of $\mathbf{NJ}$, the two functor $\pi H$ and $G$ must be isomorphic. Using this fact, it is easy to show that $H$ is weakly-full, because if $\mathrm{Hom}_{\mathbf{NJ}}(A, B)=\varnothing$, then $\mathrm{Hom}_{\mathrm{Po}(\mathcal{C})}(G(A), G(B))=\varnothing$, as $G$ is weakly-full. As $G(A) \cong \pi H(A)$ and $G(B) \cong \pi H(B)$, we have $\mathrm{Hom}_{\mathrm{Po}(\mathcal{C})}(\pi H(A), \pi H(B))=\varnothing$. Therefore, $\mathrm{Hom}_{\mathcal{C}}(H(A), H(B))=\varnothing$ which completes the proof of weakly-fullness of $H$. The proof for $(ii)$ is similar.
\end{proof}

The immediate consequence of Lemma \ref{PosetReflectionExistenceProblem} is the existence-completeness of one Heyting algebra for $\mathbf{NJ}$:

\begin{thm}\emph{(Algebraic Completeness)}
The Heyting algebra $\mathrm{Po}(\mathbf{NJ})$ and hence the family of all Heyting algebras are existence-complete for $\mathbf{NJ}$.
\end{thm}
\begin{proof}
As the identity functor $id: \mathbf{NJ} \to \mathbf{NJ}$ is weakly-full, by Lemma \ref{PosetReflectionExistenceProblem}, we can prove the claim.
\end{proof}

To push the completeness further, one may want to restrict the family of Heyting algebra to Heyting algebras of some specific form. The first immediate restriction is to Heyting algebras in the form $\mathbf{2}^{(P, \leq)}$, where $(P, \leq)$ is a poset.
To prove, it is enough to
use Theorem \ref{KripkeRepForHeyting} to embed the Heyting algebra $\mathcal{H}$ into $\mathbf{2}^{[\mathcal{H}, \mathbf{2}]}$, where $[\mathcal{H}, \mathbf{2}]$ is the poset of all $\mathrm{BC}$-functors from $\mathcal{H}$ to $\mathbf{2}$:

\begin{thm}\emph{(Kripke Completeness)}
\begin{itemize}
    \item[$(i)$] 
There is a poset $(P, \leq)$ (i.e., $[\mathrm{Po}(\mathbf{NJ}), \mathbf{2}]$) and a weakly-full $\mathrm{BC}$-functor $\mathbf{NJ} \to \mathbf{2}^{(P, \leq)}$. Therefore, the poset $\mathbf{2}^{(P, \leq)}$ and consequently the family of all Alexandrov spaces is existence-complete for $\mathbf{NJ}$.
    \item[$(ii)$]
Let $FinTree$ be the family of all Heyting algebras in the form $\mathbf{2}^{(T, \leq)}$, where $(T, \leq)$ is a finite rooted tree. Then, $FinTree$ is existence-complete for $\mathbf{NJ}$.   
\end{itemize}
\end{thm}
\begin{proof}
The proof of $(i)$ is explained above. For $(ii)$, a technique called \emph{filtration} is required to show that for given formulas $A$ and $B$, finite posets and then finite rooted trees are sufficient for the existence-completeness. We will not explain the details here. The reader must consult \cite{Chagrov}.
\end{proof}

Using finite rooted trees, one can strengthen the existence-completeness even further. Recall that a space $X$ is called \emph{metrizble} if there is a metric on $X$ inducing its topology. It is called \emph{dense-in-itself} if $Cl(X-\{x\})=X$, for any $x \in X$. The spaces $\mathbb{Q}$, $\mathbb{R}$ and $\{0, 1\}^{\mathbb{N}}$, where $\{0, 1\}$ is the discrete space with two elements are all dense-in-themselves.

\begin{thm}(McKinsey-Tarski \cite{McTarski,RasSik}) Let $X$ be a dense-in-itself metrizable space. Then, $\mathcal{O}(X)$ is existence-complete for $\mathbf{NJ}$.
\end{thm}
\begin{proof}
To prove, it is enough to show that for any finite rooted tree $(T, \leq)$, there is a weakly-full $\mathrm{BC}$-functor from $\mathbf{2}^{(T, \leq)}$ into $\mathcal{O}(X)$. For that purpose and using Example \ref{OpenMaps}, it is enough to provide an open surjection $f: X \to T$, where $T$ is equipped with the upset topology. Note that the surjectivitiy of $f$ implies that $f^{-1}: \mathbf{2}^{(T, \leq)} \to \mathcal{O}(X)$ is weakly-full. This is not an easy task and one can find its proof in \cite{Nick}.
\end{proof}

\begin{cor} Each of the Heyting algebras $\mathcal{O}(\mathbb{Q})$, $\mathcal{O}(\mathbb{R})$ and $\mathcal{O}(\{0, 1\}^{\mathbb{N}})$ are existence-complete for $\mathbf{NJ}$.
\end{cor}

Of course, any existence-complete family for $\mathbf{NJ}$ can also be used as an existence-complete family for $\mathbf{NM}$ and $\mathbf{NN}$. The reason is the well-known conservativity of $\mathbf{NJ}$ over $\mathbf{NM}$ and $\mathbf{NN}$. However, there are also some interesting existence-completeness results specific to the fragments. We present one of these results as we will need it in the next subsection.

\begin{thm} (\cite{Makkai}) \label{ExistenceCompletenessOfN} There is a weakly-full $\mathrm{AB}$-functor from $\mathbf{NM}$ to $\mathbf{2}^{(\mathbb{N}, \mid)}$, where $\mid$ is the divisibility order on $\mathbb{N}$. Therefore, $\mathbf{2}^{(\mathbb{N}, | \,)}$ is existence-complete for $\mathbf{NM}$.
\end{thm}
\begin{proof}
The first part of the proof is similar to what we explained for Heyting algebras. First, notice that $\mathrm{Po}(\mathbf{NM})$ is an almost Heyting algebra and it is enough to provide a weakly-full $\mathrm{AB}$-functor from $\mathrm{Po}(\mathbf{NM})$ to $\mathbf{2}^{(\mathbb{N}, \mid)}$.
For that purpose, we first embed $\mathrm{Po}(\mathbf{NM})$ into 
$\mathbf{2}^{(P, \leq)}$, where $(P, \leq)$ is the poset $[\mathrm{Po}(\mathbf{NM}), \mathbf{2}]$ of all $\mathrm{AB}$-functors ordered by the pointwise order. One crucial point here is that the poset $(P, \leq)$ has a terminal object, which simply is the constant map $1$. Notice that this map is an $\mathrm{AB}$-functor but not a $\mathrm{BC}$-functor, as it does not preserve the initial element $0$. Having this terminal element is the main reason to deviate from $\mathbf{NJ}$ and Heyting algebras. Then, there is a technical step changing $(P, \leq)$ into a countable poset again with a terminal element \cite{Makkai}. Finally, to provide a weakly-full $\mathrm{AB}$-functor from $\mathbf{2}^{(P, \leq)}$ to $\mathbf{2}^{(\mathbb{N}, \mid)}$, using Example \ref{p-morphism}, it is enough to provide a surjective p-morphism from $(\mathbb{N}, \mid)$ to $(P, \leq)$. This is the main part of the proof and can be found in \cite{Makkai}.
\end{proof}

Note that the poset $\mathbf{2}^{(\mathbb{N}, \mid)}$ or even more generally, the poset $\mathbf{2}^{(Q, \leq)}$, where $(Q, \leq)$ is a poset with the top element $t$ is not existence-complete for $\mathbf{NJ}$. The reason is that the formula $\neg p \vee \neg \neg p$ is not provable in intuitionistic logic, i.e., there is no map in $\mathrm{Hom}_{\mathbf{NJ}}(\top, \neg p \vee \neg \neg p)$, while for any order-preserving map $f: (Q, \leq) \to \mathbf{2}$ in $\mathbf{2}^{(Q, \leq)}$, we have $\neg f \vee \neg \neg f= 1$. The reason for the latter is that either $\neg f=0$ or $\neg f=1$. To prove, assume $(\neg f)(q) =1$ and $(\neg f)(r)=0$, for some $q, r \in Q$. Then, as $q \leq t$, we have $(\neg f)(t)=1$. However, $(\neg f)(r)=0$ means that there is $s \geq r$ such that $f(s)=1$. Hence, $f(t)=1$ which is impossible as $(\neg f)(t)=1$.

\subsubsection{BHK Existence-completeness}

One interesting family whose existence-completeness is worth investigating is the family of categories of the form $\mathbf{Set}^{\mathcal{C}}$. The reason for this interest is their role in formalizing the BHK interpretation.
For that purpose, and thanks to Lemma \ref{PosetReflectionExistenceProblem}, it is enough to compute the poset reflection of $\mathbf{Set}^{\mathcal{C}}$ and check if a family of these posets is existence-complete or not. To start, it is natural to consider the canonical choice among all categories of variable sets, i.e., the category of sets itself. However, this category is not existence-complete for $\mathbf{NJ}$ as it captures the whole classical logic.
\begin{thm}
$\mathrm{Po}(\mathbf{Set}) \cong \mathbf{2}$. Therefore, for any two propositions $A$ and $B$, $\mathsf{CPC} \vdash A \to B$ iff $\mathrm{Hom}_{\mathbf{Set}}(F(A), F(B)) \neq \varnothing$, for any weakly-full $\mathrm{BC}$-functor $F: \mathbf{NJ} \to \mathbf{Set}$.
\end{thm}
\begin{proof}
First, note that in the preorder reflection of $\mathbf{Set}$, all non-empty sets are equivalent to each other as there is a function between any two non-empty sets. As the empty set has a function into any set and there is no function from a non-empty set into the empty set, the poset reflection will be the poset $\mathbf{2}$, where $0$ represents the class of the empty set and $1$ is the class of all non-empty sets. The second part is an easy consequence of the first part.
\end{proof}

The first non-trivial family to investigate is the family of variable sets over trees.

\begin{thm}\label{PosetComputationForTrees}
For any tree $(T, \leq)$, we have $\mathrm{Po}(\mathbf{Set}^{(T, \leq)}) \cong \mathbf{2}^{(T, \leq)}$. Therefore, the family of $\mathrm{BCC}$'s in the form $\mathbf{Set}^{(T, \leq)}$, where $(T, \leq)$ is a finite rooted tree is existence-complete for $\mathbf{NJ}$.
\end{thm}
\begin{proof}
First, for any object $E$ of $\mathbf{Set}^{(T, \leq)}$, define $C(E)=\{u \in T \mid E(u)\neq \varnothing\}$ as a subset of $T$. It is easy to see that $C(E)$ is actually an upset and hence lives in $\mathbf{2}^{(T, \leq)}$. Moreover, it is clear that if there is a map $\alpha: E \to F$, then $C(E) \subseteq C(F)$. Therefore, $C$ induces an order-preserving map from $\mathrm{Po}(\mathbf{Set}^{(T, \leq)})$ to $\mathbf{2}^{(T, \leq)}$. We claim that this map is an isomorphism.
First, it is clear that it is surjective, as for any upset $U$ of $(T, \leq)$, it is enough to define $E_U(u)=\{*\}$ if $u \in U$ and otherwise $E_U(u)=\varnothing$ and define $E(i_{uv})$ as the only possible map from $E_U(u)$ to $E_U(v)$, where $i_{uv}: u \to v$ is the unique map in the poset $(T, \leq)$ from $u$ to $v$. This map exists, as if $u \leq v$, either $u \notin U$ which implies $E_U(u)=\varnothing$ or $u \in U$ which implies $v \in U$ and hence $E_U(u)=E_U(v)=\{*\}$. It is also easy to check that $E_U: (T, \leq) \to \mathbf{Set}$ is a functor and $C(E)=U$. 

Secondly, $C: \mathrm{Po}(\mathbf{Set}^{(T, \leq)}) \to \mathbf{2}^{(T, \leq)}$ is an order-embedding. To prove, assume that $C([E]) \subseteq C([F])$. We must construct a map $\alpha: E \to F$. As $C([E]) \subseteq C([F])$, either $F(u)$ is non-empty or $E(u)=F(u)=\varnothing$, for any $u \in T$. Define $M \subseteq T$ as the set of all minimal elements in $T$ such that $F(u) \neq \varnothing$. For any $m \in M$, choose one $y_m \in F(m)$ in an arbitrary fashion. We can extend the assignment $y$ to all $u$'s for which $F(u) \neq \varnothing$, because for such $u$, there is a unique $m_u \in M$ such that $m_u \leq u$. The reason simply is that the segment below $u$ is well-ordered and hence has the least element in $M$. Define
$y_u$ as $F(i_{m_uu})(y_{m_u})$. It is easy to see that if $u \leq v$ and $F(u)$ and $F(v)$ are non-empty, then $m_u=m_v$ and $F(i_{uv})(y_u)=y_v$. Now, define $\alpha: E \to F$ in the following way. For any $u \in T$, if $F(u)=\varnothing$, then $E(u)=\varnothing$, and hence we can define $\alpha_u: E(u) \to F(u)$ as the unique function over the empty set. If $F(u) \neq \varnothing$, then define
$\alpha_u$ as the constant function, mapping everything to $y_u$. Checking that $\alpha: E \to F$ is a natural transformation is straightforward.
\end{proof}

The next simple case whose poset reflection is easy to compute is the category $\mathbf{Set}^G$, where $G$ is a group.

\begin{thm}\label{PosetComputationForGSet}
Let $G$ be a group and $Sub(G)$ be the preordered set of its subgroups with the ordering $K \leq H$ defined by the existence of a $g \in G$ such that $g^{-1}Kg \subseteq H$. Then, $\mathrm{Po}(\mathbf{Set}^G)=\mathbf{2}^{Sub(G)^{op}}$.
\end{thm}
\begin{proof}
To any object $E$ in $\mathbf{Set}^{G}$, assign a subset of $Sub(G)$ by $\delta (E)=\{H \leq G \mid \exists x \in E(*) \; \forall h \in H, E(h)(x)=x \}$. 
Notice that $\delta(E)$ is downward-closed, because if $g^{-1}Kg \subseteq H$ and $H \in \delta(E)$, then there exists $x \in E(*)$ such that $E(h)(x)=x$, for any $h \in H$. Therefore, $E(g^{-1}kg)(x)=x$, for any $k \in K$ which implies that $E(k)(E(g)(x))=E(g)(x)$. Hence, $K \in \delta(E)$. Now, we show that if there is a natural transformation $\alpha : E \to F$, then $\delta(E) \subseteq \delta(F)$. Assume that $H \in \delta(E)$. Hence, there exists $x \in E(*)$ such that $E(h)(x)=x$, for any $h \in H$. As $\alpha$ is a natural transformation, we have $F(h)(\alpha_* (x))=\alpha_* (E(h)(x))=\alpha_* (x)$. Therefore, $H \in \delta(F)$ and hence $\delta(E) \subseteq \delta(F)$. Using this observation, it is clear that $\delta$ induces a map from $\mathrm{Po}(\mathbf{Set}^G)$ to $\mathbf{2}^{Sub(G)^{op}}$. We claim that this map is an isomorphism. 

To prove that it is an order-embedding, assume that $\delta(E) \subseteq \delta(F)$. We must construct a natural transformation $\alpha: E \to F$. Define the relation $\sim$ on the set $E(*)$ by $x \sim y$, if there is $g \in G$ such that $E(g)(x)=y$. It is easy to use the fact that $G$ is a group and $E$ is a functor to prove that the relation $\sim$ is an equivalence relation. Take the partition of $E(*)$ into the equivalence classes of $\sim$. To define $\alpha_*: E(*) \to F(*)$, we define $\alpha_*$ on any equivalence class, separately. Take a class $O$ and pick an arbitrary element $x_O \in O$. Define $S_{x_O}=\{g \in G \mid E(g)(x_O)=x_O\}$. It is clear that $S_{x_O} \in \delta(E)$ and hence in $\delta(F)$. Therefore, there exists $y \in F(*)$ such that for any $g \in S_{x_O}$, we have $F(g)(y)=y$. This means that if $E(g)(x_O)=x_O$ then $F(g)(y)=y$, for any $g \in G$. Define $\alpha_*(x_O)=y$ and for any $z \in O$, if $z=E(g)(x_O)$, for some $g \in G$, define $\alpha_*(z)=F(g)(y)$. Note that the definition of $\alpha_*(z)$ is independent of the choice of $g$, because if $E(g)(x_O)=E(h)(x_O)=z$, then $E(g^{-1}h)(x_O)=x_O$ and hence $F(g^{-1}h)(y)=y$ which implies $F(g)(y)=F(h)(y)$. This completes the definition of $\alpha_*$ on $O$, but as $O$ is an arbitrary class, we have a function $\alpha_*: E(*) \to F(*)$. To show that $\alpha: E \to F$ is a natural transformation, one must prove $F(g)(\alpha_* (z))=\alpha_* (E(g)(z))$, for any $g \in G$ and any $z \in E(*)$. Assume that $z \in O$, for some class $O$ and $z=E(h)(x_O)$. Notice that
$E(g)(z)=E(g)E(h)(x_O)=E(gh)(x_O)$.
Then, by definition, $\alpha_* (z)=F(h)(x_O)$. Hence, $F(g)(\alpha_* (z))=F(g)(F(h)(x_O))=F(gh)(x_O)$. The last is equal to $\alpha_* (E(g)(z))$, because $E(g)(z)=E(gh)(x_O)$ and $E(g)(z)$ belongs to the class $O$.

Finally, we have to show that $\delta$ is surjective. For any downward-closed set $I \subseteq Sub(G)$, define $E_I(*)=\Sigma_{H \in I} G/H$ and $E(g)(H, [r]_H)=(H, [gr]_H)$. Then, it is easy to see that $E_I: G \to \mathbf{Set}$ is actually a functor. Moreover, we have $\delta (E_I)=I$. The reason is that for any subgroup $K$, we have $K \in \delta (E_I)$ iff there is $H \in I$ and $r \in G$ such that for any $k \in K$, we have $[r]_H=[kr]_H$. This is simply equivalent to the existence of $H \in I$ such that $K \leq H$. As $I$ is down-ward closed, the last is equivalent to $K \in I$.
\end{proof}

\begin{rem}
Notice that if the group $G$ is abelian, then the preorder $\leq$ in Theorem \ref{PosetComputationForGSet} becomes the usual inclusion.
\end{rem}

\begin{cor}(\cite{Makkai}) \label{PosetComputationForDynamicalSystems}
$\mathrm{Po}(\mathbf{Set}^{\mathbb{Z}}) \cong \mathbf{2}^{(\mathbb{N}, \mid)}$, where $\mid$ is the divisibility relation.
\end{cor}
\begin{proof}
For $G=(\mathbb{Z}, +)$, the poset $Sub(G)$ is just $(\mathbb{N}, \mid \,)^{op}$, as the subgroups of $G$ are of the form $n\mathbb{Z}$ and $m\mathbb{Z} \subseteq n\mathbb{Z}$ iff $n \mid m$.
\end{proof}

Now, having $\mathrm{Po}(\mathbf{Set}^{G})$ computed, we can combine it with Theorem \ref{ExistenceCompletenessOfN} to prove some existence-completeness results for $\mathbf{NM}$. For that purpose, we must connect $Sub(G)^{op}$ to $(\mathbb{N}, \mid)$.
Let $G$ be a group. An element $x \in G$ is called of an \emph{infinite non-commutative order} if $x^m$ is not in the subgroup generated by $\{g^{-1}x^ng \mid g \in G\}$, for any $n > m \geq 0$. If $G$ is abelian, then $x$ is of infinite non-commutative order iff $x^n \neq e$, for any $n>0$.

\begin{lem}\label{InfiniteOrder}
If $G$ has an element of an infinite non-commutative order, then there exists a surjective p-morphism $f: Sub(G)^{op} \to (\mathbb{N}, \mid) $.
\end{lem}
\begin{proof}
Let $x$ be an element of an infinite non-commutative order. For any subgroup $H$, if there is a positive number $n$ such that $\{g^{-1}x^ng \mid g \in G \} \subseteq H$, define $f(H)$ as the least of such $n$'s and if there is no $n$ with that property, define $f(H)=0$. It is easy to prove that for any such $n$, we have $f(H) \mid n$. We first show that $f$ is order-preserving. Let $H \leq K$. If $f(H)=0$, we have $f(K) \mid f(H)$. Therefore, w.l.o.g., we can assume that $f(H) \neq 0$. As $H \leq K$, there is some $g \in G$ such that $g^{-1}Hg \subseteq K$. We claim that $r^{-1} x^{f(H)} r \in K$, for any $r \in G$. To prove, set $s=rg^{-1}$. Then,
$r^{-1}x^{f(H)}r=g^{-1}s^{-1}x^{f(H)}sg
$. By definition of $f(H)$, we have $s^{-1}x^{f(H)}s \in H$. Therefore, $r^{-1}x^{f(H)}r \in K$. As $r \in G$ is arbitrary, we have $f(K) \mid f(H)$. To prove that $f$ is surjective, for any $n \geq 0$, define $H_n$ as the subgroup generated by $\{g^{-1}x^ng | g \in G\}$. Then, as $x$ is of an infinite non-commutative order, then $f(H_n)=n$. Finally, to prove that $f$ is a p-morphism, assume that $f(H) \mid m$, for some number $m \in \mathbb{N}$. Consider $H_m$. As $f(H_m)=m$, it is enough to prove that $H_m \leq H$. For any $g \in G$, as $g^{-1}x^{f(H)}g \in H$ and $f(H) \mid m$, we also have $g^{-1}x^{m}g \in H$, for any $g \in G$. Therefore, $H_m \subseteq H$ and hence $H_m \leq H$. 
\end{proof}

\begin{thm}\label{LauchliMPC}
Let $G$ be a group that has an element of an infinite non-commutative order. Then,
there exists a weakly-full $\mathrm{AB}$-functor $F: \mathbf{NM} \to \mathbf{Set}^{G}$. Therefore, $\mathbf{Set}^{G}$ is existence-complete for $\mathbf{NM}$.
\end{thm}
\begin{proof}
By Lemma \ref{PosetReflectionExistenceProblem}, it is enough to prove the existence of a weakly-full $\mathrm{AB}$-functor from $\mathbf{NM}$ to $\mathrm{Po}(\mathbf{Set}^{G})$.
By Theorem \ref{PosetComputationForGSet}, we know that $\mathrm{Po}(\mathbf{Set}^{G})$ is $\mathbf{2}^{Sub(G)^{op}}$ and by Lemma \ref{InfiniteOrder}, there is a surjective p-morphism from $Sub(G)^{op}$ to $(\mathbb{N}, \mid)$ which implies the existence of a weakly-full $\mathrm{AB}$-functor from $\mathbf{2}^{(\mathbb{N}, \mid)}$ to $\mathbf{2}^{Sub(G)^{op}}$.  
Therefore, it is just enough to find a weakly-full $\mathrm{AB}$-functor from $\mathbf{NM}$ to $\mathbf{2}^{(\mathbb{N}, \mid)}$ which is provided by Theorem \ref{ExistenceCompletenessOfN}.
\end{proof}

\begin{rem}
One may wonder if we can extend Theorem \ref{LauchliMPC} to cover the whole $\mathbf{NJ}$. The answer is negative. As we observed before, a poset in the form $\mathbf{2}^{(Q, \leq)}$ where $(Q, \leq)$ has a top element cannot be existence-complete for $\mathbf{NJ}$. As $\mathrm{Po}(\mathbf{Set}^{G})$ is $\mathbf{2}^{Sub(G)^{op}}$ and $Sub(G)^{op}$ has the trivial subgroup $\{e\}$ as its top element, we cannot expect $\mathbf{Set}^G$ to be existence-complete for $\mathbf{NJ}$.
\end{rem}

\begin{cor}(L\"{a}uchli realizability for $\mathsf{MPC}$ \cite{Lach,Makkai})
There exists a weakly-full $\mathrm{AB}$-functor $F: \mathbf{NM} \to \mathbf{Set}^{\mathbb{Z}}$. Therefore, $\mathbf{Set}^{\mathbb{Z}}$ is existence-complete for $\mathbf{NM}$.
\end{cor}

\begin{rem}
L\"{a}uchli realizability \cite{Lach} was the first BHK style complete semantics for minimal logic. It is also possible to extend it to intuitionistic logic using a trick to interpret $\bot$ \cite{Makkai}. The modern categorical version we present here is developed and elegantly presented in \cite{Makkai}. 
\end{rem}

So far, we have computed the poset reflection of some categories of the form $\mathbf{Set}^{\mathcal{C}}$. Let us conclude this subsection by a remark that such computation is not always an easy task. For instance, there are small categories $\mathcal{C}$ such that the poset reflection of $\mathbf{Set}^{\mathcal{C}}$ is not even a set. To see such a situation,
let $X$ and $Y$ be two sets and $f, g: X \to Y$ be two functions:
\[\begin{tikzcd}[ampersand replacement=\&]
	X \&\& Y
	\arrow["f", shift left=1, from=1-1, to=1-3]
	\arrow["g"', shift right=1, from=1-1, to=1-3]
\end{tikzcd}\]
Then, define the category $\mathcal{C}_{f, g: X \to Y}$ in the following manner. For the objects, use the elements in the disjoint union $X+Y=\{(0, x) \mid x \in X\} \cup \{(1, y) \mid y \in Y\}$ of $X$ and $Y$ and for the maps, define $\mathrm{Mor}(\mathcal{C}_{f, g: X \to Y})$ as the union of the set of identity maps and the set 
\[
\{f_x: (0, x) \to (1, y) \mid f(x)=y\} \cup \{g_x: (0, x) \to (1, y) \mid g(x)=y\}
\]
where for any pair $(x, y)$ if $f(x)=y$ (resp. $g(x)=y$), there is a
map $f_x: (0, x) \to (1, y)$ (resp. $g_x: (0, x) \to (1, y)$. As no non-trivial composition is possible with our maps, we define the composition in a trivial way. In \cite{Lawvere}, Lawvere claimed that there are sets $X$ and $Y$ and maps $f, g: X \to Y$ such that $\mathrm{Po}(\mathbf{Set}^{\mathcal{C}_{f, g: X \to Y}})$
is not even a set, see \cite{menni2000exact}.

\subsection{The Equivalence Problem}

For the equivalence problem, we only focus on one $\mathrm{BCC}$ rather than a family of $\mathrm{BCC}$'s. Here is the formal definition of the completeness of a $\mathrm{BCC}$.

\begin{dfn}
A $\mathrm{BCC}$ $\mathcal{C}$ is called \emph{equivalence-complete} for $\mathbf{NJ}$, when for any two maps $f, g: A \to B$ in $\mathbf{NJ}$, if $F(f)=F(g)$, for any $\mathrm{BC}$-functor $F: \mathbf{NJ} \to \mathcal{C}$, then $f=g$. For $\mathbf{NM}$ and $\mathbf{NN}$, a similar definition is in place, replacing the $\mathrm{BC}$-structure by the $\mathrm{AB}$- or $\mathrm{CC}$-structures, respectively.
\end{dfn}

\begin{rem}
There is also a stronger version of the equivalence-completeness of $\mathcal{C}$ for $\mathbf{NJ}$ that simply demands the existence of a faithful $\mathrm{BC}$-functor $F: \mathbf{NJ} \to \mathcal{C}$. In fact, it is a uniform version of the equivalence-completeness, as we have just one $\mathrm{BC}$-functor working for all formulas $A$ and $B$.
\end{rem}

\subsubsection{Equivalence-completeness for $\mathbf{NN}$}

Historically, the equivalence problem started with and still focused on the $\mathbf{NN}$ case and the $\mathrm{CC}$-structure. The reason is the close connection this setting has with the simply typed lambda calculus as a very basic programming language.
The first two important equivalence-completeness results are the following:

\begin{thm}(\cite{Cubric})
Any $\mathrm{CC}$-functor $F: \mathbf{NJ} \to \mathbf{Set}$ that maps the atoms of $\mathcal{L}_p$ to infinite sets is faithful. Therefore, $\mathbf{Set}$ is equivalence-complete for $\mathbf{NN}$.
\end{thm}

\begin{thm}(\cite{Soloviev})\label{EqualityCompletenessOfFinSet}
$\mathbf{FinSet}$ is equivalence-complete for $\mathbf{NN}$.
\end{thm}

The equivalence-completeness of $\mathbf{FinSet}$ is crucial as it helps to prove the decidability of the equivalence problem for $\mathbf{NN}$.

\begin{cor}(\cite{Soloviev}) \label{DecidabilityOfEqualityInNN}
The equivalence problem for $\mathbf{NN}$ is decidable.
\end{cor}
\begin{proof}
It is enough to prove that the equivalence problem and its complement are both recursively enumerable. It is recursively enumerable as if two proofs are equal, \emph{there exists} a sequence of basic $\beta\eta$-equalities transforming one to the other. For the complement, if two proofs are not equal, then by Theorem \ref{EqualityCompletenessOfFinSet}, \emph{there is} an assignment of finite sets to the atoms in the proofs such that the resulting functions are not equal.
\end{proof}

It is worth mentioning that the two versions of the equivalence-completeness are not equivalent. For instance, although $\mathbf{FinSet}$ is equivalence-complete for $\mathbf{NN}$, there is no faithful $\mathrm{CC}$-functor $F: \mathbf{NN} \to \mathbf{FinSet}$. As explained in \cite{Simpson}, the reason is the existence of infinitely many maps in $\mathrm{Hom}_{\mathbf{NN}}((p \to p) \wedge p, p)$ encoding the natural numbers. To explain, for any natural number $n \in \mathbb{N}$, define the derivation $\mathsf{D}_n$ of $p$ from $(p \to p) \wedge p$, inductively by:
\begin{center}
	\begin{tabular}{c c c}
$\mathsf{D}_0$: \AxiomC{$(p \to p) \wedge p$}
\RightLabel{\footnotesize$\wedge_2 E$} 
		\UnaryInfC{$p$}
		\DisplayProof \quad \quad \quad 
  &
$\mathsf{D}_{n+1}$: \AxiomC{$(p \to p) \wedge p$}
 \RightLabel{\footnotesize$\wedge_1 E$} 
		\UnaryInfC{$p \to p$}
  
	    \AxiomC{$(p \to p) \wedge p$}
     \noLine
\UnaryInfC{$\mathsf{D}_n$}
\noLine
\UnaryInfC{$p$}
\RightLabel{\footnotesize$\to E$} 
		\BinaryInfC{$p$}
		\DisplayProof 
	\end{tabular}
\end{center}
These proofs form an infinite family $\{\mathsf{D}_n\}_{n \in \mathbb{N}}$ of distinct derivations. To see why, it is enough to show that a $\mathrm{CC}$-functor maps them to different maps. Using the freeness of $\mathbf{NN}$, define $F: \mathbf{NN} \to \mathbf{Set}$ by mapping $p$ to the set $\mathbb{N}$. Then, $F$ maps $\mathsf{D}_n$ to the function $f_n: \mathbb{N}^{\mathbb{N}} \times \mathbb{N} \to \mathbb{N}$, defined by $f_n(\alpha, m)=\alpha^n(m)$, where $\alpha^n$ means iterating $\alpha$ for $n$ many times. Substituting $\alpha$ as the successor function $s$ and setting $m=0$, it is clear that $f_n(s, 0)=n$, which implies that all $f_n$'s are distinct.

In the rest of this subsubsection, we will present a necessary and sufficient condition for a $\mathrm{CCC}$ to be equivalence-complete for $\mathbf{NN}$ in any of the two mentioned senses. First, for the weaker version, it is clear that if $\mathcal{C}$ is a preordered set, then it cannot be complete for $\mathbf{NN}$, as there are at least two non-equal projection morphisms in $\mathrm{Hom}_{\mathbf{NJ}}(p \wedge p, p)$, as discussed in Example \ref{ExamEqualityProblem}.
The good news is that this necessary condition is actually sufficient.

\begin{thm}(Simpson \cite{Simpson}) \label{EqualityCompletenessSimpson}
For any $\mathrm{CCC}$ $\mathcal{C}$, it is equivalence-complete for $\mathbf{NN}$ iff it is not a preordered set.
\end{thm}

\begin{phil}\label{NNIsOptimal}
By Theorem \ref{EqualityCompletenessSimpson}, if $\mathcal{C}$ is not equivalence-complete for $\mathbf{NN}$, then it must be a preordered set. Reading any $\mathrm{CCC}$ as a proof system, it simply means that if two non-equivalent derivations in $\mathbf{NN}$ are considered as equal in $\mathcal{C}$, then the whole notion of deduction in $\mathcal{C}$ collapses to the mere deductibility. In other words, it is impossible to add more equalities than what is governed by the $\beta \eta$-equivalences if we want to keep the notion of deductibility non-trivial. In this sense, the deduction equivalence of $\mathbf{NN}$ is optimal \cite{DosenMax}. 
\end{phil}
For the other version of the equivalence-completeness for $\mathbf{NN}$, let us first explain a necessary condition \cite{Simpson}. A map $f: A \to A$ in a $\mathrm{CCC}$ $\mathcal{C}$ is called of an \emph{infinite order} if $f^m \neq f^n$, for any $m \neq n$, where $f^k$ means the result of $k$ many compositions of $f$ with itself.
We show that the existence of a map $f: A \to A$ in $\mathcal{C}$ with an infinite order is a necessary condition, if we want a faithful functor $F: \mathbf{NN} \to \mathcal{C}$. To see why, it is enough to find a map in $\mathbf{NN}$ with an infinite order. For that purpose, consider the following derivation $\mathsf{D}$ from $[(p \to p) \wedge p] \to p$ to itself:
\begin{center}
\small
	\begin{tabular}{c}

\AxiomC{$[(p \to p) \wedge p]_1$}
 \UnaryInfC{$p \to p$} 

 \AxiomC{$[(p \to p) \wedge p]_1$}
 \UnaryInfC{$p$}

\AxiomC{$[(p \to p) \wedge p]_1$}
 \UnaryInfC{$p \to p$}

\BinaryInfC{$p$}

\BinaryInfC{$(p \to p) \wedge p$}

  \AxiomC{$[(p \to p) \wedge p] \to p$}
		\BinaryInfC{$p$} 
\RightLabel{\footnotesize$\to I_1$} 
		\UnaryInfC{$[(p \to p) \wedge p] \to p$}
		\DisplayProof 
	\end{tabular}
\end{center}
This derivation is of an infinite order in $\mathbf{NN}$. To see why, it is enough to find a $\mathrm{CC}$-functor $F: \mathbf{NN} \to \mathbf{Set}$ such that $F(\mathsf{D}^m) \neq F(\mathsf{D}^n)$, for any $m \neq n$. For that purpose, use the freeness of $\mathbf{NN}$ and 
find a $\mathrm{CC}$-functor $F: \mathbf{NN} \to \mathbf{Set}$
such that $F(p)=\mathbb{N}$. Then, one can see that $F(\mathsf{D})$ is the  function $G: \mathbb{N}^{\mathbb{N}^\mathbb{N} \times \mathbb{N}} \to \mathbb{N}^{\mathbb{N}^\mathbb{N} \times \mathbb{N}}$ defined by $G(\alpha)=\lambda fk. \alpha(f, f(k))$. Computing $G^n$, we see that $G^n(\alpha)=\lambda fk. \alpha(f, f^n(k))$ and hence $G^n(p_1)=\lambda fk. f^n(k)$. Therefore, $G^n(p_1)$'s are distinct functions, for different $n$'s which means that $G^n \neq G^m$, if $m \neq n$.

The following theorem shows that this necessary condition is actually sufficient:

\begin{thm}(Simpson \cite{Simpson})
For any $\mathrm{CCC}$ $\mathcal{C}$, there exists a faithful $\mathrm{CC}$-functor $F: \mathbf{NN} \to \mathcal{C}$ iff $\mathcal{C}$ contains an endomorphism of an infinite order.
\end{thm}

\subsubsection{Equivalence-completeness for $\mathbf{NJ}$}

The structure of the proofs in $\mathbf{NJ}$ is usually more complex than what is present in $\mathbf{NN}$. The reason simply is the rather complex behavior of the disjunction elimination rule. However, recently, the $\mathbf{NJ}$ counterpart of Theorem \ref{EqualityCompletenessOfFinSet} was also proved:

\begin{thm}(Scherer \cite{scherer,schererthesis}) \label{EqualiyCompletenessNJFinitset}
$\mathbf{FinSet}$ is equivalence-complete for $\mathbf{NJ}$.
\end{thm}
\begin{cor}
The equivalence problem for $\mathbf{NJ}$ is decidable.
\end{cor}
\begin{proof}
The proof is similar to that of Corollary \ref{DecidabilityOfEqualityInNN}.
\end{proof}

Our final task in this subsection is to generalize Theorem \ref{EqualiyCompletenessNJFinitset} from $\mathbf{FinSet}$ to any non-preorder $\mathrm{BCC}$. The basic idea behind the following line of argument is attributed to Alex Simpson \cite{SimpOnline}. First, notice that using the argument we gave for $\mathbf{NN}$, it is clear that if a $\mathrm{BCC}$ is equivalence-complete for $\mathbf{NJ}$, then it cannot be a preordered set. For sufficiency, we need the following lemma:

\begin{lem}\label{DisjointnessInBCC}
Let $\mathcal{C}$ be a $\mathrm{BCC}$. For any objects $A$, $B$ and $C$, if $i_0: A \to A+B$ and $i_1: B \to A+B$ are the injections, then:
\begin{description}
    \item[$(i)$]
$i_0$ is monic, i.e., for any maps $f, g: C \to A$ if $i_0f=i_0g: C \to A+B$, then $f=g$. 
 \item[$(ii)$]
$i_1$ is monic, i.e., for any maps $f, g: C \to B$ if $i_1f=i_1g: C \to A+B$, then $f=g$. 
    \item[$(iii)$] 
If $\mathcal{C}$ is not a preordered set, then for any maps $f: 1 \to A$ and $g: 1 \to B$, the maps $i_0f, i_1g: 1 \to A+B$ are distinct.
\end{description}
\end{lem}
\begin{proof}
For $(i)$, consider the map $i_0 \times id_A : A \times A \to (A+B) \times A$. First, we intend to define a map $j: (A+B) \times A \to A \times A$ such that $j \circ (i_0 \times id_A)=id_{A \times A}$:
\[\begin{tikzcd}
	{A \times A} && {(A+B) \times A} \\
	\\
	&& {A \times A}
	\arrow["{i_0 \times id_A}", from=1-1, to=1-3]
	\arrow["j", from=1-3, to=3-3]
	\arrow["{id_{A \times A}}"', from=1-1, to=3-3]
\end{tikzcd}\]
For that purpose, define $j'$ as the following canonical map: 
\[
j'=(\lambda_A id_{A \times A}, \lambda_A \langle p_1, p_1 \rangle): A+B \to [A, A \times A]
\]
and set $j=ev \langle j'\times id_A \rangle: (A+B) \times A \to A \times A$. It is easy to see that $j \circ (i_0 \times id_A)=id_{A \times A}$. Now, let $f, g: C \to A$ satisfy $i_0f=i_0g$. Then,
\[\begin{tikzcd}
	C &&& {A \times A} && {(A+B) \times A} \\
	\\
	&&&&& {A \times A}
	\arrow["j", from=1-6, to=3-6]
	\arrow["{i_0 \times id_A}", from=1-4, to=1-6]
	\arrow["{id_{A \times A}}"', from=1-4, to=3-6]
	\arrow["{\langle f, f \rangle}", shift left=1, from=1-1, to=1-4]
	\arrow["{\langle g, f \rangle}"', shift right=1, from=1-1, to=1-4]
\end{tikzcd}\]
the two maps of the first row are equal. Hence, $\langle f, f \rangle=\langle g, f \rangle$ which implies $f=g$. The proof for $(ii)$ is similar to $(i)$.

For $(iii)$, assume $i_0f= i_1g: 1 \to A+B$, for some $f: 1 \to A$ and $g: 1 \to B$. As $\mathcal{C}$ is not a preordered set, there are distinct morphisms $k, l: E \to F$ in $\mathcal{C}$. Therefore, the two maps $\lambda_E k, \lambda_E l: 1 \to [E, F]$ are distinct. Now, consider the maps $\lambda_E k !: A \to [E, F]$ and $\lambda_E l !: B \to [E, F]$ and note that $\lambda_E k ! f=\lambda_E k $ and $\lambda_E l ! g=\lambda_E l$.
Using the universal property for the coproduct $A+B$, we have a map $h: A+B \to [E, F]$ such that $h\circ i_0=\lambda_E f !$ and $h \circ i_1=\lambda_E l !$:
\[\begin{tikzcd}
	&&&& {[E, F]} \\
	\\
	\\
	1 && A && {A+B} && B && 1
	\arrow["{i_1}", from=4-7, to=4-5]
	\arrow["{\lambda_E k}", from=4-1, to=1-5]
	\arrow["h"', dashed, from=4-5, to=1-5]
	\arrow["{i_0}"', from=4-3, to=4-5]
	\arrow["f"', from=4-1, to=4-3]
	\arrow["g"', from=4-9, to=4-7]
	\arrow["{\lambda_E l}"', from=4-9, to=1-5]
	\arrow["{\lambda_E l!}", from=4-7, to=1-5]
	\arrow["{\lambda_E k!}"', from=4-3, to=1-5]
\end{tikzcd}\]
As $i_0f=i_1g$, we have $\lambda_E k=hi_0f=hi_1g= \lambda_E l$ which is a contradiction.
\end{proof}

\begin{cor}\label{AllInjectionsAreDifferent}
Let $\mathcal{C}$ be a non-preorder $\mathrm{BCC}$. Then, the maps $i_j: 1 \to \sum_{j=0}^n 1$ are distinct for different $j$'s.
\end{cor}
\begin{proof}
The proof is by an induction on $n$. For $n=1$, by part $(iii)$ of Lemma \ref{DisjointnessInBCC}, we have $i_0 \neq i_1$. To prove the claim for $n+1$, if $i_j=i_k$, for some $j, k \leq n$, we can use part $(ii)$ of Lemma \ref{DisjointnessInBCC} to reduce the claim to the induction hypothesis. If $j \leq n$ and $k=n+1$, we can use part $(iii)$ of Lemma \ref{DisjointnessInBCC} to reach a contradiction.
\end{proof}

\begin{lem}\label{FaithfulIntoNonpreorder}
Let $\mathcal{C}$ be a non-preorder $\mathrm{BCC}$. Then, there is a faithful $\mathrm{BC}$-functor $F: \mathbf{FinSet} \to \mathcal{C}$.
\end{lem}
\begin{proof}
For any finite set $A$, define $F(A)$ as $\sum_{a \in A} 1$ and for any function $f: A \to B$ define $F(f): \sum_{a \in A} 1 \to \sum_{b \in B} 1$ by the canonical map $(i_{f(a)})_{a \in A}$. It is easy to prove that $F$ is a $\mathrm{BC}$-functor. For faithfulness, assume that $f, g: A \to B$ are two maps such that $F(f)=F(g)$. Then, using the definition of $F(f)$ and $F(g)$, we have $i_{f(a)}=i_{g(a)}: 1 \to  \sum_{b \in B} 1$, for any $a \in A$. As $\mathcal{C}$ is not a preordered set, by Corollary \ref{AllInjectionsAreDifferent}, we have $f(a)=g(a)$ and as $a \in A$ is arbitrary, $f=g$.
\end{proof}

\begin{cor}\label{CorEqComplNJ}
A $\mathrm{BCC}$ $\mathcal{C}$ is equivalence-complete for $\mathbf{NJ}$ iff it is not a preordered set.
\end{cor}
\begin{proof}
One direction is explained before. For the other, if $f \neq g: A \to B$ in $\mathbf{NJ}$, by Theorem \ref{EqualiyCompletenessNJFinitset}, there is a $\mathrm{BC}$-functor $G: \mathbf{NJ} \to \mathbf{FinSet}$ such that $G(f) \neq G(g)$. Combined with the $\mathrm{BC}$-functor $F: \mathbf{FinSet} \to \mathcal{C}$ of Lemma \ref{FaithfulIntoNonpreorder}, we reach the $\mathrm{BC}$-functor $FG: \mathbf{NJ} \to \mathcal{C}$ and as $F$ is faithful, we have $FG(f) \neq FG(g)$. This completes the proof of the equivalence-completeness of $\mathcal{C}$ for $\mathbf{NJ}$.
\end{proof}

\begin{phil}
Similar to Philosophical Note \ref{NNIsOptimal}, Corollary \ref{CorEqComplNJ} implies that it is impossible to add more equalities than what is governed by the $\beta \eta$-equivalences to the $\mathbf{NJ}$, if we want to keep the deductibility non-trivial. In this sense, the deduction equivalence of $\mathbf{NJ}$ is optimal. 
\end{phil}

\subsection{The Identity Problem}\label{SubsectionIdentity}

In this subsection, we will focus on the last problem among the three we introduced in Section \ref{SecHistory}, i.e., the propositional identity problem. To identify all the pairs of isomorphic propositions in a given fragment, we have two approaches to follow. Either we ask syntactically for some basic isomorphisms from which all the other isomorphisms can be derived or we can go for a semantical approach of having a concrete complete category $\mathcal{C}$ such that two propositions are isomorphic iff all of their images in $\mathcal{C}$ are isomorphic. We will explain both of these approaches later. But, first, we need a diversion into the arithmetical world.

\subsubsection{Tarski's High School Algebra Problem}
Consider the first-order language $\{\times, +, (-)^{(-)}, 1\}$ for arithmetic and the equations in Table \ref{HighSchoolAlgebra}.
\begin{table}[ht!]
\centering
\begin{tabular}{ |c| c| c| }
\hline
$1 \times x=x$ & $x \times y=y \times x$ &  $x \times (y \times z)=(x \times y) \times z$ \\
\hline
$x^1=x$ &
$x^{(y\times z)}=(x^y)^z$ &
$1^x=1$ \\
\hline
$(x \times y)^z=x^z \times y^z$
&
$x + y=y + x$
&
$x + (y + z)=(x + y) + z$\\
\hline
$x \times (y+z)=x \times y + x \times z$
&
$x^{(y+z)}=x^y \times x^z$
&
\\
\hline
\end{tabular}
\caption{High School Identities}\label{HighSchoolAlgebra}
\end{table}
All these equations clearly hold in the \emph{standard model}, i.e., the set of natural numbers equipped with the usual addition, multiplication, and exponentiation functions and the number $1$ as its designated constant. In \cite{Tarski}, Tarski asked if the equalities in Table \ref{HighSchoolAlgebra} axiomatize all equalities in the language $\{\times, +, (-)^{(-)}, 1\}$ that hold over the set of positive natural numbers and conjectured that the answer is positive. For some fragments of the language, we know that this is really the case. For instance, for the fragment $\{1, \times, (-)^{(-)}\}$, the $+$-free equations in Table \ref{HighSchoolAlgebra} axiomatize all the valid equations over the fragment. However, later Wilkie \cite{Wilkie} proved that for the full language, the answer is negative. Consider the following equation called Wilkie's equation:
\[
(A^x+B^x)^y \times (C^y+D^y)^x=(A^y+B^y)^x \times (C^x+D^x)^y,
\]
where $A=1+x$, $B=1+x+x^2$, $C=1+x^3$ and $D=1+x^2+x^4$. Wilkie's equation is valid in the standard model. To see why, it is enough to note that we have the identity $AD=BC$. Therefore,
\[
A^{xy}[(A^x+B^x)^y \times (C^y+D^y)^x]=(A^x+B^x)^y \times ((AC)^y+(AD)^y)^x= 
\]
\[
(A^x+B^x)^y \times ((AC)^y+(BC)^y)^x=C^{xy}[(A^x+B^x)^y \times (A^y+B^y)^x].
\]
Similarly, we have 
\[
A^{xy}[(A^y+B^y)^x \times (C^x+D^x)^y]=(A^y+B^y)^x \times ((AC)^x+(AD)^x)^y= 
\]
\[
(A^y+B^y)^x \times ((AC)^x+(BC)^x)^y=C^{xy}[(A^y+B^y)^x \times (A^x+B^x)^y].
\]
Therefore,
\[
A^{xy}[(A^x+B^x)^y \times (C^y+D^y)^x]=A^{xy}[(A^y+B^y)^x \times (C^x+D^x)^y].
\]
As $A$ is always positive, we reach the equality we wanted to prove by canceling out the term $A^{xy}$. Notice that in all the steps of this proof except the canceling out part, we only used the high school algebra identities in Table \ref{HighSchoolAlgebra}.
Using a complex proof-theoretic argument, Wilkie \cite{Wilkie} showed that this equation is not provable by the axioms in Table \ref{HighSchoolAlgebra}. Later Gurevi\v{c} \cite{GurevichFiniteModel} provided a finite counter-model to reprove this unprovability. He also showed that the theory of all valid equalities is not finitely axiomatizable. More precisely, he considered the equation:
\[
(A^x+B_n^x)^{2^x} \times (C_n^{2^x}+D_n^{2^x})^x=(A^{2^x}+B_n^{2^x})^x \times (C_n^x+D_n^x)^{2^x},
\]
where $A=1+x$, $B_n=\sum_{i=0}^{n-1}x^i$, $C_n=1+x^n$ and $D_n=\sum_{i=0}^{n-1}x^{2i}$ and $n \geq 3$ is an odd number. Using the fact that $AD_n=B_nC_n$, one can see that the equation is valid in the standard model. Then, Gurevi\v{c} showed that for any finite set of axioms, there exists an odd $n \geq 3$ such that the corresponding equation is not provable from the axioms. Although the theory of all valid equations is not finitely axiomatizable, it is known that it is decidable \cite{Macintyre,GurevichFiniteModel}.

\subsubsection{The Propositional Identity}

In this subsubsection, we will address the propositional identity problem for the systems $\mathbf{NN}$, $\mathbf{NM}$ and $\mathbf{NJ}$. For $\mathbf{NN}$, using the positive solution for Tarski's high school algebra problem for the $+$-free fragment, one can prove:
\begin{thm} \label{PropIdentityForNN}
The following are equivalent:
\begin{description}
\item[$(i)$]
$A \cong B$ is provable from the theory $Th_{\top, \wedge, \to}$ presented in Table \ref{AxiomsForIsoCCC}. 
    \item[$(ii)$]
$A \cong B$ as two objects in $\mathbf{NN}$.
\item[$(iii)$]
$F(A) \cong F(B)$, for any $\mathrm{CC}$-functor $F: \mathbf{NN} \to \mathbf{FinSet}$.
\end{description}
\end{thm}
\begin{proof}
Proving $(ii)$ from $(i)$ is a consequence of the fact that all the isomorphisms in Table \ref{AxiomsForIsoCCC} hold in any $\mathrm{CCC}$ including $\mathbf{NN}$. Proving $(iii)$ from $(ii)$ is also easy, as functors respect isomorphisms. To prove $(i)$ from $(iii)$, first rewrite the formulas $A$ and $B$ as terms $t_A$ and $t_B$ in the language $\{\times, (-)^{(-)}, 1\}$ replacing $\top$, $\wedge$ and $\to$ with $1$, $\times$ and $(-)^{(-)}$, respectively. Now, as the relation $\cong$ over $\mathbf{FinSet}$ is just the equality between the cardinals of the finite sets, $(iii)$ implies the validity of $t_A=t_B$ in the standard model. Hence, $t_A=t_B$ is provable from the $+$-free high school algebra identities that are nothing but the axioms of $Th_{\top, \wedge, \to}$ in disguise.  
\end{proof}

\begin{table}[ht!]
\centering
\begin{tabular}{ |c| c| }
\hline
$\top \wedge A \cong A$ & $A \wedge B \cong B \wedge A$ \\
\hline
$A \wedge (B \wedge C) \cong (A \wedge B) \wedge C$
&
$\top \to A \cong A$ \\
\hline
$(A \wedge B) \to C \cong (A \to (B \to C))$
&
$A \to \top \cong \top$\\
\hline
$A \to (B \wedge C) \cong (A \to B) \wedge (A \to C)$
&
$ $
\\
\hline
\end{tabular}
\caption{$Th_{\top, \wedge, \to}$}\label{AxiomsForIsoCCC}
\end{table}

\begin{cor}
The propositional identity problem in $\mathbf{NN}$ is decidable.
\end{cor}
\begin{proof}
Using Theorem \ref{PropIdentityForNN}, $A \cong B$ holds in $\mathbf{NN}$ iff it is provable from a finitely axiomatizable theory. Hence, the relation $\cong$ is recursively enumerable. Moreover, The theorem also proves that $A \ncong B$ iff there exists a $\mathrm{CC}$-functor $F: \mathbf{NN} \to \mathbf{FinSet}$ such that $F(A) \ncong F(B)$. As any $\mathrm{CC}$-functor is essentially a numeral assignment to the atoms in $A$ and $B$ and $F(A) \ncong F(B)$ is just an inequality between natural numbers, we see that $\ncong$ is also recursively enumerable. Therefore, $\cong$ is decidable.
\end{proof}

For $\mathbf{NM}$, the propositional identities is not finitely axiomatizable as shown in \cite{Fiore}. Actually, all Gurevi\v{c}'s equations are also propositional identities, and hence, we can use the non-finite-axiomatizability of the valid equations in that language to prove the non-finite-axiomatizability of the propositional identities in  $\mathbf{NM}$.
For $\mathbf{NJ}$, in \cite{DiCosmoDufour}, it is proved that the equational theory for the arithmetical language $\{\times, +, (-)^{(-)}, 0, 1\}$ is not finitely-axiomatizable, using again Gurevi\v{c}'s equations. Hence, the propositional identities in $\mathbf{NJ}$ is not finitely axiomatizable. Moreover, as observed in \cite{Fiore}, for $\mathbf{NJ}$, the equivalence between $(ii)$ and $(iii)$ in Theorem \ref{PropIdentityForNN} breaks downs. To see why, it is enough to consider the propositions $\neg p \vee \neg \neg p$ and $\top$. Any of their interpretations in $\mathbf{FinSet}$ are isomorphic as both sides have one element. However, they are not even intuitionistically equivalent.
Having that said, for negative propositions:
\begin{thm}(\cite{Fiore})
The following are equivalent:
\begin{description}
\item[$(i)$]
$\mathsf{IPC} \vdash \neg A \longleftrightarrow \neg B$.
    \item[$(ii)$]
$\neg A \cong \neg B$ as two objects in $\mathbf{NJ}$.
\item[$(iii)$]
$F(\neg A) \cong F(\neg B)$, for any $\mathrm{BC}$-functor $F: \mathbf{NJ} \to \mathbf{FinSet}$.
\end{description}
\end{thm}

\begin{cor}
The following are equivalent:
\begin{description}
    \item[$(i)$]
$A \cong 0$ as two objects in $\mathbf{NJ}$.
\item[$(ii)$]
$F(A) \cong 0$, for any $\mathrm{BC}$-functor $F: \mathbf{NJ} \to \mathbf{FinSet}$.
\end{description} 
\end{cor}
To the best of our knowledge, the decidability of the propositional identity problem for $\mathbf{NM}$ and $\mathbf{NJ}$ is still open. For a survey on the results and the applications of the propositional identity problem, see \cite{IsoSurvey,IsomorphismOfTypes}.

\section{Higher-order Theories} \label{SectionTheories}
In the previous sections, we used categorical machinery to formalize a proof system over the \emph{propositional language} as a structured category. While these investigations are intriguing, one must acknowledge that mathematics extends far beyond this restrictive language to form propositions about various \emph{types} of entities, such as natural numbers, real numbers, points, lines, and more. It also addresses higher-order entities constructed from these \emph{basic} entities. For example, we may consider functions on natural numbers, functions on those functions, and so forth. To formalize this richer language, we need to extend the propositional language with \emph{types} over which quantifiers operate and to include \emph{higher types} that encompass functions between these types. In this section, we will discuss these higher-order languages in general and introduce a specific language for discussing natural numbers.

Let $\{T_i\}_{i \in I}$ be a family of primitive symbols. We define the set of all \emph{types} inductively as follows: each $T_i$ is a type, for any $i \in I$, and for any two types $\sigma$ and $\tau$, the expressions $\sigma \times \tau$ and $\sigma \to \tau$ are also types. The basic types represent the fundamental sets of entities we work with, while the operations $\sigma \times \tau$ and $\sigma \to \tau$ formalize the cartesian product of $\sigma$ and $\tau$, and the set of all functions from $\sigma$ to $\tau$, respectively.

Next, we need to introduce \emph{terms} as names for entities within the types. First, consider a family of function symbols $\{f^{\tau}_j(x_1^{\sigma_1}, x_2^{\sigma_2}, \ldots, x_{n_j}^{\sigma_{n_j}})\}_{j \in J}$, where $\sigma_1, \sigma_2, \ldots, \sigma_{n_j}$, and $\tau$ are types. The function symbol $f^{\tau}_j$ is intended to formalize a basic function of arity $n_j$, with inputs of types $\sigma_1, \sigma_2, \ldots, \sigma_{n_j}$ and output of type $\tau$. Now consider:

\begin{itemize}
\item 
Infinitely many variables in the form $x^{\sigma}$, for any type $\sigma$. 
    \item
The function symbol $\mathbf{p}^{\sigma \times \tau}(x^{\sigma}, y^{\tau})$, for any types $\sigma$ and $\tau$. The intended meaning of $\mathbf{p}^{\sigma \times \tau}(x^{\sigma}, y^{\tau})$ is the pair of $x^{\sigma}$ and $y^{\tau}$.
\item 
The function symbols $\mathbf{p_0}^{\sigma}(x^{\sigma \times \tau})$ and $\mathbf{p_1}^{\tau}(x^{\sigma \times \tau})$, for any types $\sigma$ and $\tau$. The intended meaning of $\mathbf{p_0}^{\sigma}(x^{\sigma \times \tau})$ and $\mathbf{p_1}^{\tau}(x^{\sigma \times \tau})$ are the projections on the first and second components, respectively.
\end{itemize}
Now, define \emph{terms} inductively by:
\begin{itemize}
\item 
any variable of type $\sigma$ is a term of type $\sigma$,
\item
for any function symbol $f^{\tau}(x_1^{\sigma_1}, x_2^{\sigma_2}, \ldots, x_{n}^{\sigma_{n}})$, (including $\mathbf{p}^{\sigma \times \tau}(x^{\sigma}, y^{\tau})$, $\mathbf{p_0}^{\sigma}(x^{\sigma \times \tau})$ and $\mathbf{p_1}^{\tau}(x^{\sigma \times \tau})$), if $\{t_k\}_{k=1}^{n}$ are terms of types $\{\sigma_k\}_{k=1}^{n}$, then $f^{\tau}(t_1, t_2, \ldots, t_{n})$ is a term of type $\tau$,
    \item 
$\lambda x^{\sigma}. t(x^{\sigma})$ is a term of type $\sigma \to \tau$, for any types $\sigma$ and $\tau$ and any term $t(x^{\sigma})$ of type $\tau$,
\item 
$ts$ is a term of type $\tau$, for any terms $t$ and $s$ of types $\sigma \to \tau$ and $\sigma$, respectively.
\end{itemize}
The intended meaning of $\lambda x^{\sigma}. t(x^{\sigma})$ is a function that maps $x^{\sigma}$ to $t(x^{\sigma})$ and $sr$ means the application of the function $s$ on the input $r$.

Finally, consider a family of predicate symbols $\{P_k(x_1^{\sigma_1}, x_2^{\sigma_2}, \ldots, x_{n_k}^{\sigma_{n_k}})\}_{k \in K}$. Then, define \emph{formulas} inductively in the following way:
\begin{itemize}
\item 
$\top$ and $\bot$ are formulas,
    \item 
$t_1=_{\sigma} t_2$ is a formula, if both $t_1$ and $t_2$ are terms of type $\sigma$, 
    \item 
$P(t_1, t_2, \ldots, t_{n})$ is a formula, if $\{t_r\}_{r=1}^n$  are terms of types $\{\sigma_r\}_{r=1}^n$, for any predicate symbol $P(x_1^{\sigma_1}, x_2^{\sigma_2}, \ldots, x_{n}^{\sigma_{n}})$,
\item 
if $A$ and $B$ are formulas, $A \wedge B$, $A \vee B$, $A \to B$ are formulas,
\item 
if $A(x^{\sigma})$ is a formula, then $\forall x^{\sigma}A(x^{\sigma})$ and $\exists x^{\sigma}A(x^{\sigma})$ are formulas.
\end{itemize}
As it is clear from the definition, to specify a higher-order language $\mathcal{L}$, it is enough to specify its basic types and its families of function and predicate symbols. We denoted these three families by $B(\mathcal{L})$, $F(\mathcal{L})$ and $P(\mathcal{L})$, respectively. We also denote the set of all types of $\mathcal{L}$ by $T(\mathcal{L})$.

Define the logic $\mathsf{IQC}^{\omega}$ over the language $\mathcal{L}$ as the theory consisting of the usual axioms of multi-sorted first-order logic with equality plus the following defining axioms:  
\begin{itemize}
    \item
$\mathbf{p}^{\sigma \times \tau}(\mathbf{p_0}^{\sigma}(x^{\sigma \times \tau}), \mathbf{p_1}^{\tau}(x^{\sigma \times \tau})) =_{\sigma \times \tau} x^{\sigma \times \tau}$, $\mathbf{p_0}^{\sigma}(\mathbf{p}^{\sigma \times \tau}(x^{\sigma}, y^{\tau})) =_{\sigma} x^{\sigma}$, and $\mathbf{p_1}^{\tau}(\mathbf{p}^{\sigma \times \tau}(x^{\sigma}, y^{\tau})) =_{\tau} y^{\tau}$,
\item 
$(\lambda x^{\sigma}. t(x^{\sigma})) y^{\sigma}=_{\tau} t(y^{\sigma})$, where $t(x^{\sigma})$ is of type $\tau$ and $t(y^{\sigma})$ is the substitution of $y^{\sigma}$ in $t(x^{\sigma})$ for $x^{\sigma}$. 
\end{itemize}
There are two non-logical axioms of interest in this chapter.  The first is the axiom $\lambda y^{\sigma}. (x^{\sigma \to \tau} \cdot y^{\sigma})=x^{\sigma \to \tau}$ or equivalently the \emph{axiom of extensionality}
\[
\mathsf{EXT}: \quad \forall f^{\sigma \to \tau} g^{\sigma \to \tau} [\forall x^\sigma
(fx =_{\tau} gx) \to (f =_{\sigma \to \tau} g)],
\]
stating that a function is uniquely determined by its action. The second is the \emph{axiom of choice}
\[
\mathsf{AC}: \quad \forall x^\sigma \exists y^\tau A(x, y) \to \exists f^{\sigma \to \tau} \forall x^\sigma A(x, fx),
\]
stating that if for any $x$ of type $\sigma$, there is a $y$ of type $\tau$ satisfying $A(x, y)$, one can find a function $f$ of type $\sigma \to \tau$ to choose such a $y$ of type $\tau$, for any $x$ of type $\sigma$.

\begin{rem}
Here are some remarks. First, when it is clear from the context, we omit the types, including in $x^{\sigma}$, function and predicate symbols, and $=_{\sigma}$. Second, if we eliminate the type constructors $\times$ and $\to$, the function symbols $\mathbf{p}$, $\mathbf{p_0}$, $\mathbf{p_1}$, and the term constructors $\lambda$ and application, along with their corresponding equalities, the result is multi-sorted first-order logic, denoted by $\mathsf{IQC}$. Third, sometimes the zero type $0$ and the sum type $+$ are added to represent the empty set and disjoint union, respectively. However, we will not include these in this chapter to keep the presentation simpler. Fourth, the language can be restricted to its fragments, as we did for the propositional language. One fragment of interest is the one without $\bot$ and $\vee$, as it simplifies the treatment of arithmetic as we will see later.
\end{rem}

There is a canonical classical interpretation of higher-order languages. It is enough to assign a non-empty set to any basic type in $B(\mathcal{L})$ and interpret the type $\sigma$ by the set $M^{\sigma}$ recursively defined by $M^{\sigma \times \tau}=M^{\sigma} \times M^{\tau}$ and $M^{\sigma \to \tau}=[M^{\sigma}, M^{\tau}]$, where $\times$ and $[-, -]$ are the cartesian product and the exponentiation in $\mathbf{Set}$. Then, we must assign functions to the function symbols in $F(\mathcal{L})$ respecting the types and extend the interpretation to all terms by composition. Note that the function symbols $\mathbf{p}$, $\mathbf{p_0}$ and $\mathbf{p_1}$ must be interpreted as the pairing functions and the projections, respectively, $\lambda x^{\sigma}. t(x)$ must be read as a function mapping $a \in M^{\sigma}$ to $t(a) \in M^{\tau}$, where $\tau$ is the type of $t(x)$ and $ts$ must be read as the usual application function. Finally, we must assign relations to the predicate symbols in $P(\mathcal{L})$ respecting the types. It is easy to see that $\mathsf{IQC^{\omega}+EXT+AC}$ is valid under any such interpretation.

For a constructive interpretation of a higher-order theory, we offer an informal interpretation here, which will be formalized later in Section \ref{SecRealizability}. We start with an informal concept of construction. To interpret types in a higher-order language, we change the category of sets in the classical interpretation to the category of \emph{constructed sets}, i.e., sets where all elements are constructed with \emph{constructed functions} serving as the morphisms. For interpreting formulas, we need to extend the BHK interpretation from Section \ref{SectionBHKInterpretation} to include equality and quantifiers. This extension can take various forms, and we outline the one that we plan to formalize later:
\begin{itemize}
\item 
a proof of $t=s$ is a verification that $t$ and $s$ are equal,
    \item 
a proof of $\forall x^{\sigma} A(x)$ is a construction that transforms any construction of $a$ to a proof of $A(a)$, for any $a$ of type $\sigma$,
    \item 
a proof of $\exists x^{\sigma} A(x)$ is a pair of a construction of an element $a$ of the type $\sigma$ and a proof of $A(a)$. 
\end{itemize}
Even at this informal level, we can verify the provability of some axioms. It is straightforward to see that every theorem of $\mathsf{IQC}^{\omega}$ is provable through this BHK interpretation. The axiom $\mathsf{EXT}$ has a proof under the BHK interpretation because if we have a proof of $\forall x^{\sigma} (fx=gx)$, then we reach a proof of $fx=gx$ for any $x$ of type $\sigma$, which implies $f=g$, thereby establishing its provability. For the axiom of choice, we need to make an additional assumption. Let us assume that each element in $\sigma$ has a unique construction. Then, if we have a construction for $\forall x^{\sigma} \exists y^{\tau} A(x, y)$, it provides an element $b$ of type $\tau$ along with its construction for any construction of any element $a$ of type $\sigma$. Since constructions for the elements of type $\sigma$ are unique, we can ignore the dependence on the construction of $a$. This effectively gives us a function $f: \sigma \to \tau$ by mapping $a$'s to $b$'s. Note that the function $f$ is induced by a construction and hence must be interpreted as a constructed function.

\subsection{Constructive Theories of Arithmetic}
\label{SecArithmetic}
In this subsection, we will introduce some theories of constructive arithmetic.
Let us start with the simpler first-order case. Fix the language $\mathcal{L}_{\mathsf{HA}}=\{0, S, +, \cdot\}$ and assume that the language does not have the connectives $\vee$ and $\bot$. This is no restriction on the theory as it is known that the propositions $S(0)=0$ and $\exists z [(z=0 \to A) \wedge (z \neq 0 \to B)]$ can play the role of $\bot$ and $A \vee B$, respectively \cite{vanDalen}. 
Define \emph{Heyting (resp. Peano) arithmetic}, denoted by $\mathsf{HA}$ (resp. $\mathsf{PA}$) as the intuitionistic (resp. classical) logic with the equality extended by the axioms:
\begin{itemize}
\item
$\forall x (S(x) \neq 0)$,
\item
$\forall xy (S(x)=S(y) \to x=y)$,
\item
$\forall x (x+0 =x)$,
\item
$\forall xy (x+S(y)=S(x+y))$,
\item
$\forall x (x \cdot 0 =0)$,
\item
$\forall xy (x \cdot S(y)=x \cdot y+x)$,
\item 
$A(0) \wedge \forall x (A(x) \to A(S(x))) \to \forall x A(x)$, for any formula $A(x)$.
\end{itemize}
It is clear that the set of natural numbers with the canonical interpretation of the language $\mathcal{L}_\mathsf{HA}$ is a model of $\mathsf{PA}$. This model is called the \emph{standard model} and will be denoted by $\mathfrak{N}$.
Here are some basic facts about the theories $\mathsf{HA}$ and $\mathsf{PA}$. First, defining the relation $x \leq y$ as $\exists z (x+z=y)$, it is possible to prove the usual properties of the ordering in $\mathsf{HA}$ such as its totality and the fact that the operations $+$ and $\cdot$ respect the ordering. 
Any quantifier in the form $\forall x (x \leq t \to A(x))$ and $\exists x (x \leq t \wedge A(x))$, for some term $t$ is called a \emph{bounded quantifier} and will be abbreviated by $\forall x \leq t \, A(x)$ and $\exists x \leq t \, A(x)$, respectively.
A formula is called \emph{bounded} if all of its quantifiers can be replaced by a bounded quantifier up to provability in $\mathsf{HA}$.
Any bounded formula $A(\bar{x})$ is decidable in $\mathsf{HA}$, i.e., $\mathsf{HA} \vdash \forall \bar{x} (A(\bar{x}) \vee \neg A(\bar{x}))$.
The theory $\mathsf{PA}$ is conservative over $\mathsf{HA}$ for formulas in the form $\forall \bar{x} \exists \bar{y} A(\bar{x}, \bar{y})$, where $A(\bar{x}, \bar{y})$ is bounded.\footnote{This means that $\mathsf{PA} \vdash \forall \bar{x} \exists \bar{y} A(\bar{x}, \bar{y})$ implies $\mathsf{HA} \vdash \forall \bar{x} \exists \bar{y} A(\bar{x}, \bar{y})$, if $A(\bar{x}, \bar{y})$ is bounded.} 
A formula is called \emph{$\exists$-free} if it has no occurrence of an existential quantifier. It is easy to see that any bounded formula is equivalent to an $\exists$-free formula in $\mathsf{HA}$. To prove, it is just enough to change any $\exists x \leq t \, A(x)$ to $\neg \forall x \leq t \, \neg A(x)$. These two formulas are equivalent as bounded formulas are decidable in $\mathsf{HA}$. It is possible to formalize some basic recursion theory in $\mathsf{HA}$. More precisely, there is a $\exists$-free formula $T(e, x, w, y)$ such that in the standard model it means that the Turing machine with the code $e$ halts on the input $x$ with the computation $w$ and the output $y$.

For the higher-order version of Heyting arithmetic, denoted by $\mathsf{HA}^{\omega}$, use the higher-order language with the basic type $N$, the function symbols $\{0, S(x^N)\}$ of type $N$, the function symbols $\mathbf{R}^{\sigma}(x^{\sigma}, y^{N \times \sigma \to \sigma}, z^N)$, for any type $\sigma$ and add no predicate symbol other than equality. Define $\mathsf{HA}^{\omega}$ as $\mathsf{IQC}^{\omega}$ plus the following axioms:
\begin{itemize}
\item
$\mathbf{R}^{\sigma}(x, y, 0) = x$ and $\mathbf{R}^{\sigma}(x, y, S(z)) = y \mathbf{p}(z, \mathbf{R}^{\sigma}(x, y, z))$,
\item 
$S(0) \neq 0$ and $S(x)=S(y) \to x=y$,
\item 
$A(0) \wedge \forall x^N (A(x) \to A(S(x)) \to \forall x^N A(x)$, for any formula $A(x)$ in the extended language.
\end{itemize}
We also define the extension $\mathsf{HA}^{\omega}_2$ of $\mathsf{HA}^{\omega}$ that we use in a restricted way later. First, add a new basic type $2$, two constants $0_2$ and $1_2$ of type $2$ and for any type $\sigma$, a function symbol $c^{\sigma}(x^{\sigma}, y^{\sigma}, z^{2})$. The intended meaning of $2$ is the set $\{0, 1\}$ and $c^{\sigma}(x^{\sigma}, y^{\sigma}, z^{2})$ is the function that reads the elements $x^{\sigma}$ and $y^{\sigma}$ from $\sigma$ and $z^{2}$ from $\{0, 1\}$ and returns $x^{\sigma}$ if $z=0$ and $y^{\sigma}$ if $z=1$. Extend the theory $\mathsf{HA}^{\omega}$ by the equations
$c^{\sigma}(x, y, 0_2)=x$, $c^{\sigma}(x, y, 1_2)=y$ and $c^{\sigma}(w0_2, w1_2, z)=wz$, for any $w$ of type $2 \to \sigma$. The resulting theory is denoted by $\mathsf{HA}^{\omega}_2$. Notice that the induction axiom operates over the whole new language including the formulas with terms of type $2$.

It is clear that $\mathsf{HA}$ can be interpreted inside $\mathsf{HA}^{\omega}$, reading any variable $x$ as a variable of the type $N$, interpreting $0$ and $S$ as the corresponding constants and defining the addition and multiplication by the function symbol $\mathbf{R}$ using the usual recursive definitions. For more on the definition and the basic properties of the defined arithmetical theories, see \cite{vanDalenII}.

Similar to $\mathsf{HA}$, the theories $\mathsf{HA}^{\omega}$ and $\mathsf{HA}^{\omega}_2$ have a canonical classical interpretation that we also call the \emph{standard model} and denote by $\mathfrak{N}$. To define the model, it is enough to map the basic type $N$ to the set of natural numbers $\mathbb{N}$ and read terms $0$ and $S(x^N)$ as the zero element and the successor function on $\mathbb{N}$. For $\mathbf{R}$, we must use the canonical function defined by the recursive definition the axioms for $\mathbf{R}$ suggests. Later in Section \ref{SecRealizabilityForArith}, we will provide a more constructive interpretation using the BHK interpretation.

\subsection{Non-classical Axioms and Consistency}

We began this chapter with a brief history of Hilbert's program and its consistency problems, which motivated the formal and precise investigation of the concept of proof. This investigation ultimately led to the development of categorical proof theory, providing concrete and distinct models for proof systems, as demonstrated in the previous sections. To close the circle, it would be interesting to apply categorical proof theory to prove the consistency of mathematical theories in general, and arithmetical theories in particular. 

For the usual arithmetical theories such as $\mathsf{HA}$ and $\mathsf{PA}$, one might argue that since mathematics is no longer in crisis, the urgency of addressing consistency problems has diminished. Today, we consider the standard model to be both secure and intuitively acceptable, making it reasonable to assert the consistency of arithmetic based on the existence of this model. However, the philosophical and mathematical discourse surrounding consistency problems is more nuanced than this straightforward objection suggests, although the objection does hold some validity.
To address the objection, it is important to note that proving consistency in the intuitive manner described above is essentially verifying the soundness of the theory with respect to our classical image. Constructive mathematics, however, encompasses various alternative theories that describe alternative worlds inconsistent with our classical understanding. Thus, proving consistency is not merely a philosophical endeavor but involves establishing the plausibility of such hypothetical worlds, where classical intuition may not be as effective as we assume. To illustrate this, we provide two examples of alternative worlds within constructive mathematics. 

The primary claim of constructive mathematics is that all mathematical entities are constructions. Within this framework, a challenge arises when dealing with infinitary objects, such as functions over natural numbers or infinite sequences of zeros and ones, which exceed the finite and bounded capacity of human construction. There are numerous approaches to handling these objects, with two extreme ends of the spectrum: the \emph{law-like} and \emph{lawless} approaches.

In the law-like approach, all infinitary objects are required to be law-like, meaning that despite their infinitary nature, they are generated by a finite set of rules or laws. The natural candidate for such a finite law is an algorithm. Consequently, this interpretation of constructive mathematics confines itself to the computable world, where everything, including functions on numbers, must be computable. This variant of constructive mathematics is known as \emph{Russian constructivism}. A typical axiom in Russian constructivism is the \emph{Church-Turing thesis}.
\[
\mathsf{CT}: \quad \forall f^{N \to N} \exists e^N \forall x^N \exists w^N \exists y^N [T(e, x, w, y) \wedge (fx=y)],
\]
stating that \emph{all total functions on natural numbers are computable}. More precisely, it asserts that for any function $f: \mathbb{N} \to \mathbb{N}$, there exists a Turing machine with the code $e$ that computes $f$, meaning that for any input $x$, there is a computation $w$ of $e$ on $x$ producing the output $y = f(x)$.
Now, regarding our consistency problem, we might ask whether $\mathsf{HA}^{\omega} + \mathsf{CT}$ is consistent. The answer is unclear because, in the usual classical framework, some functions, such as the characteristic function for the halting predicate, are uncomputable. Therefore, the theory $\mathsf{HA}^{\omega} + \mathsf{CT}$ is inconsistent with our classical understanding of the world.

At the opposite end of the spectrum is the assumption that all infinite objects are \emph{lawless}. For example, when considering functions over natural numbers, we may work with any function whose values are chosen in a lawless manner, one after another, over time. At any given moment, we can only comprehend a finite number of values in our finite minds, yet we can still discuss potentially infinite sequences of numbers constructed in this lawless fashion. One might argue that even if we include these mathematical objects in our constructive mathematics, we cannot derive any conclusions about them due to their lawless and hence completely unpredictable nature.
Brouwer's insightful idea was that if all functions are lawless, the only meaningful statements we can make about them must involve only finitely many values of the function. For instance, stating that the first element of a lawless sequence $\alpha$ is five is meaningful, whereas claiming that all its elements are zero is not. In this framework, we can accommodate lawless sequences if we restrict ourselves to discussing the \emph{continuous} properties of the functions where continuity ensures that the property is actually about a finite part of the input. This type of continuity is a defining characteristic of what is commonly referred to as \emph{Brouwerian constructivism}. A typical axiom in this form of constructivism is the \emph{continuity principle}\footnote{This axiom is formally called the \emph{weak continuity principle for numbers}. Hence, the abbreviation $\WCN$.}:
\[
\WCN: \quad \forall f^{N^N \to N} \forall \alpha^{N \to N} \exists x^N \forall \beta^{N \to N} (\alpha=_{x} \beta \to f\alpha=f\beta),
\]
where $\alpha=_x \beta$ abbreviates $\forall y^N (y < x \to \alpha y = \beta y)$, stating that the \emph{value of a total numeral function at $\alpha$ depends only on finitely many values of $\alpha$}. The $x$ bounding the cardinality of these  values is called the \emph{continuity module}. Similar to what we encountered with $\mathsf{CT}$, this axiom is also inconsistent with our usual classical image. It asserts that all total functions $f: \mathbb{N}^{\mathbb{N}} \to \mathbb{N}$ are continuous. However, using the axiom of the excluded middle, we can observe that the characteristic function of the singleton $\{\alpha_0\}$, where $\alpha_0$ is the constant zero function, is total but not continuous. In other words, the value of $f(\alpha)$ depends on all values of $\alpha(n)$, not just a finite number of them. Therefore, we face a consistency question again, this time about the theory $\mathsf{HA}^{\omega} + \WCN$. For more on non-classical constructive theories, see \cite{vanDalen,vanDalenII}.

One natural strategy for proving such consistency claims is providing a model—an alternative world where the anti-classical axioms hold. To construct these models, we use categories of constructed sets, interpreting constructions in a way that aligns with the consistency problem at hand. We then apply the BHK interpretation to validate the theory and demonstrate its consistency. Consider, for instance, the case of $\mathsf{HA}^{\omega} + \mathsf{CT}$. Here, we naturally expect to interpret the informal notion of construction as an algorithm or a computable function. The category of constructed sets then provides a \emph{computable world}, where everything, including all functions on natural numbers, is computable. This setup is precisely what we need to support the Church-Turing thesis, and by leveraging this observation, we aim to show that the BHK interpretation validates $\mathsf{CT}$.

However, under the BHK interpretation, the axiom
\[
\mathsf{CT}: \quad \forall f^{N \to N} \exists e^N \forall x^N \exists w^N \exists y^N [T(e, x, w, y) \wedge (fx=y)],
\]
asserts more than just “all total functions over natural numbers are computable.” It also requires that the process of finding the code $e$ for a function $f$ from its construction $\mathtt{f}$ be computable. Later, we will explain how to weaken $\mathsf{CT}$ to align with what we originally intended. For now, let us examine the implications of working with this current version.

As we saw above, we have two options for interpreting constructions: either as an algorithm or as a computable function, and the choice significantly impacts the validation of $\mathsf{CT}$. If we interpret $\mathtt{f}$ as a computable function, i.e., as $f$ itself, then the requirement to  recursively find a code for it becomes impossible, as we will explain later. On the other hand, if we interpret $\mathtt{f}$ as an algorithm for $f$, then $\mathsf{CT}$ simply asserts that if we have an algorithm $\mathtt{f}$ for $f$, we can find a code $e$ for $f$ in a computable manner. This is straightforward, as $e$ can simply be set as the algorithm $\mathtt{f}$ itself.

Let us now investigate the axiom
\[
\WCN: \quad \forall f^{N^N \to N} \forall \alpha^{N \to N} \exists x^N \forall \beta^{N \to N} (\alpha=_{x} \beta \to f\alpha=f\beta).
\]
from the same lens. Here, we must interpret construction as a continuous entity to get a \emph{continuous world}. Similar to $\mathsf{CT}$, the axiom not only asserts that $f(\alpha)$ depends on $x$ many initial values of $\alpha$, but also demands the continuity of the process of determining $x$ based on the constructions of $f$ and $\alpha$. When the constructions of $f$ and $\alpha$ are read as the functions themselves, this requirement becomes impossible again. However, there exists a notion of an ``algorithm" that operates on infinitary data. By interpreting the constructions of $f$ and $\alpha$ as such algorithms, one can validate the axiom, as it is possible to determine $x$ from the algorithms \cite{RealTr}.

In general, there are essentially two ways to interpret a construction: \emph{intensionally} or \emph{extensionally}. The intensional approach treats constructions as \emph{instructions} to be executed, while the extensional approach views them as \emph{functions in the traditional sense}. For example, in the above examples, the intensional interpretation considers a construction as an \emph{algorithm}, whereas the extensional interpretation views it as a computable or continuous \emph{function}.

In this chapter, we adopt the \emph{extensional} approach. A key point about this approach is that it typically tends to validate the axioms $\mathsf{EXT}$ and $\mathsf{AC}$. As mentioned in the last paragraph prior to Subsection \ref{SecArithmetic}, the validity of $\mathsf{EXT}$ is straightforward. Regarding $\mathsf{AC}$, the extensional interpretation usually treats a function as its own construction, implying that constructions are generally unique. This uniqueness leads to the validation of $\mathsf{AC}$. In contrast, under the intensional approach, a function can correspond to many different algorithms, which influences the validation of the axiom of choice.

In the following theorem, we will show that $\mathsf{CT}$ and $\WCN$ are inconsistent with $\mathsf{EXT}$ and $\mathsf{AC}$. Consequently, we demonstrate what we previously promised: that the extensional BHK interpretation cannot validate the axioms $\mathsf{CT}$ and $\WCN$.

\begin{thm} The theory
$\mathsf{HA^{\omega}+EXT+AC}$ is inconsistent with either of the axioms $\mathsf{CT}$ or $\WCN$. 
\end{thm}
\begin{proof}
We argue informally inside the theory $\mathsf{HA}^{\omega}+\mathsf{EXT}+\mathsf{AC}$. For $\mathsf{CT}$, use $\mathsf{AC}$ on $\mathsf{CT}$ to get a functional $F: N^N \to N$ that maps a total function $f: N \to N$ to one of its codes. We claim that $f=g$ iff $F(f)=F(g)$. One direction is clear. For the other, if $F(f)=F(g)$, as $F(f)$ is the code for $f$, we have $fx=F(f) \cdot x=F(g) \cdot x=gx$, where $e \cdot x$ is the output of applying the machine with the code $e$ on the input $x$. Note that the predicate $e \cdot x=y$ is definable in arithmetic by the definition $\exists w^N T(e, x, w, y)$. Therefore, $fx=gx$, for any $x$ of type $N$ and thus $f=g$, by $\mathsf{EXT}$. Now, using the equivalence between $f=g$ and $F(f)=F(g)$ and the decidability of the equality over the type $N$ in $\mathsf{HA}^{\omega}$, we have the decidability of the equality over $N^N$, i.e., 
\[
\forall f^{N \to N} g^{N \to N} [(f=g) \vee \neg (f = g)]. \quad \quad (*)
\]
Set $x$ as a parameter and define $f$ as the map that reads $w$ and outputs the truth value of the predicate $T(x, x, w, 1)$ and $g$ as the constant zero function. This is possible as $T$ is decidable and hence $f$ is total. Note that $f=g$ iff $\forall w^N \neg T(x, x, w, 1)$. Therefore,
by $(*)$, we have 
\[
\forall x^N (\forall w^N \neg T(x, x, w, 1) \vee \neg \forall w^N \neg T(x, x, w, 1)).
\]
Using this decidability, we can define a total characteristic function $h: N \to N$ that reads $x$ and outputs $1$ if $\forall w^N \neg T(x, x, w, 1)$, and $0$ if $\neg \forall w^N \neg T(x, x, w, 1)$. By $\mathsf{CT}$, there is a code for the function $h$.
Call it $e$. Then, either $e \cdot e=1$ or $e \cdot e=0$. In the first case, we have $e \cdot e=1$ which means that $\exists w^N T(e, e, w, 1)$. However, we are in the first case. Hence, we must have $\forall w^N \neg T(e, e, w, 1)$ which is impossible. If $e \cdot e=0$, then as the value of $e \cdot e$ is unique, it is not $1$ and hence it is easy to prove that $\forall w^N \neg T(e, e, w, 1)$. However, we are in the second case. Hence, $\neg \forall w^N \neg T(e, e, w, 1)$ which is again a contradiction.

For $\WCN$, we will reproduce the proof provided in \cite{Escardo}, only altering the type-theoretic language to our arithmetical language. Since functions on natural numbers and sequences of natural numbers are essentially the same entities, we will use them interchangeably to simplify our presentation. Let $0^nk^{\omega}$ denote the
sequence of $n$ many zeros followed by infinitely many $k$'s, for any $k$ and $n$ of type $N$. Then, notice that $0^nk^{\omega} =_n 0^{\omega}$
and $0^nk^{\omega}=_{(n+1)}k$.
By applying $\mathsf{AC}$ on $\WCN$, we have a functional $F$ such that for any $f^{N^N \to N}$, $\alpha^{N \to N}$ and $\beta^{N \to N}$ if $\alpha =_{F(f, \alpha)} \beta$ then $f\alpha = f\beta$. By substituting $\alpha= 0^{\omega}$, we have
\[
\forall f^{N^N \to N} \forall \beta^{N \to N} [0^{\omega} =_{M(f)} \beta \to f0^{\omega} = f\beta], \quad \quad (**)
\]
where $M(f)=F(f, 0^{\omega})$.
Let
$m = M(\lambda \gamma.0)$ and define $f^{N^N \to N}$ by
$f\beta = M(\lambda \gamma.\beta(\gamma m))$.
Note that $0^{\omega}(\gamma m) = 0$, for the variable $\gamma^{N \to N}$. Hence, by $\mathsf{EXT}$, we have $\lambda \gamma.0^{\omega}(\gamma m) = \lambda \gamma.0$. Thus, by definition, we have $f(0^{\omega}) = M(\lambda \gamma.0^{\omega}(\gamma m)) = M(\lambda \gamma.0) = m$. Therefore, by $(**)$, we have 
\[
\forall \beta^{N \to N} [0^{\omega} =_{M(f)} \beta \to m = f\beta]. \quad \quad (\dagger)
\]
On the other hand, for any $\beta^{N \to N}$, we can define $g_{\beta}=\lambda \gamma. \beta(\gamma m)$. 
Note that $M(g_{\beta})=f(\beta)$, by the definition of $f$. Moreover, notice that $g_{\beta}0^{\omega}=\beta(0)$ and $g_{\beta}(\gamma)=\beta(\gamma m)$.
Applying $(**)$ on $g_{\beta}$, we have
\[
\forall \gamma^{N \to N} [0^{\omega} =_{f\beta} \gamma \to \beta 0 = \beta (\gamma m)]. \quad \quad (\dagger \dagger)
\]
Substituting $\beta = 0^{M(f)+1}1^{\omega}$ in $(\dagger)$ and using the fact that $0^{\omega} =_{M(f)} \beta$, we have $f(\beta) = m$. Using the same $\beta=0^{M(f)+1}1^{\omega}$ in $(\dagger\dagger)$ and the equality $f(\beta) = m$, we have
\[
\forall \gamma^{N \to N} [0^{\omega} =_m \gamma \to \beta 0 = \beta (\gamma m)]. 
\]
Substituting $\gamma = 0^m(M(f) + 1)^{\omega}$, we have $0^{\omega} =_m \gamma$. Hence, $0 = \beta 0 = \beta (\gamma m) = \beta (M(f) + 1) = 1$ which is a contradiction.
\end{proof}

As a concluding part of this subsection, we will explain how to weaken the axioms $\mathsf{CT}$ and $\WCN$ to reflect our original intent. Our approach is to modify the existential quantifiers to require only the \emph{mere} existence of the code and the continuity module, rather than their constructibility via a recursive or continuous process. In this revised formulation, $\mathsf{CT}$ aligns with our original understanding, stating that for any computable function $f$, there exists a code $e$ for $f$, with no requirement for the process of finding this code to be computable. Similarly, $\WCN$ would assert that for any function $f$ and any $\alpha$, there exists an $x$ such that the value of $f(\alpha)$ depends only on the first $x$ values of $\alpha$, and the process of finding $x$ need not be continuous.

To achieve this, we need a method to \emph{collapse} proofs of a proposition into a single proof, preserving only the fact of provability while discarding additional details. This process is known as \emph{propositional truncation} \cite{PropAsType,HoTT}. However, for simplicity, we will use a similar mechanism adapted to the specific BHK interpretation employed in this chapter.

First, note that if we formalize the BHK interpretation in a set-like manner without any temporal structure, it becomes apparent that for any sentence $A$, if there is no proof of $A$, then \emph{any} construction can serve as a proof of $\neg A$, since there is no proof of $A$ to transform into a proof of $\bot$. Consequently, if $A$ does have a proof, then $\neg A$ cannot have a proof and hence \emph{any} construction serves as a proof of $\neg \neg A$, and if $A$ is not provable, then $\neg A$ is provable implying that $\neg \neg A$ has \emph{no} proof. Thus, a proof of $\neg \neg A$ conveys no information beyond the mere provability of $A$, meaning that $\neg \neg$ can function as a collapsing mechanism. Using this collapsing machine, the weaker versions of the axioms $\mathsf{CT}$ and $\WCN$ can be formulated as: 
\[
\mathsf{CT}_{\neg \neg}: \quad \forall f^{N \to N} \neg \neg \exists e^N \forall x^N \exists w^N \exists y^N [T(e, x, w, y) \wedge (fx=y)],
\]
\[
\WCN_{\neg \neg}: \quad \forall f^{N^N \to N} \forall \alpha^{N \to N} \neg \neg \exists x^N \forall \beta^{N \to N} (\alpha=_{x} \beta \to f\alpha=f\beta).
\]
where the existence of the code or the continuity module $x$ is asserted, but the requirement for finding them in a recursive or continuous manner is relaxed. 
Note that even these weakened versions remain in contradiction with our classical image, as double negation has no impact in classical mathematics. In Section \ref{SecRealizabilityForArith}, we will formalize the arguments provided here to construct models for the theories $\mathsf{HA^{\omega} + EXT + AC} + \mathsf{CT}{\neg \neg}$ and $\mathsf{HA^{\omega} + EXT + AC} + \WCN{\neg \neg}$, thereby demonstrating their consistency.

\subsection{Information Extraction}

There is also another application of categorical proof theory that is relevant not only to the alternative theories discussed in the previous subsection but also to conventional theories compatible with our classical perspective. To illustrate this, consider a simple question: Does a proof of a proposition contain more information than just its mere truth?
The answer appears to be affirmative. For example, in the context of $\mathsf{HA}$, knowing a proof of a proposition of the form $\forall x \exists y A(x, y)$ tells us that for any $n \in \mathbb{N}$, there exists an $m \in \mathbb{N}$ such that $A(n, m)$ holds in the standard model. However, the proof might contain additional information. Specifically, it may be possible to \emph{extract} an algorithm from the proof that, given $n \in \mathbb{N}$, computes an $m \in \mathbb{N}$ such that $A(n, m)$.

The traditional black-and-white approach to proofs often overlooks the possibility of intermediate solutions. It either retains all proofs with their syntactical complexities intact or shifts to a semantic view, discarding the entire structure of a proof in favor of greater flexibility. The trade-off is that this approach loses all information contained within the proof, leaving only the bare truth of its conclusion.
While ignoring the proof structure can be useful in some contexts, it may be overly simplistic in others. Therefore, we may need to find a middle ground that balances flexibility with the retention of useful information. Specifically, developing \emph{models of proofs} that preserve certain aspects of the proof structure while allowing for some flexibility could offer a more nuanced approach. This way, we can maintain essential information while adapting the proofs for practical use.

This is precisely where categorical proof theory, and specifically the BHK interpretation, becomes highly valuable. For instance, a BHK proof of the sentence $\forall x^{\sigma} \exists y^{\tau} A(x, y)$ is a construction that maps any construction of any element $a$ of type $\sigma$ to a pair consisting of a construction of an element $b$ of type $\tau$ and a proof of $A(a, b)$. Thus, if the elements in types $\sigma$ and $\tau$ are constructions for themselves, having a BHK proof for the sentence $\forall x^{\sigma} \exists y^{\tau} A(x, y)$ enables us to find a method to compute $b$ from $a$. Note that the notion of construction is not predetermined; it can be tailored to suit specific needs. For example, if we choose to interpret constructions as computable functions, then for $\sigma = \tau = N$, a BHK proof of $\forall x^{N} \exists y^{N} A(x, y)$ yields a recursive function that takes a natural number $n \in \mathbb{N}$ and computes a natural number $m \in \mathbb{N}$ such that $A(n, m)$ holds. Similarly, if we choose to interpret constructions as continuous functions, then for $\sigma = N \to N$ and $\tau = N$, a BHK proof of $\forall x^{N \to N} \exists y^{N} A(x, y)$ provides a continuous function that, given a function $\alpha: \mathbb{N} \to \mathbb{N}$, computes a natural number $m \in \mathbb{N}$ such that $A(\alpha, m)$ holds. We will revisit these applications in Section \ref{SecRealizabilityForArith}.

\section{Realizability}\label{SecRealizability}
In the previous sections, we explained the categorical formalization of propositional deductive systems. The natural next step is to extend the formalization to the higher-order setting. However, this task is far more demanding than what we covered for the propositional case, and thus it goes beyond what we can include in this chapter.\footnote{To see how such generalized settings work, see \cite{BHK,HofmannSeman,Bart}.} To provide a taste of this, though, we will present a special yet powerful enough setting called \emph{realizability}, based on the BHK interpretation we discussed earlier.

Choosing this special setting has its own benefits. First, it provides the categorical presentation of one of the most powerful techniques in proof theory of the same name. Second, our setting uses less structure, which translates to gaining more insights into the theories. For instance, we will show that any computable function provably total in $\mathsf{HA}$ and $\mathsf{PA}$ is representable in any $\mathrm{CCC}$ with enough structure. In this section, we establish the general setting for an arbitrary higher-order language, and in Section \ref{SecRealizabilityForArith}, we will apply the machinery to arithmetical theories.

\subsection{Propositional Realizability}

To formalize the higher-order BHK interpretation, we follow the format explained in Section \ref{SectionBHKInterpretation} and extended in Section \ref{SectionTheories}. However, this time, we relax the temporal structure and change our notion of construction from set functions to maps in a $\mathrm{BCC}$. First, let us start with the propositional language again.

\begin{nota}
In the realizability parts of this chapter, namely the present section and Section \ref{SecRealizabilityForArith}, we use Typewriter font (as in $\mathtt{f}: \mathtt{N} \to \mathtt{N}$) for objects and maps in our $\mathrm{BCC}$ to distinguish them from set-theoretical objects and functions.
\end{nota}

\begin{dfn}
Let $\mathcal{C}$ be a $\mathrm{BCC}$. By a \emph{$\mathcal{C}$-interpretation} $\mathfrak{I}$ for the propositional language, we mean an assignment that maps each atom $p$ in the language to a pair $(|p|, R_p)$, where $|p|$ is an object in $\mathcal{C}$ and $R_p \subseteq \mathrm{Hom}_{\mathcal{C}}(\mathtt{1}, |p|)$ is a set.
\end{dfn}
By recursion, we extend $\mathfrak{I}$ to assign an object $|A|$ to any proposition $A$ in the following way:
\begin{itemize}
\item[$\bullet$]
$|\top|=\mathtt{1}$, $|\bot|=\mathtt{0}$,
\item[$\bullet$]
$|A \wedge B|=|A| \times |B|$,
$|A \vee B|=|A| + |B|$,
$|A \to B|= [|A|, |B|]$.
\end{itemize}
Then, inductively define the relation $\mathtt{r} \Vdash A$ between the map $ \mathtt{r} \in \mathrm{Hom}_{\mathcal{C}}(\mathtt{1}, |A|)$ and the proposition $A$ in the following way:
\begin{itemize}
\item 
$\mathtt{r} \Vdash p$ if $\mathtt{r} \in R_p$, for any $\mathtt{r}: \mathtt{1} \to |p|$,
\item
$!_{\mathtt{1}} \Vdash \top$ for the unique map $!_{\mathtt{1}}: \mathtt{1} \to \mathtt{1}$,
\item
$\mathtt{r} \nVdash \bot$ for any map $\mathtt{r}: \mathtt{1} \to \mathtt{0}$,
\item
$\mathtt{r} \Vdash A \wedge B$ iff $\mathtt{p_0r} \Vdash A$ and $\mathtt{p_1r} \Vdash B$, for any $\mathtt{r}: \mathtt{1} \to |A| \times |B|$, 
\item
$\mathtt{r} \Vdash A \vee B$ iff there exists $\mathtt{s}: \mathtt{1} \to |A|$ such that $\mathtt{r}=\mathtt{i_0s}$ and $\mathtt{s} \Vdash A$ or there exists $\mathtt{t}: \mathtt{1} \to |B|$ such that $\mathtt{r}=\mathtt{i_1t}$ and $\mathtt{t} \Vdash B$, for any $\mathtt{r}: \mathtt{1} \to |A| + |B|$, 
\item
$\mathtt{r} \Vdash A \to B$ iff for any $\mathtt{s}: \mathtt{1} \to |A|$, if $\mathtt{s} \Vdash A$ then $\mathtt{r} \cdot \mathtt{s} \Vdash B$,  for any $\mathtt{r}: \mathtt{1} \to [|A|, |B|]$.
\end{itemize}
For any \emph{finite} set $\Gamma$ of formulas, define $|\Gamma|$ as $\prod_{\gamma \in \Gamma}|\gamma|$. We use an arbitrary order on the formulas in $\Gamma$, and since the product is commutative up to an isomorphism, the chosen order is actually immaterial. Then, by $\Gamma \Vdash_{\mathfrak{I}} A$, we mean the existence of a map $\mathtt{r}: |\Gamma| \to |A|$ such that $\mathtt{r} \circ \mathtt{s} \Vdash A$, for any $\mathtt{s}: \mathtt{1} \to |\Gamma|$ satisfying $\mathtt{s} \Vdash \bigwedge \Gamma$. This $\mathtt{r}$ is called a \emph{realizer} for $\Gamma \Vdash_{\mathfrak{I}} A$. When $\Gamma$ is empty, $|\Gamma|=\mathtt{1}$, by definition and $\mathtt{r}: \mathtt{1} \to |A|$ is called a realizer for $A$.  When the $\mathcal{C}$-interpretation $\mathfrak{I}$ is clear from the context, we drop the subscript $\mathfrak{I}$ in $\Vdash_{\mathfrak{I}}$. Finally, by $\Gamma \Vdash_{\mathcal{C}} A$, we mean $\Gamma \Vdash_{\mathfrak{I}} A$, for any $\mathcal{C}$-interpretation $\mathfrak{I}$.

\begin{rem}
Notice that $\mathtt{r}: |\Gamma| \to |A|$ realizes $\Gamma \Vdash_{\mathfrak{I}} A$ iff for any sequence  $\{\mathtt{s}_{\gamma}: \mathtt{1} \to |\gamma|\}_{\gamma \in \Gamma}$ of maps such that $\mathtt{s}_\gamma \Vdash \gamma$, for any $\gamma \in \Gamma$, we have $\mathtt{r} \circ \langle \mathtt{s}_{\gamma} \rangle_{\gamma \in \Gamma} \Vdash A$. This presentation is easier to work with in practice.
\end{rem}

\begin{exam}
For any $\mathcal{C}$-interpretation $\mathfrak{I}$, the identity map $\mathtt{id}_{|A|}: |A| \to |A|$ realizes $A \Vdash A$ because if $\mathtt{r} \Vdash A$ then $\mathtt{id}_{|A|} \circ \mathtt{r}=\mathtt{r} \Vdash A$. The unique map $!_{|A|}: \mathtt{0} \to |A|$ realizes $\bot \Vdash A$. To prove that, we have to show that for any $\mathtt{r}: \mathtt{1} \to \mathtt{0}$, if $\mathtt{r} \Vdash \bot$ then $!_{|A|} \circ \mathtt{r} \Vdash A $. However, $\mathtt{r} \Vdash \bot$ is impossible, by definition. The unique map $!_{|A|}: |A| \to \mathtt{1}$ realizes $A \Vdash \top$ because if $\mathtt{r} \Vdash A$ then $!_{|A|} \circ \mathtt{r}: \mathtt{1} \to \mathtt{1}$. As $\mathtt{1}$ is terminal, we have $!_{|A|} \circ \mathtt{r}=!_{\mathtt{1}}$ and as $!_{\mathtt{1}} \Vdash \top$, we reach $!_{|A|} \circ \mathtt{r} \Vdash \top$.
\end{exam}

\begin{exam}
The map $\mathtt{p_0}: |A| \times |B| \to |A|$
realizes
$\{A, B\} \Vdash A$. Let $\mathtt{r}: \mathtt{1} \to |A \times B|=|A| \times |B|$. Then, $\mathtt{p_0} \circ \mathtt{r} \Vdash A$, by definition. Similarly, one can show that $\mathtt{i_1}: |B| \to |A|+|B|$ realizes $B \Vdash A \vee B$.
\end{exam}

\begin{exam}
If $C \Vdash A$ and $C \Vdash B$ then $C \Vdash A \wedge B$. Let $\mathtt{r}: |C| \to |A|$ and $\mathtt{s}: |C| \to |B|$ be the realizers, respectively. Then, $\langle \mathtt{r}, \mathtt{s} \rangle : |C| \to |A| \times |B|$ realizes $C \Vdash A \wedge B$. To prove that, let $\mathtt{u}: \mathtt{1} \to |C|$ realize $C$. Then, by definition, $\mathtt{r \circ u} \Vdash A$ and $\mathtt{s \circ u} \Vdash B$. As $\langle \mathtt{r}, \mathtt{s} \rangle \circ \mathtt{u}=\langle \mathtt{r \circ u}, \mathtt{s \circ u} \rangle$, we reach $\langle \mathtt{r}, \mathtt{s} \rangle \circ \mathtt{u} \Vdash A \wedge B$, by definition.
\end{exam}

\begin{exam}
If $A \Vdash C$ and $B \Vdash C$ then $A \vee B \Vdash C$. Let $\mathtt{r}: |A| \to |C|$ and $\mathtt{s}: |B| \to |C|$ be the realizers, respectively. Then, $(\mathtt{r}, \mathtt{s}): |A|+|B| \to |C|$ realizes $A \vee B \Vdash C$. To prove that, let $\mathtt{u}: \mathtt{1} \to |A|+|B|$ realize $A \vee B$. Then, by definition, either there is
$\mathtt{v}: \mathtt{1} \to |A|$ such that $\mathtt{u}=\mathtt{i_0v}$ and $\mathtt{v} \Vdash A$ or there is $\mathtt{w}: \mathtt{1} \to |B|$ such that $\mathtt{u}=\mathtt{i_1w}$ and $\mathtt{w} \Vdash B$. In the former case, we have $(\mathtt{r}, \mathtt{s}) \circ \mathtt{u}=(\mathtt{r}, \mathtt{s}) \circ \mathtt{i_0v}=\mathtt{r} \circ \mathtt{v} \Vdash C$ and in the latter case $(\mathtt{r}, \mathtt{s}) \circ \mathtt{u}=(\mathtt{r}, \mathtt{s}) \circ \mathtt{i_1w}=\mathtt{s} \circ \mathtt{w} \Vdash C$.
\end{exam}

\begin{exam}
If $C \Vdash A$ and $C \Vdash A \to B$ then $C \Vdash B$.
Let $\mathtt{r}: |C| \to |A|$ and $\mathtt{s}: |C| \to [|A|, |B|]$ be the realizers, respectively. 
Then, the map $\mathtt{ev} \langle \mathtt{s}, \mathtt{r} \rangle: |C| \to |B|$ realizes
$C \Vdash B$. To prove that, let $\mathtt{u}: \mathtt{1} \to |C|$ be a map such that $\mathtt{u} \Vdash C$. Then, $\mathtt{s} \circ \mathtt{u} \Vdash A \to B$ and $\mathtt{r} \circ \mathtt{u} \Vdash A$. Hence, $(\mathtt{s} \circ \mathtt{u}) \cdot (\mathtt{r} \circ \mathtt{u}) \Vdash B$.
As
$\mathtt{ev} \langle \mathtt{s}, \mathtt{r} \rangle \circ \mathtt{u}=(\mathtt{s} \circ \mathtt{u}) \cdot (\mathtt{r} \circ \mathtt{u})$, we reach $\mathtt{ev} \langle \mathtt{s}, \mathtt{r} \rangle \circ \mathtt{u} \Vdash B$.
\end{exam}

\begin{exam}\label{RealizabilityImplication}
If $\{C, A\} \Vdash B$ then $C \Vdash A \to B$. Let $\mathtt{r}: |C| \times |A| \to |B|$ be a realizer for $\{C, A\} \Vdash B$. We claim that the map $\lambda_{|A|} \mathtt{r} : |C| \to [|A|, |B|]$ realizes $C \Vdash A \to B$. To show that, let  $\mathtt{s}: \mathtt{1} \to |C|$ be a map such that $\mathtt{s} \Vdash C$. We have to show that $(\lambda_{|A|} \mathtt{r}) \circ \mathtt{s} \Vdash A \to B$. For that purpose, let $\mathtt{t}: \mathtt{1} \to |A|$ be a map such that $\mathtt{t} \Vdash A$. Then, as $((\lambda_{|A|} \mathtt{r}) \circ \mathtt{s}) \cdot \mathtt{t}=\mathtt{r} \circ \langle \mathtt{s}, \mathtt{t} \rangle$ and we know that $\mathtt{r} \circ \langle \mathtt{s}, \mathtt{t} \rangle \Vdash B$, the proof is complete. 
\end{exam}

As might already be clear from the provided examples, our goal is to transform proofs in $\mathsf{IPC}$ into realizers, i.e., maps in $\mathcal{C}$. This transformation allows us to import the proofs into $\mathcal{C}$ and leverage the structure of $\mathcal{C}$ to gain insights about proofs.

\begin{thm}(Soundness-completeness) \label{PropositionalRealizabilitySoundness}
For any set $\Gamma \cup \{A\}$ of formulas, $\Gamma \vdash_{\mathsf{IPC}} A$ iff $\Gamma \Vdash_{\mathcal{C}} A$ for any $\mathrm{BCC}$ $\mathcal{C}$.
\end{thm}
\begin{proof}
For soundness, it is enough to use an induction on the length of the proof of $\Gamma \vdash_{\mathsf{IPC}} A$. Some cases are checked in the above examples. The rest is left to the reader. For completeness, set $\mathcal{C}=\mathbf{NJ}$ and $\mathfrak{I}(p)=(p, \mathrm{Hom}_{\mathbf{NJ}}(\top, p))$. Then, as $\Gamma \Vdash_{\mathfrak{I}} A$, it has a realizer map $\mathtt{r}: |\Gamma| \to |A|$ in $\mathbf{NJ}$. Hence, there is a derivation of $A$ from $\Gamma$ in $\mathbf{NJ}$ which implies $\Gamma \vdash_{\mathsf{IPC}} A$.
\end{proof}

\subsection{$\mathcal{C}$-sets and Higher-order Realizability}

So far, we have covered our alternative formalization for the propositional BHK interpretation. To extend this formalization to higher-order languages, we need a notion of sets whose elements are constructed by maps in $\mathcal{C}$. The following provides such a notion.

\begin{dfn}
Let $\mathcal{C}$ be a $\mathrm{BCC}$. By a \emph{$\mathcal{C}$-set}, we mean a tuple in the form $S=(|S|, |\!|S|\!|, \Vdash_S)$, where $|S|$ is an object in $\mathcal{C}$, $|\!|S|\!|$ is a set and $\Vdash_S \, \subseteq \mathrm{Hom}_{\mathcal{C}}(\mathtt{1}, |S|) \times |\!|S|\!|$ is a binary relation such that for any $x \in |\!|S|\!|$, there exists $\mathtt{x}: \mathtt{1} \to |S|$ such that $\mathtt{x} \Vdash_S x$. A \emph{$\mathcal{C}$-set map} $f:(|S|, |\!|S|\!|, \Vdash_S) \to (|T|, |\!|T|\!|, \Vdash_T)$ is a function $f: |\!|S|\!| \to |\!|T|\!|$ induced (or realized) by a map $\mathtt{f}: |S| \to |T|$ in $\mathcal{C}$, i.e., if $\mathtt{x} \Vdash_S x$ then $\mathtt{f}\mathtt{x} \Vdash_T f(x)$, for any $x \in |\!|S|\!|$ and any $\mathtt{x}: \mathtt{1} \to |S|$. It is easy to see that $\mathcal{C}$-sets together with $\mathcal{C}$-set maps form a category, denoted by $\mathcal{C}\!\!-\!\mathbf{Set}$.\footnote{Following the usual convention in the literature, $\mathcal{C}$-sets must be called assemblies over $\mathcal{C}$. We use the term $\mathcal{C}$-set because it better aligns with the intuition that it is a set whose elements are constructed via the maps in $\mathcal{C}$.}
\end{dfn}

\begin{phil}
Interpreting $\mathcal{C}$ as the world of constructions, a $\mathcal{C}$-set can be viewed as a \emph{constructed set}. More precisely, the $\mathcal{C}$-set $S = (|S|, |\!|S|\!|, \Vdash_S)$ consists of the set $|\!|S|\!|$ equipped with a type $|S|$ of constructions and a relation $\mathtt{x} \Vdash_S x$ indicating that $\mathtt{x}: \mathtt{1} \to |S|$ is a construction of $x$. The surjectivity condition ensures that every element has a construction. Similarly, a $\mathcal{C}$-set map is a function induced by changing constructions via compositions with maps within $\mathcal{C}$. 
\end{phil}

The category $\mathcal{C}\!\!-\!\mathbf{Set}$ is a $\mathrm{BCC}$. We only present the $\mathrm{BC}$-structure without the proof that it really is the $\mathrm{BC}$-structure. For more details, see \cite{Lars, LarsThesis,Bauer,bauer2}. The terminal object of $\mathcal{C}\!\!-\!\mathbf{Set}$ is the $\mathcal{C}$-set $1=(\mathtt{1}, \{*\}, \Vdash_{1})$, where $\mathtt{1}$ is the terminal object of $\mathcal{C}$ and $\Vdash_1$ is defined by $id_{\mathtt{1}} \Vdash_1 *$. The initial object of $\mathcal{C}\!\!-\!\mathbf{Set}$ is the $\mathcal{C}$-set $0=(\mathtt{0}, \varnothing, \Vdash_{0})$, where $\mathtt{0}$ is the initial object of $\mathcal{C}$ and $\Vdash_0$ is the empty relation.

For products, coproducts and exponentials, let $S=(|S|, |\!|S|\!|, \Vdash_S)$ and $T=(|T|, |\!|T|\!|, \Vdash_T)$ be two $\mathcal{C}$-sets. Then, for the product of $S$ and $T$, consider the tuple $S \times T=(|S| \times |T|, |\!|S|\!| \times |\!|T|\!|, \Vdash_{S \times T})$, where $\mathtt{x} \Vdash_{S \times T} x$ iff $\mathtt{p_0}\mathtt{x} \Vdash_S p_0(x)$ and $\mathtt{p_1}\mathtt{x} \Vdash_T p_1(x)$, for any $\mathtt{x}: \mathtt{1} \to |S| \times |T|$ and any $x \in |\!|S|\!| \times |\!|T|\!|$. It is easy to see that $S \times T$ is a $\mathcal{C}$-set. Moreover, it is clear that the set projection functions are induced by the projection maps in $\mathcal{C}$ and hence they are $\mathcal{C}$-set maps. The data of $S \times T$, together with the projections is the binary product of $S$ and $T$.
For the coproduct of $S$ and $T$, consider the tuple $S+T=(|S| + |T|, |\!|S|\!| + |\!|T|\!|, \Vdash_{S + T})$, where $\mathtt{x} \Vdash_{S + T} x$ iff there are $\mathtt{y}: \mathtt{1} \to |S|$ and $y \in |\!|S|\!|$ such that $\mathtt{x}=\mathtt{i_0}\mathtt{y}$, $x=(0, y)$ and $\mathtt{y} \Vdash_S y$ or there are $\mathtt{z}: \mathtt{1} \to |T|$ and $z \in |\!|T|\!|$ such that $\mathtt{x}=\mathtt{i_1}\mathtt{z}$, $x=(1, z)$ and $\mathtt{z} \Vdash_S z$. It is easy to see that $S + T$ is a $\mathcal{C}$-set. Moreover, the set injection functions are induced by the injection maps in $\mathcal{C}$. Therefore, we can consider them as $\mathcal{C}$-set maps. The data of $S + T$ together with the injections is the binary coproduct of $S$ and $T$.
Finally, for the exponential $[S, T]$, consider the tuple 
\[
[S, T]=([|S|, |T|], \mathrm{Hom}_{\mathcal{C}\!-\!\mathbf{Set}}(S, T),  \Vdash_{[S, T]}),
\]
where $\mathtt{f} \Vdash_{[S, T]} f$ iff $\mathtt{x} \Vdash_S x$ implies $\mathtt{f} \cdot \mathtt{x} \Vdash_T f(x)$, for any $\mathtt{x}: \mathtt{1} \to |S|$ and any $x \in |\!|S|\!|$. It is easy to see that $[S, T]$ is a $\mathcal{C}$-set. Moreover, the usual evaluation map is induced by the evaluation map in $\mathcal{C}$.  Therefore, we can consider it as a $\mathcal{C}$-set map. The data of $[S, T]$ together with the evaluation map is the exponentiation of $T$ by $S$.

With a suitable notion of constructed sets in place, we can now extend the BHK interpretation to higher-order languages.

\begin{dfn}
Let $\mathcal{L}$ be a higher-order language and $\mathcal{C}$ be a $\mathrm{BCC}$. By a \emph{$\mathcal{C}$-interpretation} of $\mathcal{L}$, we mean a tuple 
\[
\mathfrak{I}=(\{M^{\sigma}\}_{\sigma \in T(\mathcal{L})}, \{f^{\mathfrak{I}}\}_{f \in F(\mathcal{L})}, \{P^{\mathfrak{I}}\}_{P \in P(\mathcal{L})}),
\]
where $M^{\sigma}$ is a $\mathcal{C}$-set such that $|\!|M^{\sigma}|\!| \neq \varnothing$ and $M^{\sigma \times \tau}=M^{\sigma} \times M^{\tau}$ and $M^{\sigma \to \tau}=[M^{\sigma}, M^{\tau}]$, for any types $\sigma$ and $\tau$, $f^{\mathfrak{I}}: M^{\sigma_1} \times  \cdots \times M^{\sigma_n} \to M^{\tau}$ is a $\mathcal{C}$-set map, for any $n$-ary function symbol $f(x_1^{\sigma_1}, \cdots, x_n^{\sigma_n})$ of type $\tau$ and $P^{\mathfrak{I}}=(|P|, R_P)$, where $|P|$ is an object in $\mathcal{C}$ and $R_P \subseteq \mathrm{Hom}_{\mathcal{C}}(\mathtt{1}, |P|) \times |\!|M^{\sigma_1}|\!| \times \cdots \times |\!|M^{\sigma_n}|\!|$ is a relation, for any $n$-ary predicate symbol $P(x_1^{\sigma_1}, \cdots, x_n^{\sigma_n})$.
\end{dfn}

The $\mathcal{C}$-interpretation $\mathfrak{I}$ assigns a $\mathcal{C}$-set map to any function symbol of the language. We can extend the assignment to all terms of the language. However, let us first extend the language $\mathcal{L}$ by adding a constant of type $\sigma$, for any $a \in |\!|M^{\sigma}|\!|$. This is a usual technique to have a name for any element in the domain. For simplicity, we denote this constant by $a$ itself and call the extended language $\mathcal{L}(\mathfrak{I})$. Now, we interpret any $\mathcal{L}(\mathfrak{I})$-term $t(\bar{x})=t(x^{\sigma_1}, \ldots, x^{\sigma_n})$ of type $\tau$ as a $\mathcal{C}$-set map $t(\bar{x})^{\mathfrak{I}}: M^{\sigma_1} \times \cdots \times M^{\sigma_n} \to M^{\tau}$ using the interpretation of the function symbols and composition. For the new constants, we interpret the constant $a \in |\!|M^{\sigma}|\!|$ as the map $1 \to M^{\sigma}$ that picks $a$ from $|\!|M^{\sigma}|\!|$. The next step is extending the $\mathcal{C}$-interpretation $\mathfrak{I}$ to all \emph{sentences} in the language $\mathcal{L}(\mathfrak{I})$. First, assign an object in $\mathcal{C}$ to any $\mathcal{L}(\mathfrak{I})$-\emph{formula} in the following way:
\begin{itemize}
\item[$\bullet$]
$|\top|=\mathtt{1}$, $|\bot|=\mathtt{0}$, $|t=s|=\mathtt{1}$, $|P(t_1, \ldots, t_n)|=|P|$,
\item[$\bullet$]
$|A \wedge B|=|A| \times |B|$,
$|A \vee B|=|A| + |B|$,
$|A \to B|= [|A|, |B|]$,
\item[$\bullet$]
$|\forall x^{\sigma} A(x^{\sigma})|=[|M^{\sigma}|, |A(x^{\sigma})|]$, 
$|\exists x^{\sigma} A(x^{\sigma})|=|M^{\sigma}| \times |A(x^{\sigma})|$.
\end{itemize}

\begin{rem}\label{UniquenessOfProofs}
Note that for any formula $A(\bar{x})$ and any two families of terms $\bar{t}$ and $\bar{s}$ of the appropriate types, we have $|A(\bar{t})| = |A(\bar{s})|$. This demonstrates a restriction in our approach, highlighting that our object-assignment is essentially propositional and does not account for the dependency of the formulas on the variables. To emphasize this fact, we use the fixed name $|A|$ to denote $|A(\bar{t})|$, regardless of the family of terms $\bar{t}$.
\end{rem}

Define the relation $\mathtt{r} \Vdash A$ between the map $ \mathtt{r} \in \mathrm{Hom}_{\mathcal{C}}(\mathtt{1}, |A|)$ and the \emph{$\mathcal{L}(\mathfrak{I})$-sentence} $A$ in the following way:
\begin{itemize}
\item
$!_{\mathtt{1}} \Vdash \top$, for the unique map $!_{\mathtt{1}}: \mathtt{1} \to \mathtt{1}$,
\item
$\mathtt{r} \nVdash \bot$, for any map $\mathtt{r} : \mathtt{1} \to \mathtt{0}$,
\item
$!_{\mathtt{1}} \Vdash t=s$ iff $t^{\mathfrak{I}}=s^{\mathfrak{I}}$, for the unique map $!_{\mathtt{1}}: \mathtt{1} \to \mathtt{1}$,
\item
$\mathtt{r} \Vdash P(t_1, \cdots, t_n)$ iff $(\mathtt{r}, t^{\mathfrak{I}}_1, \ldots, t^{\mathfrak{I}}_n) \in R_P$, for any $\mathtt{r}: \mathtt{1} \to |P|$, where the $\mathcal{C}$-set map $t^{\mathfrak{I}}_i: 1 \to M^{\sigma_i}$ is interpreted as an element in $|\!|M^{\sigma_i}|\!|$,
\item
$\mathtt{r} \Vdash A \wedge B$ iff $\mathtt{p_0r} \Vdash A$ and $\mathtt{p_1r} \Vdash B$, for any $\mathtt{r}: \mathtt{1} \to |A| \times |B|$, 
\item
$\mathtt{r} \Vdash A \vee B$ iff there exists $\mathtt{s}: \mathtt{1} \to |A|$ such that $\mathtt{r}=\mathtt{i_0s}$ and $\mathtt{s} \Vdash A$ or there exists $\mathtt{t}: \mathtt{1} \to |B|$ such that $\mathtt{r}=\mathtt{i_1t}$ and $\mathtt{t} \Vdash B$, for any $\mathtt{r}: \mathtt{1} \to |A| + |B|$, 
\item
$\mathtt{r} \Vdash A \to B$ iff for any $\mathtt{s}: \mathtt{1} \to |A|$, if $\mathtt{s} \Vdash A$ then $\mathtt{r} \cdot \mathtt{s} \Vdash B$,  for any $\mathtt{r}: \mathtt{1} \to [|A|, |B|]$,
\item
$\mathtt{r} \Vdash \forall x^{\sigma} A(x^{\sigma})$ iff for any $\mathtt{a}: \mathtt{1} \to |M^{\sigma}|$ and any $a \in |\!|M^{\sigma}|\!|$, if $\mathtt{a} \Vdash_{M^{\sigma}} a$ then $\mathtt{r} \cdot \mathtt{a} \Vdash A(a)$, for any $\mathtt{r}: \mathtt{1} \to [|M^{\sigma}|, |A|]$, 
\item
$\mathtt{r}  \Vdash \exists x^{\sigma} A(x^{\sigma})$ iff there is $a \in |\!|M^{\sigma}|\!|$ such that $\mathtt{p_0 r} \Vdash_{M^{\sigma}} a$ and $\mathtt{p_1r} \Vdash A(a)$, for any $\mathtt{r}: \mathtt{1} \to |M^{\sigma}| \times |A|$.
\end{itemize}
Let $\bar{x}$ be a \emph{finite} set of variables of types $\sigma_1, \ldots, \sigma_n$. For any finite set of formulas $\Gamma(\bar{x})$, define $|\Gamma|$ as $\prod_{\gamma \in \Gamma}|\gamma|$. We use an arbitrary order on the formulas in $\Gamma$, and since the product is commutative up to an isomorphism, the chosen order is actually immaterial. Then, by $\Gamma(\bar{x}) \Vdash_{\mathfrak{I}} A(\bar{x})$, we mean the existence of a map $\mathtt{r}: |M^{\sigma_1}| \times  \cdots \times |M^{\sigma_n}| \times |\Gamma| \to |A|$ such that $\mathtt{r} \circ \langle \langle \mathtt{a}_i \rangle_{i=1}^n, \mathtt{s} \rangle \Vdash A(\bar{a})$, for any $\{\mathtt{a}_i: \mathtt{1} \to |M^{\sigma_i}|\}_{i=1}^n$, any $\{a_i \in |\!|M^{\sigma_i}|\!|\}_{i=1}^n$ and any $\mathtt{s}: 1 \to |\Gamma|$ such that $\mathtt{a_i} \Vdash_{M^{\sigma_i}} a_i$, for any $1 \leq i \leq n$ and $\mathtt{s} \Vdash \bigwedge \Gamma(\bar{a})$. This $\mathtt{r}$ is called a \emph{realizer} for $\Gamma(\bar{x}) \Vdash_{\mathfrak{I}} A(\bar{x})$. When $\Gamma(\bar{x})$ is empty, $\mathtt{r}$ is called a realizer for $A(\bar{x})$. Again, we used an arbitrary order on the variables in $\bar{x}$ which is immaterial.
For an infinite set $\Gamma(\bar{x})$, by $\Gamma(\bar{x}) \Vdash_{\mathfrak{I}} A(\bar{x})$, we mean the existence of a finite set $\Delta(\bar{x}) \subseteq \Gamma(\bar{x})$ such that $\Delta(\bar{x}) \Vdash_{\mathfrak{I}} A(\bar{x})$. When the $\mathcal{C}$-interpretation $\mathfrak{I}$ is clear from the context, we drop the subscript $\mathfrak{I}$ in $\Vdash_{\mathfrak{I}}$. Finally, by $\Gamma(\bar{x}) \Vdash_{\mathcal{C}} A(\bar{x})$, we mean $\Gamma(\bar{x}) \Vdash_{\mathfrak{I}} A(\bar{x})$, for any $\mathcal{C}$-interpretation $\mathfrak{I}$.

\begin{rem}
The definition of $\Gamma(\bar{x}) \Vdash A(\bar{x})$ can be somewhat ambiguous since it is not immediately clear which variables are considered free in the formulas in $\Gamma(\bar{x}) \cup A(\bar{x})$. Naturally, we would choose $\bar{x}$ to be the set of all free variables in $\Gamma(\bar{x}) \cup A(\bar{x})$. However, there are cases where other variables need to be treated as free. For instance, in the formula $A(x) \wedge B(y)$, the set of free variables is $\{x, y\}$. However, the definition of realizability for this formula requires us to consider $A(x)$ with free variables $\{x, y\}$ rather than just $\{x\}$.
Fortunately, extending the set of free variables does not impact realizability. This is because $|\!|M^{\tau}|\!| \neq \varnothing$ for any type $\tau$, which implies that there is always a map $\mathtt{a}: \mathtt{1} \to |M^{\tau}|$. Thus, using $\mathtt{a}$, the existence of a realizer $\mathtt{r}: |M^{\sigma_1}| \times \cdots \times |M^{\sigma_n}| \times |\Gamma| \to |A|$ is equivalent to the existence of a realizer $\mathtt{s}: |M^{\tau}| \times |M^{\sigma_1}| \times \cdots \times |M^{\sigma_n}| \times |\Gamma| \to |A|$, when the variable of type $\tau$ is a dummy in $A(\bar{x})$.
\end{rem}

\begin{exam}
It is straightforward to verify that the examples of propositional realizability we previously discussed extend to the case where variables are present. For instance, the following examples hold: $[A(\bar{x}) \Vdash A(\bar{x})]$, $[\bot \Vdash A(\bar{x})]$, $[A(\bar{x}) \Vdash \top]$, $[\{A(\bar{x}), B(\bar{x})\} \Vdash A(\bar{x})]$, and $[B(\bar{x}) \Vdash A(\bar{x}) \vee B(\bar{x})]$.
\end{exam}

\begin{exam}
Let $\mathtt{p}: |M^{\sigma}| \times |M^{\sigma}| \times \mathtt{1} \times |A| \to |A|$ be the projection map on the last component. Then $\mathtt{p}$ realizes $(x^{\sigma}=y^{\sigma}), A(x^{\sigma}) \Vdash A(y^{\sigma})$. The reason is that for any $\mathtt{a}, \mathtt{b}: \mathtt{1} \to  |M^{\sigma}|$, any $a, b \in |\!|M^{\sigma}|\!|$, any $\mathtt{r}: \mathtt{1} \to \mathtt{1}$ and any $\mathtt{s}: \mathtt{1} \to |A|$, if $\mathtt{a} \Vdash_{M^{\sigma}} a$, $\mathtt{b} \Vdash_{M^{\sigma}} b$, $\mathtt{r} \Vdash a=b$ and $\mathtt{s} \Vdash A(a)$, we must have $\mathtt{r}=!_{\mathtt{1}}$ and hence $a=b$, by definition. Therefore, $\mathtt{p} \circ \langle \mathtt{a}, \mathtt{b}, \mathtt{r}, \mathtt{s} \rangle= \mathtt{s} \Vdash A(b)$.
\end{exam}

\begin{exam}
Let $t(\bar{x})=t(x_1^{\sigma_1}, \ldots, x_n^{\sigma_n})$ be an $\mathcal{L}(\mathfrak{I})$-term of type $\tau$. The interpretation of $ t(\bar{x})$ is the $\mathcal{C}$-set map $t^{\mathfrak{I}}: M^{\sigma_1} \times \cdots \times M^{\sigma_n} \to M^{\tau}$ and as it is a $\mathcal{C}$-set map, it is induced by a map $\mathtt{t}: |M^{\sigma_1}| \times \cdots \times |M^{\sigma_n}| \to |M^{\tau}|$ in $\mathcal{C}$. Then, for any $\{\mathtt{a}_i: \mathtt{1} \to |M^{\sigma_i}| \}_{i=1}^n$ and $\{a_i \in |\!|M^{\sigma_i}|\!|\}_{i=1}^n$ such that $\mathtt{a}_i \Vdash_{M^{\sigma_i}} a_i$, for any $1 \leq i \leq n $, it is easy to see that $ \mathtt{t}\langle \mathtt{a}_i \rangle_{i=1}^n \Vdash_{M^{\tau}} t^{\mathfrak{I}}(\bar{a})$. This claim holds for any such $\mathtt{t}$.
\end{exam}

\begin{exam}
Let $t(\bar{x})=t(x_1^{\sigma_1}, \ldots, x_n^{\sigma_n})$ be an $\mathcal{L}(\mathfrak{I})$-term of type $\tau$. We claim that $A(t(\bar{x}), \bar{x}) \Vdash \exists y^{\tau} A(y, \bar{x}) $. Using the previous example, let $\mathtt{t}: |M^{\sigma_1}| \times \cdots \times |M^{\sigma_n}| \to |M^{\tau}|$ be a map inducing $t^{\mathfrak{I}}$. Now, define the map $\mathtt{r}=\langle \mathtt{t} \mathtt{p}, \mathtt{q} \rangle : |M^{\sigma_1}| \times \cdots \times |M^{\sigma_n}| \times |A| \to |M^{\tau}| \times |A|$, where $\mathtt{p}$ is the projection on $|M^{\sigma_1}| \times \cdots \times |M^{\sigma_n}|$ and $\mathtt{q}$ is the projection on $|A|$. Then, $\mathtt{r}$
realizes $A(t(\bar{x}), \bar{x}) \Vdash \exists y^{\tau} A(y, \bar{x}) $. To prove that, let $\{\mathtt{a}_i: \mathtt{1} \to |M^{\sigma_i}| \}_{i=1}^n$, $\{a_i \in |\!|M^{\sigma_i}|\!|\}_{i=1}^n$, and $\mathtt{s}: \mathtt{1} \to |A|$ such that $\mathtt{a}_i \Vdash_{M^{\sigma_i}} a_i$, for any $1 \leq i \leq n $ and $\mathtt{s} \Vdash A(t(\bar{a}), \bar{a})$. As $ \mathtt{t}\langle \mathtt{a}_i \rangle_{i=1}^n \Vdash_{M^{\tau}} t^{\mathfrak{I}}(\bar{a})$, we reach
$\langle \mathtt{t}\langle \mathtt{a}_i \rangle_{i=1}^n, \mathtt{s} \rangle \Vdash \exists y^{\tau} A(y, \bar{a})$. As
$\mathtt{r} \circ \langle \langle \mathtt{a}_i \rangle_{i=1}^n, \mathtt{s} \rangle=\langle \mathtt{t}\langle \mathtt{a}_i \rangle_{i=1}^n, \mathtt{s} \rangle$, the proof is complete.
\end{exam}

\begin{exam}
If $A(\bar{x}) \Vdash B(\bar{x}, y)$ and $A(\bar{x})$ has no free variable $y$, then $A(\bar{x}) \Vdash \forall y^{\tau} B(\bar{x}, y)$. Let $\mathtt{r}: |M^{\sigma_1}| \times \cdots \times |M^{\sigma_n}| \times |M^{\tau}| \times |A| \to |B|$ be the realizer for $A(\bar{x}) \Vdash B(\bar{x}, y)$. Using the commutativity of the product, we can transform $\mathtt{r}$ to the map $\mathtt{s}: |M^{\sigma_1}| \times \cdots \times |M^{\sigma_n}| \times |A| \times |M^{\tau}| \to |B|$. Then, the map $\lambda_{|M^{\tau}|} \mathtt{s} : |M^{\sigma_1}| \times \cdots \times |M^{\sigma_n}| \times |A| \to [|M^{\tau}|, |B|]$ realizes $A(\bar{x}) \Vdash \forall y^{\tau} B(\bar{x}, y)$. The proof follows a similar method to the one used for implications, as detailed in Example \ref{RealizabilityImplication}.
\end{exam}

\begin{thm}(Soundness) \label{Soundness}
For any set $\Gamma(\bar{x}) \cup \{A(\bar{x})\}$ of formulas, if $\Gamma(\bar{x}) \vdash_{\mathsf{IQC}^{\omega}+\mathsf{EXT}} A(\bar{x})$ then $\Gamma(\bar{x}) \Vdash_{\mathcal{C}} A(\bar{x})$, for any $\mathrm{BCC}$ $\mathcal{C}$.
\end{thm}
\begin{proof}
The proof is by induction on the length of the proof of $\Gamma(\bar{x}) \vdash_{\mathsf{IQC^{\omega}+EXT}} A(\bar{x})$. We have already realized the propositional rules as shown in Theorem \ref{PropositionalRealizabilitySoundness}. For the quantifier rules, as well as the axioms for equality and $\mathsf{EXT}$, some specific instances have been addressed in the examples above. The remaining cases are left as an exercise for the reader.
\end{proof}

Since our notion of realizability is a \emph{special case} of the semantics for higher-order languages, we do not anticipate completeness. The following example illustrates a concrete counterexample.

\begin{exam}(\emph{Incompleteness of realizability})
For any sentence $A$, any formula $B(x^{\sigma})$ and any $\mathcal{C}$-interpretation $\mathfrak{I}$, we show that $(A \to \exists x^{\sigma} B(x)) \Vdash_{\mathfrak{I}} \exists x^{\sigma} (A \to B(x))$. Notice that $(A \to \exists x^{\sigma} B(x)) \nvdash \exists x^{\sigma} (A \to B(x))$ in intuitionistic logic. This provides a counterexample to the completeness of our realizability notion. To prove the realizability part, note that $|A \to \exists x^{\sigma} B(x)|=[|A|, |M^{\sigma}| \times |B|]$ and $| \exists x^{\sigma} (A \to B(x))|=|M^{\sigma}| \times [|A|, |B|]$.
There are two cases to consider. Either $\Vdash_{\mathfrak{I}} A$ or $\nVdash_{\mathfrak{I}} A$. In the first case, let $\mathtt{r} \Vdash_{\mathfrak{I}} A$. Then, define 
\[\small \mathtt{F_0}=\begin{tikzcd}
	{[|A|, |M^{\sigma}| \times |B|]} && {[|A|, |M^{\sigma}| \times |B|] \times |A|} && {|M^{\sigma}|}
	\arrow["{\langle \mathtt{id}, \mathtt{r} \circ ! \rangle}", from=1-1, to=1-3]
	\arrow["{\mathtt{p_0 \circ ev}}", from=1-3, to=1-5]
\end{tikzcd}\]
\[\small \mathtt{F_1}=\begin{tikzcd}
	{[|A|, |M^{\sigma}| \times |B|] \times |A|} && {[|A|, |M^{\sigma}| \times |B|] \times |A|} && {|B|}
	\arrow["{\mathtt{id} \times (\mathtt{r} \circ !)}", from=1-1, to=1-3]
	\arrow["{\mathtt{p_1 \circ ev}}", from=1-3, to=1-5]
\end{tikzcd}\]
We claim that $\mathtt{F}=\langle \mathtt{F_0}, \lambda_{|A|}\mathtt{F_1} \rangle $ realizes  $(A \to \exists x^{\sigma} B(x)) \Vdash \exists x^{\sigma} (A \to B(x))$. Let $\mathtt{s} \Vdash A \to \exists x^{\sigma} B(x)$. Then, $\mathtt{s} \cdot \mathtt{r} \Vdash \exists x^{\sigma} B(x)$ which implies the existence of $a \in |\!|M^{\sigma}|\!|$ such that $\mathtt{p_0}(\mathtt{s} \cdot \mathtt{r}) \Vdash_{M^{\sigma}} a$ and $\mathtt{p_1}(\mathtt{s} \cdot \mathtt{r}) \Vdash B(a)$. Now, we have to show that $\mathtt{F} \circ \mathtt{s} \Vdash \exists x^{\sigma}(A \to B(x))$. First, notice that $\mathtt{p_0}(\mathtt{F} \circ \mathtt{s})=\mathtt{F_0} \circ \mathtt{s}=\mathtt{p_0}(\mathtt{s} \cdot \mathtt{r}) \Vdash_{M^{\sigma}} a$. Second, we want to show that $\mathtt{p_1}(\mathtt{F} \circ \mathtt{s}) \Vdash A \to B(a)$. Let $\mathtt{t} \Vdash A$. Then, $\mathtt{p_1}(\mathtt{F} \circ \mathtt{s}) \cdot \mathtt{t}=\mathtt{p_1}(\mathtt{s} \cdot \mathtt{r}) \Vdash B(a)$ which completes the proof.

In the second case, we assume $\nVdash_{\mathfrak{I}} A$. As $|\!|M^{\sigma}|\!| \neq \varnothing$, we can pick a map $\mathtt{a}: \mathtt{1} \to |M^{\sigma}|$ and an element $a \in |\!|M^{\sigma}|\!|$ such that $\mathtt{a} \Vdash_{M^{\sigma}} a$. Define $\mathtt{F_0}=\mathtt{a} \circ !: [|A|, |M^{\sigma}| \times |B|] \to |M^{\sigma}|$ and $\mathtt{F_1}=\mathtt{p_1 \circ ev}: [|A|, |M^{\sigma}| \times |B|] \times |A| \to |B|$. We claim that $\mathtt{F}=\langle \mathtt{F_0}, \lambda_{|A|}\mathtt{F_1} \rangle $ realizes  $(A \to \exists x^{\sigma} B(x)) \Vdash \exists x^{\sigma} (A \to B(x))$. Let $\mathtt{s} \Vdash A \to \exists x^{\sigma} B(x)$. First, notice that $\mathtt{p_0}(\mathtt{F} \circ \mathtt{s})=\mathtt{F_0} \circ \mathtt{s}=\mathtt{a} \Vdash_{M^{\sigma}} a$. Second, we have $\mathtt{p_1}(\mathtt{F} \circ \mathtt{s}) \Vdash A \to B(a)$ as there is no $\mathtt{t} \Vdash A$. 
\end{exam}

\subsection{Projective $\mathcal{C}$-sets and the Axiom of Choice}
In this subsection, we investigate the situations where the axiom of choice is realizable. To do this, we first need to introduce projective $\mathcal{C}$-sets.
\begin{dfn}
A $\mathcal{C}$-set $S=(|S|, |\!|S|\!|, \Vdash_S)$ is called \emph{projective} if $\mathtt{x_0} \Vdash_S x$ and $\mathtt{x}_1 \Vdash_S x$ implies $\mathtt{x_0}=\mathtt{x_1}$, for any $\mathtt{x_0}, \mathtt{x_1}: \mathtt{1} \to |S|$ and any $x \in |\!|S|\!|$.
\end{dfn}

\begin{thm}\label{AxiomOfChoice}
Let $\mathfrak{I}$ be a $\mathcal{C}$-interpretation that maps the type $\sigma$ to a projective $\mathcal{C}$-set. Then, the sentence 
$\forall x^\sigma \exists y^\tau A(x, y) \to \exists f^{\sigma \to \tau} \forall x^\sigma A(x, fx)$ is realizable.
\end{thm}
\begin{proof}
Note that $|\forall x^\sigma \exists y^\tau A(x, y)|=[|M^\sigma|, |M^\tau| \times |A|]$ and $|\exists f^{\sigma \to \tau} \forall x^\sigma A(x, fx)|=|M^{\sigma \to \tau}| \times [|M^\sigma|, |A|]$. Consider the maps $\mathtt{F_0}$ and $\mathtt{F_1}$ defined as below:
\[\mathtt{F_0}=\begin{tikzcd}[ampersand replacement=\&]
	{[|M^\sigma|, |M^\tau| \times |A|] \times |M^{\sigma}|} \&\& {|M^\tau| \times |A|} \&\& {|M^{\tau}|}
	\arrow["{\mathtt{ev}}", from=1-1, to=1-3]
	\arrow["{\mathtt{p_0}}", from=1-3, to=1-5]
\end{tikzcd}\]
\[\mathtt{F_1}=\begin{tikzcd}[ampersand replacement=\&]
	{[|M^\sigma|, |M^\tau| \times |A|] \times |M^{\sigma}|} \&\& {|M^\tau| \times |A|} \&\& {|A|}
	\arrow["{\mathtt{ev}}", from=1-1, to=1-3]
	\arrow["{\mathtt{p_1}}", from=1-3, to=1-5]
\end{tikzcd}\]
We claim that the map 
\[
\mathtt{F}=\langle \lambda_{|M^\sigma|} \mathtt{F_0}, \lambda_{|M^\sigma|} \mathtt{F_1} \rangle : [|M^\sigma|, |M^\tau| \times |A|] \to |M^{\sigma \to \tau}| \times [|M^{\sigma}|, |A|]
\]
realizes 
\[
\forall x^\sigma \exists y^\tau A(x, y) \Vdash \exists f^{\sigma \to \tau} \forall x^\sigma A(x, fx).
\]
Let $\mathtt{r} \Vdash \forall x^\sigma \exists y^\tau A(x, y)$. As $M^{\sigma}$ is projective, for any $a \in |\!|M^{\sigma}|\!|$, there is a unique $\mathtt{a}_0: 1 \to 
|M^{\sigma}|$ such that $\mathtt{a}_0 \Vdash_{M^{\sigma}} a$. Hence, we have $\mathtt{r} \cdot \mathtt{a}_0 \Vdash \exists y^{\tau} A(a, y)$ which implies the existence of $b \in |\!|M^{\tau}|\!|$ such that $\mathtt{p_0}(\mathtt{r} \cdot \mathtt{a}_0) \Vdash_{M^{\tau}} b$ and $\mathtt{p_1}(\mathtt{r} \cdot \mathtt{a}_0) \Vdash A(a, b)$. For any $a \in |\!|M^{\sigma}|\!|$, choose such a $b \in |\!|M^{\tau}|\!|$ and define $g: |\!|M^{\sigma}|\!| \to |\!|M^{\tau}|\!|$ by setting $g(a)=b$. We claim that $\mathtt{p_0(F \circ r)} \Vdash_{[M^{\sigma}, M^{\tau}]} g$. For the proof, let $\mathtt{a} \Vdash_{M^{\sigma}} a$. By the projectivity of $M^{\sigma}$, we have $\mathtt{a}=\mathtt{a}_0$. As $\mathtt{p_0(F \circ r)} \cdot \mathtt{a}=((\lambda_{|M^\sigma|} \mathtt{F_0}) \circ \mathtt{r}) \cdot \mathtt{a}_0=\mathtt{F_0} \circ \langle \mathtt{r}, \mathtt{a_0} \rangle=\mathtt{p_0}(\mathtt{r} \cdot \mathtt{a}_0)$ and $\mathtt{p_0}(\mathtt{r} \cdot \mathtt{a}_0) \Vdash_{M^{\tau}} b=g(a)$, we reach $\mathtt{p_0(F \circ r)} \cdot \mathtt{a} \Vdash_{M^{\tau}} g(a)$. Hence, $\mathtt{p_0(F \circ r)} \Vdash_{[M^{\sigma}, M^{\tau}]} g$. Similarly, we can prove that $\mathtt{p_1(F \circ r)} \Vdash \forall x^\sigma A(x, gx)$. Therefore, $\mathtt{F} \circ \mathtt{r} \Vdash \exists f^{\sigma \to \tau} \forall x^\sigma A(x, fx)$.
\end{proof}

Unfortunately, the exponentiation of projective $\mathcal{C}$-sets is not necessarily projective. Therefore, even if the $\mathcal{C}$-interpretation maps all basic types to projective $\mathcal{C}$-sets, this does not guarantee that all $M^{\sigma}$'s are projective. However, in certain special cases, we can ensure that all $M^{\sigma}$ are projective, which in turn implies that the axiom of choice is realizable. To explain this, we need two definitions.

\begin{dfn}
A category $\mathcal{C}$ with a terminal object $1$ is called \emph{well-pointed} if, for any two maps $f, g: A \to B$, having $f \circ h = g \circ h$ for any $h : 1 \to A$ implies $f = g$.
\end{dfn}

\begin{exam}
The categories $\mathbf{Set}$ and $\mathbf{Top}$ are well-pointed while $\mathbf{Set}^{\mathbf{2}}$ is not. For the latter claim, consider the inclusion function $i: \{0\} \to \{0, 1\}$ as an object in $\mathbf{Set}^{\mathbf{2}}$ and the map $\alpha: i \to i$:
\[\begin{tikzcd}[ampersand replacement=\&]
	{\{0, 1\}} \&\& {\{0, 1\}} \\
	\\
	{\{0\}} \&\& {\{0\}}
	\arrow["{\alpha_1}", from=1-1, to=1-3]
	\arrow["i", from=3-1, to=1-1]
	\arrow["{\alpha_0}"', from=3-1, to=3-3]
	\arrow["i"', from=3-3, to=1-3]
\end{tikzcd}\]
defined by $\alpha_0=id_{\{0\}}$ and $\alpha_1$ as the constant zero function. It is easy to see that the only map from the terminal object $1=\{*\} \to \{*\}$ to $i$ is:
\[\begin{tikzcd}[ampersand replacement=\&]
	{\{*\}} \&\& {\{0, 1\}} \\
	\\
	{\{*\}} \&\& {\{0\}}
	\arrow["{\beta_1}", from=1-1, to=1-3]
	\arrow["{id_{\{*\}}}", from=3-1, to=1-1]
	\arrow["{\beta_0}"', from=3-1, to=3-3]
	\arrow["i"', from=3-3, to=1-3]
\end{tikzcd}\]
where $\beta_0(*)=\beta_1(*)=0$. As $\alpha \circ \beta =id_i \circ \beta$ and $\alpha \neq id_i$, the category $\mathbf{Set}^{\mathbf{2}}$ is not well-pointed.
\end{exam}

\begin{rem}\label{AnotherWellpointed}
Let $\mathcal{C}$ be a well-pointed $\mathrm{CCC}$ and $f, g: 1 \to [A,B]$ be two maps. Then, having $f \cdot h = g \cdot h$, for any $h : 1 \to A$, implies $f = g$. To prove this claim, as any map $k: 1 \to [A, B]$ is in the form $\lambda_{A}k'$, for some $k': A \to B$, it is enough to pick $f', g': A \to B$ such that $f=\lambda_{A}f'$ and $g=\lambda_{A}g'$ and show that $f'=g'$. Now, let $h : 1 \to A$. Then, $f \cdot h=(\lambda_{A}f') \cdot h=f' \circ h$. Similarly, $g \cdot h=g' \circ h$. As $f \cdot h = g \cdot h$, we have $f' \circ h = g' \circ h$, for any $h: 1 \to A$. As $\mathcal{C}$ is well-pointed, $f'=g'$ which implies $f=g$.
\end{rem}

\begin{dfn}
A $\mathcal{C}$-set $S=(|S|, |\!|S|\!|, \Vdash_S)$ is called \emph{total} if for any $\mathtt{x}: \mathtt{1} \to |S|$ there is $x \in |\!|S|\!|$ such that $\mathtt{x} \Vdash_S x$.  
\end{dfn}

\begin{lem}\label{TotalProjectiveLemma}
Let $S$ and $T$ be $\mathcal{C}$-sets. Then:
\begin{description}
\item[$(i)$]
If $S$ and $T$ are projective (resp. total), then $S \times T$ is also projective (resp. total).
    \item[$(ii)$]
If $\mathcal{C}$ is well-pointed, $S$ is total and $T$ is projective, $[S, T]$ is projective.  
  \item[$(iii)$] 
If $S$ is projective and $T$ is total, $[S, T]$ is total.
\end{description}
\end{lem}
\begin{proof}
$(i)$ is easy to prove. For $(ii)$, let $\mathtt{f} \Vdash_{[S, T]} h$, $\mathtt{g} \Vdash_{[S, T]} h$, and $\mathtt{x}: \mathtt{1} \to |S|$ be an arbitrary map. As $S$ is total, there is an element $x \in |\!|S|\!|$ such that $\mathtt{x} \Vdash_S x$. Hence, we have $\mathtt{f}\cdot \mathtt{x} \Vdash_T h(x)$ and $\mathtt{g}\cdot \mathtt{x} \Vdash_T h(x)$. As $T$ is projective, we have $\mathtt{f}\cdot \mathtt{x}=\mathtt{g}\cdot \mathtt{x}$.  As $\mathtt{x}: \mathtt{1} \to |S|$ is arbitrary and $\mathcal{C}$ is well-pointed, we have $\mathtt{f}=\mathtt{g}$, by Remark \ref{AnotherWellpointed}. For $(iii)$, let $\mathtt{f}: \mathtt{1} \to [|S|, |T|]$. We define a function $f: |\!|S|\!| \to |\!|T|\!|$ such that $\mathtt{f} \Vdash_{[S, T]} f$. As $S$ is projective, for any $x \in |\!|S|\!|$, there is a unique $\mathtt{x}: \mathtt{1} \to |S|$ such that $\mathtt{x} \Vdash_S x$. Hence, $\mathtt{f} \cdot \mathtt{x}: \mathtt{1} \to |T|$. As $T$ is total, there is an element $y \in |\!|T|\!|$ such that $\mathtt{f} \cdot \mathtt{x} \Vdash_T y$. Choose such a $y$ and define $f(x)=y$. To prove $\mathtt{f} \Vdash_{[S, T]} f$, let $\mathtt{z} \Vdash_S x$. As $S$ is projective, $\mathtt{z}=\mathtt{x}$. Then, by definition, $f(x)$ is chosen such that $\mathtt{f} \cdot \mathtt{z}=\mathtt{f} \cdot \mathtt{x} \Vdash_T f(x)$. 
\end{proof}

Using Lemma \ref{TotalProjectiveLemma}, if $\mathcal{C}$ is well-pointed and the $\mathcal{C}$-interpretation $\mathfrak{I}$ maps all basic types to total projective $\mathcal{C}$-sets, then $\mathfrak{I}$ maps all types to total projective $\mathcal{C}$-sets. Consequently, the axiom of choice is realizable for any such $\mathcal{C}$-interpretation. We will use this fact later in Section \ref{SecRealizabilityForArith}.

\section{Realizability for Arithmetical Theories} \label{SecRealizabilityForArith}

In this section, we intend to apply the realizability machinery to the arithmetical theories introduced in Section \ref{SectionTheories}. To achieve this, we need two additional components. First, we must introduce two new $\mathrm{BCC}$'s to capture the worlds of computable and continuous constructions that we informally presented in Section \ref{SectionTheories}. Second, we need a categorical way of representing the set of natural numbers using a categorical gadget called the \emph{natural numbers object}. These two components will be covered in Subsections \ref{SubsecRecConWorlds} and \ref{SubsecNNO}, respectively. Then, in Subsection \ref{SubsecRealForArithmetic}, we will use the previously introduced $\mathrm{BCC}$'s to prove some arithmetical consistencies and to extract information from the arithmetical proofs.

\subsection{A Recursive and a Continuous World} \label{SubsecRecConWorlds}

In this subsection, we will introduce two categories to formalize the previously mentioned computable and continuous worlds. For the latter, the canonical candidate is the category $\mathbf{Top}$. However, $\mathbf{Top}$ does not have all the exponentials, making it somewhat cumbersome to work with. To achieve a $\mathrm{BCC}$ with a continuous nature similar to that of $\mathbf{Top}$, we need to extend topological spaces by incorporating their own local equality relations.

\begin{dfn}
An \emph{equilogical space} is a pair $(X, \sim_X)$, where $X$ is a topological space and $\sim_X$ is an equivalence relation on $X$. One can read the equilogical space $(X, \sim_X)$ as a proposition whose finitary proofs are stored in $X$ and considered up to the equivalence $\sim_X$.
An \emph{equilogical map} $f: (X, \sim_X) \to (Y, \sim_Y)$ is a continuous function $f:X \to Y$ preserving the equivalence relation, i.e., if $x \sim_X y$ then $f(x) \sim_Y f(y)$, for any $x, y \in X$.
Two equilogical maps $f, g: (X, \sim_X) \to (Y, \sim_Y)$ are considered equivalent, denoted by $f \sim g$, if $f(x) \sim_Y g(x)$, for any $x \in X$. The relation $\sim$ on the equilogical maps from $(X, \sim_X)$ to $(Y, \sim_Y)$ is an equivalence relation, respecting the composition of functions, i.e., if $f_1 \sim f_2$ and $g_1 \sim g_2$, then $f_1 \circ g_1 \sim f_2 \circ g_2$, for any $f_1, f_2: (Y, \sim_Y) \to (Z, \sim_Z)$ and $g_1, g_2: (X, \sim_X) \to (Y, \sim_Y)$.
The collection of equilogical spaces and equilogical maps (up to the equivalence) with the evident composition and identity constitutes a category, denoted by $\mathbf{Equ}$. 
\end{dfn}

\begin{phil}
An equilogical space should be understood as a constructive space, where both the set of elements and their equality are primitive notions. In constructive mathematics, every mathematical entity is a mental construction, and thus, there is no universal notion of equality applicable everywhere. One must define the corresponding equality locally for each constructed entity. Consequently, an equilogical map is simply a constructive continuous function that preserves the local identities as expected. The equivalence between equilogical maps serves as an \emph{extensionality} criterion, ensuring that a map is uniquely determined by its values.
\end{phil}

The category $\mathbf{Equ}$ is a $\mathrm{BCC}$. The terminal and the initial objects are $(\{*\}, =_{\{*\}})$ and $(\varnothing, =_{\varnothing})$, respectively. 
The product of $(X, \sim_X)$ and $(Y, \sim_Y)$ is the pair $(X \times Y, \sim_{X \times Y})$, equipped with the usual projections, where $X \times 
Y$ is the cartesian product with the product topology, and $(x_1, y_1) \sim_{X \times Y} (x_2, y_2)$ iff $x_1 \sim_X x_2$ and $y_1 \sim_Y y_2$. Note that the projection maps are continuous and respect the equivalence relation, thus they live in $\mathbf{Equ}$. The coproduct of $(X, \sim_X)$ and $(Y, \sim_Y)$ is $(X+Y, \sim_{X+Y})$, equipped with the usual injections, where $X+Y$ is the disjoint union of $X$ and $Y$ with its canonical topology and $(0, x_1) \sim_{X+Y} (0, x_2)$ iff $x_1 \sim_X x_2$ and $(1, y_1) \sim_{X+Y} (1, y_2)$ iff $y_1 \sim_Y y_2$. Note that the injections are continuous and respect the equivalence relation, thus they live in $\mathbf{Equ}$. 

For the exponential objects in $\mathbf{Equ}$, their detailed construction is unfortunately beyond the scope of this chapter and the interested reader can refer to \cite{Scot,Bauer} for more information. The only property of the construction we need here is the fact that the exponentials in $\mathbf{Top}$ (if they exist) remain intact by moving to $\mathbf{Equ}$. More precisely, define the functor $Eq: \mathbf{Top} \to \mathbf{Equ}$ mapping the space $X$ to the equilogical space $(X, =_X)$ and the function $f: X \to Y$ to the class of itself in $\mathbf{Equ}$. Then, if the exponential $[X, Y]$ with the usual set-theoretical evaluation map exists in $\mathbf{Top}$, then the exponential $[(X,=_X), (Y, =_Y)]$ is just $([X, Y], =_{[X, Y]})$ with the same evaluation map. For instance, the exponential $[(\mathbb{N}, =), (\mathbb{N}, =)]$ in $\mathbf{Equ}$ is just $(\mathbb{N}^{\mathbb{N}}, =)$ with the usual evaluation map.\\

The second category formalizes the computable world and is defined similarly to $\mathbf{Equ}$. However, in this category, spaces and continuous functions are replaced with subsets of $\mathbb{N}$ and computable functions, respectively.

\begin{exam}
Consider the collection of pairs $(A, \sim_A)$, where $A$ is a subset (not necessarily decidable or computably enumerable) of $\mathbb{N}$ and $\sim_A$ is an equivalence relation on $A$. An \emph{equivariant map} $f: (A, \sim_A) \to (B, \sim_B)$ is a \emph{partial} computable function $f: \mathbb{N} \rightharpoonup \mathbb{N}$ that maps the elements of $A$ to the elements of $B$ while preserving the equivalence relation. Two equivariant maps are considered equivalent if they induce the same function on the equivalence classes.
We can interpret the pair $(A, \sim_A)$ as representing the set of all \emph{finitary proofs} of a proposition $\phi_A$ assigned to $A$, with $\sim_A$ capturing the equivalence between these proofs. Thus, an equivariant map is viewed as an \emph{algorithmic} procedure that transforms proofs of $\phi_A$ into proofs of $\phi_B$, respecting the equivalences among these proofs.
The collection of all pairs $(A, \sim_A)$ as the objects and equivariant maps (up to the equivalence) as the morphisms together with the evident composition and identity constitutes a category, denoted by $\mathbf{Rec}$.
\end{exam}

The category $\mathbf{Rec}$ is a $\mathrm{BCC}$. Similar to $\mathbf{Equ}$, its terminal and initial objects are $(\{*\}, =_{\{*\}})$ and $(\varnothing, =_{\varnothing})$, respectively. The prodcuts and coproducts are defined similar to $\mathbf{Equ}$. However, as in $\mathbf{Rec}$, we restrict ourselves to the subsets of $\mathbb{N}$, we need to encode the cartesian product and the disjoint union of two subsets of $\mathbb{N}$ by another subset. More precisely, in $\mathbf{Rec}$, the product of $(A, \sim_A)$ and $(B, \sim_B)$ is the pair $(C, \sim_C)$, together with the maps $p_0: C \to A$ and $p_1: C \to B$, where $C=\{2^a(2b+1) \mid a \in A, b \in B\}$, $2^a(2b+1) \sim_C 2^c(2d+1)$ iff $a \sim_A c$ and $b \sim_B d$, and $p_0$ and $p_1$ are defined by $p_0(2^n(2m+1))=n$ and $p_1(2^n(2m+1))=m$. Notice that both $p_0$ and $p_1$ are computable functions and hence live in $\mathbf{Rec}$.
Similarly, for the coproduct of $(A, \sim_A)$ and $(B, \sim_B)$, we use $(C, \sim_C)$, together with the maps $i_0: A \to C$ and $i_1: B \to C$, where 
\[
C=\{2^a(2b+1) \mid (a=0 \; \text{and} \;  b \in A) \, \text{or} \; (a=1 \, \text{and} \; b \in B)\},
\]
$2^a(2b+1) \sim_C 2^c(2d+1)$ if ($a=c=0$ and $b \sim_A d$) or ($a=c=1$ and $b \sim_B d$) and $i_0$ and $i_1$ are defined by  $i_0(n)=2^0(2n+1)=2n+1$ and $i_1(n)=2(2n+1)$. Notice that both $i_0$ and $i_1$ are computable functions and hence live in $\mathbf{Rec}$.

For the exponential object $[(A, \sim_A), (B, \sim_B)]$, we use $(C, \sim_C)$, together with the map $ev: C \times A \to B$, where 
\[
C=\{e \in \mathbb{N} \mid \forall n \in A \; (e \cdot n \in B) \, \& \, \forall mn \in A \, (m \sim_A n \to e \cdot m \sim_B e \cdot n)\},
\] 
$e \sim_C f$ iff $e \cdot n \sim_B f \cdot n$, for any $n \in A$ and $ev(2^e(2a+1))=e \cdot a$, where $e \cdot a$ is the output of a universal machine applying the code $e$ of a Turing machine to the input $a$. Recall that for any partial computable map $f: \mathbb{N} \rightharpoonup \mathbb{N}$, by the $S_{mn}$-theorem, if $e$ is one of the codes for the function $(m, n) \mapsto f((2^m(2n+1))$, we have $S(e, m) \cdot n=f((2^m(2n+1))$, for any $m, n \in \mathbb{N}$. Now, for any class $[f]: (D, \sim_D) \times (A, \sim_A) \to (B, \sim_B)$, it is enough to define $\lambda_A [f]: (D, \sim_D) \to (C, \sim_C)$ as the class of the computable function $d \mapsto S(e, d)$. One can verify that this map is well-defined, equivariant, and meets the required property in the definition of the exponential object. Moreover, it is unique with respect to this property.

\subsection{Natural Numbers Object} \label{SubsecNNO}
In this subsection, we will introduce a categorical version of the set of natural numbers and represent numbers and numeral functions, i.e., the multi-variable functions over the natural numbers, by this categorical gadget. 
\begin{dfn}(\emph{Natural numbers object})
Let $\mathcal{C}$ be a $\mathrm{CCC}$. A \emph{natural numbers object} (an $\mathrm{NNO}$, for short) in $\mathcal{C}$ is an object $N$ together with maps $Z: 1 \to N$ and $s: N \to N$ such that for any object $A$ and any maps $a: 1 \to A$ and $f: A \to A$, there exists a unique map $g: N \to A$ such that the following diagram commutes:
\[\begin{tikzcd}
	1 && N && N \\
	\\
	&& A && A
	\arrow["s", from=1-3, to=1-5]
	\arrow["f"', from=3-3, to=3-5]
	\arrow["g", dashed, from=1-5, to=3-5]
	\arrow["g"', dashed, from=1-3, to=3-3]
	\arrow["Z", from=1-1, to=1-3]
	\arrow["a"', from=1-1, to=3-3]
\end{tikzcd}\]
It is easy to see that any two natural numbers objects in a category are canonically isomorphic. For any maps $a: 1 \to A$ and $f: A \to A$, we denote the unique $g$ by $R(f, a)$. Therefore, we have $R(f, a)\circ Z=a$ and $R(f, a) \circ s=f \circ R(f, a)$. Intuitively, the map $R(f, a): N \to A$ is the map defined by iterating $f$ over the initial values $a$ and the equations are just the description of this recursive definition.
\end{dfn}

\begin{rem}\label{equivalentUniquenessNNO}
Despite what we had for products, coproducts and expoentials, it is not clear how to rewrite the uniqueness condition for an $\mathrm{NNO}$ as an equality. This can be problematic in constructing the free $\mathrm{CCC}$ with an $\mathrm{NNO}$. We will come back to this point later. 
\end{rem}

\begin{rem}
Notice that the universal definition of the natural numbers object roughly states that the diagram
\[\begin{tikzcd}[ampersand replacement=\&]
	1 \&\& N \&\& N
	\arrow["Z", from=1-1, to=1-3]
	\arrow["s", from=1-3, to=1-5]
\end{tikzcd}\]
is the ``least" between all such diagrams with the fixed node $1$.
\end{rem}

\begin{exam}
In $\mathbf{Set}$, the set $\mathbb{N}$ of natural numbers together with the morphism $Z: \{*\} \to \mathbb{N}$ mapping $*$ to $0$ and $s: \mathbb{N} \to \mathbb{N}$ mapping any number to its successor is an $\mathrm{NNO}$.
In a $\mathrm{CCC}$ poset $(P, \leq)$, the $\mathrm{NNO}$ is the terminal $1$ with the canonical map $!: 1 \to 1$ as both $Z$ and $s$.
In both $\mathbf{Set}^{\mathbb{N}}$ and $\mathbf{Set}^{\mathbb{Z}}$, the pair $(\mathbb{N}, id_{\mathbb{N}})$ together with the same maps as before is an $\mathrm{NNO}$. Note that both the maps $Z$ and $s$ preserve the identity and hence are equivariant. Moreover, note that for any $a: (\{*\}, id_{\{*\}}) \to (A, \sigma_A)$ and $f: (A, \sigma_A) \to (A, \sigma_A)$, it is easy to see that $\sigma_A(f^n(a(*)))=f^n(a(*))$. Therefore, the map $g: \mathbb{N} \to A$ defined by $g(n)=f^n(a(*))$ is actually a map from the dynamic system $(\mathbb{N}, id_{\mathbb{N}})$ to the dynamic system $(A, \sigma_A)$. It is clear that this $g$ is the unique $R(f, a)$ we are looking for. More generally, the constant functor $\Delta_{\mathbb{N}}: \mathcal{C} \to \mathbf{Set}$ together with $\Delta_Z$ and $\Delta_s$ is an NNO in $\mathbf{Set}^{\mathcal{C}}$.
In $\mathbf{Equ}$, the pair $(\mathbb{N}, =_{\mathbb{N}})$ along with the set functions $Z$ and $s$ is an $\mathrm{NNO}$, where $\mathbb{N}$ is the discrete space of natural numbers. Note that both $Z$ and $s$ are 
continuous and equivariant as they preserve the equivalence relation. Moreover, for any $[x]: (\{*\}, =_{\{*\}}) \to (X, \sim_X)$ and $[f]: (X, \sim_X) \to (X, \sim_X)$, it is easy to see that the class of the map $g: \mathbb{N} \to X$ defined by $g(n)=f^n(a(*))$ is independent of the choice of $f$ and $a$. The map $g$ is clearly a continuous equivariant map from $(\mathbb{N}, =_{\mathbb{N}})$ to $(X, \sim_X)$. It is clear that the class of $g$ is the unique $R(f, a)$ we are looking for. The same construction applies to $\mathbf{Rec}$, with the important point being that the class of all computable functions is closed under primitive recursion, allowing the definition of $g$ from $a$ and $f$.
\end{exam}

For clarity in presenting some arguments, it is useful to consider the category $\mathbf{FinOrd}$, which consists of natural numbers as objects and functions between finite sets as morphisms. Specifically, an object in $\mathbf{FinOrd}$ is a natural number $m \in \mathbb{N}$, and a morphism from $m$ to $n$ is a function from the set $\{0, 1, \ldots, m-1\}$ to the set $\{0, 1, \ldots, n-1\}$. The identity and composition of morphisms are defined in the standard way.
The category $\mathbf{FinOrd}$ is a $\mathrm{BCC}$, with the terminal object being $1$, the initial object being $0$, and the product, coproduct, and exponential operations corresponding to the usual operations on natural numbers. Essentially, $\mathbf{FinOrd}$ provides a simpler version of the category $\mathbf{FinSet}$, where we work with the cardinalities of finite sets rather than the sets themselves.
In fact, by choosing an arbitrary bijection $\alpha_A: A \to \{0, 1, \ldots, Card(A)-1\}$ for any finite set $A$, we can define a functor $Card: \mathbf{FinSet} \to \mathbf{FinOrd}$. This functor maps a finite set $A$ to its cardinality $Card(A)$ and a function $f: A \to B$ to the function $\alpha_B \circ f \circ \alpha_A^{-1}: \{0, 1, \ldots, Card(A)-1\} \to \{0, 1, \ldots, Card(B)-1\}$. The functor $Card$ is a $\mathrm{BC}$-functor.

\begin{exam}(\textit{Non-existence of the $\mathrm{NNO}$}) \label{NonExistenceNNOFinSet}
The category $\mathbf{FinOrd}$ has no $\mathrm{NNO}$. For the sake of contradiction, let the number $n$ together with the maps $Z: \{0\} \to \{0, \ldots, n-1\}$ and $s: \{0, \ldots, n-1\} \to \{0, \ldots, n-1\}$ be an $\mathrm{NNO}$. Then, as $\{0, \ldots, n-1\}$ is finite, there are distinct natural numbers $p < q \in \mathbb{N}$ such that $s^p Z=s^q Z$. Now, consider the number $q+1$ as an object in $\mathbf{FinOrd}$ and the maps $a: \{0\} \to \{0, \ldots, q\}$ picking the element $0$ in $\{0, \ldots, q\}$ and $f: \{0, \ldots, q\} \to \{0, \ldots, q\}$ defined as the successor function, except that for $q$, we define $f(q)=0$. Then, as $(n, Z, s)$ is an $\mathrm{NNO}$, there is a map $g: n \to q+1$ such that $gs^kZ(0)=k$, for any natural number $k \leq q$. Therefore, as $gs^pZ=gs^qZ$, we have $p=q$ which is impossible. Note that we only used the existence of the map $R(f, a)$ to reach a contradiction and not its uniqueness.
\end{exam}

\begin{exam}\label{NonExistenceOfNNONJ}
Let $\mathcal{C}$ be a $\mathrm{BCC}$ such that there is a $\mathrm{BC}$-functor $F: \mathcal{C} \to \mathbf{FinOrd}$.
We claim that $\mathcal{C}$ has no $\mathrm{NNO}$.
First, note that $F$ is surjective on objects, because $F$ maps $\sum_{i=0}^{n-1} 1$ in $\mathcal{C}$ to $n$ in $\mathbf{FinOrd}$. Moreover, $F$
is a full functor as for any map $f: \{0, \ldots, m-1\} \to \{0, \ldots, n-1\}$ in $\mathbf{FinOrd}$, there is a map $g: \sum_{i=0}^{m-1} 1 \to \sum_{j=0}^{n-1} 1$ in $\mathcal{C}$ that uses the injections to mimic the behavior of $f$ and hence $F(g)=f$.
Using these two facts, it is easy to see that if $(N, Z, s)$ is an $\mathrm{NNO}$ in $\mathcal{C}$, then $(F(N), F(Z), F(s))$ in $\mathbf{FinOrd}$ has all the properties of an $\mathrm{NNO}$ except probably its uniqueness. However, the existence of such data is impossible as observed in Example \ref{NonExistenceNNOFinSet}. Therefore, $\mathcal{C}$ has no $\mathrm{NNO}$. One can use this observation to prove that the categories $\mathbf{FinSet}$ and $\mathbf{NJ}$ have no $\mathrm{NNO}$. For the former, use the $\mathrm{BC}$-functor $Card: \mathbf{FinSet} \to \mathbf{FinOrd}$. For the latter, by mapping the atoms $p_i$ to some numbers, we can use Theorem \ref{Freness} to find a $\mathrm{BC}$-functor $F: \mathbf{NJ} \to \mathbf{FinOrd}$. 
\end{exam}

In any $\mathrm{CCC}$ equipped with an $\mathrm{NNO}$ $(N,Z,s)$, there is a standard way to represent natural numbers and \emph{some} of numeral functions. For the former, the natural number $n$ is represented by the map $\bar{n}=s^n\circ Z:1 \to N$, iterating the successor map for $n$ many times. For the latter, we have:
\begin{dfn}\label{DefRepresentableFunctions}
Let $\mathcal{C}$ be a CCC with the NNO $(N, Z, s)$. A numeral function $f: \mathbb{N}^k \to \mathbb{N}$ is called \emph{representable} in $\mathcal{C}$ if there is a map $F: N^k \to N$ in $\mathcal{C}$ such that for any $(n_1, \ldots, n_k) \in \mathbb{N}^k$, we have
\[\begin{tikzcd}
	1 &&& {N^k} \\
	\\
	\\
	N
	\arrow["{\langle \bar{n}_1, \ldots, \bar{n}_k \rangle}", from=1-1, to=1-4]
	\arrow["F", from=1-4, to=4-1]
	\arrow["{\overline{f(n_1, \ldots, n_k)}}"', from=1-1, to=4-1]
\end{tikzcd}\]
\end{dfn}

\begin{exam}\label{AllBasicFunctions}
Let $\mathcal{C}$ be a $\mathrm{CCC}$ with the $\mathrm{NNO}$ $(N, Z, s)$. It is clear that the function $c_0: \mathbb{N}^0 \to \mathbb{N}$ mapping its one and only input to zero is representable by the map $Z: 1 \to N$. Similarly, the successor function is representable by the map $s: N \to N$. The projection function $I^k_i: \mathbb{N}^k \to \mathbb{N}$ mapping a $k$-tuple to its $i$-th argument is representable by the projection map $p_i: \prod_{i=0}^{k-1} N \to N$. For a less trivial example, consider the map $f: N \to N$ as $f=R(s \circ s, Z)$:
\[\begin{tikzcd}[ampersand replacement=\&]
	1 \&\& N \&\& N \&\& N \\
	\\
	\&\& N \&\& N \&\& N
	\arrow["Z", from=1-1, to=1-3]
	\arrow["Z"', from=1-1, to=3-3]
	\arrow["f", dashed, from=1-3, to=3-3]
	\arrow["s", from=1-5, to=1-7]
	\arrow["{s \circ s}"', from=3-5, to=3-7]
	\arrow["f", dashed, from=1-7, to=3-7]
	\arrow["f"', dashed, from=1-5, to=3-5]
\end{tikzcd}\]
It is easy to see that for any $n \in \mathbb{N}$, the following diagram commutes:
\[\begin{tikzcd}
	1 && N && N \\
	\\
	&& N && N
	\arrow["{s^n}", from=1-3, to=1-5]
	\arrow["{s^{2n}}"', from=3-3, to=3-5]
	\arrow["f", dashed, from=1-5, to=3-5]
	\arrow["f"', dashed, from=1-3, to=3-3]
	\arrow["Z", from=1-1, to=1-3]
	\arrow["Z"', from=1-1, to=3-3]
\end{tikzcd}\]
Therefore, the map $f: N \to N$ represents the numeral function $d: \mathbb{N} \to \mathbb{N}$ defined by $d(n)=2n$.
\end{exam}

\begin{exam}\label{ExampleOfRepresentables}
In $\mathbf{Set}$, $\mathbf{Set}^{\mathcal{C}}$, and $\mathbf{Equ}$, all numeral functions are representable. 
In $\mathbf{Rec}$, however, the representable functions are the total computable functions.
\end{exam}

In $\mathbf{Set}$, we often use the inductive nature of $\mathbb{N}$ to define new functions from the existing ones by what is called \emph{primitive recursion}. More precisely, given functions $g: X \to A$ and $h: \mathbb{N} \times X \times A \to A$, for some sets $X$ and $A$, we can use recursion on $n$ to define a function $f: \mathbb{N} \times X \to A$ satisfying $f(0, x) = g(x)$ and $f(n+1, x) = h(n, x, f(n, x))$. A similar situation occurs for any $\mathrm{NNO}$ in any $\mathrm{CCC}$. Here, the inductive nature of $\mathbb{N}$ is mimicked by the universality of the $\mathrm{NNO}$.

\begin{lem}(Primitive recursion) \label{PrimitiveRecursion}
Let $\mathcal{C}$ be a $\mathrm{CCC}$ with the $\mathrm{NNO}$ $(N, Z, s)$. Then, for any map $g: X \to A$ and any $h: N \times X \times A \to A$, there exists a unique $f: N \times X \to A$ such that:
\[\begin{tikzcd}
	X && {N \times X} & {N \times X} && {N \times X} \\
	\\
	&& A & {N \times X \times A} && A
	\arrow["{s \times id_X}", from=1-4, to=1-6]
	\arrow["h"', from=3-4, to=3-6]
	\arrow["f", dashed, from=1-6, to=3-6]
	\arrow["{\langle id_{N \times X}, f \rangle}"', dashed, from=1-4, to=3-4]
	\arrow["{\langle Z \circ !, id_X \rangle}", from=1-1, to=1-3]
	\arrow["g"', from=1-1, to=3-3]
	\arrow["f", dashed, from=1-3, to=3-3]
\end{tikzcd}\]
\end{lem}
\begin{proof}
Define $G: 1 \to (N \times A^X)$ as $\langle Z, \lambda_X g \rangle$ and $H': N \times A^X \times X \to A$ as $h \langle p_0, p_2, ev \langle p_1, p_2 \rangle \rangle$ and set 
$H: (N \times A^X) \to (N \times A^X)$ as $\langle sp_0, \lambda_X H' \rangle$.
By the universality of the $\mathrm{NNO}$, there exists a unique $F: N \to N \times A^X$ such that:
\[\begin{tikzcd}[ampersand replacement=\&]
	1 \&\& N \& N \&\& N \\
	\\
	\&\& {N \times A^X} \& {N \times A^X} \&\& {N \times A^X}
	\arrow["s", from=1-4, to=1-6]
	\arrow["H"', from=3-4, to=3-6]
	\arrow["F", dashed, from=1-6, to=3-6]
	\arrow["F"', dashed, from=1-4, to=3-4]
	\arrow["Z", from=1-1, to=1-3]
	\arrow["G"', from=1-1, to=3-3]
	\arrow["F", dashed, from=1-3, to=3-3]
\end{tikzcd}\]
Define $f: N \times X \to A$ by $ev \langle p^{N, A^X}_1Fp^{N, X}_0, p^{N, X}_1 \rangle $. It is easy to see that this $f$ satisfies the claimed property and it is unique in that property.
\end{proof}

\begin{rem}
Note that in the above proof, we used exponentiation to manage the parameter object $X$. If the category is cartesian but not cartesian closed, one must modify the definition of the $\mathrm{NNO}$ to a stronger form, known as a \emph{parametrized $\mathrm{NNO}$}, which directly incorporates the parameter $X$ into the definition.
\end{rem}

\begin{thm}
All primitive recursive functions are representable in any $\mathrm{CCC}$ with an $\mathrm{NNO}$.
\end{thm}
\begin{proof}
As observed in Example \ref{AllBasicFunctions}, all basic functions are clearly representable. It is also easy to see that the composition of representable functions is representable. For primitive recursion, it is sufficient to use Lemma \ref{PrimitiveRecursion}.
\end{proof}

It is clear that in a poset $\mathrm{CCC}$, all numbers are represented with the map $!: 1 \to 1$. This situation is annoying as we expect to represent numbers in an injective way. The next lemma shows that this poset case is actually the only problematic case.

\begin{lem}\label{UniqunessofNumeralRepresentation}
Let $\mathcal{C}$ be a non-preorder $\mathrm{CCC}$ with an $\mathrm{NNO}$. If $\bar{m}=\bar{n}$, then $m=n$, for any $m, n \in \mathbb{N}$.
\end{lem}
\begin{proof}
Let $pd: N \to N$ be a representation for the predecessor function in $\mathcal{C}$ and assume $\bar{m}=\bar{n}$, for some $m > n$. By composing $pd$ with $\bar{m}$ and $\bar{n}$ for $n$ many times, we have $\overline{m-n}=\bar{0}$. Therefore, it is enough to prove that $\bar{k} \neq \bar{0}$, for any $k \neq 0$. Let $f, g: A \to B$ be two different maps between two objects in $\mathcal{C}$. Define $h: N \to [A, B]$ such that:
\[\begin{tikzcd}
	1 && N && N \\
	\\
	&& {[A, B]} && {[A, B]}
	\arrow["s", from=1-3, to=1-5]
	\arrow["{(\lambda g) \circ !}"', from=3-3, to=3-5]
	\arrow["h"', dashed, from=1-3, to=3-3]
	\arrow["Z", from=1-1, to=1-3]
	\arrow["{\lambda f}"', from=1-1, to=3-3]
	\arrow["h", dashed, from=1-5, to=3-5]
\end{tikzcd}\]
It is easy to see that $h \circ \bar{0}=\lambda f$ while $h \circ \bar{k}=\lambda g$, for any $k > 0$. Therefore, if $\bar{k}=\bar{0}$, by composition with $h$, we reach $\lambda f=\lambda g$ which implies $f=g$. This is a contradiction.
\end{proof}

In a similar way to what we had for $\mathbf{NJ}$, it is also possible to come up with the ``least" $\mathrm{CCC}$ with an $\mathrm{NNO}$, where every object is constructed from $N$ and $1$ by using product and exponentiation, and any morphism is constructed from the very basic morphisms such as $!$, $p_0$, $p_1$, $ev$, $Z$, $s$, and the operations $\langle -, - \rangle$, $\lambda$, and $R$. To state a formal version of this claim, we need the following definition:

\begin{dfn}
Let $\mathcal{C}$ and $\mathcal{D}$ be two $\mathrm{CCC}$'s and $F: \mathcal{C} \to \mathcal{D}$ be a $\mathrm{CC}$-functor. We say that $F$ \emph{preserves the $\mathrm{NNO}$}, if for any $\mathrm{NNO}$
\[\begin{tikzcd}[ampersand replacement=\&]
	1 \&\& N \&\& N
	\arrow["Z", from=1-1, to=1-3]
	\arrow["s", from=1-3, to=1-5]
\end{tikzcd}\]
in $\mathcal{C}$, the diagram
\[\begin{tikzcd}[ampersand replacement=\&]
	{F(1)} \&\& {F(N)} \&\& {F(N)}
	\arrow["{F(Z)}", from=1-1, to=1-3]
	\arrow["{F(s)}", from=1-3, to=1-5]
\end{tikzcd}\] 
is an $\mathrm{NNO}$ in $\mathcal{D}$.
\end{dfn}

\begin{thm}(Free $\mathrm{CCC}$ with an $\mathrm{NNO}$)
There exists a free $\mathrm{CCC}$ with an $\mathrm{NNO}$, i.e., a $\mathrm{CCC}$ $\mathcal{C}$ with an $\mathrm{NNO}$ such that for any $\mathrm{CCC}$ $\mathcal{D}$ with an $\mathrm{NNO}$, there is a unique (up to isomorphism) $\mathrm{CC}$-functor $F: \mathcal{C} \to \mathcal{D}$, preserving the $\mathrm{NNO}$. We denote this $\mathrm{CCC}$ by $\Lambda_N$.
\end{thm}
\begin{proof}
As the conditions of being an $\mathrm{NNO}$ go beyond simple equalities, we cannot apply the usual method of generating all required objects and maps and then using a quotient to form the equivalence relation generated by the equalities. To address this, we should weaken the notion of an $\mathrm{NNO}$ to something representable by equalities, construct the free $\mathrm{CCC}$ with that weak $\mathrm{NNO}$, and then demonstrate that the weak $\mathrm{NNO}$ is actually an $\mathrm{NNO}$ in that category. For more details on this strategy, see \cite{LambekNNO}.
\end{proof}

\begin{rem}\label{LambdaNIsNotPreorder}
The category $\Lambda_N$ is not a preorder. To prove this, we assume otherwise. Thus, the two maps $Z: 1 \to N$ and $sZ: 1 \to N$ must be equal. Use the freeness of $\Lambda_N$ to get a $\mathrm{CC}$-functor $F: \Lambda_N \to \mathbf{Set}$ preserving the $\mathrm{NNO}$. Hence, $F(N) \cong \mathbb{N}$. As $Z = sZ$ in $\Lambda_N$, we have $F(Z) = F(s)F(Z): \{*\} \to F(N)$, which implies $0 = 1$ in $\mathbb{N}$, leading to a contradiction. Hence, $\Lambda_N$ is not a preorder.
\end{rem}

\begin{dfn}
A numerical function is called a \emph{primitive recursive functional} if it is representable in $\Lambda_N$.
\end{dfn}

It is reasonable to claim that the numerical primitive recursive functionals are exactly the functions that are representable in \emph{all} $\mathrm{CCC}$'s with an $\mathrm{NNO}$, simply because $\Lambda_N$ is the ``smallest" of such categories. This claim can be proved using the following lemma:

\begin{lem}\label{TransferOfRepresentability}
Let $\mathcal{C}$ and $\mathcal{D}$ be two $\mathrm{CCC}$'s with an $\mathrm{NNO}$ and $G: \mathcal{C} \to \mathcal{D}$ is a $\mathrm{CC}$-functor preserving the $\mathrm{NNO}$. Then, if the map $F: N_{\mathcal{C}}^k \to N_{\mathcal{C}}$ in $\mathcal{C}$ represents the function $f: \mathbb{N}^k \to \mathbb{N}$, then $G(F):N_{\mathcal{D}}^k \to N_{\mathcal{D}}$ in $\mathcal{D}$ also represents $f$. 
\end{lem}
\begin{proof}
As $F$ represents $f$, for any natural number $n \in \mathbb{N}$, we have $F \circ \bar{n}_{\mathcal{C}}=\overline{f(n)}_{\mathcal{C}}$. As $G$ is a CC-functor respecting the $\mathrm{NNO}$, we have $G(\bar{n}_{\mathcal{C}})=\bar{n}_{\mathcal{D}}$. Therefore, by applying $G$ on $F \circ \bar{n}_{\mathcal{C}}=\overline{f(n)}_{\mathcal{C}}$, we get $G(F) \circ \bar{n}_{\mathcal{D}}=\overline{f(n)}_{\mathcal{D}}$.
\end{proof}

\begin{cor}\label{RepresentableEverywhere}
A numerical function is a primitive recursive functional iff it is representable in all $\mathrm{CCC}$'s with an $\mathrm{NNO}$.
\end{cor}
\begin{proof}
One direction is clear as $\Lambda_N$ is itself a $\mathrm{CCC}$ with an $\mathrm{NNO}$. For the other direction, let $\mathcal{C}$ be a $\mathrm{CCC}$ with an $\mathrm{NNO}$. Define $G$ as the unique $\mathrm{CC}$-functor from $\Lambda_N$ to $\mathcal{C}$ that preserves the $\mathrm{NNO}$. By Lemma \ref{TransferOfRepresentability}, it follows that if a function is representable in $\Lambda_N$, it is also representable in $\mathcal{C}$.

\end{proof}

In the remainder of this subsection, we aim to evaluate the computational power of numerical primitive recursive functionals. To begin, we demonstrate that all such functionals are computable.

\begin{thm}\label{FunctionalsAreRecursive}
All numerical primitive recursive functionals are recursive.   
\end{thm}
\begin{proof}
Let $f: \mathbb{N}^k \to \mathbb{N}$ be a primitive recursive functional. Then, by Corollary \ref{RepresentableEverywhere}, $f$ is representable in $\mathbf{Rec}$. Hence, it is recursive, by Example \ref{ExampleOfRepresentables}. 
\end{proof}

\begin{rem}
Having established Theorem \ref{FunctionalsAreRecursive}, one might assume that all total recursive functions are representable in $\Lambda_N$. However, this is not the case. The reason is tied to the well-known informal argument that no formal system can capture all total recursive functions. For a formal proof, see \cite{LambekScott}.
\end{rem}

We have shown that all primitive recursive functions are representable in any $\mathrm{CCC}$ with an $\mathrm{NNO}$. Consequently, these functions are all primitive recursive functionals. However, the power of $\Lambda_N$ extends well beyond that of primitive recursion. This is because $\Lambda_N$ supports primitive recursion not only on numerical functions but also on functionals.
For example, consider the Ackermann function $A(m, n)$, defined by $A(0, n)=n+1$ and $A(m+1, n)=A_m^{n+1}(1)$, where $A_m$ is just the function $A(m, -): \mathbb{N} \to \mathbb{N}$. This function is \emph{not} primitive recursive. However, it is a primitive recursive functional as one can use a higher-order recursion on $m$ to define $A_m$ using the equations $A_0 = s$ and $A_{m+1} = \lambda n. A_m^{n+1}(1)$.

To formally prove the representability of the Ackermann function in $\Lambda_N$, we first consider the function $iter: \mathbb{N} \times \mathbb{N}^{\mathbb{N}} \to \mathbb{N}$ that maps $(n, f)$ to $f^{n+1}(1)$. Using the recursive definition of $iter$, which is given by the equations $iter(0, f)=f(1)$ and $iter(n+1, f)=f(iter(n, f))$, and applying Lemma \ref{PrimitiveRecursion}, we can construct a map $i: N \times N^N \to N$ in $\Lambda_N$ that simulates the behavior of $iter$. More precisely, set $X=N^N$ and $A=N$ in Lemma \ref{PrimitiveRecursion} and define $g: N^N \to N$ by $ev \langle id_{N^N}, \bar{1} \circ ! \rangle$ and $h: N \times N^N \times N \to N$ by $ev \langle p_1, p_2 \rangle$. By Lemma \ref{PrimitiveRecursion}, there is a map $i: N \times N^N \to N$ such that:
\[\begin{tikzcd}
	{N^N} && {N \times N^N} && {N \times N^N} && {N \times N^N} \\
	\\
	&& N && {N \times N^N \times N} && N
	\arrow["{\langle Z \circ !, id_{N^N} \rangle}", from=1-1, to=1-3]
	\arrow["{ev \langle id_{N^N}, \bar{1} \circ !\rangle}"', from=1-1, to=3-3]
	\arrow["i", dashed, from=1-3, to=3-3]
	\arrow["{s \times id_{N^N}}", from=1-5, to=1-7]
	\arrow["{\langle id_{N \times N^N}, i \rangle}"', dashed, from=1-5, to=3-5]
	\arrow["i", dashed, from=1-7, to=3-7]
	\arrow["{ev \langle p_1, p_2 \rangle}"', from=3-5, to=3-7]
\end{tikzcd}\]
Now, define $F: N \to N^N$ as a map induced by the universality of the $\mathrm{NNO}$:
\[\begin{tikzcd}
	1 && N &&& N \\
	\\
	&& {N^N} &&& {N^N}
	\arrow["Z", from=1-1, to=1-3]
	\arrow["{\lambda_N s}"', from=1-1, to=3-3]
	\arrow["s", from=1-3, to=1-6]
	\arrow["F"', dashed, from=1-3, to=3-3]
	\arrow["F", dashed, from=1-6, to=3-6]
	\arrow["{\lambda_N \, (i \circ (p_1 \times p_0))}"', from=3-3, to=3-6]
\end{tikzcd}\]
It is easy to see that $F: N \to N^N$ mimics the definition of $A_{-}$. Therefore, the map $ev \langle F, id_N \rangle: N \times N \to N$ represents the Ackermann function in $\Lambda_N$.

\subsection{Realizability} \label{SubsecRealForArithmetic}
From the introduction of Subsection \ref{SecArithmetic}, recall that the language of arithmetic does not include the connectives $\vee$ and $\bot$. Therefore, it suffices to work with $\mathrm{CCC}$'s rather than $\mathrm{BCC}$'s for its interpretation. Furthermore, since we aim to realize arithmetical proofs, it is natural to assume that the $\mathrm{CCC}$ has an $\mathrm{NNO}$. Thus, in this subsection, we assume that $\mathcal{C}$ is a $\mathrm{CCC}$ equipped with the $\mathrm{NNO}$  $(\mathtt{N}, \mathtt{Z}, \mathtt{s})$. 
For any such $\mathcal{C}$, there is a canonical $\mathcal{C}$-interpretation called the \emph{standard $\mathcal{C}$-interpretation} that captures the intended meaning of the arithmetical proofs. Formally, by the \emph{standard $\mathcal{C}$-interpretation} $\mathfrak{S}$ of the language of $\mathsf{HA}^{\omega}$, we mean a $\mathcal{C}$-interpretation that maps:
\begin{itemize}
    \item 
the basic type $N$ to the $\mathcal{C}$-set $N=(\mathtt{N}, \mathbb{N}, \Vdash_N)$, where $\mathtt{n} \Vdash_N n$, for any natural number $n \in \mathbb{N}$ in which $\mathtt{n}: \mathtt{1} \to \mathtt{N}$ is the representation of the natural number $n$,
    \item 
the function symbol $0$ to the $\mathcal{C}$-set map $0^{\mathfrak{S}}: 1 \to N$ picking the element $0 \in \mathbb{N}$. Note that this map is realized by $\mathtt{Z}: \mathtt{1} \to \mathtt{N}$,
    \item 
the function symbol $S(x^N)$ of type $N$ to the $\mathcal{C}$-set map $S^{\mathfrak{S}}: N \to N$ defined by $S^{\mathfrak{S}}(n)=n+1$. Note that this map is realized by $\mathtt{s}: \mathtt{N} \to \mathtt{N}$,  
\item 
the function symbol $\mathbf{R}^{\sigma}(x^{\sigma}, y^{N \times \sigma \to \sigma}, z^{N})$ of type $\sigma$ to the $\mathcal{C}$-set map $(\mathbf{R}^{\sigma})^{\mathfrak{S}}: M^{\sigma} \times [N \times M^{\sigma}, M^{\sigma}] \times N \to M^{\sigma}$ defined recursively by 
\[
(\mathbf{R}^{\sigma})^{\mathfrak{S}}(a, f, 0)=a \quad \text{and} \quad (\mathbf{R}^{\sigma})^{\mathfrak{S}}(a, f, n+1)=f(n, (\mathbf{R}^{\sigma})^{\mathfrak{S}}(a, f, n)).
\]
Using Lemma \ref{PrimitiveRecursion}, it is easy to see that this map is realized by a map in $\mathcal{C}$ constructed via a primitive recursion. 
\end{itemize}

If the basic type $2$ is also present in the language and the coproduct $\mathtt{2}=\mathtt{1}+\mathtt{1}$ exists in $\mathcal{C}$, the standard $\mathcal{C}$-interpretation must map:
\begin{itemize}
    \item 
the type $2$ to the $\mathcal{C}$-set $(\mathtt{2}, \{0, 1\}, \Vdash_2)$, where $\mathtt{i_0} \Vdash_2 0$ and $\mathtt{i_1} \Vdash_2 1$,
    \item 
the constants $0_2$ and $1_2$ of type $2$ to the $\mathcal{C}$-set maps $0_2^{\mathfrak{S}}: 1 \to 2$ and $1_2^{\mathfrak{S}}: 1 \to 2$ picking the elements $0$ and $1$ in $\{0, 1\}$, respectively. Note that these functions are realized by the maps $\mathtt{i_0}: \mathtt{1} \to \mathtt{2}$ and $\mathtt{i_1}: \mathtt{1} \to \mathtt{2}$, respectively, 
    \item 
the function symbol $c^{\sigma}(x^{\sigma}, y^{\sigma}, z^2)$ to the $\mathcal{C}$-set map $(c^{\sigma})^{\mathfrak{S}}: M^{\sigma} \times M^{\sigma} \times 2 \to M^{\sigma}$ defined by $(c^{\sigma})^{\mathfrak{S}}(a, b, 0)=a$ and $(c^{\sigma})^{\mathfrak{S}}(a, b, 1)=b$. It is easy to see that the function $(c^{\sigma})^{\mathfrak{S}}$ is realized by a map in $\mathcal{C}$. 
\end{itemize}

In the rest of this subsection, we will use only the standard $\mathcal{C}$-interpretation for realizability. Thus, we will omit the subscript $\mathfrak{S}$ in $\Vdash_{\mathfrak{S}}$ everywhere. Additionally, by a sentence, we always mean an $\mathcal{L}(\mathfrak{S})$-sentence.

\begin{rem}\label{RemarksOnReal}
Here are some remarks. First, as the language does not have the connective $\bot$, one can use induction to prove that for any sentence $A$, the set $\mathrm{Hom}_{\mathcal{C}}(\mathtt{1}, |A|)$ is non-empty. Second, for any $\exists$-free sentence, i.e., a sentence $A$ with no existential quantifiers, the object $|A|$ is always a terminal object. Therefore, there is only one map in $\mathrm{Hom}_{\mathcal{C}}(\mathtt{1}, |A|)$. Third, as we interpreted $\bot$ as $S(0)=_N 0$ and $|S(0)=_N 0|=\mathtt{1}$, we have $|\neg B|=|B \to \bot|=[|B|, \mathtt{1}] \cong \mathtt{1}$. Therefore, for any negative sentence $A$, we also have $|A| \cong \mathtt{1}$. One can read this property as an incarnation of the usual intuition that these sentences have no computational content because such sentences can have at most one realizer, encoding whether they are provable or not.
\end{rem}

\begin{exam}\label{RealizabilityForNegation}
(\emph{Realizability for negative formulas})  
For any formula $A(x^{\sigma})$, we have $\Vdash \neg A(x^{\sigma})$ iff  $\nVdash A(a)$, for any $a \in |\!|M^{\sigma}|\!|$. For the forward direction, let $\mathtt{r} \Vdash \neg A(x^{\sigma})$ and $a \in |\!|M^{\sigma}|\!|$. There is $\mathtt{a}: \mathtt{1} \to |M^{\sigma}|$ such that $\mathtt{a} \Vdash_{M^{\sigma}} a$. Hence, $\mathtt{r} \circ \mathtt{a} \Vdash \neg A(a)$. If $\mathtt{s} \Vdash A(a)$, we reach $(\mathtt{r} \circ \mathtt{a}) \cdot \mathtt{s} \Vdash S(0)=0$ which implies $S(0)=0$. Hence, $\nVdash A(a)$. For the backward direction, assume $\nVdash A(a)$, for any $a \in |\!|M^{\sigma}|\!|$. Consider $! : |M^{\sigma}| \to |\neg A| \cong \mathtt{1}$. We claim that $! \Vdash \neg A(x^{\sigma})$. To prove that, we have to show that for any $\mathtt{a}: \mathtt{1} \to |M^{\sigma}|$, any $a \in |\!|M^{\sigma}|\!|$ and any $\mathtt{s}: \mathtt{1} \to |A|$, if $\mathtt{s} \Vdash A(a)$ then $(! \circ \mathtt{a}) \cdot \mathtt{s} \Vdash S(0)=_N 0$. As there is no $\mathtt{s}: \mathtt{1} \to |A|$ satisfying $\mathtt{s} \Vdash A(a)$, we get $! \Vdash \neg A(x^{\sigma})$ which complete the proof. A similar thing is true for $\Vdash A(\bar{x})$ for a set $\bar{x}$ of variables. Especially, if $A$ is a \emph{sentence}, then  $\Vdash \neg A$ iff $\nVdash A$. 
\end{exam}

\begin{rem}\label{OnRealizabilityOfDoubleNeg}(\emph{Double negation as a collapsing operator})
Using Example \ref{RealizabilityForNegation}, one can see that $\Vdash \neg \neg A(x^{\sigma})$ iff $\Vdash A(a)$, for any $a \in |\!|M^{\sigma}|\!|$. A similar thing is true for $\Vdash \neg \neg A(\bar{x})$ for a set $\bar{x}$ of variables. Notice the difference between $\Vdash \neg \neg A(x^{\sigma})$ and $\Vdash A(x^{\sigma})$. In the latter, one must come up with a map $\mathtt{r}: |M^{\sigma}| \to |A|$ such that for any $\mathtt{a}:  \mathtt{1} \to |M^{\sigma}|$ and any $a \in |\!|M^{\sigma}|\!|$ satisfying $\mathtt{a} \Vdash_{M^{\sigma}} a$, we have $\mathtt{r} \circ \mathtt{a} \Vdash A(a)$. This simply means the existence of a \emph{uniform} way to find a realizer for $A(a)$ for $a \in |\!|M^{\sigma}|\!|$. In this sense, one can say that $\Vdash \neg \neg A(x^{\sigma})$ is actually the non-uniform version of $\Vdash A(x^{\sigma})$, where we only care about the existence of the realizer and not its uniform construction. This non-uniformity is exactly what we get from the double negation operator. We discussed this role of double negation before when we talked about the axioms $\mathsf{CT}$ and $\mathsf{WC\!-\!N}$ in Section \ref{SectionTheories}.
\end{rem}

\begin{exam}\label{RealizabilityForDisjunction}
(\emph{Realizability for disjunctive formulas}) $\Vdash A(x^{\sigma}) \vee B(x^{\sigma})$ iff there are maps $\mathtt{r}: |M^{\sigma}| \to \mathtt{N}$, $\mathtt{s}: |M^{\sigma}| \to |A|$ and $\mathtt{t}: |M^{\sigma}| \to |B|$ such that for any $\mathtt{a} : 1 \to |M^{\sigma}|$ and any $a \in |\!|M^{\sigma}|\!|$, if $\mathtt{a} \Vdash_{M^{\sigma}} a$, either $\mathtt{r} \circ \mathtt{a}=\mathtt{0}$ and $\mathtt{s} \circ \mathtt{a} \Vdash A(a)$ or $\mathtt{r} \circ \mathtt{a} = \mathtt{n}$, for some $n \geq 1$ and $\mathtt{t} \circ \mathtt{a} \Vdash B(a)$. We only prove the forward direction and leave the rest to the reader. To prove it, substitute the formula $A(x^{\sigma}) \vee B(x^{\sigma})$ with its definition $\exists y^N [(y=0 \to A(x^{\sigma})) \wedge (y \neq 0 \to B(x^{\sigma}))]$ and let $\mathtt{u}$ be its realizer. Therefore, for any $\mathtt{a} : 1 \to |M^{\sigma}|$ and any $a \in |\!|M^{\sigma}|\!|$, if $\mathtt{a} \Vdash_{M^{\sigma}} a$, there exists $n \in \mathbb{N}$ such that $\mathtt{p_0(u \circ a)}\Vdash_N n$ and  $\mathtt{p_1(u \circ a)} \Vdash  [(n=0 \to A(a)) \wedge (n \neq 0 \to B(a))]$.
There are two cases. Either $n=0$ or $n \neq 0$. In the first case, as $\mathtt{p_0(u \circ a)} \Vdash_N n$, we have $\mathtt{p_0(u \circ a)}=\mathtt{0}$. Moreover, $\mathtt{p_0p_1(u \circ a)} \Vdash (n=0 \to A(a))$. As $n=0$, we have $! \Vdash n=0$. Hence, $\mathtt{p_0p_1(u \circ a)} \cdot ! \Vdash A(a)$. In the second case, as $\mathtt{p_0(u \circ a)} \Vdash_N n$ and $n \neq 0$, we have $\mathtt{p_0(u \circ a)} = \mathtt{n}$, for some $n \geq 1$. Moreover, $\mathtt{p_1p_1(u \circ a)} \Vdash (n \neq 0 \to B(a))$. As $n \neq 0$, we have $ \nVdash n = 0$ which implies $! \Vdash n \neq 0$, by Example \ref{RealizabilityForNegation}. Hence, $\mathtt{p_1p_1(u \circ a)} \cdot ! \Vdash B(a)$. Now, define $\mathtt{r}=\mathtt{p_0u}$, $\mathtt{s}=\mathtt{ev \langle p_0p_1u, ! \rangle}$ and $\mathtt{t}=\mathtt{ev \langle p_1p_1u, ! \rangle}$. Then, as $\mathtt{r \circ a}=\mathtt{p_0(u \circ a)}$, $\mathtt{s \circ a}=\mathtt{(p_0p_1(u \circ a)}) \cdot !$, and $\mathtt{t \circ a}=(\mathtt{p_1p_1(u \circ a)}) \cdot !$, it is clear that either $\mathtt{r} \circ \mathtt{a}=\mathtt{0}$ and $\mathtt{s} \circ \mathtt{a} \Vdash A(a)$ or $\mathtt{r} \circ \mathtt{a} = \mathtt{n}$ for some $n \geq 1$ and $\mathtt{t} \circ \mathtt{a} \Vdash B(a)$. A similar thing is true for $\Vdash A(\bar{x}) \vee B(\bar{x})$ for a set $\bar{x}$ of variables. 
\end{exam}

\begin{exam}(\emph{Incompleteness of realizability})
For any \emph{sentence} $A$, the intuitionistically invalid sentence $A \vee \neg A$ is realizable. To prove, note that there are two cases to consider. Either $\Vdash A$ or $\nVdash A$. In the first case, define $\mathtt{r}=\mathtt{0}$ and set $\mathtt{s}: \mathtt{1} \to |A|$ as one of the realizers of $A$. In the second case, set $\mathtt{r}=\mathtt{1}$ and $\mathtt{t}=!: \mathtt{1} \to |\neg A| \cong \mathtt{1}$. By Example \ref{RealizabilityForNegation}, as $\nVdash A$, we have $\mathtt{t} \Vdash \neg A$. Finally, by Example \ref{RealizabilityForDisjunction}, we have $\Vdash A \vee \neg A$. 
\end{exam}

So far, we have investigated the behavior of realizability for negative and disjunctive formulas. In the next lemma, we connect the realizability to the validity in the standard model $\mathfrak{N}$. First, we need a definition. In the language of arithmetic, a type is called a \emph{$0$-type} if it is only constructed from the basic type $N$ using $\times$ but not $\to$. It is called a \emph{$1$-type}, if it is either a $0$-type or it is in the form $\sigma \to \tau$, for two $0$-types $\sigma$ and $\tau$. A formula is called a \emph{$0$-formula} (resp. \emph{$1$-formula}) if all terms in it (including variables, constants and function symbols) are of $0$-types (resp. $1$-types). A formula is called \emph{standard} if $|\!|M^{\sigma}|\!|=\mathfrak{N}^{\sigma}$, for any type $\sigma$ occurring in $A$ as the type of a term in it. For instance, any $0$-formula is a standard formula. To have another example, as the exponentiation $[K, L]$ exists in $\mathbf{Top}$, for any finite powers $K$ and $L$ of $\mathbb{N}$, the exponentiation of $[(K, =), (L, =)]$ in $\mathbf{Equ}$ is $([K, L], =)$. Additionally, as $K$ is discrete, all functions from $K$ to $L$ are continuous. Hence, if $\mathcal{C}=\mathbf{Equ}$, for any $1$-type $\sigma$, we have $|\!|M^{\sigma}|\!|=\mathfrak{N}^{\sigma}$. Therefore, all $1$-formulas are standard, when $\mathcal{C}=\mathbf{Equ}$.

\begin{rem}\label{RealizibalityForEqu}
Note that for $\mathcal{C}=\mathbf{Equ}$ and any $1$-types $\sigma$ and $\tau$, any $f \in |\!|M^{\sigma \to \tau}|\!|$ is actually a continuous function $f: \mathfrak{N}^{\sigma} \to \mathfrak{N}^{\tau}$ realized by itself as a map in $\mathbf{Equ}$. Therefore, we can write $f \Vdash_{M^{\sigma \to \tau}} f$.   
\end{rem}

\begin{lem}(Realizability vs truth) \label{TtuthLemma} 
Let $A$ be an $\exists$-free standard sentence. Then,
$\Vdash A$ iff $\mathfrak{N} \vDash A$. Especially, for any $\exists$-free $0$-sentence $A$, we have $\Vdash A$ iff $\mathfrak{N} \vDash A$.
\end{lem}
\begin{proof}
First, recall from Remark \ref{RemarksOnReal} that for any $\exists$-free sentence $A$, there is only one map in $\mathrm{Hom}_{\mathcal{C}}(\mathtt{1}, |A|)$ that we denote by $\theta_A$. Therefore, we have to show that $\theta_A \Vdash A$ iff $\mathfrak{N} \vDash A$. The proof is by induction on $A$. For the atomic formulas, the claim holds by definition and the fact that for any type in $A$ we have $|\!|M^{\sigma}|\!|=\mathfrak{N}^{\sigma}$. The conjunction case is easy. For the implication $A=B \to C$, if $\theta_{B \to C} \Vdash B \to C$ and $\mathfrak{N} \vDash B$, then by the induction hypothesis, $\theta_B \Vdash B$. Therefore, by definition, $\theta_{B \to C} \cdot \theta_B \Vdash C$. As there is just one map in $\mathrm{Hom}(\mathtt{1}, |C|)$, we have $\theta_{B \to C} \cdot \theta_B =\theta_C$. Hence, $\theta_C \Vdash C$. Again, by the induction hypothesis, $\mathfrak{N} \vDash C$.  Therefore, $\mathfrak{N} \vDash B \to C$. For the other direction, assume $\mathfrak{N} \vDash B \to C$. To prove $\theta_{B \to C} \Vdash B \to C$, we must show that $\theta_{B \to C} \cdot \mathtt{r} \Vdash C$, for any $\mathtt{r} \Vdash B$. Since, there is only one map in $\mathrm{Hom}(\mathtt{1}, |B|)$, we have $\mathtt{r}=\theta_B$. Since $\theta_B \Vdash B$, by the induction hypothesis $\mathfrak{N} \vDash B$. As $\mathfrak{N} \vDash B \to C$, we have $\mathfrak{N} \vDash C$. By the induction hypothesis, $\theta_C \Vdash C$. As there is only one map in $\mathrm{Hom}(\mathtt{1}, |C|)$, we have $\theta_{B \to C} \cdot \mathtt{r}=\theta_C$. Hence, $\theta_{B \to C} \cdot \mathtt{r}=\theta_C \Vdash C$. Therefore, $\theta_{B \to C} \Vdash B \to C$. For the universal quantifier $A=\forall x^{\sigma} B(x)$, if $\theta_{\forall x^{\sigma} B(x)} \Vdash \forall x^{\sigma} B(x)$, then for any $a \in |\!|M^{\sigma}|\!|=\mathfrak{N}^{\sigma}$, pick $\mathtt{a}: 1 \to |M^{\sigma}|$ such that $\mathtt{a} \Vdash_{M^{\sigma}} a$. Then, $\theta_{\forall x^{\sigma} B(x)} \cdot \mathtt{a} \Vdash B(a)$. By the induction hypothesis, $\mathfrak{N} \vDash B(a)$, for any $a \in \mathfrak{N}^{\sigma}$. Hence, $\mathfrak{N} \vDash \forall x^{\sigma} B(x)$. Conversely, if $\mathfrak{N} \vDash \forall x^{\sigma} B(x)$ then $\mathfrak{N} \vDash B(a)$, for any $a \in \mathfrak{N}^{\sigma}=|\!|M^{\sigma}|\!|$. By the induction hypothesis $\theta_{B(a)} \Vdash B(a)$. Now, similar to the argument for the implication, we have $\theta_{\forall x^{\sigma} B(x)} \Vdash \forall x^{\sigma} B(x)$. For the second part, note that if $A$ is a $0$-sentence, all types $\sigma$ occurring in $A$ are in the form $N \times N \times \ldots \times N$. Hence, $|\!|M^{\sigma}|\!|=\mathbb{N}^k$. 
\end{proof}

\begin{exam}(\emph{Non-classicality of realizability})\label{NonClassRealizability}
We saw that $A \vee \neg A$ is always realizable, for any sentence $A$. In this example, we show that $A(x) \vee \neg A(x)$ is \emph{not} necessarily realizable if $A(x)$ is a formula with free variables. For instance, let $A(x)=\forall w^Ny^N \neg T(x^N, x^N, w^N, y^N)$ stating that if the code $x$ of a Turing machine applies to $x$ itself, the computation does not halt. Set $\mathcal{C}=\mathbf{Rec}$.
We claim that $ \nVdash A(x) \vee \neg A(x)$.
To see why, assume $ \Vdash A(x) \vee \neg A(x)$. By Example \ref{RealizabilityForDisjunction}, there is a map $\mathtt{r}: \mathtt{N} \to \mathtt{N} $ in $\mathbf{Rec}$ such that for any natural number $n \in \mathbb{N}$, either $\mathtt{r} \circ \mathtt{n}=\mathtt{0}$ and $\Vdash A(n)$ or $\mathtt{r} \circ \mathtt{n}=\mathtt{m}$, for some $m \geq 1$ and $\Vdash \neg A(n)$. As both $A(n)$ and $\neg A(n)$ are $\exists$-free $0$-sentences, by Lemma \ref{TtuthLemma}, realizability coincides with the truth. Therefore, as $\mathtt{m} \neq \mathtt{0}$, for any $m \geq 1$, deciding $\mathtt{r} \circ \mathtt{n}=\mathtt{0}$ also decides the truth of $A(n)$ which is the halting problem for the input $n$. As $\mathtt{r}: \mathtt{N} \to \mathtt{N}$ is a map in $\mathbf{Rec}$, it is a total computable function. Hence, $\mathtt{r} \circ \mathtt{n}=\mathtt{0}$ is a decidable predicate for the input $n$. This implies that the halting problem is also decidable which is a contradiction. Hence, $ \nVdash A(x) \vee \neg A(x)$.
\end{exam}

\begin{exam}(\emph{Undecidability of equality})
Let $x$ and $y$ be two variables of type $N \to N$. We show that $(x=y) \vee (x \neq y)$ is not necessarily realizable. Set $\mathcal{C}=\mathbf{Equ}$ and assume $\Vdash (x=y) \vee (x \neq y)$. Then, by Example \ref{RealizabilityForDisjunction}, there is a map $\mathtt{r}: [\mathtt{N}, \mathtt{N}] \times [\mathtt{N}, \mathtt{N}] \to \mathtt{N}$ in $\mathbf{Equ}$ such that 
for any functions $\alpha, \beta: \mathbb{N} \to \mathbb{N}$, if we also read them as maps in $\mathbf{Equ}$ realizing themselves, either $\mathtt{r} \circ \langle \alpha, \beta \rangle=\mathtt{0}$ and $\Vdash \alpha =\beta$ or $\mathtt{r} \circ \langle \alpha, \beta \rangle=\mathtt{m}$, for some $m \geq 1$ and $\Vdash \alpha \neq \beta$.
As both $\alpha=\beta$ and $\alpha \neq \beta$ are $\exists$-free $1$-sentences, by Lemma \ref{TtuthLemma}, realizability coincides with the truth. Therefore, as $\mathtt{m} \neq \mathtt{0}$, for any $m \geq 1$, deciding  $\mathtt{r} \circ \langle \alpha, \beta \rangle=\mathtt{0}$ also decides the equality $\alpha = \beta$. However, $\mathtt{r}: [\mathtt{N}, \mathtt{N}] \times [\mathtt{N}, \mathtt{N}] \to \mathtt{N}$ is a map in $\mathbf{Equ}$. Hence, it is a continuous function and as $\{0\}$ is open in the discrete topology of $\mathbb{N}$, the subset $\Delta=\{(\alpha, \alpha) \mid \alpha: \mathbb{N} \to \mathbb{N}\}$ must be open in $\mathbb{N}^{\mathbb{N}} \times \mathbb{N}^{\mathbb{N}}$. Let $0^{\omega}$ be the constant zero function on $\mathbb{N}$. As $\Delta$ is open, there is a basic open neighbourhood $U=\prod_{i=0}^{\omega} U_i \subseteq \Delta$ around $(0^{\omega}, 0^{\omega})$. As $U_i=\mathbb{N}$ except for finitely many $i$'s, $U$ always intersects with $\mathbb{N}^{\mathbb{N}}\times \mathbb{N}^{\mathbb{N}}-\Delta$ which is a contradiction. Therefore, $\Delta$ is not open and hence $\nVdash (x=y) \vee (x \neq y)$. 
\end{exam}

So far, we have seen some examples of (un)realizability of arithmetical formulas. Next theorem shows that any provable sentence in $\mathsf{HA^{\omega}+EXT}$ is realizable. To extend the same claim to cover the axiom of choice, we need the following definition:

\begin{dfn}
The $\mathrm{NNO}$ $(\mathtt{N}, \mathtt{Z}, \mathtt{s})$ in a $\mathrm{CCC}$ is called \emph{total} if for any map $\mathtt{a}: \mathtt{1} \to \mathtt{N}$, there is a natural number $n \in \mathbb{N}$ such that $\mathtt{a}=\mathtt{n}$.
\end{dfn}

For instance, the $\mathrm{NNO}$ in the categories $\mathbf{Set}$, $\mathbf{Rec}$ and $\mathbf{Equ}$ is total.

\begin{thm}\label{ArithmeticalSoundness}
Let $\mathcal{C}$ be a $\mathrm{CCC}$ with an $\mathrm{NNO}$. If $\mathsf{HA^{\omega}+EXT} \vdash A$, then $\Vdash A$.
If $\mathtt{2}$ exists in $\mathcal{C}$, the same also holds for $\mathsf{HA^{\omega}_2+EXT}$. If $\mathcal{C}$ is also well-pointed non-preorder category with the total $\mathrm{NNO}$, we have a similar claim for $\mathsf{HA^{\omega}+EXT+AC}$. If $M^\mathtt{2}$ is also total, we have the same for $\mathsf{HA^{\omega}_2+EXT+AC}$.
\end{thm}
\begin{proof}
Using Theorem \ref{Soundness}, the logical part is realizable. The only thing to prove is the realizability of the non-logical axioms of the theories. The realizability of the defining axioms for the function symbol $\mathbf{R}^{\sigma}$ is easy to check. For $S(0) \neq 0$ and $\forall x^N (S(x)=S(y) \to (x=y))$, as they are $0$-sentences, their truth implies their realizability, by Lemma \ref{TtuthLemma}. For the induction 
\[
A(0) \wedge \forall x^N (A(x) \to A(Sx)) \to \forall x^N A(x),
\]
set $X=|A| \times [N, [|A|, |A|]]$. It is enough to find a map $\mathtt{f}: N \times X \to |A|$ such that for any $\mathtt{r}: 1 \to X$ and any $n \in \mathbb{N}$, if $\mathtt{r} \Vdash A(0) \wedge \forall x^N (A(x) \to A(Sx))$, then $\mathtt{f} \circ \langle \mathtt{n}, \mathtt{r} \rangle \Vdash A(n)$.
Consider $\mathtt{p_0}: X \to |A|$ and $\mathtt{p_1}: X \to [N, [|A|, |A|]]$ as the projection maps
and $\mathtt{e}: N \times [N, [|A|, |A|]] \times |A| \to |A|$ as the canonical evaluation map.
Then, define $\mathtt{h}: N \times X \times |A| \to |A|$ as: 
\[\begin{tikzcd}[ampersand replacement=\&]
	{N \times X \times |A|} \&\&\& {N \times [N, [|A|, |A|]] \times |A|} \&\& {|A|}
	\arrow["{\mathtt{id}_N \times \mathtt{p_1} \times \mathtt{id}_{|A|}}"', from=1-1, to=1-4]
	\arrow["{\mathtt{e}}"', from=1-4, to=1-6]
\end{tikzcd}\]
Using Lemma \ref{PrimitiveRecursion}, there is a map $\mathtt{f}: N \times X \to |A|$ such that
\[\begin{tikzcd}[ampersand replacement=\&]
	X \&\& {N \times X} \& {N \times X} \&\& {N \times X} \\
	\\
	\&\& {|A|} \& {N \times X \times |A|} \&\& {|A|}
	\arrow["{\langle \mathtt{Z} \circ !, \mathtt{id}_X \rangle}", from=1-1, to=1-3]
	\arrow["{\mathtt{p_0}}"', from=1-1, to=3-3]
	\arrow["{\mathtt{f}}", dashed, from=1-3, to=3-3]
	\arrow["{\mathtt{s} \times \mathtt{id}_X}", from=1-4, to=1-6]
	\arrow["{\langle \mathtt{id}_{N \times X}, \mathtt{f} \rangle}"', dashed, from=1-4, to=3-4]
	\arrow["{\mathtt{f}}", dashed, from=1-6, to=3-6]
	\arrow["{\mathtt{h}}"', from=3-4, to=3-6]
\end{tikzcd}\]
We claim that this $\mathtt{f}$ works. First, for any map $\mathtt{r}: 1 \to X$ and any number $n \in \mathbb{N}$, note that 
$\mathtt{f} \circ \langle \mathtt{0}, \mathtt{r} \rangle=\mathtt{p_0r}$ and $\mathtt{f} \circ \langle \mathtt{n+1}, \mathtt{r} \rangle=(\mathtt{p_1r} \cdot \mathtt{n}) \cdot \mathtt{f}(\mathtt{n}, \mathtt{r})$. Now, if $\mathtt{r} \Vdash A(0) \wedge \forall x (A(x) \to A(S(x)))$, then by induction on natural numbers, it is easy to see that $\mathtt{f} \circ \langle \mathtt{n}, \mathtt{r} \rangle \Vdash A(n)$, for any natural number $n \in \mathbb{N}$. 

For the second part, as the $\mathrm{NNO}$ is total, the $\mathcal{C}$-set $N$ is total. As $\mathcal{C}$ is not a preordered set, the $\mathcal{C}$-set $N$ is projective, by Lemma \ref{UniqunessofNumeralRepresentation}. Now, as $\mathcal{C}$ is well-pointed, the standard $\mathcal{C}$-interpretation maps all types to projective total $\mathcal{C}$-sets, by Lemma \ref{TotalProjectiveLemma}. Hence, by Theorem \ref{AxiomOfChoice}, the axiom of choice is realized.
The case involving the type $2$ is easy. In the second case, it is enough to use an argument similar to the proof of part $(iii)$ of Lemma \ref{DisjointnessInBCC} to show that $M^\mathtt{2}$ is projective. The rest is similar to the previous case.
\end{proof}

Using Theorem \ref{ArithmeticalSoundness}, one can easily establish some consistency and unprovablity statements. First, we show that the negation of some classical tautologies are consistent with $\mathsf{HA}^{\omega}+\mathsf{EXT+AC}$.

\begin{cor}\label{FailureOfExcludedMiddle}
There are $0$-formulas $A(x)$ and $B(x)$ such that $\neg \forall x^N (A(x) \vee \neg A(x))$ and $\neg[\forall x^N \neg \neg B(x) \to \neg \neg \forall x^N B(x)]$ are consistent with $\mathsf{HA}^{\omega} + \mathsf{EXT+AC}$.
\end{cor}
\begin{proof}
Set $A(x)=\forall w^N y^N \neg T(x, x, w, y)$ and $B(x)=(A(x) \vee \neg A(x))$. As $\mathbf{Rec}$ is a well-pointed non-preorder category with the total $\mathrm{NNO}$, we have the soundness of $\mathsf{HA}^{\omega} + \mathsf{EXT+AC}$. Moreover, in Example \ref{NonClassRealizability}, we showed that $\forall x^N (A(x) \vee \neg A(x))$ is not realizable for the standard $\mathbf{Rec}$-interpretation. Hence, by Example \ref{RealizabilityForNegation}, we reach $\Vdash \neg \forall x^N (A(x) \vee \neg A(x))$. For the other part, we claim that $\nVdash \forall x^N \neg \neg B(x) \to \neg \neg \forall x^N B(x)$. Assume otherwise. Then, as $\forall x^N \neg \neg B(x)$ is an intuitionistic tautology, we have $\Vdash \forall x^N \neg \neg B(x)$. Therefore, we must have $\Vdash \neg \neg \forall x^N B(x)$ which is equivalent to $\Vdash \forall x^N B(x)$. However, we showed that $\nVdash \forall x^N (A(x) \vee \neg A(x))$.
\end{proof}

The next consistency statement is the consistency of Church-Turing thesis:
\begin{cor}\label{ConsistencyOfCT}
The theory $\mathsf{HA^{\omega}+EXT+AC+CT_{\neg \neg}}$ is consistent. 
\end{cor}
\begin{proof}
Set $\mathcal{C}=\mathbf{Rec}$. As $\mathbf{Rec}$ is a well-pointed non-preorder category with the total $\mathrm{NNO}$, using Theorem \ref{ArithmeticalSoundness}, it is enough to show that
\[
\mathsf{CT}_{\neg \neg}: \quad \quad \forall f^{N \to N} \neg \neg \exists e^N \forall x^N \exists w^N \exists y^N [T(e, x, w, y) \wedge (fx=y)],
\]
is realizable. For that purpose, first notice that the elements in $|\!|[N, N]|\!|$ are just total recursive functions. Now, using Remark \ref{OnRealizabilityOfDoubleNeg}, it is enough to show that for any total recursive function $g: \mathbb{N} \to \mathbb{N}$, the statement 
\[
\exists e^N \forall x^N \exists w^N \exists y^N [T(e, x, w, y) \wedge (gx=y)],
\]
is realizable. As $g$ is computable, it has a code $m \in \mathbb{N}$. It is easy to show that $\Vdash \forall x^N \exists w^N \exists y^N [T(m, x, w, y) \wedge (gx=y)]$.
\end{proof}

\begin{rem}
As explained before, to prove the consistency of the anti-classical theory $\mathsf{HA^{\omega}+EXT+AC+CT_{\neg \neg}}$, one cannot use the usual classical image. However, as observed in Corollary \ref{ConsistencyOfCT}, one can take an alternative world, i.e., the computable world $\mathbf{Rec}$ and interpret arithmetical proofs in that world. In fact, the possibility of realizing Church-Turing thesis was quite expected in such a world as every ``function" in that world is computable. One can go even further to say that the theory $\mathsf{HA^{\omega}+EXT+AC+CT_{\neg \neg}}$ with its Church-Turing thesis is actually a device to describe such computable worlds.
\end{rem}

To have a similar consistency result for $\mathsf{HA^{\omega}+EXT+AC+WC\!-\!N_{\neg \neg}}$, we need the following lemma:

\begin{lem} \label{OnTheExistenceOftheModule}
\begin{description}
\item[$(i)$]
There exists a function $M: [\mathbb{N}^\mathbb{N}, \mathbb{N}] \times \mathbb{N}^\mathbb{N} \to \mathbb{N}$ that reads continuous functions $f: \mathbb{N}^\mathbb{N} \to \mathbb{N}$ and $\alpha: \mathbb{N} \to \mathbb{N}$ and returns a number $M(f, \alpha)$ such that for any $\beta: \mathbb{N} \to \mathbb{N}$, if $\alpha=_{M(f, \alpha)} \beta$, then $f(\alpha)=f(\beta)$. 
\item[$(ii)$]
There exists a \emph{continuous} function $M: [2^\mathbb{N}, \mathbb{N}] \to \mathbb{N}$ that reads a continuous function $f: 2^\mathbb{N} \to \mathbb{N}$ and returns a number $M(f)$ such that for any $\alpha, \beta: \mathbb{N} \to 2$, if $\alpha=_{M(f)} \beta$, then $f(\alpha)=f(\beta)$.
\end{description}
\end{lem}
\begin{proof}
First, recall from Example \ref{TopAndExp} that the function space $\mathbb{N}^{\mathbb{N}}$ as a topological space is the same as the product space $\mathbb{N}^{\omega}$. Therefore, we work with the latter as it is easier to work with. For $(i)$, for any $f$ and $\alpha$, as $f$ is continuous, $f^{-1}(\{f(\alpha)\})$ must be open. Hence, $\alpha \in f^{-1}(\{f(\alpha)\})$ has a neighbourhood in the form $\{\alpha_0\} \times \cdots \times \{\alpha_{n-1}\} \times \mathbb{N}^{\omega}$  around $\alpha$ in $f^{-1}(f(\alpha))$, where $\alpha_i \in \mathbb{N}$ is the $i$-th argument of $\alpha$. Therefore, for any $\beta \in \mathbb{N}^{\mathbb{N}}$, if $\alpha=_n \beta$, we have $f(\alpha)=f(\beta)$. Define $M(f, \alpha)$ as the least such number $n$.
By definition, it is clear that $M(f, \alpha)$ has the required property. 

For $(ii)$, for any finite string $w=w_0w_1 \ldots w_n$ of zeros and ones, define $K_{w}=\{w_0\} \times \cdots \times \{w_n\} \times \{0, 1\}^{\omega}$ as a subset of $2^\mathbb{N}$ and notice that it is both open and compact. Now, as the first step, we prove that for any $f: 2^{\mathbb{N}} \to \mathbb{N}$, there is at least one $m \in \mathbb{N}$ such that
for any $\alpha, \beta \in 2^\mathbb{N}$, if $\alpha=_m \beta$, then $f(\alpha)=f(\beta)$. Note that $2^{\mathbb{N}}$ is a compact space and as $f$ is continuous, $I=f[2^{\mathbb{N}}]$ is a compact subspace of the discrete space $\mathbb{N}$. Hence, it is finite. Moreover, as $\{k\}$ is both open and closed, then so is $f^{-1}(k)$, for any $k \in I$. Therefore, $f^{-1}(k)$ is compact. As it is open, it is a union of opens in the form $K_w$, for some finite strings $w$. By compactness, we can assume that the number of $w$'s is finite. Hence, $f^{-1}(k)=\bigcup_{i=1}^{n_k} K_{w_{k, i}}$. Pick $m$ as the maximum of the lengths of all $w_{k, i}$'s, for all $k \in I$ and $i \leq n_k$. Therefore, it is clear that if $\alpha =_m \beta$, then $\alpha \in K_{w_{k, i}}$ iff $\beta \in K_{w_{k, i}}$ which implies $\alpha \in f^{-1}(k)$ iff $\beta \in f^{-1}(k)$ and hence $f(\alpha)=f(\beta)$. 
This completes the proof of claim.
Now, define $M(f)$ as the least such number $m$. By definition, it is clear that $M(f)$ has the claimed property. The only thing to prove is the continuity of $M$. As $M(f)$ is the least possible number with the property
\[
\forall \alpha \beta \in 2^{\mathbb{N}} \, [\alpha=_m \beta \to f(\alpha)=f(\beta)],
\]
it is easy to see that
$M^{-1}(m)$ is the set of $f: 2^{\mathbb{N}} \to \mathbb{N}$ such that $f$ is constant over $K_w$, for all $w \in  \{0, 1\}^m$ but there are distinct $\alpha, \beta \in 2^{\mathbb{N}}$ such that $\alpha=_{m-1} \beta$, while $f(\alpha) \neq f(\beta)$. Therefore, one can represent $M^{-1}(m)$ as the intersection of 
\[
\bigcup_{n \in \mathbb{N}} \bigcap_{w \in \{0, 1\}^m} \{f: 2^{\mathbb{N}} \to \mathbb{N} \mid f[K_w]=\{n\} \}
\]
and
\[
\bigcup_{w \in \{0, 1\}^{m-1}} \bigcup_{\alpha \neq \beta \in K_{w}} \, \bigcup_{n \neq n' \in \mathbb{N}} [\{f: 2^{\mathbb{N}} \to \mathbb{N} \mid f(\alpha)=n\} \cap \{f: 2^{\mathbb{N}} \to \mathbb{N} \mid f(\beta)=n'\}].
\]
Note that the topology of $[2^{\mathbb{N}}, \mathbb{N}]$ is the usual compact-open topology, as observed in Example \ref{TopAndExp}. Therefore,
using the continuity of the evaluation function and the compactness of $K_w$, it is easy to prove that $M^{-1}(m)$ is open. Hence, $M$ is continuous.
\end{proof}

\begin{cor} The theory
$\mathsf{HA_2^{\omega}+EXT+AC}$ is consistent with the axioms $\WCN_{\neg\neg}$ and
$\UWCN$, where
\[
\UWCN: \quad \forall f^{2^N \to N} \exists x^N \forall \alpha^{N\to 2} \beta^{N \to 2} (\alpha=_{x} \beta \to f\alpha=f\beta),
\]
is the \emph{uniform continuity principle}, stating that for any map $f: 2^N \to N$, the value of $f$ only depends on finitely many of the values of its input. 
\end{cor}
\begin{proof}
Set $\mathcal{C}=\mathbf{Equ}$. Note that $\mathcal{C}$ is a non-preorder well-pointed category and its $\mathrm{NNO}$ is total. Moreover, $M^{\mathtt{2}}$ is total. Therefore, it is enough to show that both $\WCN_{\neg \neg}$ and $\UWCN$ are realizable. For $\WCN_{\neg \neg}$, using Remark \ref{OnRealizabilityOfDoubleNeg}, it is enough to show that for any continuous functions $f: \mathbb{N}^{\mathbb{N}} \to \mathbb{N}$ and any function $\alpha:\mathbb{N} \to \mathbb{N}$, the statement $\exists x^N \forall \beta^{N \to N} [\alpha=_x \beta \to f\alpha = f\beta]$ is realizable. Using Lemma \ref{OnTheExistenceOftheModule}, we know that $M(f, \alpha)$ provides the witness for $x$. The sentence $\forall \beta^{N \to N} [\alpha=_{M(f, \alpha)} \beta \to f\alpha = f\beta]$ is a true $\exists$-free $1$-sentence. Hence, its truth implies its realizability. For $(ii)$, it is enough to use the function $M$ introduced in Lemma \ref{OnTheExistenceOftheModule} for $2^{\mathbb{N}}$. Note that $M$ is continuous.
\end{proof}

\begin{rem}
Similar to the case of Church-Turing thesis, to prove the consistency of the anti-classical theory $\mathsf{HA^{\omega}+EXT+AC+WC-N_{\neg \neg}}$, we can interpret arithmetical proofs in an alternative world, i.e., the continuous world $\mathbf{Equ}$. Again, the interpretation was quite expected and one can say that the theory with its continuity principle is actually a device to describe such continuous worlds.
\end{rem}

In the rest of this subsection, we use the realizability machinery to extract some computational information from the arithmetical proofs. First, recall the notion of representable functions defined in Definition \ref{DefRepresentableFunctions}. Then, we have: 

\begin{thm}(Function extraction) \label{FunctionExtraction}
Let $A(x, y)$ be an $\exists$-free $0$-formula and $\mathsf{HA^{\omega}+EXT} \vdash \forall x^N \exists y^{N} A(x,y)$. Then, for any non-preorder $\mathrm{CCC}$ $\mathcal{C}$ with an $\mathrm{NNO}$, there exists a function $f: \mathbb{N} \to \mathbb{N}$ representable in $\mathcal{C}$ such that $\mathfrak{N} \vDash \forall x^N A(x,f(x))$. The same also holds for $\mathsf{HA}$ and if $A(x, y)$ is bounded, for $\mathsf{PA}$.
\end{thm}
\begin{proof}
For the first part, notice that $|\forall x^N \exists y^N A(x,y)|=[\mathtt{N}, \mathtt{N} \times |A|]$. Therefore, by Theorem \ref{ArithmeticalSoundness}, there is a map 
$
\mathtt{r}: \mathtt{1} \to [\mathtt{N}, \mathtt{N} \times |A|]
$
such that $\mathtt{r} \Vdash \forall x^N \exists y^N A(x,y)$. Therefore, for any natural number $n \in \mathbb{N}$, as $\mathtt{n} \Vdash n$, we have $\mathtt{r} \cdot \mathtt{n} \Vdash \exists y^N A(n, y)$ which means that there is a natural number $m \in \mathbb{N}$ such that $\mathtt{p_0(\mathtt{r} \cdot n)} \Vdash_{N} m$ and $\mathtt{p_1(r \cdot n)} \Vdash A(n, m)$. As $A(n, m)$ is an $\exists$-free $0$-sentence, by Lemma \ref{TtuthLemma}, we have $\mathfrak{N} \vDash A(n, m)$.
As $\mathcal{C}$ is non-preorder, the number $m \in \mathbb{N}$ is uniquely determined by $\mathtt{p_0(\mathtt{r} \cdot n)}$. Define the function $f: \mathbb{N} \to \mathbb{N}$ by $f(n)=m$. It is clear that $A(n, f(n))$ holds in the standard model, for any $n \in \mathbb{N}$. The only thing to note is that the function $f: \mathbb{N} \to \mathbb{N}$ is representable by the map $\mathtt{p_0 \circ (r \cdot id_N)}$.
For the second part, it is enough to embed $\mathsf{HA}$ into $\mathsf{HA}^{\omega}$ and note that any formula in the language of $\mathsf{HA}$ is translated into a $0$-formula as its terms are all of type $N$. For $\mathsf{PA}$, use the
conservativity of $\mathsf{PA}$ over $\mathsf{HA}$ for the formulas in the form $\forall x \exists y A(x, y)$, where $A$ is bounded and the fact that any bounded formula is equivalent to an $\exists$-free formula in $\mathsf{HA}$.
\end{proof}

\begin{cor}
If $\mathsf{HA^{\omega}+EXT} \vdash \forall x^N \exists y^N A(x,y)$, for an $\exists$-free $0$-formula $A(x, y)$, then there exists a recursive function $f: \mathbb{N} \to \mathbb{N}$ such that $\mathfrak{N} \vDash \forall x^N A(x,f(x))$. The same also holds for $\mathsf{HA}$.
\end{cor}
\begin{proof}
Set $\mathcal{C}=\mathbf{Rec}$ in Theorem \ref{FunctionExtraction} and recall that only recursive functions are representable in $\mathbf{Rec}$.
\end{proof}

To find the strongest possible application of Theorem \ref{FunctionExtraction}, we must work with the ``least" possible $\mathrm{CCC}$ with an $\mathrm{NNO}$, i.e., the free category $\Lambda_N$.

\begin{cor}\label{FunctionsOfHA}
If $\mathsf{HA^{\omega}+EXT} \vdash \forall x^N \exists y^N A(x,y)$, for an $\exists$-free $0$-formula, then there exists a primitive recursive functional $f: \mathbb{N} \to \mathbb{N}$ such that $\mathfrak{N} \vDash \forall x^N A(x,f(x))$. The same also holds for $\mathsf{HA}$ and if $A(x, y)$ is bounded, for $\mathsf{PA}$.
\end{cor}
\begin{proof}
Set $\mathcal{C}=\Lambda_N$ in Theorem \ref{FunctionExtraction}. By Remark \ref{LambdaNIsNotPreorder}, $\Lambda_N$ is not a preorder. Therefore, it is enough to recall that the numeral primitive recursive functionals are by definition the functions representable in the category $\Lambda_N$.  
\end{proof}

Finally, we can apply the extraction technique to the higher functionals:

\begin{cor}
If $\mathsf{HA^{\omega}+EXT+AC} \vdash \forall x^{N \to N} \exists y^N A(x, y)$, for an $\exists$-free $1$-formula $A(x, y)$, then there exists a continuous function $f: \mathbb{N}^{\mathbb{N}} \to \mathbb{N}$ such that $\mathfrak{N} \vDash A(\alpha,f(\alpha))$, for any function $
\alpha : \mathbb{N} \to \mathbb{N}$.
\end{cor}
\begin{proof}
As $\mathbf{Equ}$ is a well-pointed non-preorder with the total $\mathrm{NNO}$, in Theorem \ref{ArithmeticalSoundness}, set $\mathcal{C}=\mathbf{Equ}$. Notice that $|\forall x^{N \to N} \exists y^N A(x,y)|=[\mathtt{N}^{\mathtt{N}}, \mathtt{N} \times |A|]$. Therefore, there is a map $\mathtt{r}: \mathtt{1} \to  [\mathtt{N}^{\mathtt{N}}, \mathtt{N} \times |A|]$ such that $\mathtt{r} \Vdash \forall x^{N \to N} \exists y^N A(x,y)$. By definition, for any function $\alpha: \mathbb{N} \to \mathbb{N}$, we have $\mathtt{r} \cdot \alpha \Vdash \exists y^N A(\alpha, y)$ which means that there is a natural number $m \in \mathbb{N}$ such that $\mathtt{p_0}(\mathtt{r} \cdot \alpha) \Vdash_N m$ and $\mathtt{p_1}(\mathtt{r} \cdot \alpha) \Vdash A(\alpha, m)$. As $\mathbf{Equ}$ is not a preordered set, the value of $m$ is uniquely determined by $\mathtt{p_0}(\mathtt{r} \cdot \alpha)$.  Define $f: \mathbb{N}^{\mathbb{N}} \to \mathbb{N} $ by $f(\alpha)=m$. Note that $f$ is a continuous map, as it is realized by $\mathtt{p_0}(\mathtt{r} \cdot \alpha)$ living in $\mathbf{Equ}$. As $A(\alpha, f(\alpha))$ is an $\exists$-free $1$-formula and hence standard, by Lemma \ref{TtuthLemma}, we have $\mathfrak{N} \vDash A(\alpha, f(\alpha))$, for any function $\alpha: \mathbb{N} \to \mathbb{N}$.
\end{proof}

\bibliographystyle{plain}
\bibliography{main}

\end{document}